\documentclass[11pt,a4paper,leqno]{amsart}
\usepackage{amsmath,amsfonts,amssymb,amsthm, color}
\usepackage[T1]{fontenc}
\usepackage{verbatim}

\usepackage[active]{srcltx}
\usepackage{mathrsfs} 
\advance\textwidth by 2.5cm 
\advance\oddsidemargin by -1.0cm \advance\evensidemargin by -1.0cm

\newcommand\del[1]{}
\newcommand\think[1]{}
\newcommand\new[1]{}
\newcommand\zus[1]{}

\newcommand\comd[1]{} 
\newcommand\Redd[1]{} 

\def\bdm{\begin{displaymath}}
\def\edm{\end{displaymath}}
\def\bea{\begin{eqnarray}}
\def\eea{\end{eqnarray}}

\newtheorem{theorem}{Theorem}[section]

\newtheorem{lem}[theorem]{Lemma}
\newtheorem{defn}[theorem]{Definition}
\newtheorem{prop}[theorem]{Proposition}
\newtheorem{coro}[theorem]{Corollary}

\newtheorem{remark}{Remark}

\numberwithin{equation}{section}

\begin{document}

\date{\today}

\title[FSNSEs  \today]
{Large deviations for 2D-fractional stochastic Navier-Stokes equation on the torus.\\-Short Proof-}

\vspace{-0.1cm}

\author[L.{} Debbi]{Latifa Debbi.\\ {}\\  \del{\SMALL This work is supported by  Alexander von Humboldt Foundation.}}

\del{\email{ldebbi@ymath.uni-bielefeld.de}
\address{Fakultat fur Mathematik, Universitat Bielefeld. Universitatsstratsse 25, 33615, Bielefeld. Germany.}}

\address{Department of Mathematics, University of York, Heslington Road, York YO10 5DD, UK.}

\email{latifa.debbi@york.ac.uk}

\maketitle
\vspace{-2cm}
\begin{abstract}
In this note, we prove the large deviation principle for the 2D-fractional stochastic Navier-Stokes equation on the torus under the dissipation order $ \alpha \in [\frac43, 2]$.  \\

\vspace{-0.25cm}
{\bf R\'esum\'e.}

Dans cette note, on presente une demonstration courte du principe  des garandes d\'eviations pour l'equation fractionnaire stochastique de Navier-Stokes definie sur le 2D-torus sous l'order de dissipation $ \alpha \in [\frac43, 2]$.

Keywords: Fractional stochastic Navier-Stokes equation, 
Q-Wiener process, weak-strong solutions, large deviation principle, farctional Sobolev spaces.

Subjclass[2000]: {58J65, 60H15, 35R11.}
\end{abstract}

\vspace{-0.25cm}
\section{Introduction}\label{sec-intro}

In this note, we present a short proof of the large deviation principle (LDP) for the 2D-fractional stochastic Navier-Stokes equation (2D-FSNSE) on the torus:\del{ $ \mathbb{T}^2$, }
\del{\begin{equation}\label{Eq-classical-SNSE-O-p}
\Bigg\{
\begin{array}{lr}
\partial_tu= -\nu (-\Delta)^\frac\alpha2 u + (u.\nabla) u - \nabla\pi + G(t, u)\frac{\partial^2}{\partial t\partial x}W(t), \;\; t>0, \\
div u=0,\;\;\; \del{ \text{(incompressible condition)},}\\
u(0) = u_0.\\
\end{array}
\end{equation}}
\begin{equation}\label{Eq-classical-SNSE-O-p}
\Bigg\{
\begin{array}{lr}
du= (-\nu (-\Delta)^\frac\alpha2 u + (u.\nabla) u - \nabla\pi)dt + G(t, u)dW(t), \;\; t>0, \\
div u=0,\;\;\; \del{ \text{(incompressible condition)},}\\
u(0) = u_0,\\
\end{array}
\end{equation}
where the unknown is the vector $ u =(u_1, u_2 )$ and the scalar $\pi$, $ \nu>0$ is the coefficient of viscosity, $ (-\Delta)^\frac\alpha2, \; \alpha \in (0, 2]$ is the fractional power of minus-Laplcian, $ W$ is an external Gaussian noise, $ G$ and $ u_0$ are a diffusion term respectively an initial data, both to be precise later. Physically, NSE on the torus is not a realistic model, but it is used for some idealizations and for homogenization problems in turbulence see e.g \cite{Foias-book-2001}. Recently, the deterministice FNSE has attracted a great attention, see for short list \cite{Cannone-Fract-NS-08, WUJ-Global-regu-2011, WUJ-06, Zhang-FNSE-stocastic-tool-2012}. The dD-FSNSEs on the torus and on bounded domains have been studied in a general  framework in \cite{Debbi-FSNSE-13}. In particular, the proof presented in this note is based on the fractional calculus developed and the results obtained in this work and on the result about the  LDP in \cite{Millet-Chueshov-Hydranamycs-2NS-10}. In this latter, the authors established an unified approach for the LDP for a wide class of SPDEs, among them the classical 2D-SNSE. The results in this note  extend this class to cover the 2D-FSNSE on the torus. The author does not know any result concerning the LDP for the FSNSE. Results concerning the LDP for the classical 2D-SNSE, can be found in \cite{Millet-Chueshov-Hydranamycs-2NS-10, Sundar-Sri-large-deviation-NS-06}.\del{zeitoun-book-93, Freidlin-LD-book,} \\

\vspace{-0.35cm}

 It is well documented that to deal mathematically with NSE, we have to split up the equation\del{Problem \eqref{Eq-classical-SNSE-O-p}} via Helmholtz projection. In this aim, we introduce the spaces
\vspace{-0.07cm}
\begin{equation*}\label{eq-def-Torus-Lq} 
\mathbb{L}^2(\mathbb{T}^2):= \{ u \in L_2^2(\mathbb{T}^2) := (L^q(\mathbb{T}^2))^2, div u=0 \},\;\; \mathbb{H}^{\beta, 2}(\mathbb{T}^2):= H_2^{\beta, 2}(\mathbb{T}^2)\cap \mathbb{L}^2(\mathbb{T}^2),\; \beta\in \mathbb{R}_+, 
\end{equation*}
\del{\begin{equation}\label{eq-def-Torus-Hs}
\mathbb{H}^{\beta, 2}(\mathbb{T}^d):= H_d^{\beta, 2}(\mathbb{T}^d)\cap \mathbb{L}^2(\mathbb{T}^d),\; \beta\in \mathbb{R}_+,
\end{equation}}
where $ (L^2(\mathbb{T}^2))^2 $ and $ (H^{\beta, 2}(\mathbb{T}^2))^2,\;\; \beta\in \mathbb{R}$, are the 
corresponding vectorial spaces  of the following null average Lebesgue and periodic Riesz potential spaces, see e.g.\del{ \cite[ps.44-48]{Foias-book-2001}}
\cite{Debbi-scalar-active, Foias-book-2001,  
Sickel-periodic spaces-85, Sinai-Mattg-Gibbs-2001},
\begin{eqnarray*}\label{def-H-s-q}
 H^{\beta, 2}(\mathbb{T}^d):= \{\!\!\!\!&{}&\!\!\!\!\!f\in D'(\mathbb{T}^d),\; s.t.\; \hat{f}(0)=0,\; \text{and}\; |f|_{H^{\beta, 2}}:=
\sum_{k\in\mathbb{Z}_0^2}\del{(1+|k|^2)^\frac \beta2}|k|^{2\beta}\hat{f}(k)^{2}<\infty\},
\end{eqnarray*}
where $ D'(\mathbb{T}^2)$ is the topological dual of $ D(\mathbb{T}^2)$; the set of infinitely differentiable scalar-valued functions on $ \mathbb{T}^2$,  $ (\hat{f}(k)=(2\pi)^{-2} f(e^{ik\cdot}))_{k\in\mathbb{Z}^2}$ is the sequence of Fourier
coefficients corresponding to  $ f$ and $ H^{0, 2}(\mathbb{T}^2):= L^2(\mathbb{T}^2)$.
\del{\begin{equation}\label{Fourier-coeff-Distr-}
c_k=\hat{f}(k):= (2\pi)^{-d} f(e^{ik\cdot}).
\end{equation}}\del{If $ f \in L^2(\mathbb{T}^2)\subset D'(\mathbb{T}^2)$, then $(c_k)_k$ is given by 
\begin{equation}\label{Fourier-coeff}
c_k=\hat{f}(k):= (2\pi)^{-d}\langle f, e^{ik\cdot}\rangle = (2\pi)^{-d}\int_{\mathbb{T}^d}f(x)e^{-ixk}dx,
\end{equation}
where the brackets in \eqref{Fourier-coeff} stand for\del{ the duality, in particular, it also denotes the}
scalar product in the Hilbert space  $ L^2(\mathbb{T}^2)$.} The Helmholtz decomposition is given by $  L^2_2(\mathbb{T}^2) = \mathbb{L}^2(\mathbb{T}^2) \oplus\mathcal{Y}(\mathbb{T}^2), $
\del{\begin{equation}\label{direct-sum}
  L^2_2(\mathbb{T}^2):= L^2(\mathbb{T}^2)\times L^2(\mathbb{T}^2) = \mathbb{L}^2(\mathbb{T}^2) \oplus\mathcal{Y}(\mathbb{T}^2), 
\end{equation}}
where $ \mathcal{Y}(\mathbb{T}^2):= \{\nabla p,\;\; p\in H^{1, 2}(\mathbb{T}^2)\}$ and the notation $ \oplus$ stands the orthogonal sum. 
The projection $ \Pi$ on  the space $ \mathbb{L}^2(\mathbb{T}^2)$ is called 
Helmholtz projection. The Stokes operator $ A$ is defined by $ A:= -\Pi \Delta: D(A) = \mathbb{H}^{2, 2}(\mathbb{T}^2)\rightarrow \mathbb{L}^2(\mathbb{T}^2)$, see e.g. \cite{Debbi-scalar-active, Foias-book-2001,  Temam-NS-Functional-95, Temam-NS-Main-79}. Now, we introduce the stochastic term. We fix the stochastic basis $ (\Omega, \mathcal{F}, P, \mathbb{F}, W)$, where
$ (\Omega, \mathcal{F}, P) $  is a complete probability space, $\mathbb{F} := (\mathcal{F}_t)_{t\geq 0}$
is a filtration satisfying the usual conditions, i.e. $(\mathcal{F}_t)_{t\geq 0}$ is an increasing right continuous filtration containing all null sets.
The stochastic process $ W:= (W(t), t\in [0, T])$ is a Wiener process with covariance operator $ Q $ being a positive symmetric trace 
class on $ \mathbb{H}^{1, 2}(\mathbb{T}^2)$. By a  Wiener process on an abstract Hilbert space $ H$, we mean, see e.g.
\cite[Definition 2.1]{Sundar-Sri-large-deviation-NS-06} and \cite{Millet-Chueshov-Hydranamycs-2NS-10, DaPrato-Zbc-92},
\begin{defn}
A stochastic process $ W:= (W(t), t\in [0, T])$ is said to be an $H$-valued 
$ \mathcal{F}_t-$adapted  Wiener process with covariance operator $ Q$, if
\begin{itemize}
\item for all $ 0\neq h\in H$, the process $  (|Q^\frac12 h|^{-1}\langle W(t), h\rangle, t\in [0, T])$ is a standard one dimensional Brownian motion,
\item for all $  h\in H$, the process $  (\langle W(t), h\rangle, t\in [0, T])$ is a martingale adapted to $\mathbb{F}$.
\end{itemize}
\end{defn}
\noindent\del{ Otherwise,  the process
$ W:= (W(t), t\in [0, T])$ is a mean zero
 Gaussian process defined on the filtered probability space $ (\Omega,
\mathcal{F}, P, \mathbb{F} )$ with time  stationary independent  increments and covariance function given by $ \mathbb{E}[\langle W(t), f\rangle \langle W(s), g\rangle]= (t\wedge s)\langle Qf, g\rangle,   \;\;\;   t,s \geq 0,\; f, g \in H.
$
\begin{equation}\label{Eq-Cov-W}
\mathbb{E}[\langle W(t), f\rangle \langle W(s), g\rangle]= (t\wedge s)\langle Qf, g\rangle,   \;\;\;   t,s \geq 0,\; f, g \in H.
\end{equation}
\noindent } Formally,  we write $ W$ as the sum of an infinite series $W(t):= \sum_{j\in\mathbb{Z}_0^2}\beta_j(t)Q^\frac12e_j$,\del{\begin{equation}\label{Seri-W}
W(t):= \sum_{j\in\mathbb{Z}^2}\beta_j(t)Q^\frac12e_j,
\end{equation}} where $(\beta_j)_{j\in \mathbb{Z}^2_0}$ is an i.i.d.  sequence of real Brownian motions and $ (e_j)_{j\in \mathbb{Z}^2_0}$ is any orthonormal basis in $ H$.\del{(On can consider the torus $ (0, 2\pi)^2$ and basis of the Stokes eigenvalues  $ (\frac1{2\pi}\frac{k^\perp}{|k|}e^{ik\cdot})_{k\in \mathbb{Z}_0^2}$).}\del{For more illustration one can assume that the basis $ (e_j)_{j\in \Sigma}$ diagonalizes simultaneously the Stokes operator $ A$
and the Covariance  $ Q$ with  $ (q_j)_{j\in \Sigma}$  is the sequence
of the eigenvalues of  $ Q$($ Qe_j = q_j e_j, \;\;\; \text{and} \;\;\; tr(Q):= \sum_{j\in \Sigma} q_j<\infty.
$).
\del{\begin{equation}\label{cond-Q-tr}
 Qe_j = q_j e_j, \;\;\; \text{and} \;\;\; tr(Q):= \sum_{j\in \Sigma} q_j<\infty.
\end{equation}} We deal with stochastic integrals in Hilbert spaces. 
In particular, } We define (here we follow the custom to do not make a difference between a class of processes and a process representing this class) $ H_0:= Q^{\frac12}(H) $, with the scalar product $\langle \phi, \psi\rangle_{H_0}:=
\langle Q^{-\frac12} \phi, Q^{-\frac12}\psi\rangle_{H},\; \forall \phi, \psi \in H_0$,  
$ L_Q(H):=\{S: H \rightarrow H:\; SQ^\frac12 \;\text{ is  Hilbert Schmidt}\}$ and $\langle S_1, S_2\rangle_{L_Q}:= tr (S_2^*QS_1)\; \forall S_1, S_2 \in L_Q(H)$ and finaly the space 
$ \mathcal{P}_T(H):=\{\sigma \in L^2(\Omega\times [0, T];\; L_Q(H)),\text{ predictable processes\del{progressively measurable  processes}}\}$  endowed with $ \langle \sigma_1, \sigma_2 \rangle_{\mathcal{P}_T} := \mathbb{E}\int_0^T    \langle \sigma_1(s), \sigma_2(s) \rangle_{L_Q}ds.
$
\del{
\begin{eqnarray}\label{eq-def-H-0}
 H_0:&=& Q^{\frac12}(H) \;\;\;  \text{with the scalar product}\; \langle \phi, \psi\rangle_{H_0}:=
\langle Q^{-\frac12} \phi, Q^{-\frac12}\psi\rangle_{H}\del{\nonumber\\
 &{}& \; \langle \phi, \psi\rangle_{H_0}:=
\langle Q^{-\frac12} \phi, Q^{-\frac12}\psi\rangle_{H}, \;\; \forall \phi,\;  \psi\in H_0,}
\end{eqnarray}
\begin{eqnarray*}\label{eq-def-L-Q}
 L_Q(H):&=& \{S: H \rightarrow H:\; SQ^\frac12 \;\text{ is  Hilbert Schmidt}\}\; \text{and}\; \langle S_1, S_2\rangle_{L_Q}:= tr (S_2^*QS_1)\del{,\nonumber\\
&{}& \langle S_1, S_2\rangle_{L_Q}:= tr (S_2^*QS_1), \; \forall S_1, S_2 \in L_Q(H)}
\end{eqnarray*}
\begin{eqnarray}\label{eq-def-P-T}
\mathcal{P}_T(H):&=&\{\sigma \in L^2(\Omega\times [0, T];\; L_Q(H)),\;\;\text{ predictable processes\del{progressively measurable  processes}}\}, \nonumber\\
&{}& \langle \sigma_1, \sigma_2 \rangle_{\mathcal{P}_T} := \mathbb{E}\int_0^T    \langle \sigma_1(s), \sigma_2(s) \rangle_{L_Q}ds.
\end{eqnarray}}
\del{\noindent To be more precise, $ \mathcal{P}_T(H)$ is the set of equivalence classes of predictable processes, but, here we follow the custom to do not make a difference between a class of processes and a process representing this class, see e.g. \cite{Karatzas-Book}.
It is to be noted that as the operator $ Q$ is of trace class, then the canonical injection $ i: H_0\rightarrow H$ is a Hilbert-Schmidt operator. Moreover, we have  $ ii^*=Q$.} It is well known that the stochastic integral, $ (\int_0^t \sigma(s)dW(s), t\in [0, T])$,
is well defined  for all $ \sigma \in \mathcal{P}_T(H)$, see e.g. \cite{DaPrato-Zbc-92}. When projecting Eq. \eqref{Eq-classical-SNSE-O-p} on $ \mathbb{L}^2(\mathbb{T}^2)$, we get\del{ the following stochastic evolution equation } (we take $ \nu =1$ and keep the notations $ \Pi G=G$ and $ u_0:= \Pi u_0$)
\begin{equation}\label{Main-stoch-eq}
\Bigg\{
\begin{array}{lr}
 du(t)= \left(-
A_{\alpha}u(t) + B(u(t))\right)dt+ G(t, u(t))dW(t), \; 0< t\leq T,\\
u(0)= u_0,
\end{array}
\end{equation}
where $ A_\alpha:= A^\frac\alpha2$,  $ B(u):=\Pi((u.\nabla) u)$,\del{\eqref{eq-B-projection},  $ u_0:= \Pi u_0\del{= \Pi u(0, \cdot)}$, and $ W$\del{$:= (W(t), t\in [0, T])$} is the  Wiener process and $ \Pi G=G$ is still denoted by $G$.\del{, maps $ \mathbb{H}^{1, 2}(\mathbb{T}^2)$ on $L_Q(\mathbb{H}^{1, 2})$.\del{ $ G: \mathbb{H}^{1, 2}(\mathbb{T}^2)\rightarrow L_Q(\mathbb{H}^{1, 2})$. F}}} for more details about this equation, see \cite{Debbi-FSNSE-13}.

{\bf Assumptions.} Assume that $ G \in C([0, T]\times \mathbb{H}^{1, 2}; L_Q(\mathbb{H}^{1, 2})$ and that there exist positive constants $ c, \gamma$, such that for every $ t, t'\in [0, T]$ and every $ u, v \in \mathbb{H}^{1, 2}(\mathbb{T}^2)$, see\del{ these conditions in} \cite{Millet-Chueshov-Hydranamycs-2NS-10, Sundar-Sri-large-deviation-NS-06},
 \begin{equation}\label{Eq-Cond-Lipschitz-Q-G}
\!\!\!\!\!(C_1)  \; \text{Lipschitz condition:}\;\;\;\;\;\;\;\;\;\;||G(t, u)-G(t, v)||_{L_Q}\leq c|u-v|_{\mathbb{H}^{1, 2}}.
\end{equation}
\begin{equation}\label{Eq-Cond-Linear-Q-G}
\!\!\!\!\!\!\!\!\!\!\!\!\!\!\!\!\!\!\!\!\!\!\!\!\!(C_2)\;  \text{Linear growth:}\; \;\;\;\;\;\;\;\;\;\;\;\;\;||G(t, u)||_{L_Q}\leq c(1+|u|_{\mathbb{H}^{1, 2}}).
\end{equation}
 \begin{equation}\label{Eq-Holder-cond-G}
(C_3) \; \text{Time H\"older regularity:}\;\;\;\;\;\;\; ||G(t, u)-G(t', u)||_{L_Q}\leq c(1+ |u|_{\mathbb{H}^{1, 2}})|t-t'|^\gamma.
\end{equation}
\del{\begin{equation}\label{Eq-initial-cond}
\!\!\!\!\!\!\!\!\!\!\!\!\!\!(C_4) \;\;\; \text{Initial condition:} \;\;\;\;\;\;\;\;\;\;\;\;\;\;\;\;\;\;\; u_0\in L^{p}(\Omega, \mathcal{F}_0, P; \mathbb{H}^{1, 2}(\mathbb{T}^2)).
\end{equation}}

\section{Wellposedness.}\label{wellposedness}
\del{In this section, we recall the  definition of solutions we have proved in \cite{Debbi-FSNSE-13}. }

\begin{defn}\label{def-variational solution}
Let us fix the stochastic basis $ (\Omega, \mathcal{F}, P, \mathbb{F}, W)$ and give\del{ $ H$ be a separable Hilbert space and } the Gelfand triple $ V \subset H\cong H^* \subset V^*$,  where $ H$  and $ V$ are  separable Hilbert respectively separable reflexive Banach spaces and $ V^*$ is the topological dual of $V$.\del{Assume that $ u_0\in L^p(\Omega, \mathcal{F}_0, P, H)$.} A continuous $H$-valued  predictabe stochastic process $(u(t), t\in [0, T])$  is called a weak solution (strong in probability) of Equation \eqref{Main-stoch-eq} with initial condition $ u_0$, if $u(\cdot, \omega) \in X:=C([0, T]; H) \cap L^2(0, T; V)\;\;\; P-a.s.$
 \del{ \begin{equation}\label{eq-set-solu-weak}
 u(\cdot, \omega) \in L^\infty(0, T; H) \cap L^2(0, T; V)\cap C([0, T]; V_1)\;\;\; P-a.s.
 \end{equation}}
and $ P-a.s.$  the following identity holds, for all $ t\in [0, T]$ and $ \varphi \in V$,
\begin{eqnarray}\label{Eq-weak-Solution}
\langle u(t), \varphi\rangle_H &=& \langle u_0, \varphi\rangle_{H} + \int_0^t {}_{V^*}\langle A^\frac\alpha2 u(s), \varphi \rangle_{V} ds+
\int_0^t {}_{V^*}\langle B(u(s)), \varphi)\rangle_{V} ds\nonumber\\
& + & {}_{V^*}\langle\int_0^tG
(u(s))dW(s), \varphi\rangle_{V}.
\end{eqnarray}
\end{defn}
\del{\noindent The following results have been proved in  \cite{DEbbi-FSNSE-13}.}
\begin{theorem}\label{Main-theorem-strog-Torus}
Let $ \alpha \in [\frac43, 2]$ and $ u_0\in L^{p}(\Omega, \mathcal{F}_0, P; \mathbb{H}^{1, 2}(\mathbb{T}^2))$, for $p\geq 4$\del{$ u_0$  satisfying Assumption $(C_4)$}. Assume that $ G$ satisfies  $ (C_1) \& (C_2)$. Then Equation \eqref{Main-stoch-eq} admits a unique weak solution (strong in probability) $ (u(t), t\in [0, T])$ in the sense of Definition \ref{def-variational solution}, with the corresponding Gelfand triple $ \mathbb{H}^{1+\frac\alpha2, 2} (\mathbb{T}^2)\subset \mathbb{H}^{1, 2}(\mathbb{T}^2)\subset
\del{(\mathbb{H}^{1+\frac\alpha2, 2}(\mathbb{T}^2))^*=}\mathbb{H}^{1-\frac\alpha2, 2}(\mathbb{T}^2)$. 
\end{theorem}
\begin{proof}
The proof of the existence and uniqueness of the solution\del{ $u(\cdot, \omega) \in L^\infty(0, T; \mathbb{H}^{1, 2}(\mathbb{T}^2)) \cap L^2(0, T; \mathbb{H}^{1+\frac\alpha2, 2}(\mathbb{T}^2))\;a.s.$} is given in \cite{Debbi-FSNSE-13} (see also arXiv:1307.5758). The continuity\del{,i.e. $ u\in C([0, T]; \mathbb{H}^{1, 2}(\mathbb{T}^2))$} follows by application of \cite[Theorem 4.2.5]{Rockner-Pevot-06}.
\del{we follow the method used in \cite{Millet-Chueshov-Hydranamycs-2NS-10}. We give here a skech of the proof. In fact, using \eqref{Eq-weak-Solution} and the properties of $(u(t), t\in [0, T])$ above,  we conclude that $ e^{-\delta A}u(t) \in C([0, T]; \mathbb{H}^{1, 2}(\mathbb{T}^2))$ for all $ \delta>0$. To show that $ \lim_{\delta\rightarrow 0}\mathbb{E}\sup_{[0, T]}|(I- e^{-\delta A})\del{C_\delta }u(t)|^2_{\mathbb{H}^{1, 2}}=0$\del{, where $C_\delta := I- e^{-\delta A}$,} we apply Ito formula to $ |(I- e^{-\delta A})u(t)|_{\mathbb{H}^{1,2}}$, argue as in the proof of \cite[Lemma 6.1]{Debbi-FSNSE-13} and follow the steps as in  \cite{Millet-Chueshov-Hydranamycs-2NS-10}. In particular, we use the following estimation
\begin{eqnarray}
\mathbb{E}
\end{eqnarray} }  
\end{proof}

\section{Large deviations.}
Let $ \epsilon>0$ and let  $(u^\epsilon)_{\epsilon>0}$ be the family of solutions of the FSNSEs\del{fractional stochastic Navier-Stokes equations}
\begin{equation}\label{Eq-epsilon-LD}
\Bigg\{
\begin{array}{lr}
 du^\epsilon(t)= \left(-
A_{\alpha}u^\epsilon(t) + B(u^\epsilon(t))\right)dt+ \sqrt{\epsilon} G(t, u^\epsilon(t))dW(t), \; 0< t\leq T,\\
u^\epsilon(0)= \xi \in \mathbb{H}^{1,2}(\mathbb{T}^2).
\end{array}
\end{equation}
Recall that thanks to Theorem \ref{Main-theorem-strog-Torus}, $(u^\epsilon(t), t\in[0, T])$ exists and is unique with values in $ X:= C([0, T]; \mathbb{H}^{1, 2}(\mathbb{T}^2))\cap L^2(0, T;  \mathbb{H}^{1+\frac\alpha2, 2}(\mathbb{T}^2))$, $ \forall \epsilon>0$. Following the same proof as in \cite{Debbi-FSNSE-13}, we can prove that the above wellposedness result remains true for the corresponding control equation, i.e. the equation obtained from \eqref{Eq-epsilon-LD}  by replacing $ \epsilon =1$ and $ W$ by any $ \int_0^\cdot v(s)ds$, with $ v\in L^2(0, T; H_0)$. Let us denote the solution of the control equation by $u^v$ and define $ g^0:C([0, T]; H_0) \rightarrow X$ by 
$ g^0(h):=u^v$, if $ h=\int_0^\cdot v(s)ds$ and $ g^0(h)=0$, otherwise.
\del{
and consequently, there exists Moreover, there  Now, we prove that
$(u^\epsilon)_{\epsilon>0}$ satisfies LPD\del{ the large deviation principle} on $  X:=C([0, T]; \mathbb{H}^{1,2}(\mathbb{T}^2)) \cap L^2(0, T; \mathbb{H}^{1+\frac\alpha2,2}(\mathbb{T}^2))$, see e.g. for short list \cite{Millet-Chueshov-Hydranamycs-2NS-10, zeitoun-book-93, Freidlin-LD-book, Sundar-Sri-large-deviation-NS-06}.}

\del{for all $ \epsilon>0$, Equation \ref{Eq-epsilon-LD} admits a unique weak (strong in probability) solution $(u^\epsilon(t), t\in [0, T])$. The main result of this paper is to prove that the family $ (u^\epsilon, \; \epsilon \in [0, 1])$ satisfies the large deviation principle on $  X:=C([0, T]; H) \cap L^2(0, T; V) $ Now, we give the definition of the large deviation principle.
The main result of this paper is to prove that the family $ (u^\epsilon, \; \epsilon \in [0, 1])$ satisfies the large deviation principle on $  X:=C([0, T]; \mathbb{H}^{1,2}(\mathbb{T}^2)) \cap L^2(0, T; \mathbb{H}^{1+\frac\alpha2,2}(\mathbb{T}^2)) $. This result is given by the following theorem}

\begin{defn}
Let $ (u^\epsilon)_{\epsilon>0}$ be a random family on a Banach space $ X$ and let the rate function $ I: X \rightarrow [0, +\infty]$ be good, i.e. satisfying that the level set $ \{ \phi \in X: I(\phi)\leq M\}$ is a compact subset of $ X$.\del{ Let us define for $ F \in \mathcal{B}(X)$, $ I(S):= \inf_{x\in S}I(x)$.} The family $ (u^\epsilon)_{\epsilon>0}$ is said to satisfy a LDP\del{a large deviation principle} on $ X$ with rate function $ I$ if for each $ F \in \mathcal{B}(X)$, we have 
\begin{equation}
-\inf_{x\in\mathring{F}} I(x)\leq 
\liminf_{\epsilon \rightarrow 0}\epsilon \log P(u^\epsilon \in F)\leq
\limsup_{\epsilon \rightarrow 0}\epsilon \log P(u^\epsilon \in F)\leq -\inf_{x\in\bar{F}}I(x),
\end{equation}
where $ \mathring{F} $ and $\bar{F}$ are the interior respectively the closure of $ F$ in $ X$.
\del{\begin{itemize}
\item For each closed subset $ F$ of $ X$, we have (Large deviation uper bound)
\begin{equation}
\limsup_{\epsilon \rightarrow 0}\epsilon \log P(u^\epsilon \in F)\leq -I(F),
\end{equation}
\item For each open subset $ G$ of $ X$, we have (Large deviation lower bound)
\begin{equation}
\limsup_{\epsilon \rightarrow 0}\epsilon \log P(u^\epsilon \in G)\geq -I(G).
\end{equation} 
\end{itemize}}
\end{defn}
\begin{theorem}\label{Main-Theorem-LD}
Assume that $G$ satisfies $ (C_1)-(C_3)$, then the family $ (u^\epsilon)_{\epsilon>0}$ satisfies the LDP\del{large deviation principle} on $  X:=C([0, T]; H) \cap L^2(0, T; V)$, with the good rate function (by convention $ \inf(\emptyset) =\infty$)
\begin{equation}
I(f):= \inf_{\{v\in L^2(0, T; H_0):f(\cdot)=g^0(\int_0^\cdot v(s)ds)\}}\{\frac12\int_0^T|v(s)|^2_{H_0}ds\}.
\end{equation}
\del{with the convention $ \inf(\emptyset) =\infty$ and where $ g^0(\int_0^\cdot v(s)ds)$ is the solution of
\begin{equation}\label{Eq-epsilon-LD}
\Bigg\{
\begin{array}{lr}
 du^v(t)= \left(-
A_{\alpha}u^v(t) + B(u^v(t))\right)dt+ G(t, u^v(t))v(t)dt, \; 0< t\leq T,\\
u^v(0)= \xi. \del{\in \mathbb{H}^{1,2}(\mathbb{T}^2).}
\end{array}
\end{equation}}
\end{theorem}
\noindent The following results are necessary to prove Theorem \ref{Main-Theorem-LD}. They are special cases of \cite[Theorem 3.10 \& Lemma 3.11]{Debbi-FSNSE-13}.\del{, see \cite{Debbi-FSNSE-13} for details and proofs,}
\begin{theorem}\label{theo-gelfand-gene}
Let $\eta \geq 0$ and $ \alpha(\eta) :=
\max\{\frac{4-2\eta}{3},\; 2\eta\}$ if $\eta \in [0, 1)$ and equal to 
1 if  $\eta\geq 1$. Assume that $ \alpha \in [\alpha(\eta), 2]$, with $\eta \in [0, \frac{1}{2}) \cup [1, \infty)$ or
$ \alpha \in (\alpha(\eta), 2)$ with $\eta \in [\frac{1}{2}, 1)$.\del{where $ \alpha(\eta)$ is defined by \eqref{critical-value-alpha}.} Then $ B: V_\eta\times V_\eta:= 
 (\mathbb{H}^{\eta+\frac\alpha2, 2}(\mathbb{T}^2))^2\rightarrow V_\eta^* =
\mathbb{H}^{\eta-\frac\alpha2, 2}(\mathbb{T}^2)$
 is bounded.\del{ Moreover, there exists a constant $ c:=c_{\alpha, \eta}>0$, such that
\begin{equation}\label{Eq-B-u-v-V-dual}
|B(u, v)|_{\mathbb{H}^{\eta-\frac\alpha2, 2}} \leq  c |u|_{\mathbb{H}^{\eta+\frac\alpha2, 2}}
|v|_{\mathbb{H}^{\eta+\frac\alpha2, 2}}.
\end{equation}
In particular, we have the following useful cases, $ \alpha(\eta =0)= \frac43$ and $ \alpha(\eta =1)= 1$.}
\del{\begin{itemize}
\item  $ \eta =0$, $ d\in \{2, 3, 4\}$ and $ \alpha\in [\frac{d+2}{3}, 2]$.
\item  $ \eta =1$ and either  $ d=2$  and  $ \alpha \in [1, 2] $ or  $d=3$ and  $ \alpha \in (1, 2]$ or
$d\in \{4, 5, 6\}$ and  $ \alpha \in [\frac{d}{3}, 2]$.
\end{itemize}}
\end{theorem}

\begin{lem}\label{lem-further-estimation}
 For  $ \eta \geq 1$ there exits $ c_\eta=c>0$, such that
for all   $u, v \in \mathbb{H}^{\eta+1-\frac\alpha2, 2}(\mathbb{T}^2)$, 
\begin{eqnarray}\label{est-B-estimatoion-q=2-eta-to-use}
 |B(u, v)|_{\mathbb{H}^{\eta-\frac\alpha2, 2}} &\leq& c |u|_{\mathbb{H}^{\eta+1-\frac{\alpha}{2}, 2}} 
 |v|_{\mathbb{H}^{\eta+1-\frac{\alpha}{2}, 2}}.
\end{eqnarray}

\end{lem}
\noindent {\bf Proof of Theorem \ref{Main-Theorem-LD}.} We apply \cite[Theorem 3.2]{Millet-Chueshov-Hydranamycs-2NS-10}. In fact, $ (C_1)-(C_3)$ are satisfied by assumptions. Moreover, thanks to Theorem \ref{theo-gelfand-gene}, with $ \eta =1$, the bilinear form $ B: \mathbb{H}^{1+\frac\alpha2, 2}(\mathbb{T}^2)\times \mathbb{H}^{1+\frac\alpha2, 2}(\mathbb{T}^2) \rightarrow \mathbb{H}^{1-\frac\alpha2, 2}(\mathbb{T}^2)$ is continuos for all $ \alpha \in [1, 2]$. Now, it is sufficient to find an interpolation space $ \mathcal{H}$ possessing \cite[the properties (i)-(iii)]{Millet-Chueshov-Hydranamycs-2NS-10}. We consider $ \mathcal{H} = \mathbb{H}^{1+\frac\alpha4, 2}(\mathbb{T}^2)$. By interpolation, it is easy to see that there exists a constant $c>0 $ such that for all $ v\in \mathbb{H}^{1+\frac\alpha2, 2}(\mathbb{T}^2)$, we have
\begin{equation}
|v|^2_{\mathbb{H}^{1+\frac\alpha4, 2}}\leq c |v|_{\mathbb{H}^{1, 2}}|v|_{\mathbb{H}^{1+\frac\alpha2, 2}}.
\end{equation}
Furthermore, using \eqref{est-B-estimatoion-q=2-eta-to-use}, the Sobolev embedding, the condition  $ \frac43\leq \alpha\leq 2$ and Young inequality, we infer that for every $ \lambda>0$ there exists a positive constant $ c_\lambda>0$ such that for all $ v^i \in \mathbb{H}^{1+\frac\alpha2, 2}(\mathbb{T}^2), \; i\in \{1,2,3\}$, we have
\begin{eqnarray}
|\langle B(v^1, v^2), v^3\rangle_{\mathbb{H}^{1, 2}}| &\leq& |v^3|_{ \mathbb{H}^{1+\frac\alpha2, 2}}|B(v^1, v^2) |_{\mathbb{H}^{1-\frac\alpha2, 2}} \leq c|v^3|_{\mathbb{H}^{1+\frac\alpha2, 2}}|v^1|_{\mathbb{H}^{2-\frac\alpha2, 2}}|v^2 |_{\mathbb{H}^{2-\frac\alpha2, 2}}\nonumber\\
&\leq & c|v^3|_{\mathbb{H}^{1+\frac\alpha2, 2}}|v^1|_{\mathbb{H}^{1+\frac\alpha4, 2}}|v^2 |_{\mathbb{H}^{1+\frac\alpha4, 2}} 
\leq \lambda |v^3|^2_{\mathbb{H}^{1+\frac\alpha2, 2}}+ c_\lambda |v^1|^2_{\mathbb{H}^{1+\frac\alpha4, 2}}|v^2 |^2_{\mathbb{H}^{1+\frac\alpha4, 2}}.
\end{eqnarray}
The proof is achieved.
\del{\begin{Rem}

\end{Rem}}

\del{\section{Properties of the nonlinear term}\label{sec-nonlinear-prop}
\noindent Our aim in this section is to study the nonlinear operator $ B$  defined by \eqref{eq-B-projection}. 
Here $ O$ denotes either the torus
$ \mathbb{T}^d$ or a bounded domain from $ \mathbb{R}^d$ with smooth boundary as mentioned  above. 
We define the bilinear operator
$ B:(\mathcal{D}(O))^2 \rightarrow \mathbb{L}^2(O)$ and the tri-linear form
$ b:(\mathcal{D}(O))^3 \rightarrow \mathbb{R}$ by,
\begin{equation}\label{Eq-def-B-theta1-Theata2}
B(u, v): = \Pi ((u\cdot \nabla) v), \;\;\;\; \forall (u, v),
\in (\mathcal{D}(O))^2
\end{equation}
respectively,
\begin{equation}\label{Eq-3lin-form}
b(u, \theta, v):= \langle B(u, v), v\rangle,\;\;\; \forall (u, \theta, v) \in (\mathcal{D}(O))^3,
\end{equation}
where the brackets in RHS of \eqref{Eq-3lin-form} stand for the scalar product in $ \mathbb{L}^2(O)$, see e.g. \cite{Amann-solvability-NSE-2000, Farwig-Sohr-L-p-theory-2005},
\begin{eqnarray}
 \mathcal{D}(O):&=& \{u\in (C^\infty(O))^d, div u=0\,\,  \text{ and $ u$ has a compact support when } O\,
 \text{is a bounded domain}\}\nonumber\\
 &=& \left\{
 \begin{array}{lr}
 (C_0^\infty(O))^d\cap \mathbb{L}^q(O), \;\; \text{when}\;\; O \;\; \text{is bounded},\nonumber\\
 (C^\infty(O))^d\cap \mathbb{L}^q(O), \;\; \text{when}\;\; O =\mathbb{T}^d.\nonumber\\
 \end{array}
 \right.
\end{eqnarray}
\del{\mathcal{D}(O):= \{\text{the set of infinitely differential functions, with compact support in the case } O
 \text{is a bounded domain and such that }  div u=0\}.}
The bilinear operator $ B $ and the trilinear form  $ b$\del{  as defined above \eqref{Eq-def-B-theta1-Theata2} and  
\eqref{Eq-3lin-form}} have several extensions based on the $ \mathbb{}H^{\beta, q}-$norm, with $ \beta\geq 1$, 
 see e.g. \cite{Temam-NS-Functional-95} and \cite[p. 97]{Foias-book-2001} for Hilbert spaces and
\cite{Giga-Solu-Lr-NS-85}  for Banach spaces and for a general survey.\del{ of many other results.} Unfortunately, 
 due to the weakness of the fractional dissipation in our equation these extensions are
useless for  our case.
 Let us before dealing with the extensions we are interested in here, recall the following intrinsic properties
\begin{equation}\label{Eq-3lin-propsym}
b(u, \theta, v)= -b(u, v, \theta), \;\; \forall u, v, \theta  \in \mathbb{H}^{1, 2}(O).
\end{equation}
Hence
\begin{equation}\label{Eq-3lin-propnull}
b(u, v, v)= 0, \;\; \forall u, v \in \mathbb{H}^{1, 2}(O).
\end{equation}
In particular, for $ O= \mathbb{T}^2$, we have also, see e.g. either \cite[Lemma 3.1]{Temam-NS-Functional-95} or  \cite[Lemma VI.3.1]{Temam-Inf-dim-88}.
\begin{equation}\label{vanishes-bilinear-tous-H1}
 \langle B(u), u\rangle_{\mathbb{H}^{1, 2}} = \langle B(u), u\rangle_{H_2^{1, 2}} =0,\;\;\; \forall u \in D(A):= \mathbb{H}^{2, 2}(\mathbb{T}^2).
\end{equation}
Now, we cite some basic lemmas.
\begin{lem}\label{Lem-classic}
For all  $ 1\leq j\leq d$, $ \eta\geq 0$ and  $1<q<\infty$,  the operators $ A^{-\frac12}\Pi\partial_{j}$ extends uniquely to a bounded
linear operator from  $ H_d^{\eta, q}(O)$ to $ \mathbb{H}^{\eta, q}(O)$.\del{Furthermore $ \partial_{j}IA^{-\frac12}$
is a bounded operator from $ \mathbb{L}^q(\mathbb{T}^d)$ to $ L_d^q(\mathbb{T}^d)$, where $ I $ denotes the injection
$ \mathbb{L}^q(\mathbb{T}^d)\subset L_d^q(\mathbb{T}^d)$.}
\end{lem}
\begin{proof}
For $ \eta \geq 0 $ and  $ O= \mathbb{T}^d$, we 
use Marcinkiewicz's theory for the pseudodifferential  operator $ A^{-\frac12}\Pi\partial_{j}$. In fact,
the  symbol of this latter in  Fourier modes is given by the matrix $ i|k|^{-1}k_j(\delta_{m, n}-|k|^{-2}k_mk_n)_{mn}$.
See also \cite{Debbi-scalar-active} and also \cite{Kato-Ponce-86} for similar calculus for the  case $ O= \mathbb{R}^d$. 
The case $ O$ bounded and $ \eta =0 $  has been proved in 
\cite[Lemma 2.1]{Giga-Solu-Lr-NS-85}. We claim here that the method in \cite{Giga-Solu-Lr-NS-85} and also 
the proof bellow are also valid 
for $ O= \mathbb{T}^d$.
For $ \eta \geq 1$  and $ O$ is either a bounded domain of $ \mathbb{R}^d$ or  $ O= \mathbb{T}^d$,
thanks to the properties of 
Helmholtz projection, we prove for all $ \beta\geq 0$ and $1< q<\infty$,
that $ \Pi: H_d^{\beta, q}(O) \rightarrow \mathbb{H}^{\beta, q}(O)$ is well defined and bounded. 
Using this statement and arguing
 as in the proof of \cite[Lemma 2.1]{Giga-Solu-Lr-NS-85}, we get the result for $ \eta\geq 1$. The result for 
 $ 0< \eta <1$ is a consequence of the interpolation for $ \eta= 0$ and $ \eta= 1$. 
\end{proof}
\del{\begin{remark}
 Following the same ideas in Lemma \ref{Lem-classic} and in \cite[Lemma 2.1]{Giga-Solu-Lr-NS-85}, 
 we can prove the result for $ \eta\geq -1$. To be checked.
\end{remark}}

\begin{lem}\label{Giga-Mikayawa-solutionLr-NS}
Let $1<r_1\leq r_2<\infty$ and $ \frac{d}{2}(\frac{1}{r_1}-\frac{1}{r_2})\leq \epsilon <\frac{d}{2r_1}$. Then, the operator
$ A_{r_1}^{-\epsilon}$  extends uniquely to a bounded operator from
$  \mathbb{L}^{r_1}(O)$ to $ \mathbb{L}^{r_2}(O)$.
\end{lem}

\begin{proof}
Thanks to Theorem \ref{theorem-domains-A-D-A-S}\del{{eq-embedded-domain-A-beta}}, it is sufficient to show that
$$  \mathbb{H}^{2\epsilon, r_1}(O) \hookrightarrow H_d^{2\epsilon, r_1}(O)\cap \mathbb{L}^{r_1}(O)  \hookrightarrow \mathbb{L}^{r_2}(O).$$
This last follows from the Sobolev  embedding; $ H_d^{2\epsilon, r_1}(O) \hookrightarrow L_d^{r_2}(O)$,
\del{$  \mathbb{H}^{2\epsilon, r_1}(O) \hookrightarrow \mathbb{L}^{r_2}(O)$,} which is
satisfied provided $ \frac{d}{2}(\frac{1}{r_1}-\frac{1}{r_2})\leq \epsilon <\frac{d}{2r_1}$,
see e.g. \cite[Theorem 7.63 p 221 and 7.66 p 222]{Adams-Hedberg-94} for bounded domain and
 see either \cite[Theorem 3.5.4.ps.168-169 and Corollary 3.5.5.p 170]{Schmeisser-Tribel-87-book} or \cite{Sickel-periodic spaces-85} for  the torus. See also \cite{Amann-solvability-NSE-2000} for solenoidal Sobolev spaces.
\end{proof}

\noindent We call the property in Lemma \ref{Giga-Mikayawa-solutionLr-NS}, the $ L^{r_1}\rightarrow  L^{r_2}$
smoothing property of  $ A_{r_1}^{-\epsilon}$.  As an application of Lemma \ref{Giga-Mikayawa-solutionLr-NS}, we have
\begin{coro}\label{coro-lem-Giga-epsilon}
 For all $ \epsilon $, such that $ \frac{d}{2q}\leq \epsilon < \frac{d}{q}$, the operator $ A^{-\epsilon}: \mathbb{L}^{\frac q2}(O)  \rightarrow
\mathbb{L}^{q}(O)$ is bounded.
\end{coro}

\noindent The following Lemma has been proved in \cite{Giga-Solu-Lr-NS-85} for bounded domain. Our claim is that
the same proof is also valid for $ O= \mathbb{T}^d$,\del{as  the Stokes operator in the case $ O= \mathbb{T}^d$ coincides with the Laplacian, the same proof is also valid for this case,} see
similar calculus in \cite{Debbi-scalar-active} and for $ q=2 $, see e.g. \cite[p 13]{Temam-NS-Functional-95}.

\begin{lem}\label{Lemma-bound-op-RGradient-Via-gamma} \cite[Lemma 2.2]{Giga-Solu-Lr-NS-85}
Let $ 0\leq \delta < \frac12+\frac d2(1-\frac1q)$. Then
\begin{eqnarray}
|A^{-\delta}\Pi (u. \nabla) v|_{\mathbb{L}^q} \leq M|A^{\nu} u|_{\mathbb{L}^q}
|A^{\rho} v|_{\mathbb{L}^q}.
\end{eqnarray}
with some constant $ M:= M_{\delta, \rho, \nu, q}$,\del{ the operator $ (u, v) \mapsto u . \nabla v $ is
continuously extended to $B: D(A^\nu) \times D(A^\rho)
\rightarrow D(A^{-\delta})$} provided that $ \nu, \rho >0$, $\delta+ \rho > \frac12$,
$\delta+ \nu + \rho \geq \frac{d}{2q}+ \frac12$. 
\end{lem}

\noindent As a corollary  of Lemma \ref{Lemma-bound-op-RGradient-Via-gamma}, we cite the following results, which will be generalized later on.\del{These results will be for a general
framework, to prove Theorem \ref{theo-gelfand-gene}.}
\del{\begin{coro}\label{coro-lem-Giga}
\noindent Let either $ O\subset \mathbb{R}^2$ a bounded domain with smooth boundary or
\begin{itemize}
 \item For all  $( u, v) \in \mathbb{H}^{1-\frac\alpha2,2}(O^2)\times \mathbb{H}^{1+\frac\alpha2,2}(O^2)$,
there exists a constant $ c>0$, such that
\begin{eqnarray}\label{Eq-B-L-2-est}
|u\nabla v|_{\mathbb{L}^2} &\leq& c|u|_{H^{1-\frac\alpha2,2}}|v
|_{H^{1+\frac\alpha2,2}}.
\end{eqnarray}
For all  $( u, v) \in (\mathbb{H}^{1-\frac\alpha4,2}(O^2))^2$,
there exists a constant $ c>0$, such that
\begin{eqnarray}\label{Eq-B-H-alpha-2-est}
|u\nabla v|_{\mathbb{H}^{-\frac\alpha2, 2}} &\leq& c|u|_{\mathbb{H}^{1-\frac\alpha4, 2}}|u|_{\mathbb{H}^{1-\frac\alpha4, 2}}.
\end{eqnarray}
\end{itemize}
\end{coro}
\begin{proof}
We can use several methods such as,
\begin{itemize}
 \item Lemma \ref{Lem-classic}. In this proof, we  take $  \delta =0, \rho = \frac12+\frac\alpha4$ and $ \nu = \frac12-\frac\alpha4 $  to get
\eqref{Eq-B-L-2-est} and  we take $  \delta =\frac\alpha4, \rho = \nu = \frac12-\frac\alpha8$ to get \eqref{Eq-B-H-alpha-2-est}.
\item or  \cite[Theorem 4.6.1 p 190 \& Proposition Tr. 2.3.5 p 14]{R&S-96} and Lemma \eqref{Theo-pointwiseMulti-Bounded-Domain},
for $ O\subset \mathbb{R}^2$ being a bounded domain. For the case $O= \mathbb{T}^2 $, we use
\cite[Theorem IV.2.2 (ii)]{Sickel-Pontwise-Torus}, the monotonicity property  \cite[Remark 4, p 164]{Schmeisser-Tribel-87-book} and
\cite[Theorem 3.5.4.ps.168-169]{Schmeisser-Tribel-87-book}.
\item or \cite[Lemma 2.1 and in p 13]{Temam-NS-Functional-95} with $  m_3=0, m_2= \frac\alpha2, m_1= 1-\frac\alpha2$  to prove  \eqref{Eq-B-L-2-est}.
\end{itemize}
\end{proof}}

\begin{coro}\label{coro-lem-Giga}
\noindent Let either $ O= \mathbb{T}^d$ or $ O\subset \mathbb{R}^d$ be a bounded domain. Then
\begin{itemize}
 \item For $ \alpha\in (0,2) $, there exists a constant $ c:= c(\alpha, d)>0$ such that for all
$( u, v) \in \mathbb{H}^{\frac d2-\frac\alpha2,2}(O)\times \mathbb{H}^{1+\frac\alpha2,2}(O)$
\begin{eqnarray}\label{Eq-B-L-2-est}
|B(u, v)|_{\mathbb{L}^2} &\leq& c|u|_{\mathbb{H}^{\frac d2-\frac\alpha2,2}}|v
|_{\mathbb{H}^{1+\frac\alpha2,2}}.
\end{eqnarray}
\item  For $ \alpha\in (0,2] $ there exists a constant $ c:= c(\alpha, d)>0$ such that for all $( u, v) \in (\mathbb{H}^{\frac{2+d-\alpha}{4},2}(O))^2$,
\begin{eqnarray}\label{Eq-B-H-alpha-2-est}
|B(u, v)|_{\mathbb{H}^{-\frac\alpha2, 2}} &\leq& c|u|_{\mathbb{H}^{\frac{2+d-\alpha}{4},2}}
|v|_{\mathbb{H}^{\frac{2+d-\alpha}{4},2}}.
\end{eqnarray}
\end{itemize}
\end{coro}
\del{\begin{proof}
Using  Lemma \ref{Lemma-bound-op-RGradient-Via-gamma} with $  \delta =0, \rho = \frac12+\frac\alpha4$ and $ \nu = \frac d4-\frac\alpha4 $  to get
\eqref{Eq-B-L-2-est} and  with $  \delta =\frac\alpha4, \rho = \nu = \frac{2+d-\alpha}{8}$ to get 
\eqref{Eq-B-H-alpha-2-est},
see also other proof in \cite[Lemma 2.1 and p 13]{Temam-NS-Functional-95}.\del{
 with $  m_3=0, m_2= \frac\alpha2, m_1= 1-\frac\alpha2$  to prove  \eqref{Eq-B-L-2-est}.}
\end{proof}}
\noindent The following results generalize 
\cite[Lemma 2.2]{Giga-Solu-Lr-NS-85} for $ O\subset \mathbb{R}^d$ and
\cite[Lemma 1.4]{Kato-Ponce-86} for $ O=\mathbb{R}^d$.

\begin{prop}\label{Prop-Main-B}
Let $ 2<q<\infty$ and $ \eta\geq 0$. The bilinear form $B :(\mathcal{D}(O))^2
\rightarrow \mathbb{L}^q(O)$ given by Formula
\eqref{eq-B-projection}, extends uniquely, (we keep the same notation) to
\begin{eqnarray}\label{Eq-def-Bilinear-form}
B:&(\mathbb{H}^{\eta, q}(O))^2& \rightarrow \mathbb{H}^{\eta-1-\frac dq, q}(O)\nonumber\\
&(u, v) & \mapsto B(u,v):= \Pi((u\cdot\nabla) v).
\end{eqnarray}
Moreover, there exists a constant $ c:= c_{\eta, d, q}>0$  such that for all 
$ (u,  v) \in (\mathbb{H}^{\eta, q}(O))^2$,
\begin{equation}\label{eq-B-estimatoion}
 | B(u, v)|_{\mathbb{H}^{\eta-1-\frac dq, q}}\leq c |u|_{\mathbb{H}^{\eta, q}} |v|_{\mathbb{H}^{\eta, q}}.
\end{equation}
\end{prop}

\begin{proof}
Let us first mention that  to justify that the bilinear form $B :(\mathcal{D}(O))^2
\rightarrow \mathbb{L}^q(O)$  is well defined, one can use H\"older and Gagliardo-Nirenberg inequalities.
Remark also that, thanks to the free divergence,  we can rewrite  
$ B(u, v)$, for smooth functions $ u$ and $ v$, e.g. $ (u, v)\in (\mathcal{D}(O))^2$, as
\begin{equation}\label{B-theta-divergence-1}
B(u, v) = \Pi(\sum_{j=1}^d \partial_j(u_j v))= \Pi(\partial_j(u_j v)).
\end{equation}
\noindent Using  Corollary \ref{coro-lem-Giga-epsilon} and Lemma \ref{Lem-classic}, we get for  $ \eta =0$, after applying  H\"older inequality,
\begin{eqnarray}\label{est-1-semi-group-div-delta2not2-B}
|B(u, v)|_{\mathbb{H}^{-1-\frac dq, q}} &=& |A^{-\frac12-\frac d{2q}} \Pi\partial_j(u_j v)|_{\mathbb{L}^{q}}
\leq C| A^{-\frac12} \Pi\partial_j(u_j v)|_{\mathbb{L}^{\frac q2}} \nonumber\\
&\leq & C  |u_j v|_{\mathbb{L}^{\frac q2}} \leq  C |u|_{\mathbb{L}^q}| v|_{\mathbb{L}^q}
\end{eqnarray}
 and for $ 0<\eta < \frac dq$, thanks to
\cite[Theorem IV.2.2 (iii) \& Theorem III. 11. (ii)]{Sickel-Pontwise-Torus}  and
\cite[Theorem 3.5.4.ps.168-169]{Schmeisser-Tribel-87-book}, for $ O=\mathbb{T}^d$ and thanks to
\cite[Theorem 1, p. 176 and Proposition Tr 6, 2.3.5, p 14]{R&S-96} and Theorem 
\ref{Theo-pointwiseMulti-Bounded-Domain} bellow,  for the bounded domain,
\begin{eqnarray}\label{est-1-semi-group-div-delta2not2-B}
|B(u, v)|_{\mathbb{H}^{\eta-1-\frac dq, q}} &=& |A^{-\frac d{2q}} A^{\frac\eta2-\frac12}\Pi(\partial_j(u_j v))|_{\mathbb{L}^{q}}
\leq C\del{| A^{\frac\eta2}A^{-\frac12} \Pi\partial_j(u_j v)|_{\mathbb{L}^{\frac q2}} \nonumber\\
&\leq & C  }|u_j v_i|_{H^{\eta, \frac q2}} \leq  C |u|_{\mathbb{H}^{\eta, q}}| v|_{\mathbb{H}^{\eta, q}}.\nonumber\\
\end{eqnarray}

\noindent For $ \eta = \frac dq$, we use  Lemma \ref{Lem-classic},
\cite[Theorem IV.2.2 (ii)]{Sickel-Pontwise-Torus}, \cite[Theorem 3.5.4.ps.168-169]{Schmeisser-Tribel-87-book} 
and the monotonicity property in
\cite[Remark 4.p.164]{Schmeisser-Tribel-87-book} for $ O=\mathbb{T}^d$ and \cite[Theorem 4.6.1, p. 190 and Proposition Tr 6, 2.3.5, p 14]{R&S-96} and Theorem \ref{Theo-pointwiseMulti-Bounded-Domain}, we infer that
\begin{eqnarray}\label{est-1-semi-group-div-delta2not2-B-+2}
|B(u, v)|_{\mathbb{H}^{\eta-1-\frac dq, q}} \leq
 c|u_jv_i|_{L^{q}}
\leq  C  |u|_{\mathbb{H}^{\frac{d}{2q}, q}}|v|_{_{\mathbb{H}^{\frac{d}{2q}, q}}}
\leq  C  |u|_{\mathbb{H}^{\eta, q}}|v|_{\mathbb{H}^{\eta, q}}.
\end{eqnarray}
\noindent For $ \eta > \frac dq $, we use Lemma \ref{Lem-classic} (here we do not use 
Corollary \ref{coro-lem-Giga-epsilon}), the fact that 
$ \mathbb{H}^{\eta, q}(O)$  is an algebra, see e.g.
\cite[Theorem IV.2.2 (ii)]{Sickel-Pontwise-Torus} and \cite[Theorem 3.5.4.ps.168-169]{Schmeisser-Tribel-87-book} for 
$ O=\mathbb{T}^d$ and see e.g. 
\cite[Theorem 1, p. 221 and Proposition Tr 6, 2.3.5, p. 14]{R&S-96} for 
$ O\subset \mathbb{T}^d$ bounded, then we conclude \eqref{eq-B-estimatoion}.  Finally, by  density of the
space $ \mathcal{D}(O)$ in $ \mathbb{H}^{\eta, q}(O)$\del{, see the definition in Section \ref{sec-formulation},} the result is proved.
\end{proof}

\noindent The proof of Proposition \ref{Prop-Main-B} does not cover the case $ q=2$.
This is due to the fact that the results in
Lemmas \ref{Lem-classic}-\ref{Giga-Mikayawa-solutionLr-NS}  and other related tools like 
\cite[Theorem 3.5.4.ps.168-169]{Schmeisser-Tribel-87-book}
do not cover the space $ \mathbb{L}^1(O)$.
We give bellow other estimates for $ q=2$.\del{ This result is used to prove Theorem \ref{Main-theorem-martingale-solution-d}.}
\begin{prop}\label{prop-Ben-q=2}
Let $\epsilon>0$, then
the bilinear form $ B$ extends uniquely $
B:(\mathbb{L}^{2}(O))^2 \rightarrow \mathbb{H}^{-1-\epsilon-\frac d2, 2}(O)$
and there exists a constant $ c:= c_{\alpha, \epsilon, d}$ such that  for all $ (u, v) \in (\mathbb{L}^{2}(O))^2$,
\begin{equation}\label{eq-B-estimatoion-q=2}
|B(u, v)|_{H^{-1-\epsilon-\frac d2, 2}}\leq c |u|_{\mathbb{L}^{2}} |v|_{\mathbb{L}^{2}}.
\end{equation}
\end{prop}
\noindent  We omit the proof here,  as a more general one will be given
in the proof of Lemma \ref{lem-bounded-W-gamma-p}.

\begin{prop}\label{Prop-Main-I}
Assume that  $ 2< q<\infty$, $ \eta\geq 0$ and  $1 + \frac dq <\alpha\leq 2$
and  let $ u, v \in L^\infty(0, T; D(A^\frac{\eta}{2}_q))$. Then
\begin{equation}
 \int_0^. e^{-A_\alpha (.-s)}B(u(s), v(s)) ds\in L^\infty(0, T; D(A^\frac{\beta}{2}_q)),
\end{equation}
for all $ \beta $  satisfies
\begin{equation}
  0\leq \beta <\eta+ \alpha-1-\frac dq.
\end{equation}
Moreover, there exist  $ c, \mu >0$ (depend on $ d, q, \alpha, \beta$), such that
 the following estimate holds
\begin{equation}\label{est-semigroup-B-main-2}
 | A^\frac{\beta}{2}\int_0^t e^{-A_\alpha (t-s)}B(u(s), v(s)) ds|_{\mathbb{L}^q}\leq c t^{\mu}
|u|_{L^\infty(0, t; D(A^{\frac{\eta}{2}}_q))}|v|_{L^\infty(0, t; D(A^{\frac{\eta}{2}}_q))}.
\end{equation}
\end{prop}

\begin{proof}
Using Proposition \ref{Prop-Main-B} and the semigroup property \eqref{eq-semigp-property}, we infer that
\begin{eqnarray}\label{est-1-semi-group-second-term-1}
|\int_0^t e^{-A_\alpha (t-s)}B(u(s)\!\!\!\!&,& \!\!\!\!v(s))ds|_{D(A^{\frac\beta2}_q)}\nonumber\\
&\leq& C\int_0^t |A^{\frac{\beta-\eta+1+\frac dq}2} e^{-A_\alpha (t-s)}|_{\mathcal{L}(\mathbb{L}^q)}
|B(u(s), v(s))|_{\mathbb{H}^{\eta-1-\frac dq, q}} ds \nonumber\\
&\leq &  C  \int_0^t (t-s)^{-\frac{ 1}{\alpha}(\beta-\eta+1+\frac dq)}|u(s)|_{D(A_q^{\frac\eta2})}|v(s)|_{D(A_q^{\frac\eta2})}ds
\nonumber\\
&\leq &  C T^{1-\frac{ d}{\alpha}(\frac{1+\beta-\eta}{d}+\frac1q)}|u|_{L^\infty(0, T; D(A_q^{\frac\eta2}))}
|v|_{L^\infty(0, T; D(A_q^{\frac\eta2}))}.
\end{eqnarray}
\end{proof}

\begin{prop}\label{prop-est-B-eta-q=2}
Let $\eta \geq 0$ and
 \begin{equation}\label{critical-value-alpha}
\alpha(d, \eta) :=
\Bigg\{
\begin{array}{lr}
\max\{\frac{d+2-2\eta}{3},\; 2\eta+2-d\},\;\;\; \  if\; \;  \eta \in [0, \frac{d}{2})\cap 
(\frac d2-2, \frac{d}{2}), \\
1, \;\;\; if \,\,\, \eta\geq \frac d2.
\end{array}
\end{equation}
\del{\footnote{Recall $ \max^1_>\{a, b\} $ equals $ a$, if $ a>b$ and greater than $ b$ if  $ a\leq b$}}
Then for either $ \alpha \in [\alpha(d, \eta), 2)$ with $(\eta \in [0, \frac{d-1}{2})\cap 
(\frac d2-2, \frac{d-1}{2})) 
\cup [\frac d2, \infty)$ or
$ \alpha \in (\alpha(d, \eta), 2)$ with $\eta \in [\frac{d-1}{2},  \frac{d}{2})$,
 the bilinear operator $ B$ extends uniquely
$$
B:(\mathbb{H}^{\eta +\frac\alpha2, 2}(O))^2 \rightarrow \mathbb{H}^{\eta -\frac\alpha2, 2}(O)$$
and there exists a constant $ c:= c_{\alpha, \eta, d}$ such that  for all 
$ (u, v) \in (\mathbb{H}^{\eta -\frac\alpha2, 2}(O))^2$,
\begin{equation}\label{eq-B-estimatoion-q=2-eta}
|B(u, v)|_{\mathbb{H}^{\eta -\frac\alpha2, 2}}\leq c |u|_{\mathbb{H}^{\eta +\frac\alpha2, 2}} |v|_{\mathbb{H}^{\eta +\frac\alpha2, 2}}.
\end{equation}
\end{prop}

\begin{proof}
Thanks to  Lemma \ref{Lem-classic}, there exists a constant $ c>0$, such that
\begin{eqnarray}\label{eq-fact-b-u-v}
 |B(u, v)|_{\mathbb{H}^{\eta-\frac\alpha2, 2}} &\leq& c |u_jv|_{\mathbb{H}^{\eta+1-\frac{\alpha}{2}, 2}}.
\end{eqnarray}
First, let us suppose that $ \eta \geq \frac d2$. Then  $\mathbb{H}^{\eta+1-\frac\alpha2, 2}(O)$ is an algebra, 
therefore
\del{$ \eta+1-\frac{\alpha}{2} >\frac d2$, then}
\begin{eqnarray}\label{est-B-estimatoion-q=2-eta}
 |B(u, v)|_{\mathbb{H}^{\eta-\frac\alpha2, 2}} &\leq& c |u|_{\mathbb{H}^{\eta+1-\frac{\alpha}{2}, 2}} 
 |v|_{\mathbb{H}^{\eta+1-\frac{\alpha}{2}, 2}}.
\end{eqnarray}
Then  Estimate \eqref{eq-B-estimatoion-q=2-eta} follows from \eqref{est-B-estimatoion-q=2-eta} 
by using the Sobolev embedding
$ \mathbb{H}^{\eta+\frac{\alpha}{2}, 2}(O) \hookrightarrow \mathbb{H}^{\eta+1-\frac{\alpha}{2}, 2}(O)$. 
This last is guaranteed thanks to the condition
$ \alpha\geq 1= \alpha(d, \eta)$.\\
\noindent For $ 0\leq \eta < \frac d2$, we combine \eqref{eq-fact-b-u-v} and either
\cite[Theorem 4.6.1.1, Proposition Tr 6, 2.3.5]{R&S-96} (see also Appendix
\ref{Sobolev pointwise multiplication-Bounded-Domain}) for $ O\subset \mathbb{R}^d$ being a  bounded domain or
\cite[Theorem IV.2.2]{Sickel-Pontwise-Torus} and \cite[Theorem 3.5.4 \& Remark 4 p 164]{Sickel-periodic spaces-85}
for $ O= \mathbb{T}^d$, then we get,
\begin{eqnarray}\label{eq-fact-b-u-v-last-est}
|B(u, v)|_{\mathbb{H}^{\eta-\frac\alpha2, 2}} &\leq& c |u|_{\mathbb{H}^{\frac{d+2+2\eta-\alpha}{4}, 2}}
| v|_{\mathbb{H}^{\frac{d+2+2\eta-\alpha}{4}, 2}}
\end{eqnarray}
provided that $ 2\eta+2-d<\alpha$. Moreover, under the condition $ \frac{d+2-2\eta}{3}\leq \alpha$, Estimate
\eqref{eq-B-estimatoion-q=2-eta} follows from \eqref{eq-fact-b-u-v-last-est} by using the
Sobolev embedding $ \mathbb{H}^{\eta+\frac{\alpha}{2}, 2}(O) \hookrightarrow 
\mathbb{H}^{\frac{d+2+2\eta-\alpha}{4}, 2}(O)$. This achieves the proof of \eqref{eq-B-estimatoion-q=2-eta}. The intervals in the definition of $\alpha(d, \eta) $ in Formula  \eqref{critical-value-alpha}
emerge thanks to the condition $ \eta >\frac d2-2$ which guaranties that $ \frac{d+2-2\eta}{3} <2$ and to the equivalence 
$ 2+2\eta-d<\frac{d+2-2\eta}{3}\Leftrightarrow \eta< \frac{d-1}{2}$.\del{, we deduce the corresponding 
values of $ \alpha$ and $ \eta$ as cited in the proposition.}
\del{ we infer the existence of a constant $ c>0$, such that
\begin{eqnarray}
 |B(u, v)|_{\mathbb{H}^{\eta-\frac\alpha2, 2}} &\leq& c |u|_{\mathbb{H}^{\frac{d+2+2\eta-\alpha}{4}, 2}}
| v|_{\mathbb{H}^{\frac{d+2+2\eta-\alpha}{4}, 2}} \leq c |u|_{\mathbb{H}^{\eta -\frac\alpha2, 2}} |v|_{\mathbb{H}^{\eta -\frac\alpha2, 2}}.
\end{eqnarray}}
\end{proof}

\noindent The investigation of the Gelfand triple corresponding to
the fractional Navier-Stokes equation for which $ B$ can be extended to a  bounded operator,
is\del{, as mentioned in the introduction,} one of the delicate questions of the theory of fractional nonlinear equations.
\del{here, we apply Proposition \ref{prop-est-B-eta-q=2} to discuss this question.}
To characterize this feature, let us first recall the following
classical Gelfand triple
 \begin{equation}
 V_c= D((A^\frac12))=\mathbb{H}^{1, 2}(O)\hookrightarrow  H:=\mathbb{L}^2(O) \tilde{=}  H^*\hookrightarrow  V_c^*,
 \del{:= \mathbb{H}^{-1, 2}(O)}
 \end{equation}
 where $ V_c^*$ is the dual of  $ V_c$.
The operators  $ A: V_c \rightarrow  V_c^*$ and  $ B: D(B):=V_c\times  H \rightarrow  V_c^*$ are bounded.
\del{ then both the operator $ A$ and the restriction of
$ B $ on $ V_c$ (denoted also by $ B$ ) are well defined and bounded from $V_c$ to $ V_c^*$.}
For the fractional case, we have for  $ \alpha >0$,
$A_\alpha: V:=D(A_\alpha^{\frac12})=\mathbb{H}^{\frac\alpha2, 2}(O) \rightarrow 
V^*=(\mathbb{H}^{\frac\alpha2, 2}(O))^*$  is bounded.  However, for $ \alpha<2$,
the space $ V\times H $ is larger than $ D(B)$. In particular,
$ \mathbb{H}^{1}(O) \subsetneq V$. Consequently, we need to extend uniquely the operator $ B $ to a
bounded operator from  $ V$ to $ V^*$. This extension is not possible for all values of 
$ \alpha \in (0, 2)$.\del{ The result is given in the following theorem}

\begin{theorem}\label{theo-gelfand-gene}
Let  $ \alpha \in [\alpha(d, \eta), 2]$, with $\eta \in ([0, \frac{d-1}{2})\cap (\frac{d}{2}, \frac{d-1}{2})) \cup [\frac d2, \infty)$\del{$\eta \in [\max\{0, \frac{d}{2}-2\},  
\frac{d-1}{2}) \cup [\frac d2, \infty)$} or
$ \alpha \in (\alpha(d, \eta), 2]$ with $\eta \in [\frac{d-1}{2},  \frac{d}{2})$
where $ \alpha(d, \eta)$ is defined by \eqref{critical-value-alpha}.
We introduce the following  Gelfand triple
\begin{equation}\label{gelfant-triple-eta}
 V_\eta:= \mathbb{H}^{\eta+\frac\alpha2, 2}(O)\hookrightarrow \mathbb{H}^{\eta, 2}(O)\hookrightarrow
 \mathbb{H}^{\eta-\frac\alpha2, 2}(O).
\end{equation}
Then
\begin{equation}
 B: V_\eta\times V_\eta:= (D(A^{\frac\eta2+\frac\alpha4}))^2= 
 (\mathbb{H}^{\eta+\frac\alpha2, 2}(O))^2\rightarrow V_\eta^* =
\mathbb{H}^{\eta-\frac\alpha2, 2}(O).
\end{equation}
 is bounded. Moreover, there exists a constant $ c:=c_{d, \alpha, \eta}>0$, such that
\begin{equation}\label{Eq-B-u-v-V-dual}
|B(u, v)|_{\mathbb{H}^{\eta-\frac\alpha2, 2}} \leq  c |u|_{\mathbb{H}^{\eta+\frac\alpha2, 2}}
|v|_{\mathbb{H}^{\eta+\frac\alpha2, 2}}.
\end{equation}
In particular, we have the following useful cases,
\begin{itemize}
\item  $ \eta =0$, $ d\in \{2, 3, 4\}$ and $ \alpha\in [\frac{d+2}{3}, 2]$.
\item  $ \eta =1$ and either  $ d=2$  and  $ \alpha \in [1, 2] $ or  $d=3$ and  $ \alpha \in (1, 2]$ or
$d\in \{4, 5, 6\}$ and  $ \alpha \in [\frac{d}{3}, 2]$.
\end{itemize}
\end{theorem}
\begin{proof}
 The proof is a straightforward application of Proposition \ref{prop-est-B-eta-q=2} and 
the classical case for $ \alpha =2$.
\end{proof}

\del{\noindent We mention here the following useful  cases,
\begin{coro}
\begin{itemize}
 \item For $ \eta =0$, $ d\in \{2, 3\}$, $ \alpha_0(d) := \alpha_0(d, 0) =\frac{d+2}{3}\leq \alpha <2 $, we consider the Gelfant triple
\begin{equation}\label{gelfant-triple-eta=0}
 \mathbb{H}^{\frac\alpha2, 2}(O)\hookrightarrow \mathbb{L}^{ 2}(O)\hookrightarrow\mathbb{H}^{-\frac\alpha2, 2}(O).
\end{equation}
Then the operator $ B $ is extended uniquely to a bounded operator
\begin{equation}
 B:  (V_0:= \mathbb{H}^{\frac\alpha2, 2}(O))^2\rightarrow V_0^* =
\mathbb{H}^{-\frac\alpha2, 2}(O).
\end{equation}
\item For $ \eta =1$, we consider the Gelfant triple
\begin{equation}\label{gelfant-triple-eta=1}
 \mathbb{H}^{1+\frac\alpha2, 2}(O)\hookrightarrow \mathbb{H}^{1, 2}(O)\hookrightarrow\mathbb{H}^{1-\frac\alpha2, 2}(O).
\end{equation}
Then for either  $ d=2$  and  $ \alpha \in [1, 2) $ or  $d=3$ and  $ \alpha \in (1, 2) $ or
$d\in \{4, 5\}$ and  $ \alpha \in (\frac{d}{3}, 2) $
the operator $ B $ is extended uniquely to a bounded operator.
\end{itemize}
\end{coro}}

\del{
In the following lemma, we give further estimations
for the nonlinear term and for the trilinear form
\begin{lem}\label{lem-B-different-q}
\begin{itemize}
 \item $(i)$ For $ 0\leq \eta <\frac{d}{2}$ and $ \alpha \in (2+2\eta-d, 2]$ there exists a constant $ c>0$, such that
\begin{equation}\label{main-est-b-B-eta-1}
 |B(u, v)|_{\mathbb{H}^{\eta-\frac\alpha2, 2}}\leq c |u|_{\mathbb{H}^{\frac{d+2+2\eta-\alpha}{4}, 2}}
| v|_{\mathbb{H}^{\frac{d+2+2\eta-\alpha}{4}, 2}}.
\end{equation}
 \item $(i')$ Let $ d\in \{2, 3\}$, $ 1<q<\frac{d}{d-1}$  and $ 0<\alpha\leq 2$. Then, there exists a constant $ c>0$, such that
\begin{equation}
 |B(u, v)|_{\mathbb{H}^{-\frac\alpha2, q}}\leq c |u|_{\mathbb{H}^{\frac\alpha2, 2}} | v|_{\mathbb{H}^{1-\frac\alpha2, 2}},
\end{equation}
provided $ \frac dq-d+\frac\alpha2>0$.
\item $(ii)$ For $ \frac{d-2}{3}\leq \alpha<2$,
\begin{eqnarray}\label{est-nonlinear-general}
| B(u, v)|_{{\mathbb{H}^{1-\frac{\alpha}{2}, 2}} } &\leq & | u|_{{\mathbb{H}^{\frac{d+2-\alpha}4, 2}}}| v|_{{\mathbb{H}^{\frac{d+2-\alpha}4, 2}}}.
\end{eqnarray}
\del{\item $(iii)$ For $ \frac{d-2}{3}\leq \alpha<2$ and $ \frac d2 <s<\alpha$,
\begin{eqnarray}\label{est-nonlinear-general}
| \langle B(u), v\rangle|_{{\mathbb{H}^{-\frac{\alpha}{2}, 2}} } &\leq & | u|_{{\mathbb{H}^{1-\frac\alpha2, 2}}}| v|_{{\mathbb{H}^{\frac{d+2-\alpha}4, 2}}}.
\end{eqnarray}}
\item $(iii)$ For $ \frac{d-2}{3}\leq \alpha<2$ and $ u\in \mathbb{H}^{1+\frac\alpha2, 2}$,
\begin{eqnarray}\label{est-tree-linear-H1}
|\langle B(u), u\rangle_{\mathbb{H}^{1, 2}}|&\leq & c\big(| u|_{{\mathbb{H}^{1+\frac\alpha2, 2}}}^2 + | u|^2_{{\mathbb{H}^{1, 2}}}\big).
\end{eqnarray}
\item $(iv)$  Let $ 1\leq \alpha \leq 2$. For all $ (u, w)\in \mathbb{H}^{1, 2}(O)\times \mathbb{H}^{\frac\alpha2, 2}(O)$,
\begin{eqnarray}\label{3linear-H1-H-1}
|\langle B(w), u\rangle_{\mathbb{L}^{2}}|&\leq & c| u|_{\mathbb{H}^{1, 2}}|w|^{\frac2\alpha}_{\mathbb{H}^{\frac{\alpha}{2}, 2}}|w|^{2\frac{1-\alpha}\alpha}_{{\mathbb{L}^{2}}}.
\end{eqnarray}
\end{itemize}
\end{lem}

\begin{proof}
\begin{itemize}
 \item $(i)$ Using Lemma \ref{Lem-classic} and
\cite[Theorem 4.6.1.1, Proposition Tr 6, 2.3.5]{R&S-96} for bounded domain (see also Appendix
\ref{Sobolev pointwise multiplication-Bounded-Domain}). For the Torus, we use again Lemma \ref{Lem-classic} and
\cite[Theorem IV.2.2]{Sickel-Pontwise-Torus} and \cite[Theorem 3.5.4 \& Remark 4 p 164]{Sickel-periodic spaces-85}, we infer the existence of
 a constant $ c>0$, such that
\begin{eqnarray}
 |B(u, v)|_{\mathbb{H}^{\eta-\frac\alpha2, 2}} &\leq& c |u_jv|_{\mathbb{H}^{1+\eta-\frac{\alpha}{2}, 2}}
\leq c |u|_{\mathbb{H}^{\frac{d+2+2\eta-\alpha}{4}, 2}}
| v|_{\mathbb{H}^{\frac{d+2+2\eta-\alpha}{4}, 2}}.
\end{eqnarray}

 \item $(i')$ Let $ 0\leq \eta <\frac{d}{2}$ and $ \alpha \in (2+2\eta-d, 2$, Using Lemmas \ref{Lem-classic}, \cite[Theorem 1.4.4.4 ]{R&S-96} and the condition $ \frac dq-d+\frac\alpha2>0$, we infer that
\begin{eqnarray}\label{Eq-B-H-alpha-2-est-d}
|B(u, v)|_{\mathbb{H}^{-\frac\alpha2, q}}&\leq& c |u_iv_j|_{\mathbb{H}^{1-\frac\alpha2, q}} \leq c |u_i|_{\mathbb{H}^{\frac\alpha2, 2}}
 | v_j|_{\mathbb{H}^{1-\frac\alpha2, 2}}.
\end{eqnarray}
\item $(ii)$ Thanks to the condition $ \frac{d-2}{3}\leq \alpha<2$, Estimation in \eqref{est-nonlinear-general} follows from
\cite[Theorem 4.6.1.1 (iii)p 190 \& Proposition Tr 6, 2.3.5 (vii)]{R&S-96}, \cite[p.177]{Adams-Hedberg-94} and Appendix
\ref{Sobolev pointwise multiplication-Bounded-Domain} in the case of
$ O\subset \mathbb{R}^d$ is bounded and \cite[Theorem iv.2.2 (ii)]{Sickel-Pontwise-Torus},
\cite[Theorem 3.5.4 (v) p. 108 \& Remark 4 p. 104]{Schmeisser-Tribel-87-book} for the case $ O= \mathbb{T}^d$.
\item $(iii)$ Using Lemma \ref{Lem-classic},\del{ \cite[Theorem 4.6.1.1]{R&S-96} (see also Appendix \ref{Sobolev pointwise multiplication-Bounded-Domain} )}
the fact that $ \mathbb{H}^{2-\frac\alpha2, 2}$ is a multiplicative algebra
\begin{eqnarray}\label{est-tree-linear-H1-first}
|\langle B(u), u\rangle_{\mathbb{H}^{1, 2}}|&\leq &| u|_{{\mathbb{H}^{1+\frac\alpha2, 2}}}^2| B(u)|_{{\mathbb{H}^{1-\frac\alpha2, 2}}}
\leq  c| u|_{{\mathbb{H}^{1+\frac\alpha2, 2}}}\sum_{j=1}^d| u^ju|_{{\mathbb{H}^{2-\frac\alpha2, 2}}}
\leq c| u|_{{\mathbb{H}^{1+\frac\alpha2, 2}}}|u|_{{\mathbb{H}^{2-\frac\alpha2, 2}}}^2,
\end{eqnarray}
where $ \max\{\frac d2, 2-\frac\alpha2\}<s<\alpha$. Recall that $ s$ exists thanks to the conditions >>>>>>>>>>>>>>>>>
\item $(iv)$  Let $ 1\leq \alpha \leq 2$ and  $ (u, w) \in (\mathcal{D}(O))^2$. Recall that  $ (u, w) \in (\mathcal{D}(O))^2$ is dense in
$\mathbb{H}^{1, 2}(O)\times \mathbb{H}^{\frac\alpha2, 2}(O)$. Using Lemma \ref{Lem-classic},
 H\"older inequality and  Gagliardo-Nirenberg inequality, we infer
\begin{eqnarray}\label{3linear-H1-H-1-est-1}
|\langle B(w), u\rangle_{\mathbb{L}^{2}}|&\leq &
| u|_{{\mathbb{H}^{1, 2}}}|B(w)|_{{\mathbb{H}^{-1, 2}}}\leq
c| u|_{{\mathbb{H}^{1, 2}}}|w_jw|_{{\mathbb{L}^{2}}}\leq
c| u|_{{\mathbb{H}^{1, 2}}}|w|^2_{{\mathbb{L}^{4}}}\leq
c| u|_{\mathbb{H}^{1, 2}}|w|^{\frac2\alpha}_{\mathbb{H}^{\frac{\alpha}{2}, 2}}|w|^{2\frac{\alpha-1}\alpha}_{{\mathbb{L}^{2}}}.\nonumber\\
\end{eqnarray}
\end{itemize}
\end{proof}

\del{Scours:
Using Lemma \ref{Lem-classic} and
\cite[Theorem 4.6.1.1, Proposition Tr 6, 2.3.5]{R&S-96} for bounded domain (see also Appendix
\ref{Sobolev pointwise multiplication-Bounded-Domain}). For the Torus, we use again Lemma \ref{Lem-classic} and
\cite[Theorem IV.2.2]{Sickel-Pontwise-Torus} and \cite[Theorem 3.5.4 \& Remark 4 p 164]{Sickel-periodic spaces-85}}

\subsection{Applications: the Gelfand triple corresponding to the FSNSE}
The notion of Gelfand triple is  one of the main tools in the study of nonlinear partial differential equations,
in particular in the study of deterministic and stochastic
Navier-Stokes  equation and in the application of the  monotonicity method see e.g \cite{Temam-NS-Main-79}.
To characterize the specificity of the  Gelfand triple corresponding to
the fractional Navier-Stokes  equation, let us first recall the following
classical Gelfand triple
 \begin{equation}
 V_c= D((A^\frac12))\hookrightarrow  H:=\mathbb{L}^2(O) \tilde{=}  H^*\hookrightarrow  V_c^*:= \mathbb{H}^{-1, 2}(O).
 \end{equation}
The operators  $ A: V_c \rightarrow  V_c^*$ and  $ B: V_c\times  V_c \rightarrow  V_c^*$ are bounded.
\del{ then both the operator $ A$ and the restriction of
$ B $ on $ V_c$ (denoted also by $ B$ ) are well defined and bounded from $V_c$ to $ V_c^*$.}
For the fractional case, we have for  $ \alpha >0$,
$A_\alpha: V=D(A_\alpha^{\frac12}) \rightarrow V^*=H^{-\frac\alpha2, 2}(O)$  is bounded.  However for $ \alpha<2$,
the space $ V $ is larger than the domain of definition of $ B$. In particular,
$ \mathbb{H}^{1}(O) \subsetneq V$. Consequently, we need to extend uniquely the operator $ B $ to a
bounded operator from  $ V$ to $ V^*$. This  is the content of the following proposition

\begin{theorem}\label{theo-gelfand-gene}
Let $ d\in \{2, 3\}$, $ 0\leq \eta\frac{d-4}{2}$ and\del{ such that $ \frac{d+2-6\eta}{3}< 2$. We define}
\del{\begin{equation}
\alpha_G(d, \eta) :=
\Bigg\{
\begin{array}{lr}
\frac{d+2-6\eta}{3},\;\;\; \  if\; \;  d+2-6\eta \geq 0, \\
0, \;\;\; otherwise.
\end{array}
\end{equation}}
\begin{equation}\label{alpha-critical}
\alpha_G(d, \eta) := 1+ \frac{d-1+2\eta}{3}.
\end{equation}
Assume $\alpha \in [\alpha_G(d, \eta), 2]$.
Then the operator $ B $ is extended uniquely to a bounded operator
(we keep the same notation)
\begin{equation}
 B: V_\eta:= D(A^{\frac\eta2+\frac\alpha4})= \mathbb{H}^{\eta+\frac\alpha2, 2}(O)\rightarrow V_\eta^* = D(A^{\frac\eta2-\frac\alpha4})=
\mathbb{H}^{\eta-\frac\alpha2, 2}(O).
\end{equation}
\end{theorem}

\begin{proof}
\del{Using Lemma \ref{Lemma-bound-op-RGradient-Via-gamma} with $ \delta := \frac\eta2 +\frac{\alpha}{4}$,
$ \rho=\nu = \frac{d+2--2\eta-\alpha}{8}$ and $ p= 2$,  we conclude
the existence of a constant $ c>0$, such that
\begin{equation}\label{Eq-B-u-v-V-dual}
|B(u, v)|_{V_\eta^*} \leq  c|u |_{D(A^{\frac{d+2-2\eta- \alpha}4})} |v|_{D(A^{\frac{d+2 -2\eta-\alpha}4})}.
\end{equation}
Thanks to the condition $ \alpha \geq \alpha_0(d)$, we get
\begin{equation}\label{Eq-B-u-v-V-dual}
|B(u, v)|_{V_\eta^*} \leq  c |u|_{V_\eta}|v|_{V_\eta}.
\end{equation}}

\noindent Using Estimation \eqref{main-est-b-B-eta-1} and the Sobolev embedding ( remark that ) we conclude
the existence of a constant $ c>0$, such that
\begin{equation}\label{Eq-B-u-v-V-dual}
|B(u, v)|_{V_\eta^*}\leq  c|u |_{D(A^{\frac{d+2-2\eta- \alpha}4})} |v|_{D(A^{\frac{d+2 -2\eta-\alpha}4})}.
\end{equation}
Thanks to the condition $ \alpha \geq \alpha_0(d)$, we get
\begin{equation}\label{Eq-B-u-v-V-dual}
|B(u, v)|_{V_\eta^*} \leq  c |u|_{V_\eta}|v|_{V_\eta}.
\end{equation}
\end{proof}
In particular, we have
\begin{coro}
\begin{itemize}
 \item For $ d\in \{2, 3\}$, $ \alpha_G(d) := \alpha_G(d, 0) =\frac{d+2}{3}\leq \alpha <2 $,
the operator $ B $ is extended uniquely to a bounded operator
\begin{equation}
 B:  V_\eta:= D(A^{\frac\alpha4})= \mathbb{H}^{\frac\alpha2, 2}(O)\rightarrow V_\eta^* = D(A^{-\frac\alpha4})=
\mathbb{H}^{-\frac\alpha2, 2}(O).
\end{equation}
\item For $ d\in \{2, 3, \cdots, 9\}$, $ \alpha_G(d, 1) := \min\{0, \frac{d-4}{3}\}\leq \alpha <2 $, the operator $ B $
is extended uniquely to a bounded operator
\begin{equation}
 B:  V_\eta:= D(A^{\frac12+\frac\alpha4})= \mathbb{H}^{1+\frac\alpha2, 2}(O)\rightarrow V_\eta^* = D(A^{-\frac12-\frac\alpha4})=
\mathbb{H}^{-1-\frac\alpha2, 2}(O).
\end{equation}
\end{itemize}
\end{coro}}

\noindent Further estimations for the bilinear operator $ B$ and the trilinear form $ b$ are summarized in the following lemma
\begin{lem}\label{lem-further-estimation}
\begin{itemize}

\item $(i)$ Assume $0\leq  \eta <\frac d2$, and $ \alpha \in (2\eta+2-d, 2]$. Then for all  
$ (u, v)\in (\mathbb{H}^{\frac{d+2+2\eta-\alpha}{4}, 2}(O))^2$, we have 
 \begin{eqnarray}\label{B-u-v-h-alpha-2-d}
|B(u, v)|_{\mathbb{H}^{\eta-\frac\alpha2, 2}} &\leq& c |u|_{\mathbb{H}^{\frac{d+2+2\eta-\alpha}{4}, 2}}
| v|_{\mathbb{H}^{\frac{d+2+2\eta-\alpha}{4}, 2}}.
\end{eqnarray}
\item $(ii)$ For  $ \eta \geq \frac d2$ and   $u, v \in \mathbb{H}^{\eta+1-\frac\alpha2, 2}(O)$, we have
\begin{eqnarray}\label{est-B-estimatoion-q=2-eta-to-use}
 |B(u, v)|_{\mathbb{H}^{\eta-\frac\alpha2, 2}} &\leq& c |u|_{\mathbb{H}^{\eta+1-\frac{\alpha}{2}, 2}} 
 |v|_{\mathbb{H}^{\eta+1-\frac{\alpha}{2}, 2}}.
\end{eqnarray}
\item $(iii)$ Assume $ d \in \{2, 3, 4\}$ and  $ \frac d2\leq \alpha\leq 2$. For all 
$ (u, w)\in \mathbb{H}^{1, 2}(O)\times \mathbb{H}^{\frac\alpha2, 2}(O)$,
\begin{eqnarray}\label{3linear-H1-H-1}
|\langle B(w), u\rangle_{\mathbb{L}^{2}}|&\leq & c| u|_{\mathbb{H}^{1, 2}}
|w|^{\frac d{\alpha}}_{\mathbb{H}^{\frac{\alpha}{2}, 2}}
|w|^{\frac{2\alpha-d}\alpha}_{{\mathbb{L}^{2}}}.
\end{eqnarray}
\item $(iv)$   Assume $ d \in \{2,\cdots, 5\}$ and  $ \frac d3\leq \alpha<d$. 
For all $ (u, w)\in \mathbb{H}^{1+\frac\alpha2, 2}(O)\times \mathbb{H}^{\frac\alpha2, 2}(O)$,
\begin{eqnarray}\label{3linear-H1+alpha2-H-1}
|\langle B(w), u\rangle_{\mathbb{L}^{2}}|&\leq & c| u|_{\mathbb{H}^{1+\frac\alpha2, 2}}|w|^{\frac{d-\alpha}{\alpha}}_{\mathbb{H}^{\frac{\alpha}{2}, 2}}
|w|^{\frac{3\alpha-d}{\alpha}}_{{\mathbb{L}^{2}}}.
\del{|\langle B(w), u\rangle_{\mathbb{L}^{2}}|&\leq & c| u|_{\mathbb{H}^{1+\frac\alpha2, 2}}|w|^{\frac2\alpha-1}_{\mathbb{H}^{\frac{\alpha}{2}, 2}}
|w|^{\frac{3\alpha-2}\alpha}_{{\mathbb{L}^{2}}}.}
\end{eqnarray}
\item $(v)$  For all $ (u, v) \in \mathbb{H}^{\frac d{4q}, q}(O)$ with $ q>2$,
\begin{equation}\label{eq-est-H-1-d-q}
 |B(u, v)|_{\mathbb{H}^{-1, q}}\leq c |u|_{\mathbb{H}^{\frac d{2q}, q}}|v|_{\mathbb{H}^{\frac d{2q}, q}}. 
\end{equation}

\item $(vi)$  The following estiamte is a classical result. For all $ u\in \mathbb{H}^{1, 2}(O)$,
\begin{equation}\label{classical-B-H-1}
 |B(u)|_{\mathbb{H}^{-1, 2}}\leq c |u|_{\mathbb{H}^{1, 2}}|u|_{\mathbb{L}^{2}}. 
\end{equation}
\end{itemize}
\end{lem}

\begin{proof}

\begin{itemize}
\item $ (i)$-$(ii)$ Estimates \eqref{B-u-v-h-alpha-2-d} and \eqref{est-B-estimatoion-q=2-eta-to-use} are copies  of the estimates \eqref{eq-fact-b-u-v-last-est} respectively \eqref{est-B-estimatoion-q=2-eta} proved above without restriction conditions. We only 
emphasize them here.\del{ as they will be needed later, in particular, for $ \eta=0$ and $ \eta=1$.}
\item $(iii)$  Let  $ d \in \{2, 3, 4\}$,  $ \frac d2\leq \alpha\leq 2$ and  $ (u, w) \in (\mathcal{D}(O))^2$. Recall that  $ (u, w) \in (\mathcal{D}(O))^2$ is dense in
$\mathbb{H}^{1, 2}(O)\times \mathbb{H}^{\frac\alpha2, 2}(O)$. Using the group property\del{ of 
the powers of Stokes operator} $ (A^\beta)_{\beta\in\mathbb{R}}$,
H\"older inequality, Lemma \ref{Lem-classic},
 again H\"older inequality and  Gagliardo-Nirenberg inequality in the case $ \alpha>\frac d2$ and Sobolev embedding 
 in the case $ \alpha=\frac d2$, see e.g. \cite[Theorem 7.63 \& 7.66]{Adams-Hedberg-94} for bounded domain and 
 \cite[Theorem 3.5.4. \& Theorem 3.5.5.]{Schmeisser-Tribel-87-book} for the torus, we infer that
\begin{eqnarray}\label{3linear-H1-H-1-proof}
|\langle B(w), u\rangle_{\mathbb{L}^{2}}|&\leq &
| u|_{{\mathbb{H}^{1, 2}}}|B(w)|_{{\mathbb{H}^{-1, 2}}}\leq
c| u|_{{\mathbb{H}^{1, 2}}}|w_jw|_{{\mathbb{L}^{2}}}
\leq c| u|_{{\mathbb{H}^{1, 2}}}|w|^2_{{\mathbb{L}^{4}}} \nonumber\\
&\leq & c| u|_{\mathbb{H}^{1, 2}}
|w|^{\frac d{\alpha}}_{\mathbb{H}^{\frac{\alpha}{2}, 2}}
|w|^{\frac{2\alpha-d}\alpha}_{{\mathbb{L}^{2}}}.
\end{eqnarray}
\item $(iv)$  Let  $ d \in \{2, \cdots, 5\}$, $ \frac d3\leq \alpha<d$ and  $ (u, w) \in (\mathcal{D}(O))^2$. 
Using the group property of $ (A^\beta)_{\beta\in\mathbb{R}}$, H\"older inequality, Lemma \ref{Lem-classic},
 \cite[Theorem 4.6.1]{R&S-96} for the bounded domain and \cite[Theorem iv.2. ii]{Sickel-Pontwise-Torus} and 
\cite[Remark 4 p 164]{Schmeisser-Tribel-87-book} for $ O=\mathbb{T}^d$ and by interpolation, we infer that
\begin{eqnarray}\label{3linear-H1+alpha-H-1-proof}
|\langle B(w), u\rangle_{\mathbb{L}^{2}}|&\leq &
| u|_{{\mathbb{H}^{1+\frac\alpha2, 2}}}|B(w)|_{{\mathbb{H}^{-1-\frac\alpha2, 2}}}\leq
c| u|_{{\mathbb{H}^{1+\frac\alpha2, 2}}}|w_jw|_{{\mathbb{H}^{-\frac\alpha2, 2}}}
\leq c| u|_{{\mathbb{H}^{1+\frac\alpha2, 2}}}|w|^2_{{\mathbb{H}^{\frac{d-\alpha}4, 2}}} \nonumber\\
&\leq& 
c| u|_{\mathbb{H}^{1+\frac\alpha2, 2}}|w|^{\frac{d-\alpha}{\alpha}}_{\mathbb{H}^{\frac{\alpha}{2}, 2}}
|w|^{\frac{3\alpha-d}{\alpha}}_{{\mathbb{L}^{2}}}.
\end{eqnarray}
\item $(v)$ The proof of \eqref{eq-est-H-1-d-q}, follows from the first two esimates in 
\eqref{est-1-semi-group-div-delta2not2-B-+2}, with $ \eta=\frac dq$. 

\item $(vi)$ We use Lemma \ref{Lemma-bound-op-RGradient-Via-gamma}, H\"older inequality and  than Gaglairdo-Nirenberg inequality, we get 
 \begin{eqnarray}
 |B(u)|_{\mathbb{H}^{-1, 2}}&\leq& c |u_ju|_{\mathbb{L}^{2}} \leq c |u|^2_{\mathbb{L}^{4}}\leq 
|u|_{\mathbb{H}^{1, 2}}|u|_{\mathbb{L}^{2}}.
\end{eqnarray}
\end{itemize}

\end{proof}}


\del{\section{2D-Fractional stochastic Navier-Stokes equation on the Torus with smooth data}\label{sec-Torus}}

\del{\section{Local mild solutions for the multi-dimensional FSNSEs.}\label{sec-1-approx-local-solution}

In this Section, we prove\del{ the first part of Theorem \ref{Main-theorem-mild-solution-d}} (\ref{Main-theorem-mild-solution-d}.1). The scheme to prove
 the existence of maximal local mild solutions for nonlinear deterministic or stochastic partial differential equations is now standard, see e.g. \cite{DaPrato-Zbc-92,  Neerven-Evolution-Eq-08}. The key idea is to apply a fixed point theorem or equivalently  Picard iterations. In \cite{Kunze-Neereven-Cont-parm-reaction-diffusion-2012},  the authors  constructed  an approximative scheme\del{ to prove the existence of a local solution} for an abstract semilinear stochastic evolution equation on a Banach space. The scheme is based on the assumption that the linear, the nonlinear and the diffusion terms of this equation\del{, $ A, B, G$,} can be approximated in appropriate way.  In particular, they applied this scheme on some\del{ to prove the global existence for the} reaction-diffusion equations.\del{ and other related models.} In the deterministic case, one of the elegant schemes for NSE is based on the approximation of the integral representation of the equation, more precisely of the mild solution, see e.g. \cite{ Farwig-Sohr-L-p-theory-2005, Lemari-book-NS-probelems-02}. It is resumed in finding a space $\mathcal{E}_T$, called the admissible path space on which  the 
bilinear operator $ B_*$ defined by
 \begin{eqnarray}\label{B-*}
B_*\!\!\!\!\!\!\!\!\!\!&:&\!\!\!\!\!\!\!  \mathcal{E}_T \times  \mathcal{E}_T \rightarrow  \mathcal{E}_T \nonumber\\
&(u, v)& \mapsto \int_0^{\cdot}e^{-(\cdot-s)A_\alpha}B(u(s), v(s))ds = \int_0^{\cdot}e^{-(\cdot-s)A_\alpha}\Pi((u(s)\nabla) v(s))ds
 \end{eqnarray}
is locally bounded. The associated space of admissible initial conditions $E_T$ is then defined by
$ u_0 \in \mathcal{D}'(O)$ such that $(e^{-tA_\alpha}u_0, t\in[0, T])\in \mathcal{E}_T$.
The space $E_T$ is called adapted value space.\del{To the best knowledge of the author, such scheme does not exist for the stochastic case.} Our aim here is to extend this scheme for the stochastic case, in particular for the stochastic Navier-Stokes equation. We shall use the same notations, but we shall delete 
the terminologies "path space"  and "adapted", as these latter are used, in the stochastic framework, for other probabilistic notions. 
We construct the admissible space $\mathcal{E}_T$
 such as the operators $ B_*$ and $ G_*$ are locally  bounded, where $ B_*$ is given by \eqref{B-*} and $ G_*$ is defined by 
 \begin{eqnarray}\label{G-*}
G_*\!\!\!\!\!\!&:&\!\!\!\mathcal{E}_T \rightarrow  \mathcal{E}_T \nonumber\\
&u& \!\!\!\!\mapsto \int_0^{\cdot}e^{-(\cdot-s)A_\alpha}G(u(s)W(ds).
 \end{eqnarray}
\noindent For the fixed parameters $ 2< p <\infty$, $ 2< q<\infty$, $ \delta\geq 0$ and $ T>0$, we introduce the 
admissible space
\begin{eqnarray}\label{E-cal-T}
\mathcal{E}_T =\{\!\!\!&{}&\!\!\! (u(t), t\in [0, T]), \;\; D(A_q^{\frac {\delta}2})-\text{continuous} \; \mathcal{F}_t-\text{adapted stochastic
processes},  s.t.\nonumber\\
  &{}&|u |_{\mathcal{E}_{T}}:= \left(\mathbb{E}\sup_{[0, T]}|u(t)|^p_{D(A_q^{\frac{\delta}2})}\right)^\frac 1p<\infty\}
 \end{eqnarray}
and the admissible initial value space 
\begin{equation}\label{E-T}
 E_T:=\{\text{random variables}\; u_0 \text{ with values in} \; \mathcal{D}'(O), \; s.t.\;\; \del{(e^{-tA_\alpha}u_0)_{t\in[0, T]}\in \mathcal{E}_T e^{-\cdot A_\alpha}u_0 } (e^{-tA_\alpha}u_0, t\in[0, T])\in \mathcal{E}_T\}.
\end{equation}
\noindent Now, it is easy, using Lemma \ref{lem-est-z-t}, Propositions \ref{Prop-Main-I}, with $ \beta=\eta=\delta$ and Assumption $(\mathcal{B})$
to prove the existence of a maximal local mild solution. The time and space regularity of the local mild solution are consequences of the construction of the solution as an element of the space \eqref{E-cal-T}. For the convenience of the reader, we give details of this proof in Appendix \ref{appendix-local-solution}, see also \cite{Debbi-scalar-active}.
\begin{remark}\label{Rem-initial-data-mild-solu}
It is obvious that, in the frame work presented above, Assumption $(\mathcal{B})$ is not optimal and may be replaced by the condition $ u_0\in E_T$, where $ E_T$ is given by \eqref{E-T}.
\end{remark}

\section{2D-FSNSEs on the Torus with smooth data.}\label{sec-Torus}
\del{\subsection{Introduction}}  In this section, we assume\del{ study the fractional stochastic Navier-Stokes equations on the 
2D-Torus,} $ O = \mathbb{T}^2$ and prove Theorem \ref{Main-theorem-strog-Torus}. As mentioned above, we have $ (A^S)^{\frac\alpha2}= (-\Delta)^\frac\alpha2$, $0<\alpha\leq 2$. Furthermore, $ (A^S)^{\frac\alpha2}$ can  also be
defined by \eqref{construction-of-fract-bounded} and  \eqref{basis} with the explicit orthonormal
basis of eigenvectors $ (e_k(\cdot):= \frac{k^\perp}{|k|}e^{ik\cdot})_{k\in \mathbb{Z}_0^d}$ and 
$(\langle v, e_k\rangle :=\hat{v}_{k})_{k\in \mathbb{Z}_0^d}$ being the sequence of Fourier coefficients, 
see e.g. \cite[p316]{Temam-Inf-dim-88}.\del{This last equality could also be deduced easily by using an orthonormal
basis of eigenvectors of $ A$, e.g. $ (e_k(\cdot):= \frac{k^\perp}{|k|}e^{ik\cdot})_{k\in \mathbb{Z}_0^d}$
and formulas \eqref{construction-of-fract-bounded} and  \eqref{basis}, see for similar remark \cite[p316]{Temam-Inf-dim-88}.
In this case  $(\langle v, e_k\rangle :=\hat{v}_{k})_{k\in \mathbb{Z}^d}$
is the sequence of Fourier coefficients). Therefore, the fractional version of the stochastic Navier-Stokes equation and
the stochastic version of the fractional Navier-Stokes equation are equivalent.}


For $ \alpha \in [1, 2]$, we fix the densely,
continuous embedding Gelfand triple \eqref{Gelfand-triple-Torus} and 
we use the following Faedo-Galerkin approximation.
Let us fix $ n\geq 1$ and introduce the projection $ P_n$, $ n\geq 1$ on the finite space 
$ H_n \subset \mathbb{L}^2(\mathbb{T}^2)$ generated by $\{e_k, k\in \mathbb{Z}_0^d,\;s.t.\; |k|\leq n\}$.
\del{We define,
\begin{equation}\label{eq-def-B-n-G-n}
 B_n(u):= \mathcal{X}_n(|u|_{\mathbb{H}^1})B(u), \;\;\; \text{and}\;\;\;\;  G_n(u):= \mathcal{X}_n(|u|_{\mathbb{H}^1})G(u).
\end{equation}}
The Faedo-Galerkin approximation scheme is defined for the process $(u_n(t), t\in [0, T])\in H_n $ by
\begin{equation}\label{FSBE-Galerkin-approxi}
\Bigg\{
\begin{array}{lr}
du_n(t)= (-A_\alpha u_n(t) + P_nB(u_n(t))dt + P_nG(u_n(t))\,dW_n(t), \; 0< t\leq T,\\
u(0) = P_nu_0 = u_{0n},
\end{array}
\end{equation}
where $ W_n(t):= P_nW(t)= \sum_{|j|\leq n} Q^\frac12e_j\beta_j(t)$ (for example 
$ (\sum_{|j|\leq n} q^\frac12_je_j\beta_j(t))$). Let us mention here that using a similar proof as in 
 \cite{Millet-Chueshov-Hydranamycs-2NS-10, Flandoli-3DNS-Dapratodebussche-06}, we can prove that 
 $ W_n(t)$ converges to $ W(t)$  in  the space
$ L^2(\Omega, \mathbb{H}^{1, 2}(\mathbb{T}^2))$, provided  $ A^\frac12Q^\frac12$ is a  Hilbert-Schmidt in
$ \mathbb{L}^2(\mathbb{T}^2)$.
Since the finite dimensionnal space stochastic differential equation  \eqref{FSBE-Galerkin-approxi} has locally  Lipschitz and  linear growth coefficients, then  Equation \eqref{FSBE-Galerkin-approxi} admits a unique strong solution
$(u_n(t), t\in [0, T])\in L^2(\Omega; C([0, T]; H_n) )$, see e.g. \cite{Millet-Chueshov-Hydranamycs-2NS-10, Krylov-simple-proof-90, Roeckner-Zhang-tamedNS-12} and the reference therien. Furthermore, we have the following result,
\begin{lem}\label{lem-unif-bound-theta-n-H-1}
Let $ \alpha\in (0, 2)$ and $ u_0\in L^{p}(\Omega, \mathbb{H}^{1, 2}(\mathbb{T}^d))$  with $ p\geq 4$. Then
the solutions $ (u_n(t), t\in [0, T])$ of equations \eqref{FSBE-Galerkin-approxi},  $ n\in \mathbb{N}_0$,
satisfy the following estimates 
 \begin{eqnarray}\label{Eq-Ito-n-weak-estimation-1-Torus}
\sup_{n}\mathbb{E}\Big(\sup_{[0, T]}|u_n(t)|^{p}_{\mathbb{H}^{1, 2}}&+&
\int_0^T|u_n(t)|^{p-2}_{\mathbb{H}^{1, 2}}\Big(|u_n(t)|^2_{\mathbb{H}^{1+\frac\alpha2, 2}} +
|u_n(t)|_{\mathbb{H}^{\beta, q_1}}^2 \Big)dt \nonumber\\
&+& \int_0^T|u_n(t)|^4_{\mathbb{H}^{1, 2}}dt+ \int_0^T|u_n(t)|^{\frac{\alpha}{\eta}}_{\mathbb{H}^{1+\eta, 2}}dt\Big)<\infty,
\end{eqnarray}
where $ \beta \leq 1+\frac\alpha2-\frac d2+\frac{d}{q_1}$, $ 2\leq q_1<\infty$ and $ \frac\alpha p<\eta\leq \frac\alpha2$.
\begin{eqnarray}\label{Eq-B-n-weak-estimation-1-Torus}
\sup_{n}\left(\mathbb{E}\int_0^T (|P_nB(u_n(t))|^{2}_{\mathbb{H}^{1-\frac\alpha2, 2}} +|A^\frac\alpha2 u_n(t)|^{2}_{\mathbb{H}^{1-\frac\alpha2, 2}} )dt\right)<\infty.
\end{eqnarray}
\end{lem}

\begin{proof}
First, we prove the following estimate
\begin{eqnarray}\label{Eq-Ito-n-weak-estimation-torus}
\sup_{n}\mathbb{E}\left(\sup_{[0, T]}|u_n(t)|^{2}_{\mathbb{H}^{1, 2}}
+\int_0^T|u_n(t)|^2_{\mathbb{H}^{1+\frac\alpha2, 2}} dt \right)&\leq&C<\infty.
\end{eqnarray}
\del{Thanks to the definitions of $B_n$ and $ G_n$ in \eqref{eq-def-B-n-G-n}
and Assumption $ (\mathcal{A})$, we  can see that $ u_n$ satisfies the assumptions of
 \cite[Lemma 5.1]{Krylov-L-p-2010}, then} By application of Ito's formula,\del{see also  \cite[Theorem 2.1]{Krylov-L-p-2010},}  we have
\begin{eqnarray}\label{Eq-first-norm-l-2n-torus}
|u_n(t)|^{2}_{\mathbb{H}^{1, 2}}&=&|u_{n0}|^{2}_{\mathbb{H}^{1, 2}}- 2 \int_0^t \langle
u_n(s), A^\frac\alpha2u_n(s)-P_nB(u_n(s))\rangle_{\mathbb{H}^{1, 2}} ds \nonumber\\
&+& 2 \int_0^t \langle u_n(s),
P_nG(u_n(s))dW_n(s)\rangle_{\mathbb{H}^{1, 2}} +\int_0^t\sum_{|j|\leq n}|P_nG(u_n(s))Q^\frac12 e_j|_{\mathbb{H}^{1, 2}}^2 ds.
\end{eqnarray}
\del{Using the commutator estimate, see e.g. \cite{Cordoba2-Pointwise-commautator, WUJ-06}, we have

\vspace{-0.35cm}

\begin{eqnarray}\label{Eq-commutator-estimate-torus}
\langle u_n(s),A_\alpha u_n(s)\rangle_{\mathbb{H}^{1, 2}} = \langle u_n(s),A_\alpha u_n(s)\rangle_{{H}_d^1} =\int_{\mathbb{T}^2}
u_n(s)(-\Delta)^{1+\frac\alpha2}u_n(s)dx\nonumber\\ &\geq& \int_{\mathbb{T}^2}
|(-\Delta)^{\frac12+\frac\alpha4}u_n(s)|^2dx\geq c |u_n(s)|^2_{\mathbb{H}^{1+\frac\alpha2, 2}}.\nonumber\\
\end{eqnarray}}
Using the semigroup property of $ (A^\beta)_{\beta\geq0}$ and the definition of the Sobolev spaces in Section \ref{sec-formulation}, we get

\vspace{-0.35cm}

\begin{eqnarray}\label{Eq-commutator-estimate-torus}
\langle u_n(s), A_\alpha u_n(s)\rangle_{\mathbb{H}^{1, 2}} =\del{ \langle u_n(s), A_\alpha u_n(s)\rangle_{{H}_d^1} = \int_{\mathbb{T}^2}
u_n(s)(-\Delta)^{1+\frac\alpha2}u_n(s)dx=} |u_n(s)|^2_{\mathbb{H}^{1+\frac\alpha2, 2}}.
\end{eqnarray}
The term $ \langle u_n(s),  P_nB(u_n(s))\rangle_{\mathbb{H}^{1, 2}}$ 
\del{($ \langle u_n(s),  P_nB(u_n(s))\rangle_{H_n^{1, 2}}$)} in the RHS of \eqref{Eq-first-norm-l-2n-torus} vanishes thanks to
\eqref{vanishes-bilinear-tous-H1}.\del{ pages $ \langle u_n(s),  P_nB_n(u_n(s))\rangle_{\mathbb{H}^1}=0$
$ \langle u,  B(u, u)\rangle_{\mathbb{H}^1}=0$.} To estimate the stochastic term in \eqref{Eq-first-norm-l-2n-torus}, 
we  use the stochastic isometry, Minkowski and H\"older inequalities, the contraction property of $ P_n$ and
Assumption $ (\mathcal{C})$ (\eqref{Eq-Cond-Linear-Q-G}, with $ q=2$ and $ \delta =1$),   we get
\begin{eqnarray}\label{Eq-second-norm-l-2n-start-Torus}
\mathbb{E}\sup_{[0, T]}|\int_0^t \langle u_n(s)\!\!\!\!\!\!&,&\!\!\!\!
P_nG(u_n(s))dW_n(s)\rangle_{\mathbb{H}^{1, 2}} |\nonumber \\
&\leq & c\mathbb{E}\left(\int_0^T\sum_{k\leq n}
\big(\int_{\mathbb{T}^d}|A^\frac12u_n(s)|| A^\frac12P_nG(u_n(s))Q^\frac12e_k|dx\big)^2 ds\right)^\frac12\nonumber\\
&\leq & c\mathbb{E}\left(\int_0^T|u_n(s)|^2_{\mathbb{H}^{1, 2}}|| G(u_n(s))Q^\frac12||_{HS(\mathbb{H}^{1, 2})}^2ds\right)^\frac12\nonumber\\
&\leq & c\mathbb{E}\left(\int_0^T
\big(|u_n(s)|^2_{\mathbb{H}^{1, 2}}+ |u_n(s)|^4_{\mathbb{H}^{1, 2}})ds\right)^\frac12.
\end{eqnarray}

\vspace{-0.35cm}

\noindent Than, we use Young and H\"older inequalities ($\epsilon <1$), we infer that
\begin{eqnarray}\label{Eq-third-norm-l-2n-torus}
\mathbb{E}\sup_{[0, T]}|\int_0^t \langle u_n(s)\!\!\!\!\!\!&,&\!\!\!\!
P_nG(u_n(s))dW_n(s)\rangle_{\mathbb{H}^{1, 2}} |\nonumber \\
&\leq & c\mathbb{E}\left(\sup_{[0,
T]}|u_n(s)|_{\mathbb{H}^{1, 2}}\left[\int_0^T (1+
|u_n(s)|^{2}_{\mathbb{H}^{1, 2}})
ds\right]^\frac12\right)\nonumber\\
&\leq& \epsilon\mathbb{E}\sup_{[0, T]}|u_n(s)|^{2}_{\mathbb{H}^{1, 2}}+
C\mathbb{E}\int_0^T (1+ |u_n(s)|^{2}_{\mathbb{H}^{1, 2}})
ds\nonumber\\
&\leq& \epsilon\mathbb{E}\sup_{[0, T]}|u_n(s)|^{2}_{\mathbb{H}^{1, 2}}+
C\mathbb{E}\int_0^T \sup_{\tau\in[0, s]}|u_n(\tau)|^{2}_{\mathbb{H}^{1, 2}}
ds +C.
\end{eqnarray}
\noindent For the last term in the RHS of  \eqref{Eq-first-norm-l-2n-torus}, we use Assumption $
(\mathcal{C})$ (\eqref{Eq-Cond-Linear-Q-G}, with $ q=2$ and $ \delta =1$), we infer the existence of a 
positive constant $ c>0$ such that,
\begin{eqnarray}\label{Eq-last-Term-Last-estimate-torus}
|\int_0^t\int_{\mathbb{T}^2} \sum_{|j|\leq n}|A^\frac12 P_nG(u_n(s))Q^\frac12 e_j|^2 dx ds|
&\leq& \int_0^t\|G(u_n(s))Q^\frac12\|^2_{HS(\mathbb{H}^{1, 2})}ds\nonumber\\
&\leq & c\int_0^t(1+\sup_{\tau\in[0, s]}|u_n(\tau)|_{\mathbb{H}^{1, 2}}^2)ds.
\end{eqnarray}
Now, we replace \eqref{Eq-commutator-estimate-torus}, \eqref{Eq-third-norm-l-2n-torus} and
\eqref{Eq-last-Term-Last-estimate-torus} in \eqref{Eq-first-norm-l-2n-torus}, we get
\begin{eqnarray}\label{Eq-second-norm-l-2n-torus}
\mathbb{E}\left[\sup_{[0, T]}|u_n(t)|^{2}_{\mathbb{H}^{1, 2}}+   \int_0^T|u_n(s)|^2_{\mathbb{H}^{1+\frac\alpha2, 2}}ds \right]&\leq &C\left(1+
\mathbb{E}|u_{n0}|^{2}_{\mathbb{H}^{1, 2}} + 
C\int_0^T\mathbb{E}\sup_{\tau\in[0, s]}|u_n(\tau)|_{\mathbb{H}^{1, 2}}^{2}ds\right).\nonumber\\
\end{eqnarray}
By application of Gronwall's lemma for the function  $ \mathbb{E}\sup_{[0, T]}|u_n(t)|^{2}_{\mathbb{H}^{1, 2}}$, 
we get the estimation of the first term in the LHS of \eqref{Eq-Ito-n-weak-estimation-torus}, 
(recall that $\mathbb{E}|u_{n0}|^{2}_{\mathbb{H}^{1, 2}}\leq \mathbb{E}|u_{0}|^{2}_{\mathbb{H}^{1, 2}}$). 
The second term in \eqref{Eq-Ito-n-weak-estimation-torus} is
then deduced from  \eqref{Eq-second-norm-l-2n-torus} and the uniform boundedness of 
 $ \mathbb{E}\sup_{[0, T]}|u_n(t)|^{2}_{\mathbb{H}^{1, 2}}$. 
Now,\del{ we assume without lose of generality that  $ p=p'$ (
for $ p'<p$ the result is easily obtained from Estimate  \eqref{eq-estim-energy-Gene-2} bellow, 
by using H\"older inequality).} we prove
\begin{equation}\label{eq-estim-energy-Gene-2}
\mathbb{E}\sup_{[0, T]}|u_n(t)|^p_{\mathbb{H}^{1, 2}} +
\mathbb{E}\int_0^{T}|u_n(t)|^{p-2}_{\mathbb{H}^{1, 2}}|u_n(s)|_{\mathbb{H}^{1+\frac\alpha2, 2}}^2 ds \leq c(1+ \mathbb{E}|u_0|_{\mathbb{H}^{1, 2}}^p).
\end{equation}
By application of Ito's formula to the process
$ (|u_n(t)|^2_{\mathbb{H}^{1, 2}}, t\in[0, T])$ given by \eqref{Eq-first-norm-l-2n-torus}
we get, see for similar calculus e.g 
\cite{Millet-Chueshov-Hydranamycs-2NS-10, Flandoli-Gatarek-95, Sundar-Sri-large-deviation-NS-06},
\begin{eqnarray}\label{stoch-eq-product-Ito-Form-Lp-norm-SDE-torus}
| u_n(t)|^p_{\mathbb{H}^{1, 2}} &+& p
\int_0^t| u_n(s)|^{p-2}_{\mathbb{H}^{1, 2}}|u_n(s)|_{\mathbb{H}^{1+\frac\alpha2, 2}}^2 ds\nonumber\\
&\leq& |u_0|^p_{\mathbb{H}^{1, 2}}
+ \frac p2\int_0^t | u_n(s)|^{p-2}_{\mathbb{H}^{1, 2}}||P_n G(u_n(s))Q^\frac12||_{HS(\mathbb{H}^{1, 2})}^2 ds \nonumber\\
&+& \frac p2 (\frac p2-1)\int_0^t| u_n(s)|^{p-4}_{\mathbb{H}^{1, 2}}|Q^\frac12G^*(u_n(s))u_n(s)|_{\mathbb{H}^{1, 2}}^2ds\nonumber\\
&+&\frac p2 \int_0^t| u_n(s)|^{p-2}_{\mathbb{H}^{1, 2}}\langle u_n(s), P_nG(u_n(s))dW(s)\rangle_{\mathbb{H}^{1, 2}}.
\end{eqnarray}
We argue as above and use Assumption $ (\mathcal{C})$ ( \eqref{Eq-Cond-Linear-Q-G}, with $ q=2$ and $ \delta=1$). 
In particular, for the third term in the RHS of
\eqref{stoch-eq-product-Ito-Form-Lp-norm-SDE-torus}, we follow a similar calculus as in
\eqref{Eq-second-norm-l-2n-start-Torus} and \eqref{Eq-third-norm-l-2n-torus},
we infer that
\begin{eqnarray}\label{Eq-last-part-L-2power-p-torus}
\mathbb{E}\sup_{[0,T]}&{}&\left|\int_0^t | u_n(s)|^{p-2}_{\mathbb{H}^{1, 2}}\langle u_n(s), P_nG(u_n(s))u_n(s),
dW(s)\rangle_{\mathbb{H}^{1, 2}} \right|\nonumber\\
&\leq & C\mathbb{E}\left(\sup_{[0,T]}|u_n(s)|^\frac p2_{\mathbb{H}^{1, 2}}\left[  \int_0^{T} |u_n(s)|^{p-4}_{
\mathbb{H}^{1, 2}} (1+ |u_n(s)|^2_{\mathbb{H}^{1, 2}})ds\right]^\frac12\right)\nonumber\\
&\leq & C_1\mathbb{E}\sup_{[0,T]}|u_n(s)|^p_{
\mathbb{H}^{1, 2}}+ C_2\mathbb{E}\int_0^{T} |u_n(s)|^\frac p2_{
\mathbb{H}^{1, 2}}(1+ |u_n(s)|^2_{\mathbb{H}^{1, 2}})ds\nonumber\\
&\leq & C\mathbb{E}\sup_{[0,T]}|u_n(s)|^p_{
\mathbb{H}^{1, 2}}+ C\mathbb{E}\int_0^{T} \sup_{[0, s]}|u_n(r)|^p_{\mathbb{H}^{1, 2}}ds.
\end{eqnarray}
Hence,
\begin{eqnarray}\label{Eq-1-term-Lp-norm-SDE-Grn-torus}
\mathbb{E}\sup_{[0, T]}|u_n(t)|^p_{\mathbb{H}^{1, 2}} &+& c
\mathbb{E}\int_0^{T}| u_n(s)|^{p-2}_{\mathbb{H}^{1, 2}}|u_n(s)|_{\mathbb{H}^{1+\frac\alpha2, 2}}^2 ds\nonumber\\
&\leq& \mathbb{E}|u_0|^p_{\mathbb{H}^{1, 2}}
+ c\int_0^{T} \mathbb{E}\sup_{[0, s]}|u_n(r)|^{p}_{\mathbb{H}^{1, 2}}ds.
\end{eqnarray}
By application of Gronwall's lemma, we conclude that the first term in the LHS of \eqref{eq-estim-energy-Gene-2} is uniformly bounded in $n$. 
The total estimate  \eqref{eq-estim-energy-Gene-2}  follows easily from the above statement and from \eqref{Eq-1-term-Lp-norm-SDE-Grn-torus}.
The uniformity boundedness of the third, fourth and last  terms in
LHS of \eqref{Eq-Ito-n-weak-estimation-1-Torus} is a consequence of the application of Sobolev embedding
see e.g. \cite{Adams-Hedberg-94, R&S-96, Schmeisser-Tribel-87-book, Sickel-periodic spaces-85} and Appendix \ref{Appendix-Sobolev}
 respectively  H\"older inequality respectively Sobolev interpolation and
Estimate \eqref{Eq-Ito-n-weak-estimation-torus}.

\del{By application of Gronwall's lemma, we conclude that the first term in the LHS of \eqref{Eq-Ito-n-weak-estimation-1-Torus}, with $ p'=p$,
is uniformly bounded. The uniformity boundedness of the third, fourth and last  terms in
LHS of \eqref{Eq-Ito-n-weak-estimation-1-Torus} is a consequence of the application of Sobolev embedding
see e.g. see \cite{Adams-Hedberg-94, R&S-96, Schmeisser-Tribel-87-book, Sickel-periodic spaces-85} and Appendix \ref{appendix-Sobolev-Embd}
 respectively  H\"older inequality respectively Sobolev interpolation and
Estimation \eqref{Eq-Ito-n-weak-estimation-torus}.}

\vspace{0.15cm}
\del{\noindent Now we prove Estimate \eqref{Eq-B-n-weak-estimation-1-Torus}.
Thanks to the contraction of $ P_n$ on $ \mathbb{L}^2(\mathbb{T}^2)$, \eqref{Eq-B-L-2-est} and the Sobolev embedding, 
we infer that for all sequence 
 $(u_n)_n$ satisfying \eqref{Eq-Ito-n-weak-estimation-1-Torus}, we have 
\begin{eqnarray}\label{Eq-B-n-weak-estimation-1-Torus-deter-2step}
\int_0^T |P_nB(u_n(t))|^{2}_{\mathbb{L}^2} dt&\leq& C \int_0^T |u_n(t)|^2_{\mathbb{H}^{1, 2}}
|u_n(t)|^2_{\mathbb{H}^{1+\frac\alpha2, 2}}dt\leq c<\infty.
\end{eqnarray}
Moreover, it is easy to see that for $ \alpha \leq 2$,
\begin{eqnarray}\label{Eq-A-alpha-weak-estimation-1-Torus-deter-2step}
\int_0^T |A^\frac\alpha2 u_n(t)|^{2}_{\mathbb{L}^2} dt&\leq& C \int_0^T |u_n(t)|^2_{\mathbb{H}^{\alpha,2}} dt\leq
 c\int_0^T|u_n(t)|^2_{\mathbb{H}^{1+\frac\alpha2,2}}dt\leq c<\infty.
\end{eqnarray}}
\noindent Now we prove Estimate \eqref{Eq-B-n-weak-estimation-1-Torus}.
Thanks to the contraction of $ P_n$ on $ \mathbb{H}^{1-\frac\alpha2, 2}(\mathbb{T}^2)$, \eqref{est-B-estimatoion-q=2-eta} with $ \eta=1$ and \del{the Sobolev embedding,} the interpolation,   
we infer that for all sequence 
 $(u_n)_n$ satisfying \eqref{Eq-Ito-n-weak-estimation-1-Torus}, we have 
\begin{eqnarray}\label{Eq-B-n-weak-estimation-1-Torus-deter-2step}
\int_0^T |P_nB(u_n(t))|^{2}_{\mathbb{H}^{1-\frac\alpha2, 2}} dt&\leq& C \int_0^T |u_n(t)|^4_{\mathbb{H}^{2-\frac\alpha2, 2}}dt\nonumber\\
&\leq& \int_0^T |u_n(t)|^{8\frac{\alpha-1}{\alpha}}_{\mathbb{H}^{1, 2}} |u_n(t)|^{4\frac{(2-\alpha)}{\alpha}}_{\mathbb{H}^{1+\frac\alpha2, 2}}dt\leq c<\infty,
\end{eqnarray}
provided $ \alpha \in [\frac43, 2]$. Moreover, it is easy to see that 
\begin{eqnarray}\label{Eq-A-alpha-weak-estimation-1-Torus-deter-2step}
\int_0^T |A^\frac\alpha2 u_n(t)|^{2}_{\mathbb{H}^{1-\frac\alpha2, 2}} dt&\leq& \del{C \int_0^T |u_n(t)|^2_{\mathbb{H}^{\alpha,2}} dt\leq}
 c\int_0^T|u_n(t)|^2_{\mathbb{H}^{1+\frac\alpha2,2}}dt\leq c<\infty.
\end{eqnarray}

\end{proof}

\subsection*{Proof of the existence}\label{sec-subsection-existence-Torus} We shall follow for this proof a quiet standard scheme, see e.g.
\cite{Millet-Chueshov-Hydranamycs-2NS-10, Krylov-Rozovski-monotonocity-2007, Rockner-Pevot-06,  Roeckner-Zhang-tamedNS-12,
Sundar-Sri-large-deviation-NS-06}, but we shall use completely different  estimates. These latter are of fractional type and have been developed in Section \ref{sec-nonlinear-prop}.\del{completely different than the classical case} We shall focus more on the key estimates and on the main features of our equation. In deed, thanks to Lemma \ref{lem-unif-bound-theta-n-H-1} and Assumption $(\mathcal{C})$, with $ q=2$ and $ \delta =1$,
we conclude the existence of
a subsequence (we keep the same notation ) $(u_n)_n$,\del{ and adapted measurable processes $ u, F_1$ and $G_1$, such that} 
\begin{equation}\label{eq-u-first-belonging}
 u\in L^2(\Omega\times [0, T]; \mathbb{H}^{1+\frac\alpha2, 2}(\mathbb{T}^2))\cap
L^p(\Omega, L^\infty([0, T];
 \mathbb{H}^{1, 2}(\mathbb{T}^2))),
\end{equation}
\begin{equation}
F_1 \in L^2(\Omega\times [0, T]; \mathbb{L}^{2} (\mathbb{T}^2)) \;\; \text{and}\;\; 
G_1\in L^2(\Omega\times [0, T]; L_Q(\mathbb{H}^{1, 2} (\mathbb{T}^2)), s.t.
\end{equation}
\begin{itemize}
\item (1) $u_n \rightarrow u$ weakly in $ L^2(\Omega\times [0, T],; \mathbb{H}^{1+\frac\alpha2, 2}(\mathbb{T}^2))$.
\item (2) $u_n \rightarrow u$ weakly-star in $ L^p(\Omega, L^\infty([0, T]; \mathbb{H}^{1, 2}(\mathbb{T}^2)))$.
\item (3) $P_n(F(u_n):= (A^\frac\alpha2 + B)(u_n))\rightarrow F_1$ weakly in $ L^2(\Omega\times [0, T]; 
\mathbb{H}^{1-\frac\alpha2, 2} (\mathbb{T}^2))$.
\item (4) $u_n \rightarrow u$ weakly in $ L^{\frac{\alpha}{\eta}}(\Omega\times [0, T]; \mathbb{H}^{1+\eta, 2}(\mathbb{T}^2))$, for all
$ \frac\alpha p<\eta \leq \frac\alpha2 $.
\del{\item (5) $u_n(T) \rightarrow \xi$ weakly in $ L^2(\Omega, \mathbb{H}^{1, 2}(\mathbb{T}^2)))$.}
\item (5) $P_nG(u_n)\rightarrow G_1$ weakly in $ L^2(\Omega\times [0, T]; L_Q(\mathbb{H}^{1, 2} (\mathbb{T}^2))$.
\end{itemize}
The statements $ (1)-(3)$ are straightforward consequence of Lemma \ref{lem-unif-bound-theta-n-H-1}. 
Statement $ (4)$ is a consequence of  the combination of Lemma \ref{lem-unif-bound-theta-n-H-1} 
and the Sobolev interpolation.  Statement $ (5)$  holds 
thanks to the fact that $ P_n$ contracts the $\mathbb{H}^{1, 2}$-norm, 
Assumption $(\mathcal{C})$ (with $ q=2$ and $ \delta=1$) and the uniform boundedness of
 $ u_n$ in $ L^2(\Omega\times[0, T], \mathbb{H}^{1, 2}(\mathbb{T}^2))$.

Now, we construct a process $ (\tilde{u}(t), t\in [0, T])$ as
\begin{equation}\label{eq-def-u-tilde}
\tilde{u}(t) = u_0+ \int_0^t F_1(s)ds + \int_0^tG_1(s)dW(s)
\end{equation}
and prove that $ u= \tilde{u}, dt\times dP- a.e.$. Indeed, using Statement $ (1)$, Equation \eqref{FSBE-Galerkin-approxi} and  Fubini theorem, we infer that for all
$ \varphi \in L^\infty(\Omega\times[0, T], \mathbb{R})$ and $v\in \cup_{n}H_n $,
\begin{eqnarray}\label{eq-u=u-tilde-1}
 \mathbb{E}\int_0^T\langle u(t), \varphi(t) v\rangle_{\mathbb{L}^2} dt &=&
\lim_{n\rightarrow +\infty} \mathbb{E}\int_0^T\langle u_n(t), \varphi(t) v\rangle_{\mathbb{L}^2} dt\nonumber\\
&=& \lim_{n\rightarrow +\infty}\Big[\mathbb{E}\int_0^T \Big(\langle u_n(0), \varphi(t) v\rangle_{\mathbb{L}^2} +
\langle P_nF(u_n(t)), \int_t^T\varphi(s)ds v\rangle_{\mathbb{L}^2}\\
&+&\langle \int_0^t P_nG(u_n(s))dW_n(s), \varphi(t) v\rangle_{\mathbb{L}^2} \Big)dt\Big].\nonumber
\end{eqnarray}
The convergence of the terms in the RHS of \eqref{eq-u=u-tilde-1} to the terms in the RHS of \eqref{eq-def-u-tilde} is as follow. The first term is a consequence of the convergence of 
$ P_n \rightarrow I_{\mathbb{L}^2}$ with respect to the bounded linear operator topology on 
$ \mathbb{L}^2(\mathbb{T}^2)$ and the application of Lebesgue dominated convergence theorem. The second 
term converges thanks to Statement $ (3)$ and the last one converges thanks to the stochastic isometry,
Statement $ (5)$,  Lebesgue dominated  convergence theorem and Lemma \ref{lem-unif-bound-theta-n-H-1},
\del{ and to
the fact that the map $ \mathcal{T}: \mathcal{P}_T(\mathbb{L}^2) \ni \sigma \mapsto \int_0^T\sigma(s)dW(s) \in L^2(\Omega)$ is 
linear continuous} see e.g.
\cite{Millet-Chueshov-Hydranamycs-2NS-10}. Therefore,
\begin{eqnarray}\label{eq-u=u-tilde-2}
\mathbb{E}\int_0^T\langle u(t)- \Big(u(0) +
\int_0^t F_1(s)ds +  \int_0^t G_1(s)dW(s)\Big), \varphi(t) v\rangle dt=0.
\end{eqnarray}
To achieve the proof of the existence, we have to prove that  $F_1= A^\frac\alpha2 \tilde{u}+ B(\tilde{u})$ and $ G_1= G(\tilde{u})$. 
First, we prove the following key estimates
\begin{itemize}
 \item $ (\mathcal{K}_1)$  The local monotonicity property: There exists a constant $ c>0$ such that 
$\forall u, v \in \mathbb{H}^{1+\frac\alpha2}(\mathbb{T}^2)$,
\begin{eqnarray}\label{eq-a-alpha-B-G}
 -2\langle A_\alpha(u-v), u-v\rangle_{\mathbb{L}^2} &+&  2\langle B(u)-B(v), u-v\rangle_{\mathbb{L}^2} +
|| G(u) - G(v)||^2_{L_Q(\mathbb{L}^{2})}\nonumber\\
&\leq & c (1+ | v|^{\frac{2\alpha}{3\alpha-2}}_{\mathbb{H}^{1+\frac\alpha2, 2}})| u-v|^2_{\mathbb{L}^2}.
\end{eqnarray}
\item $ (\mathcal{K}_2)$ For all 
$ \psi\in L^\infty([0, T], \mathbb{R}_+)$ and\del{$ v\in L^2(\Omega\times [0, T]; H_n)$}
$ v\in L^2(\Omega\times [0, T]; \mathbb{H}^{1+\frac\alpha2}(\mathbb{T}^2))$,
\begin{eqnarray}\label{eq-def-Z-n}
 Z_n:= \int_0^T\psi(t)dt\mathbb{E}\big\{\int_0^te^{-r(s)}\big(\!\!\!\!\!&-&\!\!\!\!\!r'(s)|u_n(s)- v(s)|^2_{\mathbb{L}^2}+ 
|| P_nG(u_n(s)) - P_nG(v(s))||^2_{L_Q(\mathbb{L}^{2})}\nonumber\\
&+&2\langle F(u_n(s)) - F(v(s)), u_n(s)- v(s))\rangle_{\mathbb{L}^2}\big)ds\big\}\leq 0,
\end{eqnarray}
\end{itemize}
where $ r'(t):= c(1+ | v(t)|^{\frac{2\alpha}{3\alpha-2}}_{\mathbb{H}^{1+\frac\alpha2, 2}})$ and $c>0$ is a 
constant relevantly chosen.
In fact,  by using an elementary calculus as in \eqref{formula-B-v1-B-v2}, Property \eqref{Eq-3lin-propnull},  
H\"older inequality, Estimate 
\eqref{Eq-B-L-2-est}\del{{Eq-B-H-alpha-2-est}}, interpolation in the case $ 1< \alpha <2 $ (recall that
$ 1\leq \alpha \leq 2 \Rightarrow \mathbb{H}^{\frac\alpha2, 2}(\mathbb{T}^2) \hookrightarrow \mathbb{H}^{1-\frac\alpha2, 2}(\mathbb{T}^2)$)
and Young inequality, we infer that
\begin{eqnarray}
|\langle B(u)-B(v), u-v\rangle_{\mathbb{L}^2}| &= & |\langle B(u-v, v), u-v\rangle_{\mathbb{L}^2}|\leq
| B(u-v, v)|_{\mathbb{L}^2}|u-v|_{\mathbb{L}^2}\nonumber\\
&\leq& c|v|_{\mathbb{H}^{1+\frac\alpha2, 2}} |u-v|_{\mathbb{L}^2} |u-v|_{\mathbb{H}^{1-\frac\alpha2, 2}}\nonumber\\
&\leq& c|v|_{\mathbb{H}^{1+\frac\alpha2, 2}} |u-v|^{\frac{3\alpha-2}{\alpha}}_{\mathbb{L}^2}
|u-v|^{\frac{2-\alpha}{\alpha}}_{\mathbb{H}^{\frac\alpha2, 2}}\nonumber\\
&\leq& c|v|^{\frac{2\alpha}{3\alpha-2}}_{\mathbb{H}^{1+\frac\alpha2, 2}}|u-v|^{2}_{\mathbb{L}^2} + \frac12|u-v|^{2}_{\mathbb{H}^{\frac\alpha2, 2}}.
\end{eqnarray}
Moreover, thanks to the semigroup property of $ (A^\beta)_{\beta\geq0}$, we infer that 
\begin{eqnarray}
 -\langle A_\alpha(u-v), u-v\rangle_{\mathbb{L}^2} = - |u-v|^{2}_{\mathbb{H}^{\frac\alpha2, 2}}.
\end{eqnarray}
Therefore, there exists a constant $ c>0$ such that
\begin{eqnarray}\label{eq-last-exitence-Torus-H-1}
 -\langle A_\alpha(u-v), u-v\rangle_{\mathbb{L}^2} +  \langle B(u)-B(v), u-v\rangle_{\mathbb{L}^2} \leq
 -\frac12|u-v|^{2}_{\mathbb{H}^{\frac\alpha2, 2}} + c|v|^{\frac{2\alpha}{3\alpha-2}}_{\mathbb{H}^{1+\frac\alpha2, 2}}|u-v|^{2}_{\mathbb{L}^2}.\nonumber\\
\end{eqnarray}
Combining \eqref{eq-last-exitence-Torus-H-1} and Assumption $ (\mathcal{C})$: (\eqref{Eq-Cond-Lipschitz-Q-G} with 
$ q=2$, $ \delta=0$ and $ C_R:=C$),\del{ as mentioned in Theorem \ref{Main-theorem-strog-Torus},} 
we easily get \eqref{eq-a-alpha-B-G}. In particular, thanks to 
the contraction of $ P_n$ on $ \mathbb{L}^2(\mathbb{T}^{2})$,  Estimate \eqref{eq-a-alpha-B-G} is still valid when replacing 
$ u, v$ and $G$  by $ u_n$ (recall $ u_n$ is the solution of Equation \eqref{FSBE-Galerkin-approxi}), $v\in L^{2}(\Omega\times[0, T]; \mathbb{H}^{1+\frac\alpha2}(\mathbb{T}^{2}))$
and  $ P_nG$ respectively.\del{and thanks to 
the contraction of $ P_n$ on $ \mathbb{L}^2(O)$, we infer that 
Estimate \eqref{eq-a-alpha-B-G} is valid as well.}  Furthermore, estimating the LHS of  
\eqref{eq-def-Z-n} by  \eqref{eq-a-alpha-B-G} endowed with these latter variables, we get $ Z_n\leq 0$. Consequently $ (\mathcal{K}_1)$ and $ (\mathcal{K}_2)$ are proved.\del{with $ u= u_n(t)$ and $ v= v(t)$}
Let us also mention and recall the following two statements

\noindent (a)- If $ v\in L^{2}(\Omega\times[0, T]; \mathbb{H}^{1+\frac\alpha2}(\mathbb{T}^{2}))$,
 with $ \alpha\in [1, 2]$, then 
\begin{equation}
 \mathbb{E}\int_0^T|v(t)|^{\frac{2\alpha}{3\alpha-2}}_{\mathbb{H}^{1+\frac\alpha2, 2}}dt \leq 
 \mathbb{E}\int_0^T(1+|v(t)|^{2}_{\mathbb{H}^{1+\frac\alpha2, 2}})dt < \infty.
\end{equation}

\noindent (b)- If a sequence $(f_n)_n$ in a Hilbert space $ H$ converges weakly to $f$ then $|f|_{H}\leq \liminf_{n\rightarrow \infty}|f_n|_H$.

\vspace{0.25cm}

\noindent Now, we take $ \psi$ and $ r(t)$ as defined above. Thanks to the equality  $ u= \tilde{u}, dt\times dP- a.e.$,
statements $ (1)$ \& (b) and Fubini's theorem, we infer that 
\begin{eqnarray}\label{eq-equality-norms-1-Torus}
\int_0^T\psi(t)dt\mathbb{E}|u(s)|^2_{\mathbb{L}^2}e^{-r(t)}- \mathbb{E}|u_0|^2_{\mathbb{L}^2}
&=&\int_0^T\psi(t)dt\mathbb{E}|\tilde{u}(s)|^2_{\mathbb{L}^2}e^{-r(t)}- \mathbb{E}|u_0|^2_{\mathbb{L}^2} \nonumber\\
&\leq &\liminf_{n\rightarrow \infty}\int_0^T\psi(t)dt\mathbb{E}|u_n(s)|^2_{\mathbb{L}^2}e^{-r(t)}- \mathbb{E}|u_0|^2_{\mathbb{L}^2}.
\end{eqnarray}
By application of the Ito formula to the Ito process $ \tilde{u} $ given by \eqref{eq-def-u-tilde} 
and using\del{the first equality in \eqref{eq-equality-norms-1-Torus},} 
the equality  $ u= \tilde{u}, dt\times dP- a.e.$ and the elementary identity 
\begin{equation}\label{elme-identity-Hilbert}
 \forall f, g \in H, | f|^2_{H}= |f-g|^2_{H} +2\langle f-g, g \rangle_{H} +| g|^2_{H}
\end{equation}
with $ f= u(s)$, $ g= v(s)$ and\del{``$
|u(s)|^2_{\mathbb{L}^2} = |u(s)-v(s)|^2_{\mathbb{L}^2}
+2\langle u(s)- v(s), v(s)\rangle_{\mathbb{L}^2} + |v(s)|^2_{\mathbb{L}^2}$''  with} $ v\in 
L^{2}(\Omega\times[0, T]; \mathbb{H}^{1+\frac\alpha2}(\mathbb{T}^{2}))$, we get
\del{\begin{eqnarray}\label{eq-elementary-calculs}
\mathbb{E}\int_0^te^{-r(s)}r'(s)|u(s)|^2_{\mathbb{L}^2}ds = \mathbb{E}\int_0^te^{-r(s)}r'(s)
\big(|u(s)-v(s)|^2_{\mathbb{L}^2}
+2\langle u(s)- v(s), v(s)\rangle_{\mathbb{L}^2} + |v(s)|^2_{\mathbb{L}^2}\big)ds.\nonumber\\
\end{eqnarray}
By application of Ito formula on Ito process $ \tilde{u} $ and using the first equality in 
\eqref{eq-equality-norms-1} and the
identity \eqref{eq-elementary-calculs}, we get}
\begin{eqnarray}\label{eq-equality-norms-Torus}
\mathbb{E}|u(s)|^2_{\mathbb{L}^2}e^{-r(t)}&-&\mathbb{E}|u_0|^2_{\mathbb{L}^2}
=\mathbb{E}\int_0^t2e^{-r(s)}\langle F_1(s),  u(s)\rangle_{\mathbb{L}^2}+ \mathbb{E}\int_0^te^{-r(s)}||G_1(s)||^2_{L_Q(\mathbb{L}^2)}ds \nonumber\\
&- &  \mathbb{E}\int_0^te^{-r(s)}r'(s)\big(|u(s)-v(s)|^2_{\mathbb{L}^2}
+2\langle u(s)- v(s), v(s)\rangle_{\mathbb{L}^2} + |v(s)|^2_{\mathbb{L}^2}\big)ds.
\end{eqnarray}
Similarly, we get Identity \eqref{eq-equality-norms-Torus} for $ \mathbb{E}|u_n(s)|^2_{\mathbb{L}^2}e^{-r(t)}$ with 
$ u, F_1 $\del{in fact it is P_nF_1, but because the product <P_nF_1, u_n>, the P_n is killed} and $ G_1$ in the RHS of \eqref{eq-equality-norms-Torus} are respectively replaced by $ u_n, F(u_n), P_nG(u_n)$ (Recall that $F:= A^\frac\alpha2 + B$). Replacing Identity \eqref{eq-equality-norms-Torus} for 
$ \mathbb{E}|u(s)|^2_{\mathbb{L}^2}e^{-r(t)}$ in the LHS of the first equality in \eqref{eq-equality-norms-1-Torus} 
and Identity \eqref{eq-equality-norms-Torus} for
$ \mathbb{E}|u_n(s)|^2_{\mathbb{L}^2}e^{-r(t)}$ in the RHS of  the second Inequality \eqref{eq-equality-norms-1-Torus} and arranging terms
(in particular, we introduce the term $ G(v(s))$ and use the elementary identity \eqref{elme-identity-Hilbert}), we
infer that
\begin{eqnarray}\label{eq-equality-norms-Torus-sum-of-u}
\mathcal{E}:= \int_0^T\psi(t)dt\mathbb{E}\big\{\int_0^te^{-r(s)}\Big[\!\!\!\!&\!\!\!\!2\!\!\!\!&\!\!\!\!\langle F_1(s),  u(s)\rangle_{\mathbb{L}^2}+
||G_1(s)||^2_{L_Q(\mathbb{L}^2)}ds \nonumber\\
&- & r'(s)\big(|u(s)-v(s)|^2_{\mathbb{L}^2}
+2\langle u(s)- v(s), v(s)\rangle_{\mathbb{L}^2} \big)\Big] ds\big\}\nonumber\\
&\leq & \liminf_{n\rightarrow \infty}\big(Z_n + Y_n + X_n\big).
\end{eqnarray}
where $ Z_n$ is given by \eqref{eq-def-Z-n},
\begin{eqnarray}
Y_n:&=& \int_0^T\psi(t)dt\mathbb{E}\big\{\int_0^t e^{-r(s)}\big(-2r'(s)
\langle u_n(s)- v(s), v(s)\rangle_{\mathbb{L}^2}+
2 \langle P_nG(u_n(s)), G(v(s))\rangle_{L_Q(\mathbb{L}^2)}\nonumber\\
&+&2\langle F(u_n(s)), v(s)\rangle_{\mathbb{L}^2} +2
\langle F(v(s)), u_n(s))\rangle_{\mathbb{L}^2}-2
\langle F(v(s)), v(s))\rangle_{\mathbb{L}^2}
\big)ds\big\},\nonumber\\
\end{eqnarray}
and
\begin{eqnarray}
X_n:&=& \int_0^T\psi(t)dt\mathbb{E}\big\{\int_0^t e^{-r(s)}\big(
2 \langle P_nG(u_n(s)), P_nG(v(s))- G(v(s))\rangle_{\mathbb{L}^2}
-||P_nG(v(s))||^2_{L_Q(\mathbb{L}^2)}\big)ds\big\},\nonumber\\
\end{eqnarray}
The sequences $ (Y_n)_n$ and $ (X_n)_n$  converge to $ Y$ and $X$ respectively,
thanks to statements $(1)-(3)$ and  $(5)$, the convergence of $ P_n \rightarrow I_{\mathbb{L}^2}$,  
Assumption $(\mathcal{C})$:( \eqref{Eq-Cond-Lipschitz-Q-G} with $ q=2$, $ \delta =0$ and $ C_R:=C$), Lemma \ref{lem-unif-bound-theta-n-H-1},   
Estimate \eqref{Eq-B-L-2-est}, similar calculus as in \eqref{Eq-B-n-weak-estimation-1-Torus-deter-2step} 
and \eqref{Eq-A-alpha-weak-estimation-1-Torus-deter-2step} and the Lebesgue dominated, where
\begin{eqnarray}
Y:&=& \int_0^T\psi(t)dt\mathbb{E}\big\{\int_0^t e^{-r(s)}\big(-2r'(s)\langle u(s)- v(s), v(s)\rangle_{\mathbb{L}^2}+
2 \langle G(s), G(v(s))\rangle_{L_Q(\mathbb{L}^2)}\nonumber\\
&+&2\langle F(s), v(s)\rangle_{\mathbb{L}^2} +2
\langle F(v(s)), u(s))\rangle_{\mathbb{L}^2}-2
\langle F(v(s)), v(s))\rangle_{\mathbb{L}^2}
\big)ds\big\} \nonumber\\
\end{eqnarray}
and
\begin{eqnarray}
X:&=& -\int_0^T\psi(t)dt\mathbb{E}\big\{\int_0^t e^{-r(s)}
||G(v(s))||^2_{L_Q}\big)ds\big\}.
\end{eqnarray}
Replacing $ X$ and $ Y$ in \eqref{eq-equality-norms-Torus-sum-of-u} and taking into account \eqref{eq-def-Z-n}, we conclude that
\begin{equation}
 \mathcal{E} -X-Y \leq \liminf_{n\rightarrow \infty}Z_n \leq 0.
\end{equation}
Therefore, we get
\begin{eqnarray}\label{eq-key-leq-0-1}
\int_0^T\psi(t)dt\mathbb{E}\big\{\int_0^te^{-r(s)}\big(-r'(s)|u(s)\!\!\!&-&\!\!\! v(s)|^2_{\mathbb{L}^2}+ \langle F_1(s) -
F(v(s)), u(s)- v(s))\rangle_{\mathbb{L}^2}\nonumber\\
&+&2|| G_1(s) - G(v(s))||^2_{L_Q}\big)ds\big\}\leq 0.
\end{eqnarray}
\del{Consequently, for $  v \in L^2([\Omega\times0, T], \mathbb{H}^{1+\frac\alpha2, 2}(\mathbb{T}^2))$, $ P_nv$ satisfies \eqref{eq-key-leq-0-1}. 
By approximation, \eqref{eq-key-leq-0-1} is also satisfied by  $  v \in L^2([\Omega\times0, T], \mathbb{H}^{1+\frac\alpha2, 2}(\mathbb{T}^2))$.
In deed, the passage to the limit of the first and third terms is guaranteed thanks to the convergence of the projection 
$ P_n$ to the identity $ I_{\mathbb{H}^{s, 2}}$ for all Sobeolev spaces $ \mathbb{H}^{s, 2}(O), s\geq 0$, in particular for  $ I_{\mathbb{L}^2}(O)$
and by Assumption $ (\mathcal{C})$. To prove the convergence of the second term, it is sufficient to prove and justify the convergence
of the term $ \langle F(P_nv(s)),  P_nv(s))\rangle_{\mathbb{L}^2}$. In fact, we rewrite this last as follow
\begin{eqnarray}\label{local-est-1}
 \langle F(P_nv(s)),  P_nv(s)\rangle_{\mathbb{L}^2}&-& \langle F(v(s)),  v(s)\rangle_{\mathbb{L}^2}\nonumber\\
&=& \langle F(P_nv(s))-F(v(s)),  v(s)\rangle_{\mathbb{L}^2} + \langle F(P_nv(s)),  v(s)- P_nv(s)\rangle_{\mathbb{L}^2}.
\end{eqnarray}
Using H\"older and  Minowksky inequalities, Estimation \eqref{Eq-B-L-2-est} and the contraction property of the projection 
$ P_n$ on all Sobeolev spaces $ \mathbb{H}^{s, 2}(O), s\geq 0$, we get
\begin{eqnarray}\label{local-est-2}
|\langle F(P_nv(s)),  v(s)- P_nv(s)\rangle_{\mathbb{L}^2}|&\leq & | v(s)- P_nv(s)|_{\mathbb{L}^2}|F(P_nv(s))|_{\mathbb{L}^2}\nonumber\\
&\leq &c| v(s)- P_nv(s)|_{\mathbb{L}^2}\big( |P_nv(s)|_{\mathbb{H}^{\alpha, 2}}+   |B(P_nv(s))|_{\mathbb{L}^2}  \big)\nonumber\\
\del{&\leq &c| v(s)- P_nv(s)|_{\mathbb{L}^2}\big( |v(s)|_{\mathbb{H}^{\alpha, 2}}+  
 |P_nv(s))|_{\mathbb{H}^{1-\frac\alpha2, 2}}|P_nv(s))|_{\mathbb{H}^{1+\frac\alpha2, 2}} \big)\nonumber\\}
&\leq &c| v(s)- P_nv(s)|_{\mathbb{L}^2}\big( |v(s)|_{\mathbb{H}^{\alpha, 2}}+  
 |v(s)|_{\mathbb{H}^{1-\frac\alpha2, 2}}|v(s)|_{\mathbb{H}^{1+\frac\alpha2, 2}} \big).
\end{eqnarray}
For the second term in \eqref{local-est-1}, we use H\"older and  Minoksky inequalities, \eqref{Eq-def-B-theta1-Theata2}, 
\eqref{Eq-B-H-alpha-2-est-d}, we infer 
\begin{eqnarray}\label{local-est-2}
|\langle F(P_nv(s))&-& F(v(s)),  v(s)\rangle_{\mathbb{L}^2}|
\leq  |v(s)|_{\mathbb{H}^{+1+\frac\alpha2, 2}} |F(v(s))- F(P_nv(s))|_{\mathbb{H}^{-1-\frac\alpha2, 2}}\nonumber\\
&\leq &c |v(s)|_{\mathbb{H}^{+1+\frac\alpha2, 2}}\big( |v(s)- P_nv(s)|_{\mathbb{L}^{2}}+ 
|B(v(s)-P_nv(s), P_nv(s))|_{\mathbb{H}^{-1-\frac\alpha2, 2}} \nonumber\\
&+& |B(v(s), P_nv(s)-v(s))|_{\mathbb{H}^{-1-\frac\alpha2, 2}}\big)\nonumber\\
&\leq &c |v(s)|_{\mathbb{H}^{1+\frac\alpha2, 2}}\big( |v(s)- P_nv(s)|_{\mathbb{L}^{2}}+ 
|v(s)|_{\mathbb{H}^{1-\frac\alpha4, 2}}|P_nv(s)-v(s)|_{\mathbb{H}^{1-\frac\alpha4, 2}}\big)\nonumber\\
&\leq &c (1+|v(s)|_{\mathbb{H}^{+1+\frac\alpha2, 2}})|P_nv(s)-v(s)|_{\mathbb{H}^{1-\frac\alpha4, 2}}.
\end{eqnarray}
Finally, the integrand in RHS of \eqref{eq-key-leq-0-1} with $ v(s)$ being replaced by $ P_nv(s)$, converges by application of the Lebesgue 
dominate convergence theorem.}
\noindent Now, we take $ v = u$ in $ L^2([\Omega\times0, T], \mathbb{H}^{1+\frac\alpha2, 2}(\mathbb{T}^2))$, we conclude from \eqref{eq-key-leq-0-1}, that
\del{$|| G(s) - G(v(s))||^2_{L_Q}$} $ G(s)= G(u(s)), \; ds\times dP-a.e.$. To get the equality $ F(s)= F(u(s)), \; ds\times dP-a.e.$, 
we consider Estimate \eqref{eq-key-leq-0-1} without the last term and we introduce
$ \tilde{v} \in L^\infty(\Omega\times[0, T], \mathbb{H}^{1+\frac\alpha2}(\mathbb{T}^2))$ and a
parameter $ \lambda \in [-1, +1]$. Replacing $ v$  and $ r'(s)$ by $ u-\lambda\tilde{v}$ respectively
$ r'_\lambda(s):= c(1+|u-\lambda\tilde{v}|^{\frac{2\alpha}{3\alpha-2}}_{\mathbb{H}^{1+\frac\alpha2}})$, we get
\begin{eqnarray}\label{eq-equality-F-Fn-*}
\mathbb{E}\int_0^Te^{-r_\lambda(s)}\big(-r'_\lambda(s)\lambda^2|\tilde{v}(s)|^2_{\mathbb{L}^2}+ 2\lambda\langle F(s) -
F(u(s)-\lambda \tilde{v}(s)), \tilde{v}(s))\rangle_{\mathbb{L}^2}\big)ds\leq 0.\nonumber\\
\end{eqnarray}
Dividing on $ \lambda<0$ and on $ \lambda>0$, we conclude that, when $ \lambda \rightarrow 0$, the limit of
the LHS of \eqref{eq-equality-F-Fn-*} exists and vanishes. Moreover, using the fact that
$ \tilde{v} \in L^\infty(\Omega\times[0, T], \mathbb{H}^{1+\frac\alpha2}(\mathbb{T}^2))$, 
the continuity of $ r_\lambda$ and 
$ r'_\lambda$  with respect to $ \lambda$ and the Lebesgue 
dominated convergence theorem, we conclude that the first term in $ \frac1\lambda$LHS of \eqref{eq-equality-F-Fn-*} 
vanishes and also 
\begin{eqnarray}\label{eq-equality-F-Fn-1}
\mathbb{E}\int_0^Te^{-r_0(s)}\langle F(s) -
F(u(s)), \tilde{v}(s))\rangle_{\mathbb{L}^2}ds= 0.
\end{eqnarray}
The justification of the use of the Lebesgue dominated convergence theorem is due to, the positivity of 
$ r_\lambda(s)$, Inequality
\del{$ |B(f)|_{\mathbb{H}^{-1, 2}} \leq |f|_{\mathbb{H}^{1, 2}}|f|_{\mathbb{L}^{2}}$}\eqref{classical-B-H-1}, Minikowskii inequality,
the conditions $ 0< \alpha\leq 2$ and $ |\lambda|\leq 1$,
the statements $ (1) \;\& \; (3)$ and the definition of $ \tilde{v}$. In fact,
\del{\begin{eqnarray}
 e^{-r_\lambda(s)}|\langle F(s) &-&
F(u(s)-\lambda \tilde{v}(s)), \tilde{v}(s))\rangle_{\mathbb{L}^2}|\nonumber\\
&\leq&
 |\tilde{v}(s)|_{\mathbb{L}^2}|F(s) -
F(u(s)-\lambda \tilde{v}(s))|_{\mathbb{L}^2} \nonumber\\
&\leq&
 |\tilde{v}(s)|_{\mathbb{L}^2}\big(|F(s)|_{\mathbb{L}^2} + |u(s)|_{\mathbb{H}^{\frac\alpha2, 2}}+ |\tilde{v}(s)|_{\mathbb{H}^{\frac\alpha2, 2}}
|B(u(s)-\lambda \tilde{v}(s))|_{\mathbb{L}^2} \nonumber\\
&\leq&
|\tilde{v}(s)|_{\mathbb{L}^2}\big(|F(s)|_{\mathbb{L}^2} + |u(s)|_{\mathbb{H}^{\frac\alpha2, 2}}+ |\tilde{v}(s)|_{\mathbb{H}^{\frac\alpha2, 2}}
|B(u(s)-\lambda \tilde{v}(s))|_{\mathbb{L}^2} \nonumber\\
&\leq&
 |\tilde{v}(s)|_{\mathbb{L}^2}|F(s) -
F(u(s)-\lambda \tilde{v}(s))|_{\mathbb{L}^2} \nonumber\\
\end{eqnarray}}
\begin{eqnarray}\label{just-domin-Torus}
 e^{-r_\lambda(s)}|\langle F(s) &-&
F(u(s)-\lambda \tilde{v}(s)), \tilde{v}(s))\rangle_{\mathbb{L}^2}|\nonumber\\
&\leq&
 |\tilde{v}(s)|_{\mathbb{H}^{1, 2}}\big[|F(s)|_{\mathbb{L}^2}+
|u(s)|_{\mathbb{H}^{\alpha-1, 2}}+ |\tilde{v}(s)|_{\mathbb{H}^{\alpha-1, 2}}+ |B(u(s)-\lambda \tilde{v}(s))|_{\mathbb{H}^{-1, 2}}\big] \nonumber\\
&\leq& c
 |\tilde{v}(s)|_{\mathbb{H}^{1, 2}}\big[|F(s)|_{\mathbb{L}^2}+
|u(s)|_{\mathbb{H}^{\alpha-1, 2}}+ |\tilde{v}(s)|_{\mathbb{H}^{\alpha-1, 2}}+
|u(s)|_{\mathbb{H}^{1, 2}}|u(s)|_{\mathbb{L}^{ 2}} \nonumber\\
&+& |\tilde{v}(s)|_{\mathbb{H}^{1, 2}}|\tilde{v}(s)|_{\mathbb{L}^{ 2}}
+|u(s)|_{\mathbb{H}^{1, 2}}|\tilde{v}(s)|_{\mathbb{L}^{ 2}} + |\tilde{v}(s)|_{\mathbb{H}^{1, 2}}|u(s)|_{\mathbb{L}^{ 2}}\big].
\end{eqnarray}
This ends the proof of the existence of a solution $ (u(t), t\in [0, T])$ belonging to the first intersection in \eqref{eq-set-solu-weak-torus} and satisfying by construction  \eqref{cond-solu-torus-H1}.\del{ and  $ u(\cdot, \omega) \in L^\infty(0, T; \mathbb{H}^{1, 2}(\mathbb{T}^2)) \cap 
L^2(0, T; \mathbb{H}^{1+\frac\alpha2, 2}(\mathbb{T}^2)), \; P-a.s.$}

\subsection*{Proof of the time regularity.}

To prove the continuity of the trajectories of the weak solution 
$ (u(t), t\in[0, T])$, i.e. $ u(\cdot, \omega)\in C([0, T], \mathbb{L}^2(O)), \; P-a.s.$, we apply \cite[Proposition VII.3.2.2]{Metivier-book-SPDEs-88}, see also
\cite[Proposition 2.5]{Sundar-Sri-large-deviation-NS-06}. We consider the dense Gelfand Triple 
$$ \mathbb{H}^{1, 2}(\mathbb{T}^2) \hookrightarrow \mathbb{L}^{2}(\mathbb{T}^2) \hookrightarrow 
(\mathbb{H}^{1, 2}(\mathbb{T}^2))^*=\mathbb{H}^{-1, 2}(\mathbb{T}^2).$$
Using  \eqref{classical-B-H-1}, Sobolev Embedding\del{,  the trivial identity $ |u|\leq 1+|u|^2$ } and  \eqref{cond-solu-torus-H1}, we infer that 
$ B(u(\cdot, \omega))\in L^2(0, T; \mathbb{H}^{-1, 2}(\mathbb{T}^2)),\; P-a.s.$. In fact,
\begin{equation}\label{cont-1}
\mathbb{E} \int_0^T|B(u(s))|_{\mathbb{H}^{-1, 2}}^{2}ds \leq c \mathbb{E} \int_0^T|u(s)|^4_{\mathbb{H}^{1, 2}}ds
\leq c (1+\mathbb{E}\sup_{[0, T]}|u(s)|^p_{\mathbb{H}^{1, 2}})<\infty. 
\end{equation}
And that for $\alpha\leq 2$,
\begin{equation}\label{cont-1-1}
\mathbb{E} \int_0^T|A_\alpha u(s)|^2_{\mathbb{H}^{-1, 2}}ds \leq c \mathbb{E} \int_0^T|u(s)|^2_{\mathbb{H}^{\alpha-1, 2}}ds
\del{\leq c \mathbb{E}\sup_{[0, T]}|u(s)|^2_{\mathbb{H}^{1, 2}}}\leq c (1+\mathbb{E}\sup_{[0, T]}|u(s)|^p_{\mathbb{H}^{1, 2}})<\infty. 
\end{equation}
Moreover, we prove that\del{ the trajectories of the 
continuous} the martingale $ M(t):= \int_0^tG(u(s))dW(s)$ belongs to $ L^2(\Omega, \del{L^\infty}C(0, T; \mathbb{H}^{1, 2}(\mathbb{T}^2)))$. In fact, we use 
Burkholdy-Davis-Gandy inequality, Assumption $ (\mathcal{C})$: (\eqref{Eq-Cond-Linear-Q-G}
with $ q=2$ and $ \delta=1$), \eqref{cond-solu-torus-H1}, we obtain 
\begin{eqnarray}\label{cont-2}
\mathbb{E}\sup_{[0, T]} |\int_0^tG(u(s))dW(s)|^2_{\mathbb{H}^{1, 2}}
&\leq& c \mathbb{E}\int_0^T|G(u(s))|^2_{L_Q(\mathbb{H}^{1, 2})}ds
\leq c \mathbb{E}\int_0^T(1+|u(s)|^2_{\mathbb{H}^{1, 2}})ds \nonumber\\
&\leq& c (1+ \mathbb{E}\sup_{[0, T]}|u(s)|^p_{\mathbb{H}^{1, 2}})<\infty.
\end{eqnarray}
Hence from \eqref{cond-solu-torus-H1}, \eqref{cont-1}, \eqref{cont-1-1}  and \eqref{cont-2}, 
we establish the existence of a subset $ \Omega'\subset \Omega$ (independent of "t"),  such that $ P(\Omega')=0$ and
$ F(u(\cdot, \omega))\in L^2(0, T; \mathbb{H}^{-1, 2}(\mathbb{T}^2))$, $  u(\cdot, \omega) \;\text{and}\;  M(\cdot, \omega) 
\in L^\infty(0, T; \mathbb{H}^{1, 2}(\mathbb{T}^2))$, $\forall \omega \in \Omega'^c$. These ingredients are enough to apply  \cite[Proposition VII.3.2.2]{Metivier-book-SPDEs-88}\del{\cite[Proposition 2.5.]{Sundar-Sri-large-deviation-NS-06}}, hence we get the result. It is important to mention that the property $  u(\cdot, \omega) \;\text{and}\;  M(\cdot, \omega) 
\in L^\infty(0, T; \mathbb{H}^{1, 2}(\mathbb{T}^2))$ is more what we need here. In deed, it is sufficient to prove $  u(\cdot, \omega) \;\text{and}\;  M(\cdot, \omega) 
\in L^2(0, T; \mathbb{H}^{1, 2}(\mathbb{T}^2))$. By the above two subsections, the proof of $ (3.7.1)$ is achieved.

\del{ Therefore, we have
\begin{eqnarray}\label{eq-equality-F-Fn-1}
\lim_{\lambda\rightarrow 0}\mathbb{E}\int_0^Te^{-r_\lambda(s)}\big\{2\langle F(s)&-&
F(u(s)), \tilde{v}(s))\rangle_{\mathbb{L}^2}+
\big(-r'_\lambda(s)\lambda|\tilde{v}(s)|^2_{\mathbb{L}^2}\nonumber\\
&+& 2\langle F(u(s)) - F(u(s)-\lambda \tilde{v}(s), \tilde{v}(s))\rangle_{\mathbb{L}^2}\big)\big\}ds= 0.
\end{eqnarray}
\begin{eqnarray}\label{eq-equality-F-Fn-1}
\lim_{\lambda\rightarrow 0}\mathbb{E}\int_0^Te^{-r_\lambda(s)}\big\{2\langle F(s)&-&
F(u(s)), \tilde{v}(s))\rangle_{\mathbb{L}^2}+
\big(-r'_\lambda(s)\lambda|\tilde{v}(s)|^2_{\mathbb{L}^2}\big\}ds\nonumber\\
&+& \lim_{\lambda\rightarrow 0}\mathbb{E}\int_0^T2e^{-r_\lambda(s)}\langle F(u(s)) - F(u(s)-\lambda \tilde{v}(s)), \tilde{v}(s))
\rangle_{\mathbb{L}^2}\big)ds= 0.
\end{eqnarray}}

\del{WITOUT STOP TIME\subsection{Proof of pathwise uniqueness.}\label{sec-subsection-uniqueness-Torus}
Let $ u^1$ and $ u^2$ be two weak solutions of Equation \eqref{Main-stoch-eq} with $ 1<\alpha \leq 2$ as constructed in Subsection
\ref{sec-subsection-existence-Torus}. Let $ w:=  u^1- u^2$, then  $ w$ satisfied the following equation
\begin{eqnarray}
 w(t) = \int_0^t \big(-A_\alpha w(s)+ B(w(s), u^1(s))&+& B(u^2(s), w(s))\big)ds \nonumber\\
&+& \int_0^t\big(G(u^1(s))- G(u^2(s))\big)dW(s).
\end{eqnarray}
Using Ito formula to the product $ e^{-r(t)}|w(t)|^2_{\mathbb{L}^{2}}$, with $(r(t), t\in [0, T])$ is a
real positive stochastic process to be precise later, Property \eqref{Eq-3lin-propnull},
condition \eqref{Eq-Cond-Lipschitz-Q-G} with $ q=2$ and $\delta =0$ (locally Lipschitz), 
Estimation \eqref{3linear-H1-H-1}, Young inequality  and arguing as in the proof of 
\eqref{Eq-Ito-n-weak-estimation-1-Torus} by replacing
the spaces $ \mathbb{H}^{1, 2}(\mathbb{T}^2)$ and $ \mathbb{H}^{1+\frac\alpha2, 2}(\mathbb{T}^2)$ by
$ \mathbb{L}^{2}(\mathbb{T}^2)$ respectively $ \mathbb{H}^{\frac\alpha2, 2}(\mathbb{T}^2)$.  we infer that 
\begin{eqnarray}\label{Torus-uniquenss-ito-formula}
\mathbb{E}\!\!\!\!&{}&\!\!\!\!e^{-r(t)}|w(t)|^2_{\mathbb{L}^{2}}+
2\mathbb{E}\int_0^{t}e^{-r(s)}|w(s)|^2_{\mathbb{H}^{\frac\alpha2, 2}}ds \nonumber\\
&\leq& \mathbb{E}\int_0^{t}e^{-r(s)}|| G(u^1(s))- G(u^2(s))||^2_{L_Q(\mathbb{L}^2)}ds\nonumber\\
& -& \mathbb{E}\int_0^{t} e^{-r(s)} \big(2\langle B(w(s),  w(s)),  u^1(s)\rangle-
r'(s)|w(s)|^2_{\mathbb{L}^{2}}\big)ds\nonumber\\
& \leq & c_N\mathbb{E}\int_0^{t}e^{-r(s)}\big(|w(s)|^2_{\mathbb{L}^2}+
|u^1(s)|_{\mathbb{H}^{1, 2}}|w(s)|^{\frac2\alpha}_{\mathbb{H}^{\frac{\alpha}{2}, 2}}
|w(s)|^{2\frac{\alpha-1}\alpha}_{{\mathbb{L}^{2}}}-
r'(s)|w(s)|^2_{\mathbb{L}^{2}}\big)ds\nonumber\\
& \leq & c_N
\mathbb{E}\int_0^{t}e^{-r(s)}\big(|w(s)|^2_{\mathbb{L}^2}+ 2 c|w(s)|^2_{\mathbb{H}^{\frac\alpha2, 2}}+
2c_1|u^1(s)|^{\frac{\alpha}{\alpha-1}}_{\mathbb{H}^{1, 2}}|w(s)|^2_{\mathbb{L}^{2}}-
r'(s)|w(s)|^2_{\mathbb{L}^{2}}\big)ds.\nonumber\\
\end{eqnarray}
Now, we choose $ c<1$ and $ r'(s)= 2c_1|u^1(s)|^{\frac{\alpha}{\alpha-1}}_{\mathbb{H}^{1, 2}}$ and replace in
\eqref{Torus-uniquenss-ito-formula},
we end up by the simple formula
\begin{eqnarray}
\mathbb{E}e^{-r(t)}|w(t)|^2_{\mathbb{L}^{2}}&+& 2(1-c)\mathbb{E}\int_0^{t}
e^{-r(s)}|w(s)|^2_{\mathbb{H}^{\frac\alpha2, 2}}ds
\leq c_N\mathbb{E}\int_0^{t}e^{-r(s)}|w(s)|^2_{\mathbb{L}^2}ds.\nonumber\\
\end{eqnarray}
Than by application of Gronwall's lemma, we get
$ \forall t\in [0, T, \, e^{-r(t)}|w(t)|^2_{\mathbb{L}^{2}} = 0, \, P-a.s.$ as
$P( e^{-c\int_0^{t}|u^1(s)|^{\frac{\alpha}{\alpha-1}}_{\mathbb{H}^{1, 2}}ds} <\infty)= 1 $.
The proof is achieved once we remark that thanks to \eqref{cond-solu-torus-H1},
$ P( \int_0^T|u^1(s)|^{\frac{\alpha}{\alpha-1}}_{\mathbb{H}^{1, 2}}ds <\infty)= 1$.}
\del{\noindent To prove the $ H^1-$continuity in \eqref{Eq-H-1-regu-2D-torus}, we use again \cite[Proposition VII.3.2.2]{Metivier-book-SPDEs-88}, with the gelfand triple \eqref{Gelfand-triple-Torus} and argue as above. In particular, we have 
\begin{itemize}
\item thanks to \eqref{est-B-estimatoion-q=2-eta-to-use} with $ \eta=1$ and by interpolation than by application \eqref{cond-solu-torus-H1} (remark that $ 4\frac{\alpha-1}{4-\alpha}<4$), we infer that $ B(u(\cdot, \omega))\in L^\frac{\alpha+2}{4-\alpha}(0, T; \mathbb{H}^{1-\frac\alpha2, 2}(\mathbb{T}^2)),\; P-a.s.$. In fact,
\begin{eqnarray}\label{cont-1-power-funct-alpha}
\mathbb{E} \int_0^T|B(u(s))|_{\mathbb{H}^{1-\frac\alpha2, 2}}^{\frac{\alpha+2}{4-\alpha}}ds &\leq & c \mathbb{E} \int_0^T|u(s)|^{2\frac{\alpha+2}{4-\alpha}}_{\mathbb{H}^{2-\frac\alpha2, 2}}ds
\leq c \mathbb{E} \int_0^T\Big(|u(s)|^{\frac{4-\alpha}{\alpha+2}}_{\mathbb{H}^{1+\frac\alpha2, 2}}
|u(s)|^{2\frac{\alpha-1}{\alpha+2}}_{\mathbb{H}^{1, 2}}\Big)^{2\frac{\alpha+2}{4-\alpha}}ds\nonumber\\
&\leq& c \mathbb{E} \int_0^T|u(s)|^{2}_{\mathbb{H}^{1+\frac\alpha2, 2}}
|u(s)|^{4\frac{\alpha-1}{4-\alpha}}_{\mathbb{H}^{1, 2}}ds
\del{\leq c \mathbb{E}\sup_{[0, T]}|u(s)|^2_{\mathbb{H}^{1, 2}}}<\infty. 
\end{eqnarray}
\item thanks to $\alpha\leq 2$, we have $ \frac{\alpha+2}{4-\alpha} \leq 2$ and
\begin{equation}\label{cont-1-1}
\mathbb{E} \int_0^T|A_\alpha u(s)|_{\mathbb{H}^{1-\frac\alpha2, 2}}^{\frac{\alpha+2}{4-\alpha}}ds \leq c \mathbb{E} \int_0^T|u(s)|_{\mathbb{H}^{1+\frac\alpha2, 2}}^{\frac{\alpha+2}{4-\alpha}}ds
\leq c (1+\mathbb{E}\int_0^T|u(s)|^2_{\mathbb{H}^{1+\frac\alpha2, 2}}ds)<\infty. 
\end{equation}

\item Now, we prove that\del{To prove that the trajectories of the 
continuous martingale $ M(t):= \int_0^tG(u(s))dW(s)$ belong $ P-a.s.$ to} $ M\in L^{\frac{2+\alpha}{2(\alpha-1)}}(\Omega, L^\infty(0, T; \mathbb{H}^{1+\frac\alpha2, 2}(\mathbb{T}^2)))$, we use 
Burkholdy-Davis-Gandy inequality, Assumption $ (\mathcal{C})$: (\eqref{Eq-Cond-Linear-Q-G}
with $ q=2$ and $ \delta=1$), \eqref{cond-solu-torus-H1}, we obtain 
\begin{eqnarray}\label{cont-2}
\mathbb{E}\sup_{[0, T]} |\int_0^tG(u(s))dW(s)|^2_{\mathbb{H}^{1, 2}}
&\leq& c \mathbb{E}\int_0^T|G(u(s))|^2_{L_Q(\mathbb{H}^{1, 2})}ds
\leq c \mathbb{E}\int_0^T(1+|u(s)|^2_{\mathbb{H}^{1, 2}})ds \nonumber\\
&\leq& c (1+ \mathbb{E}\sup_{[0, T]}|u(s)|^2_{\mathbb{H}^{1, 2}})<\infty.
\end{eqnarray}
\end{itemize}
Hence from \eqref{cond-solu-torus-H1}, \eqref{cont-1}, \eqref{cont-1-1}  and \eqref{cont-2}, 
we establish the existence of a subset $ \Omega'\subset \Omega$ (independent of "t"), 
such that $ P(\Omega')=0$ and
$ F(u(\cdot, \omega))\in L^2(0, T; \mathbb{H}^{-1, 2}(\mathbb{T}^2)),\; \text{and} \; u(\cdot, \omega),  M(\cdot, \omega) 
\in L^\infty(0, T; \mathbb{H}^{1, 2}(\mathbb{T}^2))\; \forall \omega \in \Omega'^c$. These ingredients are enough to apply  \cite[Proposition VII.3.2.2]{Metivier-book-SPDEs-88}\del{\cite[Proposition 2.5.]{Sundar-Sri-large-deviation-NS-06}}, hence we get the result. }

\subsection*{Proof of the pathwise uniqueness.}\label{sec-subsection-uniqueness-Torus}
Let $ u^1$ and $ u^2$ be two weak solutions of Equation \eqref{Main-stoch-eq} satisfying \eqref{eq-set-solu-weak-torus} 
and \eqref{cond-solu-torus-H1}. Let $ w:=  u^1- u^2$, then  $ w$ satisfies the following equation
\begin{eqnarray}\label{eq-uniq-w}
 w(t) = \int_0^t \big(-A_\alpha w(s)+ B(w(s), u^1(s))+ B(u^2(s), w(s))\big)ds + \int_0^t\big(G(u^1(s))- G(u^2(s))\big)dW(s).\nonumber\\
\end{eqnarray}
For $ N>0$, we define the stopping times, $ \tau_N^i: \inf\{t\in (0, T); |u^i(t)|_{\mathbb{L}^{2}}> N\}\wedge T, i=1, 2$, with the understanding that $ \inf(\emptyset)= +\infty$ and define
$ \tau_N:=\min_{i\in\{1,2\}}\{ \tau_N^i\}$.
Using Ito formula for the product $ e^{-r(t)}|w(t)|^2_{\mathbb{L}^{2}}$, with $(r(t), t\in [0, T])$ being a
 positive real stochastic process to be defined later, Property \eqref{Eq-3lin-propnull},
Assumption $(\mathcal{C})$ (\eqref{Eq-Cond-Lipschitz-Q-G} with $ q=2$ and $\delta =0$, locally Lipschitz), 
Estimate \eqref{3linear-H1+alpha2-H-1}, Young inequality  and arguing as in the proof of \eqref{Eq-Ito-n-weak-estimation-torus}\del{
\eqref{Eq-Ito-n-weak-estimation-1-Torus}} with the replacement of
the spaces $ \mathbb{H}^{1, 2}(\mathbb{T}^2)$ and $ \mathbb{H}^{1+\frac\alpha2, 2}(\mathbb{T}^2)$ by
$ \mathbb{L}^{2}(\mathbb{T}^2)$ respectively $ \mathbb{H}^{\frac\alpha2, 2}(\mathbb{T}^2)$,  we infer that for $ 1\leq \alpha< 2$ 
(here we omit to writ the proof for the dissipative regime, as it is classical.)
\del{\begin{eqnarray}\label{Torus-uniquenss-ito-formula-H-1}
\mathbb{E}\!\!\!\!&{}&\!\!\!\!e^{-r(t\wedge \tau_N)}|w(t\wedge \tau_N)|^2_{\mathbb{L}^{2}}+
2\mathbb{E}\int_0^{t\wedge \tau_N}e^{-r(s)}|w(s)|^2_{\mathbb{H}^{\frac\alpha2, 2}}ds \nonumber\\
&\leq& \mathbb{E}\int_0^{t\wedge \tau_N}e^{-r(s)}|| G(u^1(s))- G(u^2(s))||^2_{L_Q(\mathbb{L}^2)}ds\nonumber\\
& -& \mathbb{E}\int_0^{t\wedge \tau_N} e^{-r(s)} \big(2\langle B(w(s)),  u^1(s)\rangle-
r'(s)|w(s)|^2_{\mathbb{L}^{2}}\big)ds\nonumber\\
& \leq & c_N\mathbb{E}\int_0^{t\wedge \tau_N}e^{-r(s)}\big(|w(s)|^2_{\mathbb{L}^2}+
|u^1(s)|_{\mathbb{H}^{1, 2}}|w(s)|^{\frac2\alpha}_{\mathbb{H}^{\frac{\alpha}{2}, 2}}
|w(s)|^{2\frac{\alpha-1}\alpha}_{{\mathbb{L}^{2}}}-
r'(s)|w(s)|^2_{\mathbb{L}^{2}}\big)ds\nonumber\\
& \leq & c_N
\mathbb{E}\int_0^{t\wedge \tau_N}e^{-r(s)}\big(|w(s)|^2_{\mathbb{L}^2}+ 2 c|w(s)|^2_{\mathbb{H}^{\frac\alpha2, 2}}+
2c_1|u^1(s)|^{\frac{\alpha}{\alpha-1}}_{\mathbb{H}^{1, 2}}|w(s)|^2_{\mathbb{L}^{2}}-
r'(s)|w(s)|^2_{\mathbb{L}^{2}}\big)ds.\nonumber\\
\end{eqnarray}}
\begin{eqnarray}\label{Torus-uniquenss-ito-formula-H-1}
\mathbb{E}\!\!\!\!&{}&\!\!\!\!e^{-r(t\wedge \tau_N)}|w(t\wedge \tau_N)|^2_{\mathbb{L}^{2}}+
2\mathbb{E}\int_0^{t\wedge \tau_N}e^{-r(s)}|w(s)|^2_{\mathbb{H}^{\frac\alpha2, 2}}ds \nonumber\\
&\leq& \mathbb{E}\int_0^{t\wedge \tau_N}e^{-r(s)}|| G(u^1(s))- G(u^2(s))||^2_{L_Q(\mathbb{L}^2)}ds\nonumber\\
& -& \mathbb{E}\int_0^{t\wedge \tau_N} e^{-r(s)} \big(2\langle B(w(s)),  u^1(s)\rangle-
r'(s)|w(s)|^2_{\mathbb{L}^{2}}\big)ds\nonumber\\
& \leq & c_N\mathbb{E}\int_0^{t\wedge \tau_N}e^{-r(s)}\big(|w(s)|^2_{\mathbb{L}^2}+
| u^1|_{\mathbb{H}^{1+\frac\alpha2, 2}}|w|^{\frac{2-\alpha}{\alpha}}_{\mathbb{H}^{\frac{\alpha}{2}, 2}}
|w|^{\frac{3\alpha-2 }{\alpha}}_{{\mathbb{L}^{2}}}-
r'(s)|w(s)|^2_{\mathbb{L}^{2}}\big)ds\nonumber\\
& \leq & c_N
\mathbb{E}\int_0^{t\wedge \tau_N}e^{-r(s)}\big(|w(s)|^2_{\mathbb{L}^2}+ 2 c|w(s)|^2_{\mathbb{H}^{\frac\alpha2, 2}}+
2c_1|u^1(s)|^{\frac{2\alpha}{3\alpha-2}}_{\mathbb{H}^{1+\frac\alpha2, 2}}|w(s)|^2_{\mathbb{L}^{2}}-
r'(s)|w(s)|^2_{\mathbb{L}^{2}}\big)ds.\nonumber\\
\end{eqnarray}
We choose $ c<1$ and $ r'(s)= 2c_1|u^1(s)|^{\frac{2\alpha}{3\alpha-2}}_{\mathbb{H}^{1+\frac\alpha2, 2}}$ and replace in
\eqref{Torus-uniquenss-ito-formula-H-1}, we end up with the simple formula
\begin{eqnarray}
\mathbb{E}e^{-r(t\wedge \tau_N)}|w(t\wedge \tau_N)|^2_{\mathbb{L}^{2}}&+& 2(1-c)\mathbb{E}\int_0^{t\wedge \tau_N}
e^{-r(s)}|w(s)|^2_{\mathbb{H}^{\frac\alpha2, 2}}ds\nonumber\\
&\leq& c_N\mathbb{E}\int_0^{t}e^{-r(s\wedge \tau_N)}|w(s\wedge \tau_N)|^2_{\mathbb{L}^2}ds.
\end{eqnarray}
By application of Gronwall's lemma, we get
$ \forall t\in [0, T]$ the random variable $|w(t\wedge \tau_N)|^2_{\mathbb{L}^{2}} = 0, \, P-a.s.$ as much as 
$ P( \int_0^{t\wedge \tau_N }|u^1(s)|^{\frac{2\alpha}{3\alpha-2}}_{\mathbb{H}^{1+\frac\alpha2, 2}}ds <\infty)= 1$. 
This last statement is confirmed thanks to 
\eqref{cond-solu-torus-H1} and the condition $ 1\leq \alpha<2$. 
The proof is then achieved once we remark that thanks to Chebyshev inequality and \eqref{cond-solu-torus-H1},
we have $ \lim_{N\rightarrow \infty}\tau_N = T, P-a.s.$ and thanks to the $\mathbb{L}^2-$continuity of $ (w(t), t\in[0, T])$.
\subsection*{Proof of the space regularity.}
In the aim to get Estimate \eqref{propty-of-2D-global-mild-sol},\del{(the first term of this estmate),}\del{\eqref{est-u-torus-H-1-q-1}} we use the regularization effect of the vorticity. Let $ (u(t), t\in [0, T])$ be a
weak solution of Equation \eqref{Main-stoch-eq} in the sense of Definition
\ref{def-variational solution}. Thanks to Appendix \ref{sec-Passage Velocity-Vorticity}, the $ curl u \del{\in L^{2}(\mathbb{T}^2)}$ is a weak solution 
of  Equation \eqref{Eq-vorticity-Torus-2-diff}. We know from \cite{Debbi-scalar-active} that this equation\del{ \eqref{Eq-vorticity-Torus-2-diff}} admits a unique global solution  which is simultaneously weak and mild and satisfies
\begin{equation}\label{eq-est-theta-brut}
 \mathbb{E}\Big(\sup_{[0, T]}| \theta(t)|_{L^q}^q + \int_0^T
 | \theta(t)|_{H^{\frac\alpha2, 2}}^2dt\Big)<\infty,
\end{equation}
for $q_0, q $ and $ \alpha$
being characterized as in $(6.3.2.)$\del{Theorem \ref{Main-theorem-strog-Torus}} and  
provided that $ curl u_0$\del{$ curl u_0\in L^p(\Omega, \mathcal{F}_0, P; L_1^{q_0}(O))$} fulfills  \eqref{eq-curl-u-0-torus} and 
$ \tilde{G}$, defined by \eqref{inte-g-tilde}, satisfies the Lipschitz and the growth conditions, i.e.  $ \tilde{G}$ satisfies \eqref{Eq-Cond-Lipschitz-Q-G} and \eqref{Eq-Cond-Linear-Q-G}, with   
$R_Q(\mathbb{L}^2, \mathbb{H}^{\delta, q})$ in the LHSs and 
$ \mathbb{H}^{\delta, q}$ in the RHSs are replaced by 
$R_Q(L^2, L^q) $ and $ L^q$ respectively. As \eqref{eq-curl-u-0-torus} is fulfilled by assumption, we check that the two latter conditions are also satisfied. In fact, thanks to the definition of $ \tilde{G}$, Assumption $(\mathcal{C})$, with $ \delta =1$
and Lemma \ref{lem-R1-bounded},\del{the fact that the
operator $\mathcal{R}^{1}$ is bounded, on $\mathbb{H}^{s, q}(\mathbb{T}^2)$,} we get
\begin{eqnarray}
|| \tilde{G}(\theta)||_{R_\gamma(L^2, L^q)}&=&|\big[\sum_{k\in\Sigma}| curl G(\mathcal{R}^1(\theta))Q^\frac12 e_k|^2\big]^\frac12|_{L^q}
\leq c|\big[\sum_{k\in\Sigma}| \partial_j G(\mathcal{R}^1(\theta))Q^\frac12 e_k|^2\big]^\frac12|_{L_2^q}\nonumber\\
&\leq & c ||G(\mathcal{R}^1(\theta))||_{R_\gamma(\mathbb{L}^2, \mathbb{H}^{1, q})}\leq c(1+|\mathcal{R}^1(\theta)|_{\mathbb{H}^{1, q}})\leq c(1+|\theta|_{L^q}). 
\del{\sum_{k\in\Sigma}| curl \sigma_k(\mathcal{R}^1(\theta(s)))|_{L^q}\nonumber\\
&\leq &c \sum_{k\in\Sigma}| \sigma_k(\mathcal{R}^1(\theta(s)))|_{H^{1,q}}
\leq c (1+| \mathcal{R}^1(\theta(s))|_{H^{1,q}})\leq c (1+| \theta(s)|_{H^{1,q}}).}
\end{eqnarray} 
By the same way, we prove the Lipschitz condition. Estimate \eqref{propty-of-2D-global-mild-sol} follows from \eqref{eq-est-theta-brut} and Lemma \ref{lem-basic-curl-gradient}.

\noindent {\bf Preliminary Notations \& General Remarks}
Let  $ \mathbb{N}_k:=\del{\mathbb{N}-}\{j\in \mathbb{N},\; s.t.\; j>k\}$\del{, $ \mathbb{Z}^d:=\{(k_1, k_2, \cdots, k_d),\; k_j\in \}$} and 
$ \mathbb{Z}^d_0:=\mathbb{Z}^d-\{0\}$. 
For $ d \in \mathbb{N}_0$, we denote by $ \mathbb{T}^d$  the $d-$dimensional torus and by $ D(\mathbb{T}^d)$ the set of infinitely differentiable scalar-valued (complex) functions on $ \mathbb{T}^d$.\del{$ C^k(\mathbb{T}^d), k\in \mathbb{N}_0, $ is the set of  $k-$differentiable functions on $ \mathbb{T}^d$,} By a domain " $ O$" we mean an open non empty set.
For either $ O=\mathbb{T}^d $ or $ O\subset \mathbb{R}^d$ bounded, we define $ H_l^{\beta, q}(O):= (H^{\beta, q}(O))^l, l\in \mathbb{N}_0, 
\beta \in \mathbb{R}, 1<q<\infty$, in particular for $ \beta =0$,  $L^q_l(O):= (L^q(O))^l$. Recall that $ H^{\beta, q}(O)$, according to $ O$, are either the Sobolev spaces on a bounded domain or the null average periodic Sobolev spaces on the torus. $ C_0^\infty(O) $ is the set of infinitely differentiable real functions\del{ with values in $\mathbb{R}$}\del{ or in $\mathbb{C}$),} with compact support on the bounded domain $ O\subset \mathbb{R}^d(O)$, $ \mathring{H}_d^{\beta, q}(O), \beta \in \mathbb{R}_+, 1<q<\infty$ is the completion of $ C_0^\infty(O) $ in $ H_d^{\beta, q}(O)$, with $ O\subset \mathbb{R}^d$ bounded. $ \partial_{x_j}$ stands for the partial derivative
with respect to the component $ x_j$, sometimes we also use the notation  $ \partial_j$. We use the notation $ |\cdot|_{X}$ to indicate the norm in $ X$. For simplicity, we denote the norm of a matrix by the corresponding scalar  space notation of the components or by a symbol of this space. The Sobolev norms used are those defined by Riesz-potential. We denote by $\mathcal{L}(X)$ the set of  bounded linear operators on a Banach space $ X$.  
The notation $|.|_{L^{r_1}\rightarrow  L^{r_2}}$
stands for the usual norm of bounded operators from  $L^{r_1}$ to $  L^{r_2}$. 
The classification of the subcritical,  critical and
superctitical regimes corresponds to  $ \alpha \in (1, 2)$, $ \alpha =1$ and $ \alpha \in (0, 1)$ respectively.
The dissipative ( sometimes called also the Laplacian dissipation) and the hyperdissipative regimes correspond to $ \alpha =2$ respectively to 
$ \alpha >2$. \del{The abbreviations (FSNSE), (SFNSE)and (FNSE) are used respectively for 
fractional 
stochastic Navier-Stokes equation, the
stochastic fractional Navier-Stokes equation, the deterministic fractional 
stochastic Navier-Stokes equation.}The abbreviations (FSNSE), (SNSE) and (FNSE) are used respectively for 
fractional 
stochastic Navier-Stokes equation, the
stochastic Navier-Stokes equation and the deterministic fractional 
stochastic Navier-Stokes equation. The abbreviation i.i.d  means independent and identically distributed.
 $ \{a_1, a_2\}\leq_k b$ (respectively $ \{a_1, a_2\}\geq_k b$) means $ a_k\leq b$, $ a_j< b, \; j\neq k$ 
and $ a_1=a_2 < b$ (respectively $
a_k\geq b$, $ a_j> b, \; j\neq k$ and $ a_1=a_2 > b$).  The expression $ q\leq_\infty q_0$ means 
$ q\leq q_0<\infty$ and $ q< q_0=\infty$.\del{ $ \max^1_>\{a, b\} $ equals $ a$, if $ a>b$ and strictly greater 
than $ b$ if  $ a\leq b$. A mathematical notion with statement is noted by the usual notation with subscription
 "s" e.g. $ [_s=[$ if the statement "s" is satisfied and $ [_s=($ if not.}
We say that  $ q^*$ is the  conjugate of $q $, if for $ 1<q<\infty$,
 $ q^*$ satisfies the equation $ \frac 1q+\frac{1}{q^*} = 1$  and for $ q=1$ respectively 
$ q=\infty$, $ q^*=\infty$ respectively $ q^*=1$.
We define, in distribution sense, the curl of a vector \del{distribution} field 
 $ v=(v_1, v_2)$ by $ curl v := \partial_1v_2 -\partial_2v_1 $. The vorticity matrix of a $dD-$vector field $ v $ on $ \mathbb{R}^d$\del{ to $ \mathbb{R}^d$} is the null diagonal, antisymmetric matrix defined by $ \Omega(v):= ((\Omega(v))_{i, j})_{1\leq i, j\leq d}$, where $ (\Omega(v))_{i, j}:= \partial_i v_j-\partial_j v_i$. For $ d=2$, the vorticity $ \Omega(v)$ is identified to the scalar function $curl v$ and for $ d=3$ to the transpose of the $3D-$vector function $ (\partial_2 v_3-\partial_3 v_2, \partial_1 v_3-\partial_3 v_1, \partial_1 v_2-\partial_2 v_1)$. In Appendix \ref{Sobolev pointwise multiplication-Bounded-Domain}, we have proved that if a Sobolev\del{ embedding and} pointwise multiplication estimate is satisfied for Sobolev spaces on $ \mathbb{R}^d$ and if $ O\subset \mathbb{R}^d$ is a "good" bounded domain, then this pointwise multiplication estimate is also valid for Sobolev spaces on bounded domains. Therefore, in many cases, we referee directely to the source of the estimate  on $ \mathbb{R}^d$.\del{ Moreover, we
also pass from vectorial valued functions to real values functions and vise-versa without 
mentioning every time.}  We use the Einstein summation
convention. Constants vary from line to line and we often delete
their dependence on parameters.

\del{\section{Global existence and uniqueness of the mild solution of the multi-dimensional FSNSEs.}\label{sec-global-mild-solution}
Here we prove (\ref{Main-theorem-mild-solution-d}.2-3). Recall that we have constructed in Section \ref{sec-1-approx-local-solution}, see also Appendix \ref{appendix-local-solution}, a maximal local mild solution $ (u, \tau_\infty)$, where $ \tau_\infty $  is defined by 
\eqref{Eq-def-tau-n-delta} and \eqref{Eq-def-tau-delta} and $ u(t)=u_N(t),$ for $ t\leq \tau_N$\footnote{we use capital N to avoid confusion with the Faedo-Galerkin approximations.}, provided that $ \alpha \in(1+\frac dq, 2]$, $ q>d$ and $u_0$ and $ G$ satisfying assumptions $ (\mathcal{B})$ and $ (\mathcal{C})$ respectively. The solution can start from an $ \mathbb{L}^q-$ initial data, therefore we take  $ \delta =0$. To prove that the local solution is global, it is sufficient to prove that the $ \mathbb{L}^q-$norm does not explode,  i.e. for $ P-a.s.$, $ \tau_\infty =T$. 

\vspace{0.25cm}

In the case $ O=\mathbb{T}^2$, thanks to Lemma \ref{lem-basic-curl-gradient} and \cite[Theorem 2.6]{Debbi-scalar-active} and according to the cases cited in $ (\ref{Main-theorem-mild-solution-d}.2)$, we infer that
\begin{equation}\label{est-nabla-u-theta-Global-existence}
 \exists c>0, s.t. \forall N\in \mathbb{N}_0, \mathbb{E}\sup_{[0, \tau_N]}|\nabla u(t)|^q_{q}\leq c \mathbb{E}\sup_{[0, \tau_N]}|\theta(t)|^q_{L^{q}}\leq c<\infty.
\end{equation}
Thus both the $\mathbb{L}^q$  and the $\mathbb{H}^{1, q}$  norms do not explode, therefore the solution is global. The estimation \eqref{est-nabla-u-theta-Global-existence} is also enough to prove the uniqueness as we shall show for the general case bellow, see also \cite{Debbi-scalar-active}. The regularity in \eqref{propty-of-2D-global-mild-sol} is a consequence of the application of \cite[Theorem 2.6]{Debbi-scalar-active} and Lemma \ref{lem-basic-curl-gradient} as we have seen in Section \ref{sec-Torus}.

\vspace{0.25cm}

\noindent For the general case, we adapt to the stochastic framework, the method used in \cite{Giga-al-Globalexistence-2001} and than apply Hasminskii's criteria \cite{BrzezniakDebbi1,  Brzez-Beam-eq, Hasminski-book-80}.\del{ It is easily seen that Thanks to Assumption $ (\mathcal{C}_b)$ the OU process, $ (z(t\wedge \tau_\infty), t\in[0, T])$ defined by  \eqref{Eq-z-t}, enjoys
\begin{equation}\label{cond-z-nabla-z}
 \mathbb{E}\sup_{[0, \tau)}(|z(s)|^2_{\mathbb{L}^q}+ |\nabla z(s)|^2_{q}) 
\leq c<\infty.
\end{equation}} Using Lemma \ref{Giga-Mikayawa-solutionLr-NS} and Lemma \ref{Lem-semigroup}, the continuity of the 
Helmholtz projection and  H\"older inequality, we get $P-a.s$ for all $ N\in \mathbb{N}_0$ and  $ t\in [0, T]$, 
\begin{eqnarray}\label{est-b-global-mild-solu-0}
|\int_0^{t\wedge\tau_N} e^{-(t\wedge\tau_N-s)A_\alpha}B(u(s))ds|_{\mathbb{L}^q}
&\leq & c \int_0^{t\wedge\tau_N} |A_\alpha^\frac d{\alpha q} e^{-(t\wedge\tau_N-s)A_\alpha}|_{\mathcal{L}(\mathbb{L}^{\frac q2})} |\Pi((u(s)\nabla)u(s))|_{\mathbb{L}^{\frac q2}}ds\nonumber\\
&\leq& c\int_0^{t\wedge\tau_N} (t\wedge\tau_N-s)^{-\frac d{\alpha q}}  |u(s)|_{\mathbb{L}^q}|\nabla u(s)|_{q}ds.
\end{eqnarray}
Therefore, as the local solution $ (u, \tau_\infty)$ satisfies Equation \eqref{Eq-Mild-Solution-stoped} up to $ \tau_\infty$, we infer that \del{for all $N \in \mathbb{N}_0 $, we have $P-a.s$ for all $ N\in \mathbb{N}_0$ and $ t\in [0, T]$,\del{ Thus for all $ t\leq \tau_\infty$ we have}
\begin{equation}\label{v-mild-sol-aux-pbm}
 |u(t\wedge\tau_N)|_{\mathbb{L}^q}= |e^{-tA_\alpha}u_0|_{\mathbb{L}^q}+ |\int_0^t e^{-(t-s)A_\alpha}B(u(s))|_{\mathbb{L}^q} + |z(s)|_{\mathbb{L}^q}.
\end{equation}}
\begin{eqnarray}\label{est-b-global-mild-solu-0}
|u(t\wedge\tau_N)|_{\mathbb{L}^q}&\leq& c\Big(|u_0|_{\mathbb{L}^q}+ c\int_0^{t\wedge\tau_N} (t\wedge\tau_N-s)^{-\frac d{\alpha q}}  |u(s)|_{\mathbb{L}^q}|\nabla u(s)|_{q}ds+ |z(t\wedge\tau_N)|_{\mathbb{L}^q}\Big),\nonumber\\
\end{eqnarray}
where the stopped Ornstein-Uhlenbeck process $( z(t\wedge\tau_N), t\in [0, T])$ is obtained as in  Remark \ref{Rem-1}. By application of  Gronwall's lemma, we conclude that  $P-a.s$ for all $ N\in \mathbb{N}_0$ and $ t\in [0, T]$, 
\del{ for all $ t\leq \tau_N$,}
\begin{eqnarray}
|u(t\wedge\tau_N)|_{\mathbb{L}^q}&\leq& c\Big(|u_0|_{\mathbb{L}^q}+ \sup_{[0, \tau_\infty)}|z(s)|_{\mathbb{L}^q}\Big)\exp{c\int_0^{t\wedge\tau_N} (t\wedge\tau_N -s)^{-\frac d{\alpha q}}|\nabla u(s)|_{q}ds}.\nonumber\\
\end{eqnarray}
\noindent We apply the increasing function $\ln^+x=\max\{0,\ln x\}$  to
 both sides of the above inequality and than we use the classical
inequalities for $ a, b > 0$, $\ln^+(a+b)\leq \ln^+(a)+ \ln^+(b)$, $
\ln^+(ab)\leq \ln^+(a)+\ln^+(b)$ and $a\leq 1+a^2$, the elementary property 
\begin{equation}\label{elementary-property}
\exists c>0,\;\;s.t.\; ln^+x \leq x +c,\;\; \forall x>0.
\end{equation}
\noindent Estimate  \eqref{Eq-nabla-z-t}\del{{cond-z-nabla-z}} and  Condition \eqref{cond-global-mild-solu}, we get for all $ N\in \mathbb{N}_0$ and $ t\in[0, T]$,
\begin{eqnarray}
\mathbb{E}\ln^+|u(t\wedge \tau_N)|_{\mathbb{L}^q}&\leq& c\Big(\mathbb{E}\ln^+|u_0|_{\mathbb{L}^q}+ \mathbb{E}\ln^+\sup_{[0, \tau_\infty)}|z(s)|_{\mathbb{L}^q} +\mathbb{E}\sup_{[0, \tau_\infty)}\int_0^t (t-s)^{-\frac d{\alpha q}}|\nabla u(s)|_{q}ds\Big)\nonumber\\
&\leq& c\Big(1+\mathbb{E}\ln^+|u_0|_{\mathbb{L}^q}\Big)<\infty.
\end{eqnarray}
Now, we introduce the function $ V: u\in \mathbb{L}^q(O) \mapsto \mathbb{R}_+ \ni V(u)=ln^+(|u|_{\mathbb{L}^q})$. 
Obviously  $V$ is uniformly continuous on bounded sets and from the 
calculus above, it is easy to check that $ V$ is a Lyapunov function.  In fact, 
\begin{eqnarray} V &\geq&  0 \; \text{on}\;\;  \mathbb{L}^q(O), \label{2.1}
\\
q_N &:=& \inf_{|u|_{\mathbb{L}^q\ge N}} ln^+(|u|_{\mathbb{L}^q}) \to \infty \text{ as
 }\;N\to\infty, \label{2.2}
\\
\mathbb{E} ln^+(|u_0|_{\mathbb{L}^q}) &<& c+ \mathbb{E}|u_0|_{\mathbb{L}^q}<\infty, \label{2.3}
\\ 
\text{and} &&\nonumber\\
\mathbb{E} V( u(t\land \tau_N)) &\leq&  c\bigl(1+\mathbb{E} V(u_0) \bigr)
\text{ for all } t\ge 0,\;  N\in\mathbb{ N}_0. \label{2.5}
\end{eqnarray}
Consequently, 
\begin{eqnarray}\label{eq-prob-tau-n=0}
\lim_{N\rightarrow +\infty} P \bigl\{ \tau_N<t\bigr\} &\leq&
 \lim_{N\rightarrow +\infty} \frac1{q_N}\mathbb{E}
\bold 1_{\{\tau_N <t\}} V( u(t\land \tau_N))\nonumber\\
&\leq& \lim_{N\rightarrow +\infty} \frac1{q_N}c\bigl(1+\mathbb{E} V( u_0)\bigr)=0.
\end{eqnarray}
\noindent By this we achieve the proof of the global existence of the mild solution. The next step is to prove the uniqueness. Let  $ (u^1(t), t\in [0, T])$ and  $ (u^2(t), t\in [0, T])$  be two  mild solutions
satisfying\del{ the estimation \eqref{eq-last-est} and} $ P(u^1(0) = u^2(0)= u_0)=1$  and let $
u^0:= u^1- u^2$. We define for $ R>0$ the stopping times
\begin{equation}
 \tau^i_R := \inf\{t\in (0, T), \;\; s.t. \;\; |u^i(t) |_{\mathbb{L}^q}> R\}, \;\;\; \inf(\emptyset)=+\infty\;\;\; \text{and}\;\;\; \tau_R:= \tau^1_R\wedge \tau^2_R.
\end{equation}
\del{with the understanding that $ \inf(\emptyset)=+\infty$. } We show that for all $ R>0 $, the stopped processes  of $( u^1(t), t\in [0, T\wedge \tau_R])$ and $ ( u^2(t), t\in [0, T\wedge \tau_R]) $  are modifications of each other.
Than thanks \del{to the facts,
\begin{equation}\label{eq-lim-tau-n}
\lim_{R\rightarrow \infty}P(\tau_R<T) \leq \sum_{i=1}^2 \lim_{R\rightarrow \infty}P(\tau^i_R<T)
\leq \sum_{i=1}^2 \lim_{R\rightarrow \infty} \frac{\mathbb{E}|u^i(t)|^k_{L^q}}{R^k}=0,
\end{equation}}
to \eqref{eq-prob-tau-n=0} we conclude that the same is true for $ (u^1(t), t\in[0, T])$ and $ (u^2(t), t\in[0, T])$.
The stopped process $u^0$ satisfies the following equation $ P-a.s$ for all $ t\in [0, T]$, see \cite{BrzezniakDebbi1, Brzez-Beam-eq, Debbi-scalar-active},
\begin{equation}\label{eq-theta-0}
u^0(t\wedge \tau_R)= I^1_{\tau_R}(t\wedge \tau_R) +I^2_{\tau_R}(t\wedge \tau_R),
\end{equation}
where
\begin{equation}
 I^1_{\tau_R}(t\wedge \tau_R):=\int_0^{t\wedge \tau_R}
 e^{(t\wedge \tau_R-s)A_\alpha}(B(u^1(s\wedge \tau_R))-B(u^2(s\wedge \tau_R)))ds
\end{equation}
and
\begin{equation}
I^2_{\tau_R}(t\wedge \tau_R):= \int_0^{t}
1_{[0, \tau_R)}(s) e^{(t-s)A_\alpha}(G(u^1(s\wedge \tau_R))-G(u^2(s\wedge \tau_R)))dW(s).
\end{equation}
 First let us remark that 
\begin{equation}\label{eq-I-1-out-Tau}
1_{[0, \tau_R]}(t) I^1_{\tau_R}(t\wedge \tau_R)= 1_{[0, \tau_R]}(t)\int_0^{t}
 1_{[0, \tau_R]}(s)e^{(t-s)A_\alpha}(B(u^1(s\wedge \tau_R))-B(u^2(s\wedge \tau_R)))ds.
\end{equation}
Now let $ p>2$. Using formulas \eqref{formula-B-v1-B-v2} and \eqref{eq-I-1-out-Tau}, a similar calculation as in the proof of 
Proposition \ref{Prop-Main-I} and H\"older inequality in $ L^1(0, t; t^{-\gamma}dt)$ with $ \gamma := \frac1\alpha(1+\frac dq)<1$, 
we estimate $I^1_{\tau_R}(t\wedge \tau_R)$ as follow
\begin{eqnarray}\label{eq-est-B-theta-0}
\mathbb{E}\!\!\!\!\!&|&\!\!\!\!\!1_{[0, \tau_R]}(t)I^1_{\tau_R}(t\wedge \tau_R)|_{\mathbb{L}^q}^p
= \mathbb{E} |1_{[0, \tau_R]}(t)\int_0^{t}
 e^{(t-s)A_\alpha}1_{[0, \tau_R]}(s)(B(u^1(s\wedge \tau_R))-B(u^2(s\wedge \tau_R)))ds|^p_{\mathbb{L}^q} \nonumber\\
&\leq& c\int_0^{t}(t-s)^{-\gamma}\mathbb{E}\Big[|1_{[0, \tau_R]}(s)
(u^1(s\wedge \tau_R)-u^2(s\wedge \tau_R))|^p_{\mathbb{L}^q}(\sup_{[0, \tau_R]}(|u^1(s)|_{L^q}^p+ |u^2(s)|_{\mathbb{L}^q}^p)) \Big]ds \nonumber\\
&\leq& cR^p\int_0^{t}(t-s)^{-\gamma}\mathbb{E}|1_{[0, \tau_R]}(s)
(u^1(s\wedge \tau_R)-u^2(s\wedge \tau_R))|^p_{\mathbb{L}^q} ds. \nonumber\\
\end{eqnarray}
\noindent Furthermore, using \cite[Proposition 4.2.]{Neerven-Evolution-Eq-08} and a similar calculation 
as in Lemma \ref{lem-est-z-t}, in particular Assumption $(\mathcal{C})$, we infer the existence of
$ c_R>0, \; 0<\gamma_1<\frac12$ such that
\begin{eqnarray}\label{eq-est-G-theta-0}
\mathbb{E}\!\!\!\!\!&[&\!\!\!\!\!1_{[0, \tau_R]}(t)I^2_{\tau_R}(t\wedge \tau_R)|_{\mathbb{L}^q}^p\nonumber\\
&=&\mathbb{E} |1_{[0, \tau_R]}(t)\int_0^{t}
 e^{(t-s)A_\alpha}1_{[0, \tau_R]}(s)(G(u^1(s\wedge \tau_R))-G(u^2(s\wedge \tau_R)))dW(s)|^p_{\mathbb{L}^q} \nonumber\\
&\leq& c\mathbb{E}\Big(\int_0^{t}(t-s)^{-\gamma_1}||1_{[0, \tau_R]}(s)(
G(u^1(s\wedge \tau_R))-G(u^2(s\wedge \tau_R)))||^2_{R_\gamma(\mathbb{L}^2, \mathbb{L}^q)}ds \Big)^\frac p2 \nonumber\\
&\leq& c_R\int_0^{t}(t-s)^{-\gamma_1}\mathbb{E}|1_{[0, \tau_R]}(s)
(u^1(s\wedge \tau_R)-u^2(s\wedge \tau_R))|^p_{\mathbb{L}^q} ds. 
\end{eqnarray}
Using  \eqref{eq-est-B-theta-0} and \eqref{eq-est-G-theta-0} and \eqref{eq-theta-0}, we infer the existence of
an $L^1-$integrable function  $ \psi: (0, T]\rightarrow \mathbb{R}_+$ such that
\begin{equation}\label{eq-theta-0-with-phi}
\mathbb{E}|1_{[0, \tau_R]}u^0(t\wedge \tau_R)|_{\mathbb{L}^q}^p\leq c_R\int_0^{t}\psi(t-s)\mathbb{E}|1_{[0, \tau_R]}(s)
u^0(s\wedge \tau_R)|^p_{\mathbb{L}^q} ds.
\end{equation}
Using Gronwall Lemma, we get for all $ t\in [0, T]$,
\begin{equation}
 | u^1(t\wedge\tau_R)- u^2(t\wedge\tau_R)|_{\mathbb{L}^q}=0. \;\;\; a.s., \;\;\; \forall t\in [0, T].
\end{equation}
Thanks to Estimate \eqref{eq-prob-tau-n=0} and to the contonuity of $ (u^1(t), t\in [0, T])$ and $ (u^2(t), t\in [0, T])$ the pathwise uniqueness holds.\del{ we infer that for all $ t\in [o, T]$,}
\del{\begin{equation}
  u^1(t)= u^2(t)=0. \;\;\; a.s. \;\;\; \del{\forall t\in [0, T]}
\end{equation}}
\del{\noindent To complete the proof of the uniqueness let us recall that Estimation \eqref{eq-last-est} is satisfied by the construction of the solution ($ \mathbb{E}\sup_{[0, T]}|u(t)|^p_{\mathbb{L}^q}<\infty$).}

\del{\noindent The whole calculus above remains valid to prove the results in $(3.6.3)$, in the exception that we have not \eqref{est-nabla-u-theta-Global-existence}. Hence, to estimate \eqref{eq-lastln-u}, we use \eqref{cond-mild-3d} and \eqref{elementary-property}.} 

\section{Martingale solution of the multi-dimensional FSNSEs.}\label{sec-Marting-solution}
In this section, we prove Theorem \ref{Main-theorem-martingale-solution-d}. The main ingredients are Faedo-Galerkin approximations,
compactness, Skorokhod embedding theorem and the representation theorem. In particular, once we prove Lemma \ref{lem-bounded-W-gamma-p} bellow, we can follow the same scheme e.g. as in \cite{Capinski-peaszat-Mart-SNSE-01, Flandoli-Gatarek-95}, see also similar calculus for the fractional stochastic scalar active equation in \cite{Debbi-scalar-active}. Thus we omit to give full details.\del{ in Appendix \ref{Appendix-Martingale-solu}.}

\begin{lem}\label{lem-bounded-W-gamma-p}
The sequence $ (u_n)_n$ of solutions of the equations \eqref{FSBE-Galerkin-approxi} is uniformly bounded in the space
\begin{eqnarray}\label{Eq-W-}
L^2(\Omega, W^{\gamma, 2}(0, T; \mathbb{H}^{-\delta', 2}(O))
\del{H_d^{-\delta', 2}(O))}\cap L^2(0, T; \mathbb{H}^{\frac\alpha2, 2}(O))),
\end{eqnarray}
where  $ \delta'\geq_1\max\{\alpha, 1+\frac d{2}\}$ and $ \gamma <\frac12$.
\end{lem}
\begin{proof}
Thanks to Lemma \ref{lem-unif-bound-theta-n-H-1-domain}, it is sufficient to prove that $ (u_n(t), t\in [0, T])$ is
uniformly bounded in $L^2(\Omega, W^{\gamma, 2}(0, T; \mathbb{H}^{-\delta', 2}(O))$. We recall that the Besov-Slobodetski space $W^{\gamma,
p}(0, T; E)$, with $ E$ being a Banach space, $ \gamma \in (0, 1)$ and $ p\geq 1$, is the space
of all $ v\in L^P(0, T; E) $ such that
\begin{eqnarray}
||v||_{W^{\gamma, p}}:= \left(\int_0^T|v(t)|_E^pdt+
\int_0^T\int_0^T\frac{|v(t)-v(s)|_E^p}{|t-s|^{1+\gamma p}}
dtds\right)^{\frac1p}<\infty.
\end{eqnarray}
\noindent As  $(u_n(t), t\in [0, T])$ is the strong solution  of the finite dimensional  stochastic
differential equation \eqref{FSBE-Galerkin-approxi}, then  $u_n(t)$  is the solution of the stochastic integral equation 
\begin{equation}\label{FSBE-Integ-solu-Galerkin-approxi}
u_n(t)= P_nu_0 + \int_0^t(-A_\alpha u_n(r) + P_nB(u_n(r))dr + \int_0^tP_nG(u_n(r))\,dW_n(r),\; a.s.,\\
\end{equation}
for all $t\in [0, T]$. We denote by 
\begin{equation}\label{Eq-Drift-term}
I(t):=  \int_0^t(-A_\alpha u_n(r) + P_nB(u_n(r))dr
\end{equation}
and
\begin{equation}\label{Eq-Drift-term}
J(t):= \int_0^tP_nG
(u_n(r))\,dW_n(r).
\end{equation}
\noindent We prove that $ I(\cdot)$ is uniformly bounded in  $L^2(\Omega; W^{\gamma, 2}(0, T; \mathbb{H}^{-\delta', 2}(O))$ 
and that the stochastic term $ J(\cdot )$ is uniformly bounded in $ L^2(\Omega; W^{\gamma, 2}(0, T;  \mathbb{L}^2(O))$, for all $ \gamma<\frac12$.\del{ as the stochastic term $ J$ is more regular then the drift term, see e.g. \cite[Lemma 2.1]{Flandoli-Gatarek-95}. \\}
Let  $ \phi \in \mathbb{H}^{{\delta'}, 2}(O)$, using Identity \eqref{Eq-3lin-propsym}, we get
\begin{eqnarray}
| {}_{\mathbb{H}^{-{\delta'}, 2}}\langle P_nB(u_n(r)), \phi\rangle_{\mathbb{H}^{{\delta'}, 2}}|
&=& |\langle u_n(r) \cdot
\nabla P_n\phi, u_n(r)\rangle_{\mathbb{L}^{2}}|\nonumber\\
&\leq& |\nabla P_n\phi|_{L^\infty}| u_n(r)|^2_{\mathbb{L}^{2}}.
\end{eqnarray}
Thanks to \cite[Remark 4 p 164, Theorem 3.5.4.ps.168-169 and Theorem 3.5.5 p 170]{Schmeisser-Tribel-87-book} for $ O= \mathbb{T}^d$, to
\cite[Theorem 7.63  and point 7.66]{Adams-Hedberg-94} for $ O$ being a bounded domain and to
the condition $ {\delta'}>1+\frac d{2}$,
we deduce for  $ 0<\epsilon < \delta'-1-\frac d{2}$,
 $$ |\nabla P_n\phi|_{L^\infty} \leq c |\nabla P_n\phi|_{H^{\epsilon+\frac d2, 2}} \leq c
|\phi|_{H_d^{1+\epsilon+\frac d2, 2}} \leq c |\phi|_{\mathbb{H}^{\delta', 2}}.$$
Therefore,
\begin{eqnarray}\label{}
 |P_nB(u_n(r))|_{\mathbb{H}^{-\delta', 2}}\leq c |u_n(r)|_{\mathbb{L}^{2}}^2
\end{eqnarray}
 and
\begin{eqnarray}\label{eq-unif-int-I(t)}
\int_0^T|I(t)|_{\mathbb{H}^{-\delta', 2}}^2dt&\leq& c\int_0^T\int_0^t \big(|(-A_\alpha
u_n(r)|^2_{\mathbb{H}^{-\delta', 2}} + |P_nB(u_n(r))|^2_{\mathbb{H}^{-\delta', 2}}\big)drdt\nonumber\\
&\leq& c\int_0^T\int_0^t\big(|
u_n(r)|^2_{\mathbb{L}^{ 2}} + |u_n(r)|^4_{\mathbb{L}^{2}}\big)drdt.
\end{eqnarray}
Moreover, using H\"older inequality and arguing as before, we get for $ t\geq s > 0$,
\begin{eqnarray}\label{eq-unif-int-I(t)-I(s)}
|I(t)- I(s)|^2_{\mathbb{H}^{-\delta', 2}}&=& |\int_s^t(-A_\alpha u_n(r) +
P_nB(u_n(r))dr|^2_{\mathbb{H}^{-{\delta'}, 2}}\nonumber\\
&\leq & C(t-s)\left(\int_s^t(| u_n(r)|^2_{\mathbb{L}^{2}} +
|u_n(r)|^4_{\mathbb{L}^{2}})dr \right).
\end{eqnarray}
From \eqref{eq-unif-int-I(t)},  \eqref{eq-unif-int-I(t)-I(s)} and \eqref{Eq-Ito-n-weak-estimation-1-bounded}, 
we have for  $ \gamma <\frac12$,
\begin{eqnarray}\label{eq-unif-int-I(t)x2}
\mathbb{E}\big(\int_0^T|I(t)|_{\mathbb{H}^{-\delta', 2}}^2dt&+&\int_0^T\int_0^T
\frac{|I(t)- I(s)|^2_{\mathbb{H}^{-\delta',
2}}}{|t-s|^{1+2\gamma }} dtds\big)^{\frac12} \nonumber\\
&\leq& C\mathbb{E}\left(\int_0^T(| u_n(r)|^2_{\mathbb{L}^{2}} +
|u_n(r)|^4_{\mathbb{L}^{2}})dr
\right)^\frac12 \leq C<\infty.
\end{eqnarray}
Now, we estimate the stochastic term $ J$. 
Using the stochastic isometry, the contraction property of $ P_n$ and Assumption $(\mathcal{C})$,( Condition  \eqref{Eq-Cond-Linear-Q-G} with $ q=2$ and $ \delta =0$), we get
\begin{eqnarray}
\int_0^T\mathbb{E}|\int_0^tP_nG(u_n(r))dW_n(r)|_{\mathbb{L}^{ 2}}^2dt&\leq&
C\int_0^T\mathbb{E}\int_0^t||G(u_n(r))||^2_{L_Q(\mathbb{L}^2)}drdt\nonumber\\
&\leq&
C\int_0^T\mathbb{E}\int_0^t(1+|u_n(r)|^2_{\mathbb{L}^2})drdt \leq c<\infty.
\end{eqnarray}
Moreover, for $ t\geq s> 0$ and $ \gamma <\frac12$, the same ingredients
above yield to
\begin{eqnarray}
\mathbb{E}\int_0^T\int_0^T\frac{|J(t)-
J(s)|^2_{\mathbb{L}^{2}}}{|t-s|^{1+2\gamma }} dtds &\leq&
C\mathbb{E}\int_0^T\int_0^T\frac{\int_s^t||G(u_n(r))||^2_{L_Q(\mathbb{L}^2)}dr}{|t-s|^{1+2\gamma
}} dtds \nonumber\\
&\leq& C\mathbb{E}\sup_{[0, T]}(1+|u_n(t)|^2_{\mathbb{L}^{2}})
\int_0^T\int_0^T|t-s|^{-2\gamma } dtds \leq c <\infty.
\end{eqnarray}
The proof of the lemma is now completed.
\end{proof}

To prove the existence of a martingale solution, we use \del{consider the Gelfand triplet \eqref{Gelfand-triple-Domain} and 
 use lemmas \ref{lem-unif-bound-theta-n-H-1-domain} and \ref{lem-bounded-W-gamma-p} 
and }the following compact embedding, see \cite[Theorem 2.1]{Flandoli-Gatarek-95},
\begin{equation}
 W^{\gamma, 2}(0, T; \mathbb{H}^{-\delta', 2}(O))\cap \mathbb{L}^2(0, T; \mathbb{H}^{\frac\alpha2, 2}(O)) 
 \hookrightarrow L^2(0, T; \mathbb{L}^2(O)).
\end{equation}
\del{$ L^2(0, T; \mathbb{H}^{\frac\alpha2, 2}(O)) \hookrightarrow L^2(0, T; \mathbb{L}^2(O))$,} Therefore, we deduce that the sequence of laws $ (\mathcal{L}(u_n))_n$  is tight on $ L^2(0, T; \mathbb{L}^2(O))$. 
Thanks to Prokhorov's theorem there exists a  subsequence, still denoted $ (u_n)_n$, for which  the sequence of laws $ (\mathcal{L}(u_n))_n$ converges  weakly on $ L^2(0, T; \mathbb{L}^2(O))$  to a probability measure $ \mu$. By Skorokhod's embedding theorem, we can construct a probability basis $ (\Omega_*, F_*, \mathbb{F}_*,  P_*)$  and a sequence of $ L^2(0, T; \mathbb{L}^2(O))\cap C([0, T]; \mathbb{H}^{-\delta', 2}(O))-$random variables
$ (u^*_n)_n$ and $ u^*$ such that  $\mathcal{L}(u^*_n) = \mathcal{L}(u_n), \forall n \in \mathbb{N}_0$,  $\mathcal{L}(u^*) = \mu$ and
$ u^*_n \rightarrow u^* a.s.$ in $ L^2(0, T; \mathbb{L}^2(O))\cap C([0, T]; \mathbb{H}^{-\delta', 2}(O))$. Moreover,  $ u^*_n(\cdot, \omega) \in C([0, T]; H_n)$. Thanks to  Lemma \ref{lem-unif-bound-theta-n-H-1-domain} and to the equality in law, we infer that the sequence  $ u^*_n$ converges weakly in $  L^2(\Omega\times [0, T]; \mathbb{H}^{\frac\alpha2, 2}(O))$ and weakly-star in $ L^p(\Omega, L^\infty([0, T]; \mathbb{L}^{2}(O))$ to a limit $ u^{**}$. It is easy to see that $u^{*} = u^{**},\;  dt\times dP-a.e.$ and
\begin{eqnarray}\label{eq-bound-u-*-n-u-*}
\mathbb{E}_*\sup_{[0, T]}| u^*(s)|^p_{\mathbb{L}^2}+ \mathbb{E}_*\int_0^T| u^*(s)|^2_{ \mathbb{H}^{\frac\alpha2, 2}}ds \leq c<\infty.
\end{eqnarray}
 
\del{\begin{equation}
u^{*}(\cdot, \omega)\in L^2(0, T; \mathbb{H}^{\frac\alpha2, 2}(O))\cap L^\infty(0, T; \mathbb{L}^2(O)).
\end{equation} }
We introduce the filtration 
\begin{equation}
(\mathit{G}_n^*)_t:= \sigma\{u^{*}_n(s), s\leq t\}
\end{equation}
and construct (with respect to $ (\mathit{G}_n^*)_t$) the time continuous square integrable martingale
$ (M_n(t), t\in [0, T])$ with trajectories in
$ C([0, T]; \mathbb{L}^2(O))$ by
\begin{equation}
M_n(t):= u_n^*(t) - P_nu_0+\int_0^t A_\alpha u_n^*(s) ds -\int_0^t P_nB(u_n^*(s))ds.
\end{equation}
The equality in law yields to the fact that the quadratic variation is given by 
\begin{equation}
\langle\langle M_n\rangle\rangle_t= \int_0^tP_nG(u^*_n(s))QG(u^*_n(s))^*ds,
\end{equation}
where $ G(u^*_n(s))^*$ is the adjoint of $G(u^*_n(s))$. We prove that, for $ a.s.$, $ M_n(t)$ converges weakly in $ \mathbb{H}^{-\delta', 2}(O)$ to  the martingale $ M(t)$, for all $ t\in [0, T]$, where $ M(t)$ is given by 
\begin{equation}\label{eq-M(t)}
M(t):= u^*(t) - u_0+\int_0^t A_\alpha u^*(s) ds -\int_0^t B(u^*(s))ds.
\end{equation}
\del{In fact, for all $ v\in V_2:= \mathbb{H}^{\delta', 2}$ with $ \delta'>1+\frac d2$, we have $ P-a.s.$,  $ \langle P_nu_0, v\rangle $  converges to $ \langle u_0, v\rangle $, thanks to the fact that $ P_n \rightarrow I$ in $ \mathbb{L}^2(O)$, $ \langle u_n^*(t), v\rangle $  converges to $ \langle u^*(t), v\rangle $ as a consequence of the 
the a.s. convergence in $ L^2(0, T; \mathbb{L}^2(O))$ (in fact, we speak about the convergence of a subsequence but as usual we keep the same notation) and  the weak convergence and the continuity in $\mathbb{H}^{-\delta', 2}(O))$, the term $\int_0^t A_\alpha u_n^*(s) ds$
converges thanks to the weak convergence $ L^2(0, t; \mathbb{L}^2(O))$, for all $ t\in [0, T]$ and the elementry inequality $ \langle A_\alpha u_n^*(t), v\rangle = \langle u_n^*(t), A_\alpha v\rangle $ with $ v \in \mathbb{H}^{\delta', 2}(O)$ and  $\delta'>1+\frac d2>\alpha $. The convergence of $\int_0^t \langle B(u_n^*(s)), v\rangle ds$ is completely described in \cite[Appendix 2]{Flandoli-Gatarek-95}, in particular the condition $ \delta'>1+\frac d2$ implies that $ \partial_j v \in C^0(O)$, which we need to do the calculus. To prove that $ M(t)$ is  a quadratic martingale, we see that for all $ \phi \in C_b(L^2(0, s; \mathbb{L}^2(O)))$ and $ v\in \mathcal{D}(O)$
\begin{equation}
\mathbb{E}(\langle M(t)- M(s), v\rangle \phi(u^*|_{[0, s]}))= \lim_{n\rightarrow +\infty} \mathbb{E}(\langle M_n(t)- M_n(s), v\rangle \phi(u^*|_{[0, s]}))=0
\end{equation}
and 
\begin{eqnarray}
&{}&\mathbb{E}(\langle M(t), v\rangle \langle M(t), y\rangle - \langle M(s), v\rangle \langle M(s), y\rangle -\int_s^t\langle G^*(u^*(r)P_nv, G^*(u^*(r)P_ny \rangle dr)\phi(u^*|_{[0, r]}))\nonumber\\
&=& \lim_{n\rightarrow +\infty} \mathbb{E}(\langle M_n(t), v\rangle \langle M_n(t), y\rangle - \langle M_n(s), v\rangle \langle M_n(s), y\rangle -\int_s^t\langle G^*(u_n^*(r)P_nv, G^*(u_n^*(r)P_ny \rangle dr)\phi(u_n^*|_{[0, r]}))\nonumber\\
&=& 0
\end{eqnarray}}
Some of the main ingredients are the $a.s.$ convergence of $ u_n^*$ in $ L^2(0, T; \mathbb{L}^2)$, $ \partial_j\phi \in C^0 $ and therefore we can estimate $\int_0^t \langle B(u_n^*(s)), v\rangle ds$ by $\int_0^t|B(u_n^*(s))|_{L^1} |v|_{C^1} ds$. Now we apply the representation theorem \cite[Theorem 8.2]{DaPrato-Zbc-92}, we infer that there exists a probability basis $ (\Omega^*, \mathcal{F}^*, P^*, \mathbb{F}^*, W^*)$ such that  
\begin{equation}
M(t)=\int_0^tG(u^*(s))W^*(ds).
\end{equation}
If moreover, $ \alpha \in [\alpha_0(d):= 1+\frac{d-1}{3}, 2]$, then thanks to 
Burkholdy-Davis-Gandy inequality, Assumption \eqref{Eq-Cond-Linear-Q-G}, with $q=2, \delta =0$ and  \eqref{eq-bound-u-*-n-u-*}
\begin{eqnarray}\label{cont-2-martg-stochas}
\mathbb{E}\sup_{[0, T]} |\int_0^tG(u^*(s))dW^*(s)|^2_{\mathbb{L}^{2}}
&\leq& c \mathbb{E}\int_0^T|G^*(u^*(s))|^2_{L_Q(\mathbb{L}^{2})}ds
\leq c (1+ \mathbb{E}\sup_{[0, T]}|u^*(s)|^2_{\mathbb{L}^{2}})<\infty.\nonumber\\
\end{eqnarray}
Further more, using Estimate \eqref{B-u-v-h-alpha-2-d} with $ \eta=0$, the Sobolev embedding 
$  \mathbb{H}^{\frac{\alpha}{2}, 2}(O) \hookrightarrow \mathbb{H}^{\frac{d+2-\alpha}{4}, 2}(O)$, ( $ 1+\frac{d-1}{3}\leq \alpha \leq 2$) 
and the boundedness of the  operator $ A_\alpha: \mathbb{H}^{\frac{\alpha}{2}, 2}(O) \rightarrow \mathbb{H}^{\frac{-\alpha}{2}, 2}(O)$, we get
\begin{eqnarray}\label{cont-1-martg}
\mathbb{E}\int_0^T\big( |A_\alpha u^*(s)|_{\mathbb{H}^{-\frac\alpha2, 2}}&+& |B(u^*(s))|_{\mathbb{H}^{-\frac\alpha2, 2}}\big)ds \leq c
\mathbb{E} \int_0^T\big( |u^*(s)|_{\mathbb{H}^{\frac\alpha2, 2}} + |u^*(s)|^2_{\mathbb{H}^{\frac{d+2-\alpha}{4}, 2}}\big)ds\nonumber\\
&\leq& c (1+\mathbb{E} \int_0^T|u^*(s)|^2_{\mathbb{H}^{\frac\alpha2, 2}}ds)<\infty. 
\end{eqnarray}
Therefore, using the densely embedding $ \mathbb{H}^{\delta', 2}(O) \hookrightarrow \mathbb{H}^{\frac\alpha2, 2}(O)$, we can take the equality \eqref{Eq-weak-Solution} in the $\mathbb{H}^{\frac\alpha2, 2}- \mathbb{H}^{-\frac\alpha2, 2}-$duality.

\vspace{0.25cm}

{\bf Pathewise uniqueness  (\ref{Main-theorem-martingale-solution-d}.3).}
To prove the uniqueness of the martingale solution under the condition \eqref{uniquness-con-martg}, we follow the scheme of the uniqueness in  
Section \ref{sec-Torus} taking into account the changes of the norms. Let
$ (u^1(t), t\in [0, T])$ and $ (u^2(t), t\in [0, T])$ be two martingale solutions on the same probability basis $ (\Omega^*, \mathbb{F}^*, P^*, W^*)$ and 
such that $ (u^1(t), t\in [0, T])$ satisfies \eqref{uniquness-con-martg}.\del{\subsection*{Proof of pathwise uniqueness for the martingale solution.}
\label{sec-subsection-uniqueness-martingale}.} We define  $ \tau_N^i$, $ \tau_N$  and  $ w:=  u^1- u^2$ as in Section \ref{sec-Torus}. 
Then $ w$ satisfies\del{ in $ \mathbb{H}^{-\frac\alpha2, 2}(O)$}
Equation \eqref{eq-uniq-w}  with $ W$ replaced by $ W^*$. We use 
 Ito formula to the product $ e^{-r(t)}|w(t)|^2_{\mathbb{L}^{2}}$, with $(r(t):=
c\int_0^t|u^1(s)|^{\frac{4\alpha}{3\alpha-d-2}}_{\mathbb{H}^{\frac{d+2-\alpha}{4}, 2}}ds, t\in [0, T])$, 
Identity  \eqref{formula-B-v1-B-v2},
\del{Property \eqref{Eq-3lin-propnull},}condition \eqref{Eq-Cond-Lipschitz-Q-G}, with $ q=2, \delta=0$, Estimate \eqref{B-u-v-h-alpha-2-d} with $ \eta=0$ 
and argue as around  \eqref{Torus-uniquenss-ito-formula-H-1} we get the proof. Combinning this latter result with Yamada-Watabnabe theorem \cite{Kurtz-Yamada-07, Ondrejat-thesis-03}, the global existence of a unique weak-strong solution follows.

\del{\section{{} {d}D-Fractional stochastic Navier-Stokes equation on Bounded domain and on the Torus with no smooth data}\label{sec-Domain}}

\del{\section{Maximal local weak-strong solutions for the multi-dimensional FSNSEs.}\label{sec-Domain}

In this section, we prove \del{the first part of Theorem \ref{Main-theorem-boubded-2}}$(3.8.1)$. But first let us illustrate the fact  that the 2D-FSNSE exhibits the same difficulty to prove the existence of the global solution as the 3D-NSE. This part will be also used in Section \ref{sec-global-weak-solution}. We follow a similar calculus as 
in Section \ref{sec-Torus}, replacing Property \eqref{vanishes-bilinear-tous-H1} by Property
\eqref{Eq-3lin-propnull}\del{is intrinsic for that  calculus and is valid neither for bounded domains of $ \mathbb{R}^d$ nor for the $ 3D-$torus, we use the property
\eqref{Eq-3lin-propnull}.} and considering\del{ for $ d\in \{2, 3\}$ and $ \alpha_0(d):= 1+\frac{d-1}3\leq \alpha\leq 2$,} 
the  densely continuously embedding Gelfand triple \eqref{Gelfand-triple-Domain}. We obtain the following Lemma
\begin{lem}\label{lem-unif-bound-theta-n-H-1-domain}
Let $ d\in \{2, 3\}$, $  \alpha_0(d):= 1+\frac{d-1}3\leq \alpha\leq 2$ and $ u_0\in L^{p}(\Omega, \mathbb{L}^{2}(O)), p\geq 4$ and let $ G$ satisfying Assumption $ (\mathcal{C})$ (\eqref{Eq-Cond-Linear-Q-G} with $ q=2$ and $ \delta =0$). Then
the solutions $ (u_n(t), t\in [0, T])$ of the equations \eqref{FSBE-Galerkin-approxi},  $ n\in \mathbb{N}_0$,
satisfy the following estimates

\begin{eqnarray}\label{Eq-Ito-n-weak-estimation-1-bounded}
\sup_{n}\mathbb{E}\Big(\sup_{[0, T]}|u_n(t)|^{p}_{\mathbb{L}^{ 2}}&+&
\int_0^T|u_n(t)|^{p-2}_{\mathbb{L}^{ 2}}\Big(|u_n(t)|^2_{\mathbb{H}^{\frac\alpha2, 2}} +
|u_n(t)|_{\mathbb{H}^{\beta, q_1}}^2 \Big)dt \nonumber\\
&+& \int_0^T|u_n(t)|^4_{\mathbb{L}^{ 2}}dt+ \int_0^T|u_n(t)|^{\frac{\alpha}{\eta}}_{\mathbb{H}^{\eta, 2}}dt\Big)<\infty,
\end{eqnarray}
where $ \beta\leq \frac\alpha2-\frac d2+\frac d{q_1}$, $ 2\leq q_1<\infty$ and $ \frac\alpha p<\eta\leq \frac\alpha2$.
\del{\item \begin{eqnarray}\label{Eq-B-n-weak-estimation-1-bounded}
\sup_{n}\left(\mathbb{E}\int_0^T |P_nB_n(u_n(t))|^{\frac{2\alpha}{4-\alpha}}_{\mathbb{H}^{-\frac\alpha2, 2}}+
\int_0^T |A^\frac\alpha2 u_n(t))|^{2}_{\mathbb{H}^{-\frac\alpha2, 2}} \right) dt <\infty.\nonumber\\
\end{eqnarray}}
\begin{eqnarray}\label{Eq-B-n-weak-estimation-1-bounded}
\sup_{n}\left(\mathbb{E}\int_0^T (|P_nB(u_n(t))|_{\mathbb{H}^{-\frac\alpha2, 2}}+| A^\frac\alpha2 u_n(t))|_{\mathbb{H}^{-\frac\alpha2, 2}})^{\frac{2\alpha}{d+2-\alpha}}dt \right) <\infty.
\end{eqnarray}
\end{lem}
\begin{proof}
The proof of \eqref{Eq-Ito-n-weak-estimation-1-bounded} follows exactly as for \eqref{Eq-Ito-n-weak-estimation-1-Torus} by replacing
the spaces $ \mathbb{H}^{1, 2}(\mathbb{T}^2)$ and $ \mathbb{H}^{1+\frac\alpha2, 2}(\mathbb{T}^2)$ respectively by
$ \mathbb{L}^{2}(O)$ and $ \mathbb{H}^{\frac\alpha2, 2}(O)$.
For the first term in the Estimate \eqref{Eq-B-n-weak-estimation-1-bounded}, we use the contraction property of $ P_n$, 
Estimate\del{ \eqref{B-u-v-h-alpha-2-d} with $ \eta =0$ (which coincides with} \eqref{Eq-B-H-alpha-2-est}\del{\eqref{Eq-B-H-alpha-2-est}}
and the Sobolev interpolation (recall that thanks to the condition $1+ \frac{d-1}3\leq  \alpha \leq 2$, we have 
the following embedding 
$ \mathbb{H}^{\frac\alpha2, 2}(O) \hookrightarrow \mathbb{H}^{\frac{d+2-\alpha}{4}, 2}(O)\hookrightarrow \mathbb{L}^2(O)$), we end up, for $1+ \frac{d-1}3<  \alpha \leq 2$,  with
\begin{eqnarray}\label{unif-estint-B-u}
\mathbb{E}\int_0^T |P_nB(u_n(t))|_{\mathbb{H}^{-\frac\alpha2}}^{\frac{2\alpha}{d+2-\alpha}}dt &\leq& c
\mathbb{E}\int_0^T |u_n(t)|_{\mathbb{H}^{\frac{d+2-\alpha}{4}, 2}}^{\frac{4\alpha}{d+2-\alpha}} dt \nonumber\\
&\leq& c
\mathbb{E}\int_0^T \big(|u_n(t)|^{\frac{d+2-\alpha}{2\alpha}}_{\mathbb{H}^{\frac\alpha2, 2}} 
|u_n(t)|^{\frac{3\alpha-d-2}{2\alpha}}_{\mathbb{L}^{2}}\big)^{\frac{4\alpha}{d+2-\alpha}} dt\nonumber\\
&\leq& c
\mathbb{E}\int_0^T |u_n(t)|^{2}_{\mathbb{H}^{\frac\alpha2, 2}} 
|u_n(t)|^{2\frac{3\alpha-d-2}{d+2-\alpha}}_{\mathbb{L}^{2}} dt.
\end{eqnarray}
The last term in the RHS of \eqref{unif-estint-B-u} is uniformly bounded thanks to \eqref{Eq-Ito-n-weak-estimation-1-bounded} and the condition
$ 2\frac{3\alpha-d-2}{d+2-\alpha}\leq p$. But this last is guaranteed thanks to $ 2\frac{3\alpha-d-2}{d+2-\alpha}\leq 4\leq p$. The case $1+ \frac{d-1}3=\alpha$ is easily obtained by application of Estimation \eqref{Eq-B-H-alpha-2-est}.
The second term in the RHS of \eqref{Eq-B-n-weak-estimation-1-bounded} is uniformly bounded thanks to
the fact that $ A: V:= D(A^\frac\alpha4) \rightarrow V^*$ is bounded, the condition $ \alpha \leq 1+\frac d2$ which yileds to 
$ \frac{2\alpha}{d+2-\alpha} \leq 2$ and thus we get
\begin{eqnarray}
\mathbb{E}\int_0^T |A^\frac\alpha2 u_n(t)|^{\frac{2\alpha}{d+2-\alpha}}_{\mathbb{H}^{-\frac\alpha2, 2}} dt
\leq c\mathbb{E}\int_0^T |u_n(t)|^{\frac{2\alpha}{d+2-\alpha}}_{\mathbb{H}^{\frac\alpha2, 2}} dt\leq c
\mathbb{E}\int_0^T (1+| u_n(t)|^{2}_{\mathbb{H}^{\frac\alpha2, 2}} )dt <\infty.\nonumber\\
\end{eqnarray}
Finaly we apply  Estimate \eqref{Eq-Ito-n-weak-estimation-1-bounded}.
\end{proof}

\noindent {\bf Existence of the solution.}
Assume that $1+ \frac{d-1}3<\alpha \leq 2$. Thanks to \eqref{Eq-Ito-n-weak-estimation-1-bounded}  and \eqref{Eq-B-n-weak-estimation-1-bounded}, we conclude the existence of
a subsequence, which is still denoted by $(u_n)_n$,\del{ and adapted processes $ u, F_2, G_2 $, such that}
\begin{equation}\label{eq-u-first-belonging}
 u\in L^2(\Omega\times [0, T]; \mathbb{H}^{\frac\alpha2, 2}(O))\cap
L^p(\Omega, L^\infty([0, T]; \mathbb{L}^{2}(O))),
\end{equation}
\begin{equation}
F_2 \in L^{\frac{2\alpha}{d+2-\alpha}}(\Omega\times [0, T];
\mathbb{H}^{-\frac\alpha2, 2} (O))\;\; \text{and}\;\; 
G_2\in L^2(\Omega\times [0, T]; L_Q(\mathbb{L}^{ 2} (O))), s.t.
\end{equation}
\begin{itemize}
\item (1') $u_n \rightarrow u$ weakly in $ L^2(\Omega\times [0, T]; \mathbb{H}^{\frac\alpha2, 2}(O)))$.
\item (2') $u_n \rightarrow u$ weakly-star in $ L^p(\Omega, \mathbb{L}^\infty([0, T]; \mathbb{L}^{ 2}(O)))$,
\item (3') $P_nF(u_n):= A^\frac\alpha2 u_n + P_nB(u_n)\rightarrow F_2$ weakly in $ L^{\frac{2\alpha}{d+2-\alpha}}(\Omega\times [0, T];
\mathbb{H}^{-\frac\alpha2, 2} (O))$.
\item (4')$u_n \rightarrow u$ weakly in $ L^{\frac{\alpha}{\eta}}(\Omega\times [0, T]; \mathbb{H}^{\eta, 2}(O))$, for all
$ \frac\alpha p<\eta \leq \frac\alpha2 $. \del{\footnote{Remark that $ \frac43 < \alpha \leq 2 \Leftrightarrow 1\leq \frac{2\alpha}{4-\alpha}\leq 2$}}
\item (5') $P_nG(u_n)\rightarrow G_2$ weakly in $ L^2(\Omega\times [0, T]; L_Q(\mathbb{L}^{ 2} (O)))$.

\end{itemize}
To prove the existence of a weak-strong solution of \eqref{Main-stoch-eq}, we can follow the same scheme as in Section \ref{sec-Torus} with the replacement of the spaces $ \mathbb{H}^{1, 2}(\mathbb{T}^2)$ and $ \mathbb{H}^{1+\frac\alpha2, 2}(\mathbb{T}^2)$ by
$ \mathbb{L}^{2}(O)$ and $ \mathbb{H}^{\frac\alpha2, 2}(O)$ respectively.  We construct a
process $ \tilde{\tilde {u}}$  as in \eqref{eq-def-u-tilde}, with $ F_1$ and $ G_1$ are replaced by $ F_2$ respectively  $ G_2$.
The proof of the statement $ u= \tilde{\tilde{u}},\; dt\times dP-a.e.$ can be done exactly as in Section \ref{sec-Torus}
with the brackets now stand for the $ V-V^*$-duality.
To check the  main key estimates, we use \eqref{formula-B-v1-B-v2}, \eqref{Eq-3lin-propnull}, H\"older inequality, 
\eqref{Eq-B-H-alpha-2-est}\del{{Eq-B-H-alpha-2-est}}, Sobolev interpolation\del{(recall that
$ 1+\frac{d-1}{3}< \alpha <2 \Rightarrow \mathbb{H}^{\frac\alpha2, 2}(O) \hookrightarrow \mathbb{H}^{\frac{d+2-\alpha}{4}, 2}(O) $)}
and Young inequality, we get
\begin{eqnarray}\label{ineq-B-u-v-v-local}
|{}_{V^*}\langle B(u)-B(v), u-v\rangle_{V} |&= & |{}_{V^*}\langle B(u-v, v), u-v\rangle_{V}|\leq
| B(u-v, v)|_{\mathbb{H}^{-\frac\alpha2, 2}}|u-v|_{\mathbb{H}^{\frac\alpha2, 2}}\nonumber\\
&\leq& c|v|_{\mathbb{H}^{\frac{d+2-\alpha}{4}, 2}} |u-v|_{\mathbb{H}^{\frac\alpha2, 2}}|u-v|_{\mathbb{H}^{\frac{d+2-\alpha}{4}, 2}}\nonumber\\
&\leq & c|v|_{\mathbb{H}^{\frac{d+2-\alpha}{4}, 2}} |u-v|^{\frac{d+2+\alpha}{2\alpha}}_{\mathbb{H}^{\frac\alpha2, 2}}
|u-v|_{\mathbb{L}^2}^{\frac{3\alpha-d-2}{2\alpha}}\nonumber\\
&\leq& c|v|^{\frac{4\alpha}{3\alpha-d-2}}_{\mathbb{H}^{\frac{d+2-\alpha}{4}, 2}}|u-v|^{2}_{\mathbb{L}^2} + \frac12|u-v|^{2}_{\mathbb{H}^{\frac\alpha2, 2}}.
\end{eqnarray}
Using the semigroup property of $ (A^\beta)_{\beta\geq0}$ and Assumption $ (\mathcal{C})$ ( \eqref{Eq-Cond-Lipschitz-Q-G}, with $ \delta =0$, $ q=2$ and $ C_R:=c$), we confirm\del{ 
\begin{eqnarray}\label{eq-monoto-local}
 -2{}_{V^*}\langle A_\alpha(u-v), u-v\rangle_{V} &+&  2{}_{V^*}\langle B(u)-B(v), u-v\rangle_{V} + 
|| G(u) - G(v)||_{L_Q(\mathbb{L}^2)}\nonumber\\
&\leq&
 -|u-v|^{2}_{\mathbb{H}^{\frac\alpha2, 2}} + c(1+|v|^{\frac{4\alpha}{3\alpha-d-2}}_{\mathbb{H}^{\frac{d+2-\alpha}{4}, 2}})
|u-v|^{2}_{\mathbb{L}^2}.
\end{eqnarray}
Consequently, we get}
\begin{itemize}
 \item $ (\mathcal{K}'_1)$- The local monotonicity property: There exists a constant $ c>0$ such that $\forall u, v \in \mathbb{H}^{\frac\alpha2}(O)$,
\begin{eqnarray}\label{eq-a-alpha-B-G-d}
 -2{}_{V^*}\langle A_\alpha(u-v), u-v\rangle_{V} &+&  2{}_{V^*}\langle B(u)-B(v), u-v\rangle_{V} +
|| G(u) - G(v)||_{L_Q(\mathbb{L}^2)}\nonumber\\
&\leq & r'(t)| u-v|^2_{\mathbb{L}^2}.
\end{eqnarray}
\end{itemize}
where $ r'(t):= c(1+ |v(t)|^{\frac{4\alpha}{3\alpha-d-2}}_{\mathbb{H}^{\frac{d+2-\alpha}{4}, 2}})$ and $c>0$ is a constant relevantly chosen.

\del{The main obstacle which prevent us in this stage to follow  the same steps as in Section \ref{sec-Torus} is the fact that we can not prove 
 for  $ v\in L^2(\Omega\times[0, T], \mathbb{H}^{\frac\alpha2, 2}(O))\cap  L^p(\Omega, L^\infty([0, T]; \mathbb{L}^{2}(O))$, that
$ v \in L^{\frac{4\alpha}{3\alpha-d-2}}(\Omega \times[0, T]; \mathbb{H}^{\frac{d+2-\alpha}{4}, 2}(O))$, unless $ \alpha \geq 1+\frac d2$. 
In fact, by interpolation,  we get
\begin{equation}\label{eq-impossible-local}
\mathbb{E}\int_0^T |v(t)|^{\frac{4\alpha}{3\alpha-d-2}}_{\mathbb{H}^{\frac{d+2-\alpha}{4}, 2}}dt \leq 
\mathbb{E}\int_0^T |v(t)|^{2\frac{d+2-\alpha}{3\alpha-d-2}}_{\mathbb{H}^{\frac\alpha2, 2}}|v(t)|^{2}_{\mathbb{L}^{2}} dt.
\end{equation}
Our claim here is that thanks to \eqref{Eq-Ito-n-weak-estimation-1-bounded}, a sufficient condition for the convergence of the integral in the RHS of 
\eqref{eq-impossible-local} is that $ 2\frac{d+2-\alpha}{3\alpha-d-2} \leq 2 \Leftrightarrow  \alpha \geq 1+\frac d2$. 
Remark that the classical values $ d=2, \alpha =2$ and $ d=3, \alpha=\frac52$ known in the literature for the
dD-NSEs are special cases of our condition above. 
The obstacle mentioned in \eqref{eq-impossible-local} is well known for the classical 3D-NSE  but not for the 2D-NSE. 
As we are using completely different calculus than the classical one, we need here to prove that our technique is optimal. Simultaneousily, we shall prove that our claim above is true.
In fact,  let $ \alpha =2=d$, then by interpolation, we get
\begin{equation}\label{eq-interp-alpha=2}
 |v|^{\frac{4\alpha}{3\alpha-d-2}}_{\mathbb{H}^{\frac{d+2-\alpha}{4}, 2}}\leq c  
( |v|^{\frac12}_{\mathbb{H}^{1, 2}}|v|^{\frac12}_{\mathbb{L}^{2}})^{4}
\leq  c|v|^{2}_{\mathbb{H}^{1, 2}}|v|^{2}_{\mathbb{L}^{2}}.
\end{equation}
Replacing \eqref{eq-interp-alpha=2} in \eqref{eq-impossible-local} and using the fact that $ u \in L^2(\Omega\times[0, T], \mathbb{H}^{1, 2}(O))\cap  L^p(\Omega, L^\infty([0, T]; \mathbb{L}^{2}(O))\cap L^4(\Omega\times[0, T], \mathbb{H}^{\frac{d+2-\alpha2}{4}, 2}(O))$ and the interpolation , thanks to \eqref{Eq-Ito-n-weak-estimation-1-bounded},  the fact that 

Follow the same machenery as in Section \ref{sec-Torus},  we prove the existence and uniquness of global solution for the 2D-NSE.}

\vspace{0.25cm}
The main obstacle which prevent us in this stage to follow  the same steps as in Section \ref{sec-Torus} is the fact that we are unable to prove that the solution \del{  $ u\in L^2(\Omega\times[0, T], \mathbb{H}^{\frac\alpha2, 2}(O))\cap  L^p(\Omega, L^\infty([0, T]; \mathbb{L}^{2}(O))$ is not enough to get} 
$ u \in L^{\frac{4\alpha}{3\alpha-d-2}}(\Omega \times[0, T]; \mathbb{H}^{\frac{d+2-\alpha}{4}, 2}(O))$, unless we suppose that $ \alpha \geq 1+\frac d2$.  In fact, under the condition $ 2\frac{d+2-\alpha}{3\alpha-d-2} \leq 2 \Leftrightarrow  \alpha \geq 1+\frac d2$ and using the interpolation and Estimate \eqref{Eq-Ito-n-weak-estimation-1-bounded},  we conclude that 
\begin{equation}\label{eq-impossible-local}
\sup_{n}\mathbb{E}\int_0^T |u_n(t)|^{\frac{4\alpha}{3\alpha-d-2}}_{\mathbb{H}^{\frac{d+2-\alpha}{4}, 2}}dt \leq 
c\sup_{n}\mathbb{E}\int_0^T |u_n(t)|^{2\frac{d+2-\alpha}{3\alpha-d-2}}_{\mathbb{H}^{\frac\alpha2, 2}}|u_n(t)|^{2}_{\mathbb{L}^{2}} dt<\infty.
\end{equation}
Remak that under the condition $\alpha \geq 1+\frac d2$, the regime is either dissipative or hyperdissipative. The proof of the existence and the uniqueness of the global solution for the dD-FSNSE under these two regimes is  classical.\del{for which we know that the calculus is easier to get  $ (3)'$. Than following} In particular, one can follow the same machinery as in Section \ref{sec-Torus} with the relevant changes mentioned above. The obstacle mentioned in \eqref{eq-impossible-local} is similar to the well known one for the classical 3D-NSE  but not for the 2D-NSE. To support more our claim mentioned in the begining of this section and in Section \ref{sec-intro}, we emphasize that the  2D-SNSE is\del{or 2D-NSE are} covered by our technique and this proves that this latter is optimal. Moreover, we can remark also that the values,  ($ \alpha\geq 1+\frac d2$),  ($ d=2, \alpha =2$) and ($ d=3, \alpha\geq\frac52$), known in the literature for the
dD-NSEs emerge in our setting in a natural way.\del{are special cases of our condition above.}

\del{one can see that the classical  2D-SNSE or 2D-NSE are 
covered by our method and this proves that our techniquethe optimality of our technique.}

\del{t is of great importance here to emphasis the optimality of our technique as  the classical  2D-SNSE or 2D-NSE are 
covered. by our method, this proves that our technique is optimal.}

\vspace{0.25cm}

\del{To skirt this difficulty,\del{ for the fractional case,} we use an approximation approach.\del{ as we have done for the local mild solution.} We consider Equation \eqref{Main-stoch-eq} with $ B$  replaced by $ B\pi_m$ and $ \pi_m$ is defined in Lemma \ref{lem-Lipschitz-pi-n} with $  X:=\mathbb{H}^{\frac{d+2-\alpha}{4}, 2}(O)$. i.e. we consider, for $m \in \mathbb{N}_0$, the equations\del{   
Otherwise, we consider Equation \eqref{Eq-approx-n} with $ G$ globally Lipschitz, therefore, we do not need to use $G(\pi_m) $,}
\begin{equation}\label{Eq-approx-m}
\Bigg\{
\begin{array}{lr}
 du^m= \big(-
A_{\alpha}u^m(t) + B(\pi_m u^m(t))\big)dt+ G(u^m(t))dW(t), \; 0< t\leq T,\\
u_m(0)= u_0.
\end{array}
\end{equation}
First, let us mention that, using  Formula \eqref{construction-of-fract-bounded} and the fact that $ div \pi_m u=0$, it is easy to see that  $ \pi_m$ is well defined. Moreover, we have for all  $(u, v)\in D(B(\cdot, \cdot))$,
\begin{eqnarray}\label{eq-relation-B-B-m}
B(\pi_m(u), \pi_m(v)) &=& \mathcal{X}_m(u, v)B(u, v), 
\end{eqnarray}
where 
\begin{eqnarray}\label{eq-x-m}
\mathcal{X}_m(u, v):&=&
\min\{1, m{|u|^{-1}_{\mathbb{H}^{\frac{d+2-\alpha}{4}, 2}}},
 m{|v|^{-1}_{\mathbb{H}^{\frac{d+2-\alpha}{4}, 2}}}, m^2{|u|^{-1}_{\mathbb{H}^{\frac{d+2-\alpha}{4}, 2}}|v|^{-1}_{\mathbb{H}^{\frac{d+2-\alpha}{4}, 2}}}\}
\end{eqnarray}
In the aim to prove the existence and the uniqueness of the global weak solution of Equation \eqref{Eq-approx-m}, for all $ m\in \mathbb{N}_0$,
we use the variational approach, see e.g. \cite{Krylov-Rozovski-monotonocity-2007, Metivier-book-SPDEs-88, Rockner-Pevot-06}. Or equivalently we can follow either the scheme above or the scheme in Section \ref{sec-Torus} with the relevant changes mentioned before.
Let  $ u, v, \theta \in \mathbb{H}^{\frac\alpha2, 2}(O)$, we have the following properties

{\bf $(a_1)$  The monotonicity.} Using \eqref{formula-B-v1-B-v2}, H\"older inequality, \eqref{B-u-v-h-alpha-2-d} with $ \eta =0$ (or \eqref{Eq-B-H-alpha-2-est}), Lemma \ref{lem-Lipschitz-pi-n}, Sobolev interpolation (recall that
$ 1+\frac{d-1}{3}< \alpha <2$)\del{(recall that
$ 1+\frac{d-1}{3}< \alpha <2 \Rightarrow \mathbb{H}^{\frac\alpha2, 2}(O) \hookrightarrow \mathbb{H}^{\frac{d+2-\alpha}{4}, 2}(O) $)}
and Young inequality, we infer that
\del{\begin{eqnarray}\label{ineq-B-u-v-v-local-n}
|{}_{V^*}\langle B(\pi_nu)&-&B(\pi_nv), u-v\rangle_{V} |\nonumber\\
&= & |{}_{V^*}\langle B(\pi_nu-\pi_nv, \pi_nv), u-v\rangle_{V}
+{}_{V^*}\langle B(\pi_nu, \pi_nu-\pi_nv), u-v\rangle_{V}|\nonumber\\
&\leq & |u-v|_{\mathbb{H}^{\frac\alpha2, 2}}\big(
| B(\pi_nu-\pi_nv, \pi_nv)|_{\mathbb{H}^{-\frac\alpha2, 2}}+| B(\pi_nu-\pi_nv, \pi_nv)|_{\mathbb{H}^{-\frac\alpha2, 2}}\big)\nonumber\\
&\leq & c|u-v|_{\mathbb{H}^{\frac\alpha2, 2}}
|\pi_nu-\pi_nv|_{\mathbb{H}^{\frac{d+2-\alpha}{4}, 2}}
\big(|\pi_nu|_{\mathbb{H}^{\frac{d+2-\alpha}{4}, 2}}+|\pi_nv|_{\mathbb{H}^{\frac{d+2-\alpha}{4}, 2}}\big)\nonumber\\
&\leq& c|v|_{\mathbb{H}^{\frac{d+2-\alpha}{4}, 2}} |u-v|_{\mathbb{H}^{\frac\alpha2, 2}}|u-v|_{\mathbb{H}^{\frac{d+2-\alpha}{4}, 2}}\nonumber\\
&\leq & c|v|_{\mathbb{H}^{\frac{d+2-\alpha}{4}, 2}} |u-v|^{\frac{d+2+\alpha}{2\alpha}}_{\mathbb{H}^{\frac\alpha2, 2}}
|u-v|_{\mathbb{L}^2}^{\frac{3\alpha-d-2}{2\alpha}}\nonumber\\
&\leq& c|v|^{\frac{4\alpha}{3\alpha-d-2}}_{\mathbb{H}^{\frac{d+2-\alpha}{4}, 2}}|u-v|^{2}_{\mathbb{L}^2} + \frac12|u-v|^{2}_{\mathbb{H}^{\frac\alpha2, 2}}.
\end{eqnarray}}
\begin{eqnarray}\label{ineq-B-u-v-v-local-m}
{}_{V^*}\langle B(\pi_mu)&-&B(\pi_mv), u-v\rangle_{V} \nonumber\\
&= & {}_{V^*}\langle B(\pi_mu, \pi_mu-\pi_mv), u-v\rangle_{V}+{}_{V^*}\langle B(\pi_mu-\pi_mv, \pi_mv), u-v\rangle_{V} \nonumber\\
&\leq & |u-v|_{\mathbb{H}^{\frac\alpha2, 2}}\big(
| B(\pi_mu, \pi_mu-\pi_mv)|_{\mathbb{H}^{-\frac\alpha2, 2}}+| B(\pi_mu-\pi_mv, \pi_mv)|_{\mathbb{H}^{-\frac\alpha2, 2}}\big)\nonumber\\
&\leq & c|u-v|_{\mathbb{H}^{\frac\alpha2, 2}}
|\pi_mu-\pi_mv|_{\mathbb{H}^{\frac{d+2-\alpha}{4}, 2}}
\big(|\pi_mu|_{\mathbb{H}^{\frac{d+2-\alpha}{4}, 2}}+|\pi_mv|_{\mathbb{H}^{\frac{d+2-\alpha}{4}, 2}}\big)\nonumber\\
&\leq& cm |u-v|_{\mathbb{H}^{\frac\alpha2, 2}}|u-v|_{\mathbb{H}^{\frac{d+2-\alpha}{4}, 2}}\nonumber\\
\del{&\leq & cn|u-v|^{\frac{d+2+\alpha}{2\alpha}}_{\mathbb{H}^{\frac\alpha2, 2}}
|u-v|_{\mathbb{L}^2}^{\frac{3\alpha-d-2}{2\alpha}}\nonumber\\}
&\leq& cm^{\frac{4\alpha}{3\alpha-d-2}}|u-v|^{2}_{\mathbb{L}^2} + \frac12|u-v|^{2}_{\mathbb{H}^{\frac\alpha2, 2}}.
\end{eqnarray}
Therefore, using the semigroup property of $ (A_\beta)_{\beta\geq 0}$, 
\eqref{ineq-B-u-v-v-local-m} and Assumption $ (\mathcal{C})$ (with $ q=2$, $ \delta=0$ and $c_R=c$), we get the monotonicity property\del{we conclude the monotonicity property of the term  $ -A_\alpha+B\pi_m$. In particular, thanks to Assumption $ (\mathcal{C})$,  Estimation \eqref{eq-monoto-local} with $ B\pi_m$ and $ m$ are in the place of $ B$ and $ |v|_{\mathbb{H}^{\frac{d+2-\alpha}{4}, 2}}$ respectively, holds. }
\begin{eqnarray}
-2{}_{V^*}\langle A_\alpha (u-v), u-v\rangle_{V} &+& 2{}_{V^*}\langle B(\pi_m u)- B(\pi_m v), u-v\rangle_{V} \nonumber\\ 
&+& ||G(u)- G(v)||^2_{HS(\mathbb{L}^2)}\leq  c_m|u-v|^2_{\mathbb{L}^{2}}.
\end{eqnarray}

{\bf $(a_2)$  The coercitivity.} \del{ Thanks to  H\"older inequality, \eqref{B-u-v-h-alpha-2-d} with $ \eta =0$ (or \eqref{Eq-B-H-alpha-2-est}), Lemma \ref{lem-Lipschitz-pi-n}, interpolation and Young inequality, we infer that 
\begin{eqnarray}\label{eq-coercive}
{}_{V^*}\langle B(\pi_m u), u\rangle_{V}&\leq &
|u|_{\mathbb{H}^{\frac\alpha2, 2}}|B(\pi_m u)|_{\mathbb{H}^{-\frac\alpha2, 2}} \leq  
|u|_{\mathbb{H}^{\frac\alpha2, 2}}|\pi_m u|^2_{\mathbb{H}^{\frac{d+2-\alpha}4, 2}}\nonumber\\
&\leq & m |u|_{\mathbb{H}^{\frac\alpha2, 2}}| u|_{\mathbb{H}^{\frac{d+2-\alpha}4, 2}}\leq 
m |u|^{\frac{\alpha+d+2}{2\alpha}}_{\mathbb{H}^{\frac\alpha2, 2}}|u|^ {\frac{3\alpha-d-2}{2\alpha}}_{\mathbb{L}^{2}}
\nonumber\\
&\leq & \frac12|u|^2_{\mathbb{H}^{\frac\alpha2, 2}} + m^{\frac{4\alpha}{3\alpha-d-2}} 2^{\frac{\alpha}{\alpha+d+2}}|u|^ {2}_{\mathbb{L}^{2}}.
\end{eqnarray}
Now i}It is easy, using the semigroup property of $ (A_\beta)_{\beta\geq 0}$,\del{\eqref{eq-coercive}}\eqref{eq-relation-B-B-m}, \eqref{eq-x-m},  \eqref{Eq-3lin-propnull} and Assumption $ (\mathcal{C})$(with $q=2$, $ \delta=0$ and $c_R=c$) that  
\begin{eqnarray}
-2{}_{V^*}\langle A_\alpha u, u\rangle_{V} + 2{}_{V^*}\langle B(\pi_m u), u\rangle_{V} + ||G(u)||^2_{L_Q(\mathbb{L}^2)}+ 2|u|^2_{\mathbb{H}^{\frac\alpha2, 2}}&\leq & c_m(1+|u|^2_{\mathbb{L}^{2}}).\nonumber\\
\end{eqnarray}

{\bf $(a_3)$  The growth.} Thanks to  \eqref{Eq-B-H-alpha-2-est}, Lemma \ref{lem-Lipschitz-pi-n} and Sobolev embedding, we infer that 
\begin{eqnarray}\label{eq-growth}
| A_\alpha u|_{\mathbb{H}^{-\frac\alpha2, 2}} + |B(\pi_m u)|_{\mathbb{H}^{-\frac\alpha2, 2}} &\leq & (m+1)|u|_{\mathbb{H}^{\frac\alpha2, 2}}. 
\end{eqnarray}

{\bf $(a_4)$  The hemicontinuity.} The hemicintinuity is obtained thanks to the continuity of the following real functions 
on $ \mathbb{R}$
\begin{equation}
s\mapsto |u+sv|_{\mathbb{H}^{\frac{d+2-\alpha}{4}, 2}}\;\;\text{and}\;\;\; 
s\mapsto \mathcal{X}_m(u+sv, u+sv).
\end{equation}
\del{and
\begin{equation}
s\mapsto \phi_{u, v, \theta}(s):=
\Big\{
\begin{array}{lr}
{}_{V^*}\langle B(u+sv), \theta\rangle_{V}, \;\; |u+sv|_{\mathbb{H}^{\frac{d+2-\alpha}{4}, 2}}\leq m\nonumber\\
\frac{m^2}{|u+sv|^2_{\mathbb{H}^{\frac{d+2-\alpha}{4}, 2}}}\;\;{}_{V^*}\langle B(u+sv), \theta\rangle_{V},\geq m \;\; |u+sv|_{\mathbb{H}^{\frac{d+2-\alpha}{4}, 2}}. \nonumber\\
\end{array}
\end{equation}}
\del{\begin{equation}
s\mapsto {}_{V^*}\langle B(\pi_m (u+sv)), \theta\rangle_{V}= 
\phi_{u, v}(s) ,
\end{equation}}
\del{Now it is easy to check, using \eqref{Eq-B-H-alpha-2-est}, that in addition to the monotonicity property, the coefficients of the approximated equation satisfy the hemicontinuity, the coercitivity and the growth properties.} 
\noindent Consequently, for all $ m \in \mathbb{N}_0$, there exists a unique global solution $ (u^m(t), t\in [0, T])$ of Equation \eqref{Eq-approx-m} in the sense of Definition \ref{def-variational solution},  i.e. $ (u^m(t), t\in [0, T])$ satisfies, Equation \eqref{Eq-weak-Solution}, with $ V_2= V=\mathbb{H}^{\frac\alpha2, 2}(O)$, $ V_1= \mathbb{H}^{-\delta', 2}(O)$, $ \delta'>1+\frac d2$, \eqref{eq-set-solu-weak} and \eqref{eq-mart-l-2solu} with the constant $ c$ in the RHS of the later estimate depends on $m$ (or equivalently $ (u^m(t), t\in [0, T])$ satisfies \eqref{eq-weak-l-2solu} up to $ T$. \del{
\begin{equation}
\mathbb{E}\sup_{[0, T]}|u^m(t)|^p_{\mathbb{L}^2}+ \mathbb{E}\int_0^T\del{|u^m(t)|^p_{\mathbb{L}^2}}|u^m(t)|^2_{\mathbb{H}^{\frac\alpha2, 2}}dt\leq c_m<\infty.
\end{equation}}The $\mathbb{H}^{-\delta', 2}$-continuity is obtained by a standard way using Section \ref{sec-nonlinear-prop}. In particular,  one of the main ingredients of the proof  is Formula \eqref{eq-B-estimatoion-q=2}\del{{eq-est-H-1-d-q}}. Thanks to Remark \ref{Rem-1}, we deduce that the trajectories of 
$ (u^m(t), t\in [0, T])$ are $\mathbb{L}^2-$weakly continuous, thus  the random time  
\begin{equation}\label{eq-stop-time-weak-solu-L2}
 \xi_m:=\inf\{t\in (0, T), s.t.\;\; |u^m(t)|_{\mathbb{H}^{\frac{d+2-\alpha}{4}, 2}} > m\}\wedge T,
\end{equation}
with the understanding that $ \inf(\emptyset)=+\infty$, is a stopping time, see complete proof in Appendix \ref{append-stop-time}. The sequence $(\xi_m)_m $ is an increasing sequence, hence the following random time exists and is then a predictable stopping time.
\begin{equation}
 \xi := \lim_{m\rightarrow +\infty}\xi_m.
\end{equation}
Thanks to the uniqueness of $ (u^m(t), t\in [0, T])$, we can define the process \del{( as for the mild solution)} $ (u(t), \;  t\in [0, \xi) )$ by 
$ (u(t):=u^m(t), t\in [0, \xi_m))$. Remark that as much as $ u(t)$ stays in the ball $ B_{\mathbb{H}^{\frac{d+2-\alpha}{4}, 2}}(0, m)$, it is a solution of our main equation \eqref{Main-stoch-eq}, hence  $ (u, \xi)$ is a maximal local weak-strong solution.}
\del{\noindent Now, we assume that $ u\in L^2(\Omega\times[0, T], \mathbb{H}^{\frac\alpha2, 2}(O)) $,  
$ v\in L^2(\Omega\times[0, T], \mathbb{H}^{\frac\alpha2, 2}(O)) \cap  
L^{\frac{4\alpha}{3\alpha-d-2}}(\Omega\times[0, T], \mathbb{H}^{\frac{d+2-\alpha}{4}, 2}(O))$ and define 
$ r'(t):= c(1+ |v(t)|^{\frac{4\alpha}{3\alpha-d-2}}_{\mathbb{H}^{\frac{d+2-\alpha}{4}, 2}})$. Let $ \tilde{\xi}$ be any predictable stopping time.
We follow a similar calculus as in 
Section \ref{sec-Torus}, taking in consideration the changes mentioned in the beginning of this section and integrating on $ (0, \tilde{\xi})$.
Then, we end up with the estimation
\begin{eqnarray}\label{eq-key-leq-0}
\int_0^T\psi(t)dt\mathbb{E}\big\{\int_0^{t\wedge \tilde{\xi}}\!\!\!\!\!\!&{}&\!\!\!\!\!\! e^{-r(s)}
\big(-r'(s\wedge \tilde{\xi})|u(s\wedge \tilde{\xi})- v(s\wedge \tilde{\xi})
|^2_{\mathbb{L}^2}+ 2|| G_2(s\wedge \tilde{\xi}) - G(v(s\wedge \tilde{\xi}))||^2_{L_Q(\mathbb{L}^2)}\nonumber\\
&+& {}_{V^*}\langle F_2(s\wedge \tilde{\xi}) -
F(v(s\wedge \tilde{\xi})), u(s\wedge \tilde{\xi})- v(s\wedge \tilde{\xi}))\rangle_{V}\big)ds\big\}\leq 0.
\end{eqnarray}
\noindent We define, for $ N\in \mathbb{N}_0$, the predictable stopping time 
\begin{equation}
 \xi_N:=\inf\{t\in [0, T], s.t. |u(t)|_{\mathbb{H}^{\frac{d+2-\alpha}{4}, 2}} \geq N\}\wedge T,
\end{equation}
with the understanding that $ \inf(\emptyset)=+\infty$. Let 
\begin{equation}
 \xi := \lim_{N\rightarrow +\infty}\xi_N.
\end{equation}
It is well known that $ \xi$ exists thanks to the monotonicity of $ \xi_N$.
We replace in \eqref{eq-key-leq-0},  $ \tilde{\xi}$ by  $ \xi_N$ and take $ v(t\wedge \xi_N)= u(t\wedge \xi_N)$, 
\del{in $ L^2([\Omega\times0, T], \mathbb{H}^{\frac\alpha2, 2}(O))$,}we conclude from \eqref{eq-key-leq-0}, that for all fixed $ N$,
\del{$|| G(s) - G(v(s))||^2_{L_Q}$}$ G_2(s\wedge \xi_N)= G(u(s\wedge \xi_N)), \; ds\times dP-a.e.$. 
To get the equality $ F_2(s\wedge \xi_N)= F(u(s\wedge \xi_N)))$, we consider Estimation \eqref{eq-key-leq-0} without the last term and introduce
$ \tilde{v} \in L^\infty(\Omega\times[0, T], \mathbb{H}^{\frac\alpha2}(\mathbb{T}^2))$, the
parameter $ \lambda \in [-1, +1]$ and argue as in Section \ref{sec-Torus}, we get \del{. Then, replacing $ v$  and $ r(s)$ by 
$ u-\lambda\tilde{v}$ respectively
$ r'_\lambda(s):= c(1+|u-\lambda\tilde{v}|^{\frac{2\alpha}{3\alpha-2}}_{\mathbb{H}^{1+\frac\alpha2}})$, we get
\begin{eqnarray}\label{eq-equality-F-Fn}
\mathbb{E}\int_0^Te^{-r_\lambda(s)}\big(-r'_\lambda(s)\lambda^2|\tilde{v}(s)|^2_{\mathbb{L}^2}+ 2\lambda\langle F(s) -
F(u(s)-\lambda \tilde{v}(s)), \tilde{v}(s))\rangle_{\mathbb{L}^2}\big)ds\leq 0.\nonumber\\
\end{eqnarray}
Dividing on $ \lambda<0$ and on $ \lambda>0$, we conclude that if the limit of
the LHS of \eqref{eq-equality-F-Fn} exists, when $ \lambda \rightarrow 0$, then it vanishes.
For the first term in the LHS of \eqref{eq-equality-F-Fn}, we use the fact that
$ \tilde{v} \in L^\infty(\Omega\times[0, T], \mathbb{H}^{1+\frac\alpha2}(\mathbb{T}^2))$ and then we calculate the integral. Then it is easy that
the limit of this term vanishes. For the second term we use the dominated convergence theorem, we get}

\begin{eqnarray}\label{eq-equality-F-Fn-1}
\mathbb{E}\int_0^{T\wedge \xi_N}e^{-r_0(s\wedge \xi_N)}{}_{V^*}\langle F_2(s\wedge \xi_N) -
F(u(s\wedge \xi_N)), \tilde{v}(s\wedge \xi_N))\rangle_{V}ds= 0.
\end{eqnarray}
The justification of the application of the dominated convergence theorem  follows as in Section \ref{sec-Torus} using \eqref{Eq-B-H-alpha-2-est}.}\del{for $ s\leq \xi_N$
\begin{eqnarray}
|{}_{V^*}\langle F_2(s) &-&
F(u(s)-\lambda \tilde{v}(s)), \tilde{v}(s))\rangle_{V}|\leq
 |\tilde{v}(s)|_{\mathbb{H}^{\frac\alpha2, 2}}\big( |F_2(s)|_{\mathbb{H}^{-\frac\alpha2, 2}}+
|u(s)|_{\mathbb{H}^{\frac\alpha2, 2}}+ |\tilde{v}(s)|_{\mathbb{H}^{\frac\alpha2, 2}}\nonumber\\
&+& |B(u(s)-\lambda \tilde{v}(s))|_{\mathbb{H}^{-\frac\alpha2, 2}} \nonumber\\
&\leq&
 |\tilde{v}(s)|_{\mathbb{H}^{\frac\alpha2, 2}}\big( |F_2(s)|_{\mathbb{H}^{-\frac\alpha2, 2}}+
|u(s)|_{\mathbb{H}^{\frac\alpha2, 2}}+ |\tilde{v}(s)|_{\mathbb{H}^{\frac\alpha2, 2}}+
|u(s)|_{\mathbb{H}^{1, 2}}|u(s)|_{\mathbb{L}^{ 2}} \nonumber\\
&+& |\tilde{v}(s)|_{\mathbb{H}^{1, 2}}|\tilde{v}(s)|_{\mathbb{L}^{ 2}}
+|u(s)|_{\mathbb{H}^{1, 2}}|\tilde{v}(s)|_{\mathbb{L}^{ 2}} + |\tilde{v}(s)|_{\mathbb{H}^{1, 2}}|u(s)|_{\mathbb{L}^{ 2}}.
\end{eqnarray}By this procedure, we have constructed a sequence of progressively measurable processes $ (u_N)_{N\in \mathbb{N}_1}$ resolving up to a stoping time $ \xi_n$ te dD-FSNSE} 
\del{\section{Global existence and uniqueness of solution for 2D-FSNSEs.}\label{sec-global-mild-weak-solution}
We assume that $d=2$, $2< q<\infty$, $ p\geq 4$ and $ u_0\in L^p(\Omega, \mathbb{H}^{1, q}(O))$.  
We prove the global existence and the uniqueness of  mild  and weak solutions for multiplicative stochastic 2D-FSNSE.}}

\section{Global existence and uniqueness of a weak solution of the multi-dimensional FSNSEs.}\label{sec-global-weak-solution}

In this section, we prove  Theorem \ref{Main-theorem-boubded-2}. \del{the global existence and the uniqueness of the weak solution for the dD-FSNSE \eqref{Main-stoch-eq}.} But first let us illustrate the fact  that the 2D-FSNSE exhibits the same difficulty to prove the existence of the global solution as the 3D-NSE.  We follow a similar calculus as 
in Section \ref{sec-Torus}, replacing Property \eqref{vanishes-bilinear-tous-H1} by Property
\eqref{Eq-3lin-propnull}\del{is intrinsic for that  calculus and is valid neither for bounded domains of $ \mathbb{R}^d$ nor for the $ 3D-$torus, we use the property
\eqref{Eq-3lin-propnull}.} and considering\del{ for $ d\in \{2, 3\}$ and $ \alpha_0(d):= 1+\frac{d-1}3\leq \alpha\leq 2$,} 
the  densely continuously embedding Gelfand triple \eqref{gelfant-triple-eta}, with $ \eta=0$.\del{{Gelfand-triple-Domain}} We obtain the following Lemma
\begin{lem}\label{lem-unif-bound-theta-n-H-1-domain}
Let $ d\in \{2, 3\}$, $  \alpha_0(d):= 1+\frac{d-1}3\leq \alpha\leq 2$ and $ u_0\in L^{p}(\Omega, \mathbb{L}^{2}(O)), p\geq 4$ and let $ G$ satisfying Assumption $ (\mathcal{C})$ (\eqref{Eq-Cond-Linear-Q-G} with $ q=2$ and $ \delta =0$). Then
the solutions $ (u_n(t), t\in [0, T])$ of the equations \eqref{FSBE-Galerkin-approxi},  $ n\in \mathbb{N}_0$,
satisfy the following estimates

\begin{eqnarray}\label{Eq-Ito-n-weak-estimation-1-bounded}
\sup_{n}\mathbb{E}\Big(\sup_{[0, T]}|u_n(t)|^{p}_{\mathbb{L}^{ 2}}&+&
\int_0^T|u_n(t)|^{p-2}_{\mathbb{L}^{ 2}}\Big(|u_n(t)|^2_{\mathbb{H}^{\frac\alpha2, 2}} +
|u_n(t)|_{\mathbb{H}^{\beta, q_1}}^2 \Big)dt \nonumber\\
&+& \int_0^T|u_n(t)|^4_{\mathbb{L}^{ 2}}dt+ \int_0^T|u_n(t)|^{\frac{\alpha}{\eta}}_{\mathbb{H}^{\eta, 2}}dt\Big)<\infty,
\end{eqnarray}
where $ \beta\leq \frac\alpha2-\frac d2+\frac d{q_1}$, $ 2\leq q_1<\infty$ and $ \frac\alpha p<\eta\leq \frac\alpha2$.
\del{\item \begin{eqnarray}\label{Eq-B-n-weak-estimation-1-bounded}
\sup_{n}\left(\mathbb{E}\int_0^T |P_nB_n(u_n(t))|^{\frac{2\alpha}{4-\alpha}}_{\mathbb{H}^{-\frac\alpha2, 2}}+
\int_0^T |A^\frac\alpha2 u_n(t))|^{2}_{\mathbb{H}^{-\frac\alpha2, 2}} \right) dt <\infty.\nonumber\\
\end{eqnarray}}
\begin{eqnarray}\label{Eq-B-n-weak-estimation-1-bounded}
\sup_{n}\left(\mathbb{E}\int_0^T (|P_nB(u_n(t))|_{\mathbb{H}^{-\frac\alpha2, 2}}+| A^\frac\alpha2 u_n(t))|_{\mathbb{H}^{-\frac\alpha2, 2}})^{\frac{2\alpha}{d+2-\alpha}}dt \right) <\infty.
\end{eqnarray}
\end{lem}
\begin{proof}
The proof of \eqref{Eq-Ito-n-weak-estimation-1-bounded} follows exactly as for \eqref{Eq-Ito-n-weak-estimation-1-Torus} by replacing
the spaces $ \mathbb{H}^{1, 2}(\mathbb{T}^2)$ and $ \mathbb{H}^{1+\frac\alpha2, 2}(\mathbb{T}^2)$ respectively by
$ \mathbb{L}^{2}(O)$ and $ \mathbb{H}^{\frac\alpha2, 2}(O)$.
For the first term in the Estimate \eqref{Eq-B-n-weak-estimation-1-bounded}, we use the contraction property of $ P_n$, 
Estimate\del{ \eqref{B-u-v-h-alpha-2-d} with $ \eta =0$ (which coincides with} \eqref{Eq-B-H-alpha-2-est}\del{\eqref{Eq-B-H-alpha-2-est}}
and the Sobolev interpolation (recall that thanks to the condition $1+ \frac{d-1}3\leq  \alpha \leq 2$, we have 
the following embedding 
$ \mathbb{H}^{\frac\alpha2, 2}(O) \hookrightarrow \mathbb{H}^{\frac{d+2-\alpha}{4}, 2}(O)\hookrightarrow \mathbb{L}^2(O)$), we end up, for $1+ \frac{d-1}3<  \alpha \leq 2$,  with
\begin{eqnarray}\label{unif-estint-B-u}
\mathbb{E}\int_0^T |P_nB(u_n(t))|_{\mathbb{H}^{-\frac\alpha2}}^{\frac{2\alpha}{d+2-\alpha}}dt &\leq& c
\mathbb{E}\int_0^T |u_n(t)|_{\mathbb{H}^{\frac{d+2-\alpha}{4}, 2}}^{\frac{4\alpha}{d+2-\alpha}} dt \nonumber\\
&\leq& c
\mathbb{E}\int_0^T \big(|u_n(t)|^{\frac{d+2-\alpha}{2\alpha}}_{\mathbb{H}^{\frac\alpha2, 2}} 
|u_n(t)|^{\frac{3\alpha-d-2}{2\alpha}}_{\mathbb{L}^{2}}\big)^{\frac{4\alpha}{d+2-\alpha}} dt\nonumber\\
&\leq& c
\mathbb{E}\int_0^T |u_n(t)|^{2}_{\mathbb{H}^{\frac\alpha2, 2}} 
|u_n(t)|^{2\frac{3\alpha-d-2}{d+2-\alpha}}_{\mathbb{L}^{2}} dt.
\end{eqnarray}
The last term in the RHS of \eqref{unif-estint-B-u} is uniformly bounded thanks to \eqref{Eq-Ito-n-weak-estimation-1-bounded} and the condition
$ 2\frac{3\alpha-d-2}{d+2-\alpha}\leq p$. But this last is guaranteed thanks to $ 2\frac{3\alpha-d-2}{d+2-\alpha}\leq 4\leq p$. The case $1+ \frac{d-1}3=\alpha$ is easily obtained by application of Estimation \eqref{Eq-B-H-alpha-2-est}.
The second term in the RHS of \eqref{Eq-B-n-weak-estimation-1-bounded} is uniformly bounded thanks to
the fact that $ A: V:= D(A^\frac\alpha4) \rightarrow V^*$ is bounded, the condition $ \alpha \leq 1+\frac d2$ which yileds to 
$ \frac{2\alpha}{d+2-\alpha} \leq 2$ and thus we get
\begin{eqnarray}
\mathbb{E}\int_0^T |A^\frac\alpha2 u_n(t)|^{\frac{2\alpha}{d+2-\alpha}}_{\mathbb{H}^{-\frac\alpha2, 2}} dt
\leq c\mathbb{E}\int_0^T |u_n(t)|^{\frac{2\alpha}{d+2-\alpha}}_{\mathbb{H}^{\frac\alpha2, 2}} dt\leq c
\mathbb{E}\int_0^T (1+| u_n(t)|^{2}_{\mathbb{H}^{\frac\alpha2, 2}} )dt <\infty.\nonumber\\
\end{eqnarray}
Finaly we apply  Estimate \eqref{Eq-Ito-n-weak-estimation-1-bounded}.
\end{proof}

\noindent {\bf Existence of the solution.}
Assume that $1+ \frac{d-1}3<\alpha \leq 2$. Thanks to \eqref{Eq-Ito-n-weak-estimation-1-bounded}  and \eqref{Eq-B-n-weak-estimation-1-bounded}, we conclude the existence of
a subsequence, which is still denoted by $(u_n)_n$,\del{ and adapted processes $ u, F_2, G_2 $, such that}
\begin{equation}\label{eq-u-first-belonging}
 u\in L^2(\Omega\times [0, T]; \mathbb{H}^{\frac\alpha2, 2}(O))\cap
L^p(\Omega, L^\infty([0, T]; \mathbb{L}^{2}(O))),
\end{equation}
\begin{equation}
F_2 \in L^{\frac{2\alpha}{d+2-\alpha}}(\Omega\times [0, T];
\mathbb{H}^{-\frac\alpha2, 2} (O))\;\; \text{and}\;\; 
G_2\in L^2(\Omega\times [0, T]; L_Q(\mathbb{L}^{ 2} (O))), s.t.
\end{equation}
\begin{itemize}
\item (1') $u_n \rightarrow u$ weakly in $ L^2(\Omega\times [0, T]; \mathbb{H}^{\frac\alpha2, 2}(O)))$.
\item (2') $u_n \rightarrow u$ weakly-star in $ L^p(\Omega, \mathbb{L}^\infty([0, T]; \mathbb{L}^{ 2}(O)))$,
\item (3') $P_nF(u_n):= A^\frac\alpha2 u_n + P_nB(u_n)\rightarrow F_2$ weakly in $ L^{\frac{2\alpha}{d+2-\alpha}}(\Omega\times [0, T];
\mathbb{H}^{-\frac\alpha2, 2} (O))$.
\item (4')$u_n \rightarrow u$ weakly in $ L^{\frac{\alpha}{\eta}}(\Omega\times [0, T]; \mathbb{H}^{\eta, 2}(O))$, for all
$ \frac\alpha p<\eta \leq \frac\alpha2 $. \del{\footnote{Remark that $ \frac43 < \alpha \leq 2 \Leftrightarrow 1\leq \frac{2\alpha}{4-\alpha}\leq 2$}}
\item (5') $P_nG(u_n)\rightarrow G_2$ weakly in $ L^2(\Omega\times [0, T]; L_Q(\mathbb{L}^{ 2} (O)))$.

\end{itemize}
To prove the existence of a weak-strong solution of \eqref{Main-stoch-eq}, we can follow the same scheme as in Section \ref{sec-Torus} with the replacement of the spaces $ \mathbb{H}^{1, 2}(\mathbb{T}^2)$ and $ \mathbb{H}^{1+\frac\alpha2, 2}(\mathbb{T}^2)$ by
$ \mathbb{L}^{2}(O)$ and $ \mathbb{H}^{\frac\alpha2, 2}(O)$ respectively.  We construct a
process $ \tilde{\tilde {u}}$  as in \eqref{eq-def-u-tilde}, with $ F_1$ and $ G_1$ are replaced by $ F_2$ respectively  $ G_2$.
The proof of the statement $ u= \tilde{\tilde{u}},\; dt\times dP-a.e.$ can be done exactly as in Section \ref{sec-Torus}
with the brackets now stand for the $ V-V^*$-duality.
To check the  main key estimates, we use \eqref{formula-B-v1-B-v2}, \eqref{Eq-3lin-propnull}, H\"older inequality, 
\eqref{Eq-B-H-alpha-2-est}\del{{Eq-B-H-alpha-2-est}}, Sobolev interpolation\del{(recall that
$ 1+\frac{d-1}{3}< \alpha <2 \Rightarrow \mathbb{H}^{\frac\alpha2, 2}(O) \hookrightarrow \mathbb{H}^{\frac{d+2-\alpha}{4}, 2}(O) $)}
and Young inequality, we get
\begin{eqnarray}\label{ineq-B-u-v-v-local}
|{}_{V^*}\langle B(u)-B(v), u-v\rangle_{V} |&= & |{}_{V^*}\langle B(u-v, v), u-v\rangle_{V}|\leq
| B(u-v, v)|_{\mathbb{H}^{-\frac\alpha2, 2}}|u-v|_{\mathbb{H}^{\frac\alpha2, 2}}\nonumber\\
&\leq& c|v|_{\mathbb{H}^{\frac{d+2-\alpha}{4}, 2}} |u-v|_{\mathbb{H}^{\frac\alpha2, 2}}|u-v|_{\mathbb{H}^{\frac{d+2-\alpha}{4}, 2}}\nonumber\\
&\leq & c|v|_{\mathbb{H}^{\frac{d+2-\alpha}{4}, 2}} |u-v|^{\frac{d+2+\alpha}{2\alpha}}_{\mathbb{H}^{\frac\alpha2, 2}}
|u-v|_{\mathbb{L}^2}^{\frac{3\alpha-d-2}{2\alpha}}\nonumber\\
&\leq& c|v|^{\frac{4\alpha}{3\alpha-d-2}}_{\mathbb{H}^{\frac{d+2-\alpha}{4}, 2}}|u-v|^{2}_{\mathbb{L}^2} + \frac12|u-v|^{2}_{\mathbb{H}^{\frac\alpha2, 2}}.
\end{eqnarray}
Using the semigroup property of $ (A^\beta)_{\beta\geq0}$ and Assumption $ (\mathcal{C})$ ( \eqref{Eq-Cond-Lipschitz-Q-G}, with $ \delta =0$, $ q=2$ and $ C_R:=c$), we confirm\del{ 
\begin{eqnarray}\label{eq-monoto-local}
 -2{}_{V^*}\langle A_\alpha(u-v), u-v\rangle_{V} &+&  2{}_{V^*}\langle B(u)-B(v), u-v\rangle_{V} + 
|| G(u) - G(v)||_{L_Q(\mathbb{L}^2)}\nonumber\\
&\leq&
 -|u-v|^{2}_{\mathbb{H}^{\frac\alpha2, 2}} + c(1+|v|^{\frac{4\alpha}{3\alpha-d-2}}_{\mathbb{H}^{\frac{d+2-\alpha}{4}, 2}})
|u-v|^{2}_{\mathbb{L}^2}.
\end{eqnarray}
Consequently, we get}
\begin{itemize}
 \item $ (\mathcal{K}'_1)$- The local monotonicity property: There exists a constant $ c>0$ such that $\forall u, v \in \mathbb{H}^{\frac\alpha2}(O)$,
\begin{eqnarray}\label{eq-a-alpha-B-G-d}
 -2{}_{V^*}\langle A_\alpha(u-v), u-v\rangle_{V} &+&  2{}_{V^*}\langle B(u)-B(v), u-v\rangle_{V} +
|| G(u) - G(v)||_{L_Q(\mathbb{L}^2)}\nonumber\\
&\leq & r'(t)| u-v|^2_{\mathbb{L}^2}.
\end{eqnarray}
\end{itemize}
where $ r'(t):= c(1+ |v(t)|^{\frac{4\alpha}{3\alpha-d-2}}_{\mathbb{H}^{\frac{d+2-\alpha}{4}, 2}})$ and $c>0$ is a constant relevantly chosen.

\vspace{0.25cm}
The main obstacle which prevent us in this stage to follow  the same steps as in Section \ref{sec-Torus} is the fact that we are unable to prove that the solution \del{  $ u\in L^2(\Omega\times[0, T], \mathbb{H}^{\frac\alpha2, 2}(O))\cap  L^p(\Omega, L^\infty([0, T]; \mathbb{L}^{2}(O))$ is not enough to get} 
$ u \in L^{\frac{4\alpha}{3\alpha-d-2}}(\Omega \times[0, T]; \mathbb{H}^{\frac{d+2-\alpha}{4}, 2}(O))$, unless we suppose that $ \alpha \geq 1+\frac d2$.  In fact, under the condition $ 2\frac{d+2-\alpha}{3\alpha-d-2} \leq 2 \Leftrightarrow  \alpha \geq 1+\frac d2$ and using the interpolation and Estimate \eqref{Eq-Ito-n-weak-estimation-1-bounded},  we conclude that 
\begin{equation}\label{eq-impossible-local}
\sup_{n}\mathbb{E}\int_0^T |u_n(t)|^{\frac{4\alpha}{3\alpha-d-2}}_{\mathbb{H}^{\frac{d+2-\alpha}{4}, 2}}dt \leq 
c\sup_{n}\mathbb{E}\int_0^T |u_n(t)|^{2\frac{d+2-\alpha}{3\alpha-d-2}}_{\mathbb{H}^{\frac\alpha2, 2}}|u_n(t)|^{2}_{\mathbb{L}^{2}} dt<\infty.
\end{equation}
Remak that under the condition $\alpha \geq 1+\frac d2$, the regime is either dissipative or hyperdissipative. The proof of the existence and the uniqueness of the global solution for the dD-FSNSE under these two regimes is  classical.\del{for which we know that the calculus is easier to get  $ (3)'$. Than following} In particular, one can follow the same machinery as in Section \ref{sec-Torus} with the relevant changes mentioned above. The obstacle mentioned in \eqref{eq-impossible-local} is similar to the well known one for the classical 3D-NSE  but not for the 2D-NSE. To support more our claim mentioned in the begining of this section and in Section \ref{sec-intro}, we emphasize that the  2D-SNSE is\del{or 2D-NSE are} covered by our technique and this proves that this latter is optimal. Moreover, we can remark also that the values,  ($ \alpha\geq 1+\frac d2$),  ($ d=2, \alpha =2$) and ($ d=3, \alpha\geq\frac52$), known in the literature for the
dD-NSEs emerge in our setting in a natural way. In addition, we can see here the importance of the Sobolev space $ \mathbb{H}^{\frac{d+2-\alpha}{4}, 2}(O)$. To the best knowledge of the author, this space emerges here for the first time.\del{To the best knowledge of the author, the role of this space appears here. }

\vspace{0.25cm}

Now, we return to our main goal in this section and let us start by the case 
 $ O=\mathbb{T}^2$.  Thanks to the  conditions in (\ref{Main-theorem-boubded-2}.1) and arguing as in the proof of the regularity in Section \ref{sec-Torus}, we infer that the maximal solution  $ (u, \xi)$ satisfies 
\begin{equation}\label{eq-2D-weak-sol-up-xi}
 \mathbb{E}\sup_{[0, \xi)}|u(t)|^q_{\mathbb{H}^{1, q}}+ \mathbb{E}\int_0^{T\wedge \xi}|u(t)|^2_{\mathbb{H}^{1+\frac\alpha2, 2}}dt\leq c<\infty.
\end{equation}
We denote by $\mathscr{E}$ the set of predictable stochastic processes $ (v(t), t\in [0, T])$ (or the extsension of $ v$ in the case $v$ is defined up to a stopping time)\del{(we can also take  an extension of $v$ if this latter is only defined up to a stopping time)} satisfying that there exists a stopping time $ \tau$ such that 
$ v \in  L^2(\Omega\times[0, \tau); \mathbb{H}^{\frac\alpha2, 2}(\mathbb{T}^2))$ and the process $ (\nabla v(t), t\in [0, \tau)) $ can be extended (we keep the some notation) to
$ \nabla v \in L^{(1-\frac{2d}{\alpha q})^{-1}}(\Omega\times[0, T]; L^{q}(\mathbb{T}^2))$, with the norm of $ \nabla v$ in this space is uniformly bounded, i.e. independently of the extension. We claim that  
 $\mathscr{E} \neq  \varnothing$. In fact, let us define,
 $ (\tilde{v}(t), t\in [0, T])$, by 
  $\tilde{v}(t):= u(t\wedge \xi), \forall t\in [0, T])$, where $ (u, \xi)$ is our maximal local solution. We have, for $ q$ characterized as in $(\ref{Main-theorem-mild-solution-d}.2)$, (bellow d=2)
\begin{eqnarray}
\del{\mathbb{E}\int_0^T|\nabla v_N(t)|^{\frac{1}{1-\frac{2d}{\alpha q}}}_{L^q_{2\times 2}} dt &\leq &}
 \mathbb{E}\int_0^{T}|\nabla u(t\wedge \xi_N)|^{\frac{1}{1-\frac{2d}{\alpha q}}}_{q} dt
 \leq c\mathbb{E}\int_0^T|\theta(t\wedge \xi_N)|^{\frac{1}{1-\frac{2d}{\alpha q}}}_{L^q})dt \leq c\mathbb{E}\int_0^T(1+|\theta(t)|^q_{L^q})dt<\infty.\nonumber\\
\end{eqnarray}
Therfore\del{it is easy to check thanks to Estimate  \eqref{eq-bale-kato-majda-con} that} $ \tilde{v}\in \mathscr{E}$. Remark that the condition $ \frac{1}{1-\frac{2d}{\alpha q}}\leq q\Leftrightarrow 1+\frac{2d}{\alpha}\leq q$, see Remark \ref{Rem-2}. Now, we shall look for a solution in the set $\mathscr{E}$. We can go back to the calculus above in this section \del{of Section \ref{sec-Domain}} and we repeat the same calculus until Estimate \eqref{ineq-B-u-v-v-local}, which we treat now  as follow. Using H\"older twice ($ 1/q+1/q'=1/2$), 
Gaglairdo-Nirenberg and than Young inequalities, we get (recall $V:=\mathbb{H}^{\frac\alpha2, 2} (O)$)
\begin{eqnarray}\label{ineq-B-u-v-v-global}
|{}_{V^*}\langle B(u)&-&B(v), u-v\rangle_{V} |\leq
| (u-v)\nabla v|_{L^2_2}|u-v|_{\mathbb{L}^{2}}
\leq
|u-v|_{L^{q'}_2}|\nabla v|_{q}|u-v|_{\mathbb{L}^{2}}\nonumber\\
&\leq& c|\nabla v|_{q}|u-v|^{2-\frac{2d}{\alpha q}}_{\mathbb{L}^{2}}|u-v|^{\frac{2d}{\alpha q}}_{\mathbb{H}^{\frac\alpha2, 2}}
\leq  c|\nabla v|^{\frac{1}{1-\frac{2d}{\alpha q}}}_{q}|u-v|^{2}_{\mathbb{L}^{2}}+ c|u-v|^{2}_{\mathbb{H}^{\frac\alpha2, 2}}.
\end{eqnarray}
We take $ r'(t):= c(1+ |\nabla v(t)|^{\frac{1}{1-\frac{2d}{\alpha q}}}_{q})$ with relevant constant $ c>0$. Than, we apply the whole machinery as in Section \ref{sec-Torus}\del{ and  the estimations as in Section \ref{sec-Domain}} to get the existence of the global solution.\del{, we prove then the existence of global weak solution.}\del{to prove the existence of the weak (strong in probability) solution of Equation \eqref{Main-stoch-eq} in the set $\mathscr{E}$. We justify the application of the dominated convergence theorem as follow,
The only remained calculus is to give the justification to the 
application of the dominated convergence theorem as in \eqref{just-domin-Torus}} To prove the uniqueness of the solution in the set 
$\mathscr{E}$, we follow the steps as in Section \ref{sec-Torus}. In particular, in Formula \eqref{Torus-uniquenss-ito-formula-H-1}, 
we estimate the term $ \langle B(w(s)), u^1(s)\rangle = - \langle B(w(s), u^1(s)), w(s)\rangle$ using \eqref{ineq-B-u-v-v-global}. 
The existence and the uniqueness hold, therefore the local solution \del{constructed in Section \ref{sec-Domain}} is global and unique. The estimate \eqref{propty-of-2D-global-mild-sol} is obtained from \eqref{eq-2D-weak-sol-up-xi}.\del{as in the proof of the regularity in Section \ref{sec-Torus}.}

\vspace{0.25cm}

For the general case $ (\ref{Main-theorem-boubded-2}.2) $,\del{we know from Section \ref{sec-Domain}, that there exists at least one maximal local solution. I} if a maximal local weak solution enjoys \eqref{eq-bale-kato-majda-con}, then we have $\mathscr{E} \neq  \varnothing $ and thus we follow the proof above (for $ O=\mathbb{T}^2$) to get the results. If a maximal local weak solution enjoys Condition \eqref{eq-other-bale-kato-majda-con}, then the set $ \mathscr{E}_1 \neq  \varnothing$, where $ \mathscr{E}_1$ is the set of predictable  stochastic processes $ (v(t), t\in [0, T])$ (or the extsension of $ v$ in the case $v$ is defined up to a stopping time) satisfying that there exists a predictable stopping time $ \tau$ such that $ v \in  L^2(\Omega\times[0, \tau); \mathbb{H}^{\frac\alpha2, 2}(\mathbb{T}^2))$ and can be extended (we keep the some notation) to
$ v \in L^{\frac{4\alpha}{3\alpha-d-2}}(\Omega\times[0, T]; \mathbb{H}^{\frac{d+2-\alpha}{4}, 2}(O))$ uniformly, i.e. with the norm of $ v$ in this space is uniformly bounded independently of the extension. Now, we can continue from Estimate \eqref{ineq-B-u-v-v-local}\del{{eq-monoto-local}} and follow the proof as above and as in Section \ref{sec-Torus}.

\del{   Now, for $ O=\mathbb{T}^2$, the condition \eqref{eq-bale-kato-majda-con} follows from Lemma \ref{lem-basic-curl-gradient} and \cite[Theorem 2.6.]{Debbi-scalar-active}. In fact, 
\noindent The verification of the application of \cite[Theorem 2.6.]{Debbi-scalar-active} and of the obtention of the regularity  \eqref{propty-of-2D-global-mild-sol} are done exactly as in Section \ref{sec-global-mild-solution}.}

\del{We start doing the calculus in a general setting. 
In Section \ref{sec-Domain}, we have constructed  a local maximal weak solution $ (u, \xi)$ satisfying \eqref{eq-weak-l-2solu} up to $ \xi_N$ for all $ N\in \mathbb{N}_0$. We assume that $ (u, \xi)$ enjoys also Estimate  \eqref{eq-bale-kato-majda-con}.  We go back to the firts part of Section \ref{sec-Domain} and we repeat the same calculus until Estimate \eqref{ineq-B-u-v-v-local}, which we treat now  as follow. Using H\"older twice ($ 1/q+1/q'=1/2$), 
Gaglairdo-Nirenberg and than Young inequalities, we get (recall $V:=\mathbb{H}^{\frac\alpha2, 2} (O)$)
\begin{eqnarray}\label{ineq-B-u-v-v-global}
|{}_{V^*}\langle B(u)&-&B(v), u-v\rangle_{V} |\leq
| (u-v)\nabla v|_{L^2_2}|u-v|_{\mathbb{L}^{2}}
\leq
|u-v|_{L^{q'}_2}|\nabla v|_{q}|u-v|_{\mathbb{L}^{2}}\nonumber\\
&\leq& c|\nabla v|_{q}|u-v|^{2-\frac{2d}{\alpha q}}_{\mathbb{L}^{2}}|u-v|^{\frac{2d}{\alpha q}}_{\mathbb{H}^{\frac\alpha2, 2}}
\leq  c|\nabla v|^{\frac{1}{1-\frac{2d}{\alpha q}}}_{q}|u-v|^{2}_{\mathbb{L}^{2}}+ c|u-v|^{2}_{\mathbb{H}^{\frac\alpha2, 2}}.
\end{eqnarray}
We take $ r'(t):= c(1+ |\nabla v(t)|^{\frac{1}{1-\frac{2d}{\alpha q}}}_{q})$ with relevant constant $ c>0$ and 
$ v\in \mathscr{E}$, where $\mathscr{E}$ is the set of progressively measurable stochastic processes satisfying 
$ v \in  L^2(\Omega\times[0, T]; \mathbb{H}^{\frac\alpha2, 2}(\mathbb{T}^2))$ and 
$ \nabla v \in L^{(1-\frac{2}{\alpha q})^{-1}}(\Omega\times[0, T]; L^{q}(\mathbb{T}^2))$. We claim that  
 $\mathscr{E} \neq  \varnothing$. In fact, let us define for $ N$ fixed,
 $ (v_N(t), t\in [0, T])$, by 
  $v_N(t):= u(t\wedge \xi_N), \forall t\in [0, T])$, where 
$ (u, \xi)$ is our local solution\del{ (constructed in Section \ref{sec-Domain}).} Therfore, it is easy to check thanks to Estimate  \eqref{eq-bale-kato-majda-con} that $ v_N\in \mathscr{E}$. 
Than, we apply the whole machinery as in Section \ref{sec-Torus} and  estimation as in Section \ref{sec-Domain} 
to prove the existence of the weak (strong in probability) solution of 
Equation \eqref{Main-stoch-eq} in the set $\mathscr{E}$.\del{We justify the 
application of the dominated convergence theorem as follow,
The only remained calculus is to give the justification to the 
application of the dominated convergence theorem as in \eqref{just-domin-Torus},} To prove the uniqueness of the solution in the set 
$\mathscr{E}$, we follow the steps as in Section \ref{sec-Torus}. In particular, in Formula \eqref{Torus-uniquenss-ito-formula-H-1}, 
we estimate the term $ \langle B(w(s)), u^1(s)\rangle = - \langle B(w(s), u^1(s)), w(s)\rangle$ using \eqref{ineq-B-u-v-v-global}. 
The existence and the uniqueness hold, therefore the local solution constructed in Section \ref{sec-Domain} is global and unique.

Now, for $ O=\mathbb{T}^2$, the condition \eqref{eq-bale-kato-majda-con} follows from Lemma \ref{lem-basic-curl-gradient} and \cite[Theorem 2.6.]{Debbi-scalar-active}. In fact, for $ q\geq 4$,
\begin{eqnarray}
\del{\mathbb{E}\int_0^T|\nabla v_N(t)|^{\frac{1}{1-\frac{2}{\alpha q}}}_{L^q_{2\times 2}} dt &\leq &}
 \mathbb{E}\int_0^{T}|\nabla u(t\wedge \xi_N)|^{\frac{1}{1-\frac{2}{\alpha q}}}_{q} dt
 \leq c\mathbb{E}\int_0^T|\theta(t\wedge \xi_N)|^{\frac{1}{1-\frac{2}{\alpha q}}}_{L^q})dt \leq c\mathbb{E}\int_0^T(1+|\theta(t)|^q_{L^q})dt<\infty.\nonumber\\
\end{eqnarray}
\noindent The verification of the application of \cite[Theorem 2.6.]{Debbi-scalar-active} and of the obtaintion of the regularity  \eqref{propty-of-2D-global-mild-sol} are done exactly as in Section \ref{sec-global-mild-solution}.}

\del{We denote by\del{ $ \mathfrak{V}$,  } $ \mathscr{E}$ the set of progressively measurable stochastic processes $ v$, such that 
$(v, \tau)$ is a maximal weak solution for  \eqref{Main-stoch-eq} with $\tau$ being a stopping time. 
Thanks to Appendix \ref{sec-Passage Velocity-Vorticity}, $ \theta := curl u$ is a local weak solution for the scalar active equation 
\eqref{Eq-vorticity-Torus-2-diff}

In this section, we prove the global existence and the uniqueness of the weak solution for the 2D-FNSE \eqref{Main-stoch-eq}.  
Thanks to Section \ref{sec-Domain}, we have   constructed a local weak solution $ (u, \xi)$ satisfying \eqref{est-local-u-tau-principle} and up to $ \xi_N$ for $ N\in \mathbb{N}_0$.
Under hypothesis in $ 8.3.2.$\del{ of Theorem \ref{Main-theorem-boubded-2} for this case ($d=2$) and} and using Appendix \ref{sec-Passage Velocity-Vorticity}
and \cite{Debbi-scalar-active}, we 
infer that (here we take, for simplicity, $ q= q_0\geq 6$)  
\begin{equation}\label{est-nabla-u-theta-Global-existence}
 \exists c>0, s.t. \forall N\in \mathbb{N}_0,  \mathbb{E}\sup_{[0, \xi_N)}|\nabla u(t)|^q_{L_{2\times 2}^q}\leq c \mathbb{E} \sup_{[0, \xi_N]}
|\theta(t)|^q_{L^q}\leq c <\infty.
\end{equation}
Recall $ \theta:= curl u$. Now, we change the calculus in \eqref{ineq-B-u-v-v-local} as follow. Using H\"older twice ($ 1/q+1/q'=1/2$), 
Gaglairdo-Nirenberg and than Young inequalities, we get
\begin{eqnarray}\label{ineq-B-u-v-v-global}
|{}_{V^*}\langle B(u)&-&B(v), u-v\rangle_{V} |\leq
| (u-v)\nabla v|_{L^2_2}|u-v|_{\mathbb{L}^{2}}
\leq
|u-v|_{L^{q'}_2}|\nabla v|_{L^q_2}|u-v|_{\mathbb{L}^{2}}\nonumber\\
&\leq& c|\nabla v|_{L^q_{2\times 2}}|u-v|^{2-\frac4{\alpha q}}_{\mathbb{L}^{2}}|u-v|^{\frac4{\alpha q}}_{\mathbb{H}^{\frac\alpha2, 2}}
\leq  c|\nabla v|^{\frac{1}{1-\frac{2}{\alpha q}}}_{L^q_{2\times 2}}|u-v|^{2}_{\mathbb{L}^{2}}+ c|u-v|^{2}_{\mathbb{H}^{\frac\alpha2, 2}}.
\end{eqnarray}
We take $ r'(t):= c(1+ |\nabla v(t)|^{\frac{1}{1-\frac{2}{\alpha q}}}_{L^q_{2\times 2}})$ with relevant constant $ c>0$ and 
$ v\in \mathscr{E}$, where $\mathscr{E}$ is the set of progressively measurable stochastic processes satisfying 
$ v \in  L^2(\Omega\times[0, T]; \mathbb{H}^{\frac\alpha2, 2}(\mathbb{T}^2))$ and 
$ \nabla v \in L^{(1-\frac{2}{\alpha q})^{-1}}(\Omega\times[0, T]; L_{2\times 2}^{q}(\mathbb{T}^2))$. We claim that  
 $\mathscr{E} \neq  \varnothing$. In fact, let us define for $ N$ fixed,
 $ (v_N(t), t\in [0, T])$, by 
  $v_N(t):= u(t\wedge \xi_N), \forall t\in [0, T])$, where 
$ (u, \xi)$ is the local solution constructed in Section \ref{sec-Domain}. It is easy to check, thanks to 
\eqref{est-nabla-u-theta-Global-existence} and the fact that $ q\geq 4$ and $ \frac43\leq \alpha \leq 2$, that $ v_N\in \mathscr{E}$. Indeed, we get
\begin{eqnarray}
\mathbb{E}\int_0^T|\nabla v_N(t)|^{\frac{1}{1-\frac{2}{\alpha q}}}_{L^q_{2\times 2}} dt &\leq &
 \mathbb{E}\int_0^{T}|\nabla u(t\wedge \xi_N)|^{\frac{1}{1-\frac{2}{\alpha q}}}_{L^q_{2\times 2}} dt
 \leq c\mathbb{E}\int_0^T|\theta(t\wedge \xi_N)|^{\frac{1}{1-\frac{2}{\alpha q}}}_{L^q})dt \nonumber\\
& \leq & c\mathbb{E}\int_0^T(1+|\theta(t)|^q_{L^q})dt<\infty.
\end{eqnarray}
\del{\begin{equation}
 \mathbb{E}\int_0^T|\nabla u(t)|^{\frac{1}{1-\frac{2}{\alpha q}}}_{L^q_{2\times 2}} dt
\leq c\mathbb{E}\int_0^T|\theta(t)|^{\frac{1}{1-\frac{2}{\alpha q}}}_{L^q})dt \leq c\mathbb{E}\int_0^T(1+|\theta(t)|^q_{L^q})dt<\infty.
\end{equation}}
Than, we apply the whole machinery as in Section \ref{sec-Torus} and  estimation as in Section \ref{sec-Domain} 
to prove the existence of the weak (strong in probability) solution of 
Equation \eqref{Main-stoch-eq} in the set $\mathscr{E}$.\del{We justify the 
application of the dominated convergence theorem as follow,
The only remained calculus is to give the justification to the 
application of the dominated convergence theorem as in \eqref{just-domin-Torus},} To prove the uniqueness of the solution in the set 
$\mathscr{E}$, we follow the steps as in Section \ref{sec-Torus}. In particular, in Formula \eqref{Torus-uniquenss-ito-formula-H-1}, 
we estimate the term $ \langle B(w(s)), u^1(s)\rangle = - \langle B(w(s), u^1(s)), w(s)\rangle$ using \eqref{ineq-B-u-v-v-global}. 
The existence and the uniqueness hold, therefore the local solution constructed in Section \ref{sec-Domain} is global and unique.

\begin{remark}
For $ d=3$, we can follow \cite{Giga-al-Globalexistence-2001}......
\end{remark}}


\del{
In Section \ref{sec-Domain}, we have constructed a maximal local weak solution $ (u, \xi)$ satisfying, for all $ N\in \mathbb{N}_0$, \eqref{eq-weak-l-2solu} up to $ \xi_N$. 
Now, for $ O=\mathbb{T}^2$, thanks to the  conditions in $(3.8.2)$ and arguing as in the proof of the regularity in Section \ref{sec-Torus}, we infer that the maximal solution  $ (u, \xi)$ satisfies 
\begin{equation}\label{eq-2D-weak-sol-up-xi}
 \mathbb{E}\sup_{[0, \xi)}|u(t)|^q_{\mathbb{H}^{1, q}}+ \mathbb{E}\int_0^{T\wedge \xi}|u(t)|^2_{\mathbb{H}^{1+\frac\alpha2, 2}}dt\leq c<\infty.
\end{equation}
We denote by $\mathscr{E}$ the set of progressively measurable stochastic processes $ (v(t), t\in [0, T])$ (we can take also an extension of $v$ if this latter is only defined up to a stopping time) satisfying that there exists a predictable stopping time $ \tau$ such that 
$ v \in  L^2(\Omega\times[0, \tau); \mathbb{H}^{\frac\alpha2, 2}(\mathbb{T}^2))$ and the process $ (\nabla v(t), t\in [0, \tau)) $ can be extended (we keep the some notation) to
$ \nabla v \in L^{(1-\frac{2d}{\alpha q})^{-1}}(\Omega\times[0, T]; L^{q}(\mathbb{T}^2))$, with the norm of $ \nabla v$ in this space is uniformly bounded, i.e. independently of the extension. We claim that  
 $\mathscr{E} \neq  \varnothing$. In fact, let us define for $ N$ fixed,
 $ (v_N(t), t\in [0, T])$, by 
  $v_N(t):= u(t\wedge \xi_N), \forall t\in [0, T])$, where $ (u, \xi)$ is our maximal local solution. We have, for $ q$ characterized as in $(3.6.2)$,(here d=2)
\begin{eqnarray}
\del{\mathbb{E}\int_0^T|\nabla v_N(t)|^{\frac{1}{1-\frac{2d}{\alpha q}}}_{L^q_{2\times 2}} dt &\leq &}
 \mathbb{E}\int_0^{T}|\nabla u(t\wedge \xi_N)|^{\frac{1}{1-\frac{2d}{\alpha q}}}_{q} dt
 \leq c\mathbb{E}\int_0^T|\theta(t\wedge \xi_N)|^{\frac{1}{1-\frac{2d}{\alpha q}}}_{L^q})dt \leq c\mathbb{E}\int_0^T(1+|\theta(t)|^q_{L^q})dt<\infty.\nonumber\\
\end{eqnarray}
Therfore\del{it is easy to check thanks to Estimate  \eqref{eq-bale-kato-majda-con} that} $ v_N\in \mathscr{E}$. Remark that the condition $ \frac{1}{1-\frac{2d}{\alpha q}}\leq q\Leftrightarrow 1+\frac{2d}{\alpha}\leq q$, see Remark \ref{Rem-2}. Now, we shall look for a solution in the set $\mathscr{E}$. We can go back to the first part of Section \ref{sec-Domain} and we repeat the same calculus until Estimate \eqref{ineq-B-u-v-v-local}, which we treat now  as follow. Using H\"older twice ($ 1/q+1/q'=1/2$), 
Gaglairdo-Nirenberg and than Young inequalities, we get (recall $V:=\mathbb{H}^{\frac\alpha2, 2} (O)$)
\begin{eqnarray}\label{ineq-B-u-v-v-global}
|{}_{V^*}\langle B(u)&-&B(v), u-v\rangle_{V} |\leq
| (u-v)\nabla v|_{L^2_2}|u-v|_{\mathbb{L}^{2}}
\leq
|u-v|_{L^{q'}_2}|\nabla v|_{q}|u-v|_{\mathbb{L}^{2}}\nonumber\\
&\leq& c|\nabla v|_{q}|u-v|^{2-\frac{2d}{\alpha q}}_{\mathbb{L}^{2}}|u-v|^{\frac{2d}{\alpha q}}_{\mathbb{H}^{\frac\alpha2, 2}}
\leq  c|\nabla v|^{\frac{1}{1-\frac{2d}{\alpha q}}}_{q}|u-v|^{2}_{\mathbb{L}^{2}}+ c|u-v|^{2}_{\mathbb{H}^{\frac\alpha2, 2}}.
\end{eqnarray}
We take $ r'(t):= c(1+ |\nabla v(t)|^{\frac{1}{1-\frac{2d}{\alpha q}}}_{q})$ with relevant constant $ c>0$. Than, we apply the whole machinery as in Section \ref{sec-Torus} and  the estimations as in Section \ref{sec-Domain} to get the existence of the global solution.\del{, we prove then the existence of global weak solution.}\del{to prove the existence of the weak (strong in probability) solution of Equation \eqref{Main-stoch-eq} in the set $\mathscr{E}$. We justify the application of the dominated convergence theorem as follow,
The only remained calculus is to give the justification to the 
application of the dominated convergence theorem as in \eqref{just-domin-Torus}} To prove the uniqueness of the solution in the set 
$\mathscr{E}$, we follow the steps as in Section \ref{sec-Torus}. In particular, in Formula \eqref{Torus-uniquenss-ito-formula-H-1}, 
we estimate the term $ \langle B(w(s)), u^1(s)\rangle = - \langle B(w(s), u^1(s)), w(s)\rangle$ using \eqref{ineq-B-u-v-v-global}. 
The existence and the uniqueness hold, therefore the local solution constructed in Section \ref{sec-Domain} is global and unique. The estimate \eqref{propty-of-2D-global-mild-sol} is obtained from \eqref{eq-2D-weak-sol-up-xi}.\del{as in the proof of the regularity in Section \ref{sec-Torus}.}

\vspace{0.25cm}

For the general case $ (3.8.3)$,\del{we know from Section \ref{sec-Domain}, that there exists at least one maximal local solution. I} if a maximal local weak solution enjoys \eqref{eq-bale-kato-majda-con}, then we have $\mathscr{E} \neq  \varnothing $ and thus we follow the proof above (for $ O=\mathbb{T}^2$) to get the results. If a maximal local weak solution enjoys Condition \eqref{eq-other-bale-kato-majda-con}, then the set $ \mathscr{E}_1 \neq  \varnothing$, where $ \mathscr{E}_1$ is the set of progressively measurable stochastic processes $ (v(t), t\in [0, T])$ (or the extsension of $ v$ in the case $v$ is defined up to a stopping time) satisfying that there exists a predictable stopping time $ \tau$ such that $ v \in  L^2(\Omega\times[0, \tau); \mathbb{H}^{\frac\alpha2, 2}(\mathbb{T}^2))$ and can be extended (we keep the some notation) to
$ v \in L^{\frac{4\alpha}{3\alpha-d-2}}(\Omega\times[0, T]; \mathbb{H}^{\frac{d+2-\alpha}{4}, 2}(O))$ uniformly, i.e. with the norm of $ v$ in this space is uniformly bounded independently of the extension. Now, we can continue from Estimate \eqref{ineq-B-u-v-v-local}\del{{eq-monoto-local}} and follow the proof as above and as in Section \ref{sec-Torus}.

\del{   Now, for $ O=\mathbb{T}^2$, the condition \eqref{eq-bale-kato-majda-con} follows from Lemma \ref{lem-basic-curl-gradient} and \cite[Theorem 2.6.]{Debbi-scalar-active}. In fact, 
\noindent The verification of the application of \cite[Theorem 2.6.]{Debbi-scalar-active} and of the obtention of the regularity  \eqref{propty-of-2D-global-mild-sol} are done exactly as in Section \ref{sec-global-mild-solution}.}

\del{We start doing the calculus in a general setting. 
In Section \ref{sec-Domain}, we have constructed  a local maximal weak solution $ (u, \xi)$ satisfying \eqref{eq-weak-l-2solu} up to $ \xi_N$ for all $ N\in \mathbb{N}_0$. We assume that $ (u, \xi)$ enjoys also Estimate  \eqref{eq-bale-kato-majda-con}.  We go back to the firts part of Section \ref{sec-Domain} and we repeat the same calculus until Estimate \eqref{ineq-B-u-v-v-local}, which we treat now  as follow. Using H\"older twice ($ 1/q+1/q'=1/2$), 
Gaglairdo-Nirenberg and than Young inequalities, we get (recall $V:=\mathbb{H}^{\frac\alpha2, 2} (O)$)
\begin{eqnarray}\label{ineq-B-u-v-v-global}
|{}_{V^*}\langle B(u)&-&B(v), u-v\rangle_{V} |\leq
| (u-v)\nabla v|_{L^2_2}|u-v|_{\mathbb{L}^{2}}
\leq
|u-v|_{L^{q'}_2}|\nabla v|_{q}|u-v|_{\mathbb{L}^{2}}\nonumber\\
&\leq& c|\nabla v|_{q}|u-v|^{2-\frac{2d}{\alpha q}}_{\mathbb{L}^{2}}|u-v|^{\frac{2d}{\alpha q}}_{\mathbb{H}^{\frac\alpha2, 2}}
\leq  c|\nabla v|^{\frac{1}{1-\frac{2d}{\alpha q}}}_{q}|u-v|^{2}_{\mathbb{L}^{2}}+ c|u-v|^{2}_{\mathbb{H}^{\frac\alpha2, 2}}.
\end{eqnarray}
We take $ r'(t):= c(1+ |\nabla v(t)|^{\frac{1}{1-\frac{2d}{\alpha q}}}_{q})$ with relevant constant $ c>0$ and 
$ v\in \mathscr{E}$, where $\mathscr{E}$ is the set of progressively measurable stochastic processes satisfying 
$ v \in  L^2(\Omega\times[0, T]; \mathbb{H}^{\frac\alpha2, 2}(\mathbb{T}^2))$ and 
$ \nabla v \in L^{(1-\frac{2}{\alpha q})^{-1}}(\Omega\times[0, T]; L^{q}(\mathbb{T}^2))$. We claim that  
 $\mathscr{E} \neq  \varnothing$. In fact, let us define for $ N$ fixed,
 $ (v_N(t), t\in [0, T])$, by 
  $v_N(t):= u(t\wedge \xi_N), \forall t\in [0, T])$, where 
$ (u, \xi)$ is our local solution\del{ (constructed in Section \ref{sec-Domain}).} Therfore, it is easy to check thanks to Estimate  \eqref{eq-bale-kato-majda-con} that $ v_N\in \mathscr{E}$. 
Than, we apply the whole machinery as in Section \ref{sec-Torus} and  estimation as in Section \ref{sec-Domain} 
to prove the existence of the weak (strong in probability) solution of 
Equation \eqref{Main-stoch-eq} in the set $\mathscr{E}$.\del{We justify the 
application of the dominated convergence theorem as follow,
The only remained calculus is to give the justification to the 
application of the dominated convergence theorem as in \eqref{just-domin-Torus},} To prove the uniqueness of the solution in the set 
$\mathscr{E}$, we follow the steps as in Section \ref{sec-Torus}. In particular, in Formula \eqref{Torus-uniquenss-ito-formula-H-1}, 
we estimate the term $ \langle B(w(s)), u^1(s)\rangle = - \langle B(w(s), u^1(s)), w(s)\rangle$ using \eqref{ineq-B-u-v-v-global}. 
The existence and the uniqueness hold, therefore the local solution constructed in Section \ref{sec-Domain} is global and unique.

Now, for $ O=\mathbb{T}^2$, the condition \eqref{eq-bale-kato-majda-con} follows from Lemma \ref{lem-basic-curl-gradient} and \cite[Theorem 2.6.]{Debbi-scalar-active}. In fact, for $ q\geq 4$,
\begin{eqnarray}
\del{\mathbb{E}\int_0^T|\nabla v_N(t)|^{\frac{1}{1-\frac{2}{\alpha q}}}_{L^q_{2\times 2}} dt &\leq &}
 \mathbb{E}\int_0^{T}|\nabla u(t\wedge \xi_N)|^{\frac{1}{1-\frac{2}{\alpha q}}}_{q} dt
 \leq c\mathbb{E}\int_0^T|\theta(t\wedge \xi_N)|^{\frac{1}{1-\frac{2}{\alpha q}}}_{L^q})dt \leq c\mathbb{E}\int_0^T(1+|\theta(t)|^q_{L^q})dt<\infty.\nonumber\\
\end{eqnarray}
\noindent The verification of the application of \cite[Theorem 2.6.]{Debbi-scalar-active} and of the obtaintion of the regularity  \eqref{propty-of-2D-global-mild-sol} are done exactly as in Section \ref{sec-global-mild-solution}.}

 \del{we conclude that $ \nabla u$ enjoys \eqref{est-nabla-u-theta-Global-existence} for $ \tau_N$ here is replaced by $ \xi_m$.Appendix \ref{sec-Passage Velocity-Vorticity}
and \cite{Debbi-scalar-active}, we  infer that (here we take, for simplicity, $ q= q_0\geq 6$)  
\begin{equation}\label{est-nabla-u-theta-Global-existence}
 \exists c>0, s.t. \forall N\in \mathbb{N}_0,  \mathbb{E}\sup_{[0, \xi_N)}|\nabla u(t)|^q_{L_{2\times 2}^q}\leq c \mathbb{E} \sup_{[0, \xi_N]}
|\theta(t)|^q_{L^q}\leq c <\infty.
\end{equation}
Recall $ \theta:= curl u$.}


\del{In this section, we prove Theorem \ref{Main-theorem-martingale-solution-d}.\del{ The method is quiet standard, see for Banach spaces 
\cite{Debbi-scalar-active} and for Hilbert space \cite{Flandoli-Gatarek-95}.} The main ingredients are Faedo-Galerkin approximations,
compactness and Skorokhod embedding theorem, see for similar calculus in
\cite{Debbi-scalar-active, Flandoli-Gatarek-95}.

\begin{lem}\label{lem-bounded-W-gamma-p}
The sequence $ (u_n)_n$ of solutions of Equations \eqref{FSBE-Galerkin-approxi} is uniformly bounded in the space
\begin{eqnarray}\label{Eq-W-}
L^2(\Omega, W^{\gamma, 2}(0, T; \del{\mathbb{H}^{-\delta', 2}(O))}
H_d^{-\delta', 2}(O))\cap L^2(0, T; \mathbb{H}^{\frac\alpha2, 2}(O))),
\end{eqnarray}
where  $ \delta'\geq_1\max\{\alpha, 1+\frac d{2}\}$ and $ \gamma <\frac12$.
\end{lem}
\begin{proof}
Thanks to Lemma \ref{lem-unif-bound-theta-n-H-1-domain}, it is sufficient to prove that $ (u_n(t), t\in [0, T])$ is
uniformly bounded in $L^2(\Omega, W^{\gamma, 2}(0, T; \mathbb{H}^{-\delta', 2}(O))$. We recall that the Besov-Slobodetski space $W^{\gamma,
p}(0, T; E)$, with $ E$ being a Banach space, $ \gamma \in (0, 1)$ and $ p\geq 1$, is the space
of all $ v\in L^P(0, T; E) $ such that
\begin{eqnarray}
||v||_{W^{\gamma, p}}:= \left(\int_0^T|v(t)|_E^pdt+
\int_0^T\int_0^T\frac{|v(t)-v(s)|_E^p}{|t-s|^{1+\gamma p}}
dtds\right)^{\frac1p}<\infty.
\end{eqnarray}
\noindent As  $(u_n(t), t\in [0, T])$ is the strong solution  of the finite dimensional  stochastic
differential equation \eqref{FSBE-Galerkin-approxi}, then  $u_n(t)$  is the solution of the stochastic integral equation 
\begin{equation}\label{FSBE-Integ-solu-Galerkin-approxi}
u_n(t)= P_nu_0 + \int_0^t(-A_\alpha u_n(r) + P_nB(u_n(r))dr + \int_0^tP_nG(u_n(r))\,dW_n(r),\; a.s.,\\
\end{equation}
for all $t\in [0, T]$. We denote by 
\begin{equation}\label{Eq-Drift-term}
I(t):=  \int_0^t(-A_\alpha u_n(r) + P_nB(u_n(r))dr
\end{equation}
and
\begin{equation}\label{Eq-Drift-term}
J(t):= \int_0^tP_nG
(u_n(r))\,dW_n(r).
\end{equation}
\noindent We prove that $ I(\cdot)$ is uniformly bounded in  $L^2(\Omega; W^{\gamma, 2}(0, T; H_d^{-\delta', 2}(O))$ 
and that the stochastic term $ J(\cdot )$ is uniformly bounded in $ L^2(\Omega; W^{\gamma, 2}(0, T;  \mathbb{L}^2(O))$, for all $ \gamma<\frac12$.\del{ as the stochastic term $ J$ is more regular then the drift term, see e.g. \cite[Lemma 2.1]{Flandoli-Gatarek-95}. \\}
Let  $ \phi \in H_d^{{\delta'}, 2}(O)$, using Identity \eqref{Eq-3lin-propsym}, we get
\begin{eqnarray}
| {}_{H_d^{-{\delta'}, 2}}\langle P_nB(u_n(r)), \phi\rangle_{H_d^{{\delta'}, 2}}|
&=& |\langle u_n(r) \cdot
\nabla P_n\phi, u_n(r)\rangle_{L_d^{2}}|\nonumber\\
&\leq& |\nabla P_n\phi|_{L^\infty}| u_n(r)|^2_{\mathbb{L}^{2}}.
\end{eqnarray}
Thanks to \cite[Remark 4 p 164, Theorem 3.5.4.ps.168-169 and Theorem 3.5.5 p 170]{Schmeisser-Tribel-87-book} for $ O= \mathbb{T}^d$, to
\cite[Theorem 7.63  and point 7.66]{Adams-Hedberg-94} for $ O$ being a bounded domain and to
the condition $ {\delta'}>1+\frac d{2}$,
we deduce for  $ 0<\epsilon < \delta'-1-\frac d{2}$,
 $$ |\nabla P_n\phi|_{L^\infty} \leq c |\nabla P_n\phi|_{H^{\epsilon+\frac d2, 2}} \leq c
|\phi|_{H_d^{1+\epsilon+\frac d2, 2}} \leq c |\phi|_{H_d^{\delta', 2}}.$$
Therefore,
\begin{eqnarray}\label{}
 |P_nB(u_n(r))|_{H_d^{-\delta', 2}}\leq c |u_n(r)|_{\mathbb{L}^{2}}^2
\end{eqnarray}
 and
\begin{eqnarray}\label{eq-unif-int-I(t)}
\int_0^T|I(t)|_{H_d^{-\delta', 2}}^2dt&\leq& c\int_0^T\int_0^t \big(|(-A_\alpha
u_n(r)|^2_{\mathbb{H}^{-\delta', 2}} + |P_nB(u_n(r))|^2_{H_d^{-\delta', 2}}\big)drdt\nonumber\\
&\leq& c\int_0^T\int_0^t\big(|
u_n(r)|^2_{\mathbb{L}^{ 2}} + |u_n(r)|^4_{\mathbb{L}^{2}}\big)drdt.
\end{eqnarray}
Moreover, using H\"older inequality and arguing as before, we get for $ t\geq s > 0$,
\begin{eqnarray}\label{eq-unif-int-I(t)-I(s)}
|I(t)- I(s)|^2_{H_d^{-\delta', 2}}&=& |\int_s^t(-A_\alpha u_n(r) +
P_nB(u_n(r))dr|^2_{\mathbb{H}^{-{\delta'}, 2}}\nonumber\\
&\leq & C(t-s)\left(\int_s^t(| u_n(r)|^2_{\mathbb{L}^{2}} +
|u_n(r)|^4_{\mathbb{L}^{2}})dr \right).
\end{eqnarray}
From \eqref{eq-unif-int-I(t)},  \eqref{eq-unif-int-I(t)-I(s)} and \eqref{Eq-Ito-n-weak-estimation-1-bounded}, 
we have for  $ \gamma <\frac12$,
\begin{eqnarray}\label{eq-unif-int-I(t)x2}
\mathbb{E}\big(\int_0^T|I(t)|_{H_d^{-\delta', 2}}^2dt&+&\int_0^T\int_0^T
\frac{|I(t)- I(s)|^2_{H_d^{-\delta',
2}}}{|t-s|^{1+2\gamma }} dtds\big)^{\frac12} \nonumber\\
&\leq& C\mathbb{E}\left(\int_0^T(| u_n(r)|^2_{\mathbb{L}^{2}} +
|u_n(r)|^4_{\mathbb{L}^{2}})dr
\right)^\frac12 \leq C<\infty.
\end{eqnarray}
Now, we estimate the stochastic term $ J$. 
Using the stochastic isometry, the contraction property of $ P_n$ and\del{ Assumption $
(\mathcal{C})$} Condition  \eqref{Eq-Cond-Linear-Q-G}, we get
\begin{eqnarray}
\int_0^T\mathbb{E}|\int_0^tP_nG(u_n(r))dW_n(r)|_{\mathbb{L}^{ 2}}^2dt&\leq&
C\int_0^T\mathbb{E}\int_0^t||G(u_n(r))||^2_{L_Q(\mathbb{L}^2)}drdt\nonumber\\
&\leq&
C\int_0^T\mathbb{E}\int_0^t(1+|u_n(r)|^2_{\mathbb{L}^2})drdt \leq c<\infty.
\end{eqnarray}
Moreover, for $ t\geq s> 0$ and $ \gamma <\frac12$, the same ingredients
above yield to
\begin{eqnarray}
\mathbb{E}\int_0^T\int_0^T\frac{|J(t)-
J(s)|^2_{\mathbb{L}^{2}}}{|t-s|^{1+2\gamma }} dtds &\leq&
C\mathbb{E}\int_0^T\int_0^T\frac{\int_s^t||G(u_n(r))||^2_{L_Q(\mathbb{L}^2)}dr}{|t-s|^{1+2\gamma
}} dtds \nonumber\\
&\leq& C\mathbb{E}\sup_{[0, T]}(1+|u_n(t)|^2_{\mathbb{L}^{2}})
\int_0^T\int_0^T|t-s|^{-2\gamma } dtds \leq c <\infty.
\end{eqnarray}
The proof of the Lemma is now completed.
\end{proof}
\begin{remark}
It is also possible, see e.g. \cite{Flandoli-Gatarek-95}, to 
\del{replace in \eqref{Eq-W-}, }prove the boundedness of the sequence $ (u_n)_n$ in 
\begin{eqnarray}
L^2(\Omega, W^{\gamma, 2}(0, T; D(A^{-\delta'})\cap L^2(0, T; \mathbb{H}^{\frac\alpha2, 2}(O))),
\end{eqnarray}

\end{remark}
\del{\noindent {\bf Proof of the existence of martingale solution
\del{Theorem \ref{Main-theorem-martingale-solution-d}}}}
To prove the existence of a martingale solution, we consider the Gelfand triplet \eqref{Gelfand-triple-Domain} and 
 use lemmas \ref{lem-unif-bound-theta-n-H-1-domain} and \ref{lem-bounded-W-gamma-p} 
and the compact embedding, see \cite[Theorem 2.1]{Flandoli-Gatarek-95}
\begin{equation}
 W^{\gamma, 2}(0, T; H_d^{-\delta', 2}(O))\cap \mathbb{L}^2(0, T; \mathbb{H}^{\frac\alpha2, 2}(O)) 
 \hookrightarrow L^2(0, T; \mathbb{L}^2(O)).
\end{equation}
\del{$ L^2(0, T; \mathbb{H}^{\frac\alpha2, 2}(O)) \hookrightarrow L^2(0, T; \mathbb{L}^2(O))$,}   Then we deduce
 that the sequence of laws $ (\mathcal{L}(u_n))_n$  is tight on $ L^2(0, T; \mathbb{L}^2(O))$. 
 Thanks to Prokhorov's theorem
there exists a  subsequence, still denoted $ (u_n)_n$, for which  the sequence of laws $ (\mathcal{L}(u_n))_n$ converges  weakly, on
$ L^2(0, T; \mathbb{L}^2(O))$,  to a probability measure $ \mu$. By Skorokhod's embedding theorem, we can construct a probability basis
$ (\Omega^*, F^*, \mathbb{F}^*,  P^*)$  and a sequence of $ L^2(0, T; \mathbb{L}^2(O))\cap C([0, T]; H_d^{-\delta', 2}(O))-$random variables
$ (u^*_n)_n$ and $ u^*$ such that  $\mathcal{L}(u^*_n) = \mathcal{L}(u_n), \forall n \in \mathbb{N}_0$,  $\mathcal{L}(u^*) = \mu$ and
$ u^*_n \rightarrow u^* a.s.$ in $ L^2(0, T; \mathbb{L}^2(O))\cap C([0, T]; H_d^{-\delta', 2}(O))$. Therefore, we conclude that
 for all $ n\in \mathbb{N}$,
\begin{eqnarray}\label{eq-bound-u-*-n-u-*}
\mathbb{E}\sup_{[0, T]}| u^*_n(s)|^p_{\mathbb{L}^2}+ \mathbb{E}\int_0^T| u^*_n(s)|_{ \mathbb{H}^{\frac\alpha2, 2}}ds
+\mathbb{E}\sup_{[0, T]}| u^*(s)|^p_{\mathbb{L}^2}+ \mathbb{E}\int_0^T| u^*(s)|^2_{ \mathbb{H}^{\frac\alpha2, 2}}ds \leq c<\infty.\nonumber\\
\end{eqnarray}
Consequently, the sequence  $ u^*_n$ converges weakly in $  L^2(\Omega\times [0, T]; \mathbb{H}^{\frac\alpha2, 2}(O))$ to a limit $ u^{**}$.
It is easy to see that $u^{*} = u^{**},  P\times dt a.e.$. We construct, as in \cite{Flandoli-Gatarek-95}, the process $ (M_n(t), t\in [0, T])$ with trajectories in
$ C([0, T]; \mathbb{L}^2(O))$ by
\begin{equation}
M_n(t):= u_n^*(t) - P_nu_0+\int_0^t A_\alpha u_n^*(s) ds -\int_0^t P_nB(u_n^*(s))ds.
\end{equation}
...>>>>>>>>>>>>>>>>>>>>>>>>>>>>>>>
We follow the same steps as in \cite{Flandoli-Gatarek-95} (hence we omit here the details), we end up with the statement that $ u^*$ is solution
in $ V^*:= \mathbb{H}^{\frac\alpha2, 2}(O)$ to Equation \eqref{Eq-weak-Solution} with $ W$ being replaced by $ W^*$.

\del{It is easy to see, thanks to Equation \eqref{FSBE-Galerkin-approxi} and to the equality in Law of $ u_n $ and $ u_n^*$,
that $ (M_n(t), t\in [0, T])$ is a square integrable martingale with respect to the filtration 
$(\mathcal{G}_n)_t:=\sigma\{u_n^*(s), s\leq t\}$. The remain part of the proof follows the same steps as in \cite{Flandoli-Gatarek-95},
hence we omit here.
\del{$ \mathbb{G}^* := (\mathcal{G}^*_t)_t$,  with $ \mathcal{G}^*_t $ is the $ \sigma-$algebra generated by
$ \cup_n (\mathcal{G}_n)_t:=\sigma\{u_n^*(s), s\leq t\}$.}}

The continuity of the trajectories of the solution $ u $ follows by a similar calculus as in 
Section \ref{sec-Torus} with the Gelfand triplet in \eqref{Gelfand-triple-Domain} and application of 
\cite[Proposition 2.5.]{Sundar-Sri-large-deviation-NS-06}. In deed,  we have thanks to 
Burkholdy-Davis-Gandy inequality, Assumption \eqref{Eq-Cond-Linear-Q-G}, \eqref{eq-bound-u-*-n-u-*}, we get
\begin{eqnarray}\label{cont-2-martg-stochas}
\mathbb{E}\sup_{[0, T]} |\int_0^tG(u(s))dW^*(s)|^2_{\mathbb{L}^{2}}
&\leq& c \mathbb{E}\int_0^T|G^*(u^*(s))|^2_{L_Q(\mathbb{L}^{2})}ds
\leq c (1+ \mathbb{E}\sup_{[0, T]}|u^*(s)|^2_{\mathbb{L}^{2}})<\infty.\nonumber\\
\end{eqnarray}
Moreover, using \eqref{Eq-B-H-alpha-2-est}, the Sobolev embedding 
$ \mathbb{H}^{\frac{d+2-\alpha}{4}, 2}(O)\subset \mathbb{H}^{\frac{\alpha}{2}, 2}(O)$, ( $ 1+\frac{d-1}{3}\leq \alpha \leq 2$) 
and the boundedness of the  operator $ A_\alpha: \mathbb{H}^{\frac{\alpha}{2}, 2}(O) \rightarrow \mathbb{H}^{\frac{-\alpha}{2}, 2}(O)$, we get
\begin{eqnarray}\label{cont-1-martg}
\mathbb{E}\int_0^T\big( |A_\alpha u^*(s)|_{\mathbb{H}^{-\frac\alpha2, 2}}&+& |B(u^*(s))|_{\mathbb{H}^{-\frac\alpha2, 2}}\big)ds \leq c
\mathbb{E} \int_0^T\big( |u^*(s)|_{\mathbb{H}^{\frac\alpha2, 2}} + |u^*(s)|^2_{\mathbb{H}^{\frac{d+2-\alpha}{4}, 2}}\big)ds\nonumber\\
&\leq& c (1+\mathbb{E} \int_0^T|u^*(s)|^2_{\mathbb{H}^{\frac\alpha2, 2}}ds)<\infty. 
\end{eqnarray}
>>>>>>>>>>>>>>>>>>>>>>>>>>>>>>>>>>>>>>>>>>>>}
\del{\begin{equation}\label{uniquness-con-martg}
 P^*( u (\cdot, \omega)\in L^{\frac{4\alpha}{3\alpha-d-2}}(0, T; \mathbb{H}^{\frac{d+2-\alpha}{4}, 2}(O)))=1.
\end{equation} \eqref{3linear-H1-H-1}, Young inequality,  we infer that
\begin{eqnarray}\label{Torus-uniquenss-ito-formula}
\mathbb{E}\!\!\!\!&{}&\!\!\!\!e^{-r(t\wedge \tau)}|w(t\wedge \tau)|^2_{\mathbb{L}^{2}}+
2\mathbb{E}\int_0^{t\wedge \tau}e^{-r(s)}|w(s)|^2_{\mathbb{H}^{\frac\alpha2, 2}}ds \nonumber\\
&\leq& \mathbb{E}\int_0^{t\wedge \tau}e^{-r(s)}|| G(u^1(s))- G(u^2(s))||^2_{HS_Q(\mathbb{L}^2)}\nonumber\\
& -& \mathbb{E}\int_0^{t\wedge \tau} e^{-r(s)} \big(2\langle B(w(s),  w(s)),  u^1(s)\rangle-
r'(s)|w(s)|^2_{\mathbb{L}^{2}}\big)ds\nonumber\\
& \leq & c_N\mathbb{E}\int_0^{t\wedge \tau}e^{-r(s)}\big(|w(s)|^2_{\mathbb{L}^2}+
|u^1(s)|_{\mathbb{H}^{?, 2}}|w(s)|^{\frac2\alpha}_{\mathbb{H}^{\frac{\alpha}{2}, 2}}
|w(s)|^{2\frac{1-\alpha}\alpha}_{{\mathbb{L}^{2}}}-
r'(s)|w(s)|^2_{\mathbb{L}^{2}}\big)ds\nonumber\\
& \leq & c_N
\mathbb{E}\int_0^{t\wedge \tau}e^{-r(s)}\big(|w(s)|^2_{\mathbb{L}^2}+ 2 c|w(s)|^2_{\mathbb{H}^{\frac\alpha2, 2}}+
2c_1|u^1(s)|^{\frac{\alpha}{\alpha-1}}_{\mathbb{H}^{1, 2}}|w(s)|^2_{\mathbb{L}^{2}}-
r'(s)|w(s)|^2_{\mathbb{L}^{2}}\big)ds.\nonumber\\
\end{eqnarray}
Now, we choose $ c<1$ and $ r'(s)= 2c_1|u^1(s)|^{\frac{\alpha}{\alpha-1}}_{\mathbb{H}^{1, 2}}$ and replace in
\eqref{Torus-uniquenss-ito-formula},
we end up by the simple formula
\begin{eqnarray}
\mathbb{E}e^{-r(t\wedge \tau)}|w(t\wedge \tau)|^2_{\mathbb{L}^{2}}&+& 2(1-c)\mathbb{E}\int_0^{t\wedge \tau}
e^{-r(s)}|w(s)|^2_{\mathbb{H}^{\frac\alpha2, 2}}ds
\leq c_N\mathbb{E}\int_0^{t\wedge \tau}e^{-r(s)}|w(s)|^2_{\mathbb{L}^2}ds.\nonumber\\
\end{eqnarray}
Than by application of Gronwall's lemma, we get
$ \forall t\in [0, T, \, e^{-r(t\wedge \tau)}|w(t\wedge \tau)|^2_{\mathbb{L}^{2}} = 0\, P-a.s.$ as
$P( e^{-c\int_0^{t\wedge \tau}|u^1(s)|^{\frac{\alpha}{\alpha-1}}_{\mathbb{H}^{1, 2}}ds} <\infty)= 1 $.
The proof is achieved once we remark that thanks to Chebyshev inequality and \eqref{cond-on the solution-Torus},
we have $ \lim_{N\rightarrow \infty}\tau_N = T a.s.$  and
$ P( \int_0^T|u^1(s)|^{\frac{\alpha}{\alpha-1}}_{\mathbb{H}^{1, 2}}ds <\infty)= 1$
\del{$ P( e^{-2c_1|u^1(s)|^{\frac{\alpha}{\alpha-1}}_{\mathbb{H}^{1, 2}}}<\infty)= 1$,
we conclude that
$\forall t\in [0, T] w(t)= 0, P-a.s.$}}

\del{\begin{lem}\label{lem-bounded-W-gamma-p}
The sequence $ (u_n)_n$ of solutions of Equations \eqref{FSBE-Galerkin-approxi} is uniformly bounded in the space
\begin{eqnarray}\label{Eq-W-}
L(\Omega, W^{\gamma, 2}(0, T; \mathbb{H}^{-\delta', 2}(O))\cap \mathbb{L}^2(0, T; \mathbb{H}^{\frac\alpha2, 2}(O)),
\end{eqnarray}
where  $ \delta'\geq_1\max\{\alpha, 1+\frac d{2}\}$.
\end{lem}
\begin{proof}
Thanks to Lemma \ref{lem-unif-bound-theta-n-l-2}, it is sufficient to prove that $ (u_n(t), t\in [0, T])$ is
uniformly bounded in $L(\Omega, W^{\gamma, 2}(0, T; \mathbb{H}^{-\delta', 2}(O))$. We recall that the Besov-Slobodetski space $W^{\gamma,
p}(0, T; E)$, with $ E$ being a Banach space, $ \gamma \in (0, 1)$ and $ p\geq 1$, is the space
of all $ v\in L^P(0, T; E) $ such that
\begin{eqnarray}
||v||_{W^{\gamma, p}}:= \left(\int_0^T|v(t)|_E^pdt+
\int_0^T\int_0^T\frac{|v(t)-v(s)|_E^p}{|t-s|^{1+\gamma p}}
dtds\right)^{\frac1p}<\infty.
\end{eqnarray}
\noindent As  $(u_n(t), t\in [0, T])$ is the solution  of the finite dimensional  stochastic
differential equation then, for $t\in [0, T]$,  we rewrite  $u_n(t)$  as
\begin{equation}\label{FSBE-Integ-solu-Galerkin-approxi}
u_n(t)= P_nu_0 + \int_0^t(-A_\alpha u_n(r) + P_nB_n(u_n(r))dr + \int_0^tP_nG_n(u_n(r))\,dW_n(r)\; a.s.\\
\end{equation}
We denote
\begin{equation}\label{Eq-Drift-term}
I(t):=  \int_0^t(-A_\alpha u_n(r) + P_nB_n(u_n(r))dr
\end{equation}
and
\begin{equation}\label{Eq-Drift-term}
J(t):= \int_0^tP_nG_n(u_n(r))\,dW^n(r).
\end{equation}
\noindent We prove that $ I(\cdot)$ is uniformly bounded in  $ W^{\gamma, 2}(0, T; \mathbb{H}^{-\delta', 2}(O))$
and that the stochastic term $ J(\cdot )$
is uniformly bounded in $ W^{\gamma, 2}(0, T;  L^2(\mathbb{T}^d)$, for all $ \gamma<\frac12$. \\
Let  $ \phi \in H^{{\delta'}, q^*}(\mathbb{T}^d)$. Recall that thanks to the choice of $ \delta'$ and to
\cite[Remark 4 p 164, Theorem 3.5.4.ps.168-169 and Theorem 3.5.5 p 170]{Schmeisser-Tribel-87-book}, we have
$  H^{{\delta'}, q^*}(\mathbb{T}^d) \hookleftarrow  L^2(\mathbb{T}^d)$. Therefore by use of Identity \eqref{Eq-3lin-propsym},
 we get
\begin{eqnarray}
| {}_{H^{-{\delta'}, q}}\langle P_nB_n(u_n(r)), \phi\rangle_{H^{{\delta'}, q^*}}|
&=& \mathcal{X}(|u_n(r)|_{L^2})|{}_{L^{2}}\langle \mathcal{R}^{\gamma, \sigma}u_n(r) \cdot \nabla P_n\phi, u_n(r)\rangle_{L^{2}}|\nonumber\\
&\leq& |\nabla P_n\phi|_{L^\infty}| \mathcal{R}^{\gamma, \sigma}u_n(r)|_{L^{2}}|u_n(r)|_{L^2}.
\end{eqnarray}
Thanks to \cite[Remark 4 p 164, Theorem 3.5.4.ps.168-169 and Theorem 3.5.5 p 170]{Schmeisser-Tribel-87-book} and to
the condition $ {\delta'}>1+\frac d{q^*}$,
we deduce for  $ 0<\epsilon < \delta'-1-\frac d{q^*}$,
 $$ |\nabla P_n\phi|_{L^\infty} \leq c |\nabla P_n\phi|_{H^{\epsilon+\frac d2, 2}} \leq c |\phi|_{H^{1+\epsilon+\frac d2, 2}}
\leq c |\phi|_{H^{{\delta'}, q^*}}.$$
Using Lemma \ref{lem-R-bounded-H-s} and the fact that $ 2\leq q <\infty$, we infer
\begin{eqnarray}\label{}
 |P_nB_n(u_n(r))|_{H^{-{\delta'}, q}}\leq c |u_n(r)|_{L^{q}}^2.
\end{eqnarray}
Hence,
\begin{eqnarray}\label{eq-unif-int-I(t)}
\int_0^T|I(t)|_{H^{-{\delta'}, q}}^2dt&\leq& c\int_0^T\int_0^t \big(|(-A_\alpha
u_n(r)|^2_{H^{-{\delta'}, q}} + |P_nB(u_n(r))|^2_{H^{-{\delta'}, q}}\big)drdt\nonumber\\
&\leq& c\int_0^T\int_0^t\big(|
u_n(r)|^2_{L^{ q}} + |u_n(r)|^4_{L^{q}}\big)drdt.
\end{eqnarray}
Moreover, using H\"older inequality and arguing as before, we get for $ t\geq s > 0$,
\begin{eqnarray}\label{eq-unif-int-I(t)-I(s)}
|I(t)- I(s)|^2_{H^{-{\delta'}, q}}&=& |\int_s^t(-A_\alpha u_n(r) +
P_nB_n(u_n(r))dr|^2_{H^{-{\delta'}, q}}\nonumber\\
&\leq & C(t-s)\left(\int_s^t(| u_n(r)|^2_{L^{ q}} +
|u_n(r)|^4_{L^{q}})dr \right).
\end{eqnarray}
Hence, from \eqref{eq-unif-int-I(t)} and \eqref{eq-unif-int-I(t)-I(s)}, we have for  $ \gamma <\frac12$,
\begin{eqnarray}\label{eq-unif-int-I(t)x2}
\mathbb{E}(\int_0^T|I(t)|_{H^{-{\delta'}, q}}^2dt&+&\int_0^T\int_0^T\frac{|I(t)- I(s)|^2_{H^{-\delta',
q}}}{|t-s|^{1+2\gamma }} dtds)^{\frac12} \nonumber\\
&\leq& C\mathbb{E}\left(\int_0^T(| u_n(r)|^2_{L^{ q}} +
|u_n(r)|^4_{L^{q}})dr
\right)^\frac12 \leq C<\infty.
\end{eqnarray}
The RHS of \eqref{eq-unif-int-I(t)x2} is uniformly bounded thanks to Estimation \eqref{Eq-Ito-n-weak-estimation-Lq}.
Now, we estimate the stochastic term $ J$. This last is more regular
then the drift term. We  prove $ J \in L(\Omega, W^{\gamma, 2}(0, T;
L^{q}(\mathbb{T}^d))$.
Using the stochastic isometry and Assumption $
(\mathcal{A})$, we get
\begin{eqnarray}
\int_0^T\mathbb{E}|\int_0^tP_nG_n(u_n(r))dW^n(r)|_{L^{ q}}^2dt&\leq&
C\int_0^T\mathbb{E}\int_0^t||G(u_n(r))Q^\frac12||^2_{R_\gamma(L^2, L^q)}drdt\nonumber\\
&\leq&
C\int_0^T\mathbb{E}\int_0^t(1+|u_n(r)|^2_{L^q})drdt <\infty.
\end{eqnarray}
Moreover, for $ t\geq s> 0$ and $ \gamma <\frac12$, the same ingredients
above yield to
\begin{eqnarray}
\mathbb{E}\int_0^T\int_0^T\frac{|J(t)-
J(s)|^2_{L^{2}}}{|t-s|^{1+2\gamma }} dtds &\leq&
C\mathbb{E}\int_0^T\int_0^T\frac{\int_s^t||G(u_n)Q^\frac12||^2_{R_\gamma(L^2, L^q)}}{|t-s|^{1+2\gamma
}} dtds \nonumber\\
&\leq& C\mathbb{E}\sup_{[0, T]}(1+|u_n(t)|^2_{L^{q}})
\int_0^T\int_0^T|t-s|^{-2\gamma } dtds <\infty.
\end{eqnarray}
The proof is now completed.
\end{proof}
\noindent {\bf Proof of  Theorem \ref{Main-theorem-mild-solution-3}.}
\noindent To prove  Theorem \ref{Main-theorem-mild-solution-3}, we can proceed by tow methods which seem, following our technique,
equivalent.
\begin{itemize}
 \item {\bf Method 1. Hilbert space setting.}  Thanks to Lemmas  \ref{lem-unif-bound-theta-n-L-2} and
\ref{lem-bounded-W-gamma-p}  with $ q=2$, we prove the first statement of Theorem \ref{Main-theorem-mild-solution-3}, i.e.
we prove the existence of a martingale solution
$(u^*(t), t\in [0, T])$ for Equation  \eqref{Main-stoch-eq} with $ (\sigma, \gamma)\in C_a$, in the sense of
Definition \eqref{def-martingle-solution}. Indeed, one can follow
\cite{Flandoli-Gatarek-95} combined with Lemma \ref{lem-R-bounded-H-s} and Proposition \ref{prop-Ben-q=2}.
To prove the second statement, we prove that $(u^*(t), t\in [0, T])$ satisfies
Equation \eqref{Eq-weak-Solution},  with $ \varphi \in H^{{\eta}, q^*}(\mathbb{T}^d)$ with $ q^*$ being the conjugate of $ q>2$.
To get the estimations \eqref{eq-set-solu-weak} and \eqref{Eq-reg-theta-mild}, we use Lemma \ref{lem-unif-bound-theta-n-l-2}.

\item {\bf Method 2. Banach space setting.} The idea here is to  extend the techniques used for the Hilbert space setting to Banach
spaces $ L^q(\mathbb{T}^d)$ with $ q>2 $. Let $ 2\leq q\leq_\infty q_0$. We prove that the
family of Laws of solutions $ (\mathcal{L}(u_n))_n$ is tight in $
L^2(0, T; L^{q}(\mathbb{T}^d))$. In fact, this latter is deduced thanks to Lemmas
 \ref{lem-unif-bound-theta-n-l-2} and  \ref{lem-bounded-W-gamma-p} and  the following embedding, see \cite[Theorem 2.1]{Flandoli-Gatarek-95} and
\cite[Remark 4 p 164, Theorem 3.5.4.ps.168-169 and Theorem 3.5.5 p 170]{Schmeisser-Tribel-87-book},
\begin{equation}
 L^2(0, T; H^{\frac\alpha2, 2})\cap W^{\gamma,
2}(0, T; H^{-{\delta'}, q}) \hookrightarrow L^2(0, T; L^{q}),
\end{equation}
provided $ d(1-\frac2q)\leq \alpha \leq 2$. This means that $ q $ and $d$ should enjoy the relation
$ d-2\leq \frac{2d}{q}$. This last is trivial for  $ d\leq \alpha$ and $ q\in[2, +\infty)$.
For $ d>\alpha$, we assume $  q\in [2, \frac{2d}{d-\alpha}]$. Therefore, under the conditions in Theorem \ref{Main-theorem-mild-solution-3} and
thanks to Skorokhod Theorem, we infer the existence of a probabilistic basis
$ (\Omega^*, \mathcal{F}^*, \mathbb{P}^*, \mathbb{F}^*, W^*)$ and  $L^2(0, T; L^{q})\cap C([0, T]; H^{-{\delta'}, q})$-valued random variables
$u^*$  and $(u^*_n)_n$ defined on this basis such that   $
\mathcal{L}(u_n^*) = \mathcal{L}(u_n) $ and  $u^*_n \rightarrow u^*, \;\; P-a.s. $
in $ L^2(0, T; L^{q})\cap C([0, T]; H^{-{\delta'}, q})$. Moreover, thanks to the above equality in law,
$(u_n^*(t), t\in[0, T])$ satisfies \eqref{Eq-Ito-n-weak-estimation-1} for all $ n\in \mathbb{N}_0$,
\eqref{Eq-Ito-n-weak-estimation-Lq} and \eqref{Eq-Ito-n-weak-estimation}.
Hence we deduce that $(u^*(t), t\in[0, T])$ satisfies  \eqref{eq-set-solu-weak} and \eqref{Eq-reg-theta-mild}.
To prove that $(u^*(t), t\in[0, T])$ is a weak solution of Equation \eqref{Main-stoch-eq} in the sense
of Definition \ref{def-variational solution}, i.e.  $ u^*$ satisfies Equation \eqref{Eq-weak-Solution}, we use Lemma \ref{lem-R-bounded-H-s} and
proceed as in \cite{Flandoli-Gatarek-95}.
\end{itemize}

\noindent {\bf Proof of Part 3. of Theorem \ref{Main-theorem-mild-solution}}

\noindent Let us remark, that thanks to Lemma  \ref{lem-unif-bound-theta-n-L-2} and  Lemma
\ref{lem-bounded-W-gamma-p}  with $ q=2$ in this later, we can prove Part 3. of
Theorem \ref{Main-theorem-mild-solution}, see, for this technique in Hilbert space, e.g. \cite{Flandoli-Gatarek-95, Rockner-quasi-geostro-1}.
As our method to prove the existence of the global mild solution is to use the result of the existence of the martingale solution,
we are then interested to extend this result for $ L^q-$spaces with $ q \geq 2$. We can proceed by two ways.
One way, is to prove that the $ L^2$-martingale solution  satisfies equation \eqref{Eq-weak-Solution} with $ \phi \in H^{\delta, q}$.
The second way is to prove the existence of martingale solution in the $ L^2(0, T; L^q).$ The techniques we are using seems equivalent in this task,
hence we will follow the second way.  It is obvious that our technique to prove
the martingale solution in $ L^q-$spaces is still valid for $ q=2$ without any conditions on the space dimension.

\noindent  We prove that the
family of laws of the solutions $ (\mathcal{L}(u_n))_n$ is tight in $
L^2(0, T; \mathbb{L}^{q}(O))$. In fact, the tightness of  $(\mathcal{L}(u_n))_n$ is deduced thanks to Lemmas
 \ref{lem-unif-bound-theta-n-l-2} and  \ref{lem-bounded-W-gamma-p} and  the following embedding
\begin{equation}
 L^2(0, T; H^{\frac\alpha2, 2})\cap W^{\gamma,
2}(0, T; H^{-\delta, q^*}) \hookrightarrow L^2(0, T; \mathbb{L}^{q}),
\end{equation}
provided $ \alpha \geq d(1-\frac2q)$*, see, e.g. \cite[Theorem 2.1]{Flandoli-Gatarek-95},
\cite[Theorem 7.63. p 221 + 7.66, p 222]{Adams-Hedberg-94}.   Therefore, we infer the existence of a probabilistic basis
$ (\Omega^*, \mathcal{F}^*, \mathbb{P}^*, \mathbb{F}^*)$ and an $L^2(0, T; \mathbb{L}^{q})\cap C([0, T]; H^{-\delta, q^*})$-valued random variables
$u^*$  and $u^*_n$ defined on this basis and such that \del{$
\mathcal{L}(u^*) = \mu$,}  $
\mathcal{L}(u_n^*) = \mathcal{L}(u_n) $ and  $u^*_n \rightarrow u^*, \;\; P-a.s. $
in $ L^2(0, T; \mathbb{L}^{q})\cap C([0, T]; H^{-\delta, q^*})$. Moreover, thanks to the above equality in law, the sequence,
$(u_n^*)_n$ satisfies \eqref{Eq-Ito-n-weak-estimation-1}, \eqref{Eq-Ito-n-weak-estimation-Lq} and \eqref{Eq-Ito-n-weak-estimation}.
Hence, we deduce that $u^*_n \rightarrow u^*, \;\; \mathbb{P}^*-a.s. $ weakly in $ L^2(\Omega\times [0, T]; H^{\frac\alpha2, 2})$
\footnote{In fact,$u^*_n \rightarrow u^*$ weakly to other limit which could be easily identified with $u^*$} and
\begin{equation}
 u^*(\cdot, \omega)\in L^2(0, T;H^{\frac\alpha2, 2}(O))\cap  L^\infty(0, T;\mathbb{L}^{q}(O)), \;\;\; P-a.s..
\end{equation}
Using the classical scheme in particular Skorokhod Theorem see \cite{Flandoli-Gatarek-95},  we prove $ u^*$ is a weak solution of the equation \eqref{Main-stoch-eq} in the sense
of Definition \ref{def-variational solution}, i.e.  $ u^*$ satisfies equation \eqref{Eq-weak-Solution}.

\noindent {\bf End of the proof} }


\del{\section{special Section Resolution of an auxiliary problem}\label{sec-2-approxim}

This definition is a simulation of Definition 4.4 in \cite{Daprato-Debussche-Martingale-Pbm-2-3-NS08}
\begin{defn}
We say that an $\mathcal{F}-$adapted process is a local weak (strong) solution for Equation \eqref{Main-stoch-eq}, if  there exists an increasing
sequence of stopping time $ (\tau_N)_N$ such that Equation \eqref{Eq-weak-Solution} is satisfied for all $ t\leq \tau_N$.
\end{defn}

\begin{lem}
 \begin{equation}
  |\rangle B(u-v, v), u-v\langle|\leq c|u-v|_{\mathbb{H}^{\frac\alpha2, 2}}^2 + |u-v|_{\mathbb{L}^2}^2|v|^{\frac{2}{\theta}}_{\mathbb{H}^{\frac\alpha2, q}}
 \end{equation}
with $ \theta := 2+\frac{2d}{\alpha}(\frac{1q} +\frac{1-s}{d})$, s>>>>>>>>>>>
\end{lem}

In this section, we introduce the following auxiliary problem SPDE
\begin{equation}\label{FSB-approx-Xi}
\Bigg\{
\begin{array}{lr}
dv(t)= (-A_\alpha v(t) + B(\xi(t), v(t))dt + G(v(t))\,dW(t), \; 0< t\leq T,\\
v(0) = u_0,
\end{array}
\end{equation}
where $ \xi_t$ is an $ \mathcal{F}-$adapted process satisfying for some $ \delta \geq  \frac{2d-1-\alpha}{4}$,
\begin{equation}
 \mathbb{E}[\sup_{[0, T]}|\xi(t)|_{H^{\delta, 2}}^p]<\infty.
\end{equation}
We will use the monotonicity method to prove the existence of a variational solution to , for more details about this method see e.g.
\cite{Krylov-Rozovski-monotonocity-2007, Pardoux-thesis, Rockner-Pevot-06} and  Appendix \ref{appendix-Monotonocity}. But, we will only give the
following definition of a variational solution of equation \eqref{FSB-approx-Xi}, see \cite[Definition 3.4.]{Krylov-Rozovski-monotonocity-2007}.
Let us give the Gelfant triple
\begin{equation}\label{Gelfant-triple}
 V \hookrightarrow H\tilde{=}  H^*\hookrightarrow  V^*,
\end{equation}
where $ H $ is a Hilbert space.
\begin{defn}\label{Def-solution-variational}
A strong continuous $ L^2-$valued $ \mathcal{F}_t-$adapted process $ (u(t))_{t\in [0, T]}$ is called a solution of equation \eqref{FSB-approx-Xi} iff

\begin{itemize}
 \item $ u \in V $ a.e. in $ (t, \omega)$ and for some given $ p>0$,
\begin{equation}\label{cond-Def-Kry-Rozo}
 \mathbb{E}\left(| u(t)|_{L^2}^p +\int_0^T|u(t)|_V^2dt \right).
\end{equation}
\item there exists a set $ \Omega'\subset \Omega $ of probability one on which for all $ t\in [0, T]$, the following equality holds in $ V^*$,
\begin{equation}\label{Eq-Def-Kry-Rozo}
 u(t)= u_0+ \int_0^t \left(-A_\alpha u(s) + B(\xi(s), u(s)) \right) ds + \int_0^t G(u(s))\,dW(s).
\end{equation}
\end{itemize}
\end{defn}

Let us here make precise the specificity of the application of the monotonicity method on the fractional dissipative problems, relatively
to the classical onces.  It is well known that the main ingredient of this
method is first to consider a  Gelfant triple \eqref{Gelfant-triple},  such that the drift term (let us denote by $ A(t, \omega, v)$) is well
defined and bounded as an operator from $V$ to $ V^*$.  In the case of the second order differential equations (classical case)
(e. g.  $ A= \Delta$ with linearity $ B(v)$ given by the gradient such as in our case, the Gelfant triplet is taken as
 \begin{equation}
 V= D((A^\frac12))=H_0^{1, 2}(\mathbb{T}^d) \hookrightarrow  H:= L^2(\mathbb{T}^d))^d \tilde{=}  H^*\hookrightarrow  V^*:= H^{-1, 2}(\mathbb{T}^d).
 \end{equation}
 It is obvious that $ V\subset D(B) =H^{1, 2}(\mathbb{T}^d)$. Hence,  the restriction of
$ B $ on $ V$ denoted also by $ B$ and the operator $ A$ are both  well defined and bounded from $V$ to $ V^*$.

In the fractional case, we can consider the Gelfant triple
\begin{equation}
 V=D(A_\alpha^\frac12)=H_0^{\frac\alpha2, 2} \hookrightarrow H^{0, 2}\tilde{=}  H^*\hookrightarrow  V^*:= H^{-\frac\alpha2, 2}.
\end{equation}
Hence, it is also obvious that $A_\alpha: V=H_0^{\frac\alpha2, 2} \rightarrow V^*=H^{-\frac\alpha2, 2}$  is bounded, however for $ \alpha<2$,
the space $ V $ is larger than the domain of definition ob $ B$. In particular,
$ H^{1}_0(\mathbb{T}^d) \subsetneq V$. Consequently,  to be able to apply the  monotonicity theorem,
we need to extend uniquely the operator $ B $ to a  bounded operator on $ V$. This will be  the content of the following proposition \del{This will be  the content of the following two propositions.
The first proposition is more relevant for the small regularization effect of the operator $ \mathcal{R}^{\gamma, \sigma}$, i.e. small $ \gamma $,
in particular for $ \gamma=1 $. In the second proposition, we investigate more the large regularization effect}
\begin{prop}
 Let $ d\in \{1, 2, 3\}$ and  $ \alpha > \frac{d+2}{3}$. Then the operator $ B $ is extended uniquely to a bounded operator
(we keep the same notation)  $ B:V \rightarrow V^*$.

\end{prop}

\begin{proof}
For generality reasons, we consider the bilinear form $ B( u, v)$ as defined in \eqref{B-theta1-theta2-divergence}.
Thanks to the fact that $ \sigma \in \sum_d^0$ and Lemma \ref{Lem-classic}, it is easy to deduce that  there exists a constant $ c>0$, s.t.
\begin{equation}
|B(u, v)|_{V^*} \leq \ c\sum_{j =1}^d| \mathcal{R}_j^{\gamma, \sigma}u \cdot  v|_{H^{\frac{2- \alpha}2}}.
\end{equation}
Using the pointwise multiplication in Theorem \ref{theor-1-SR},  Lemma \ref{lem-R-bounded-H-s} and the condition that $ \alpha > \frac{d+2}{3} $,
we  get
\begin{equation}\label{Eq-B-u-v-V-dual}
|B(u, v)|_{V^*} \leq \ c|u |_{H^{\frac{d+2- \alpha}4, 2}} |v|_{H^{\frac{d+2- \alpha}4, 2}}\leq c |u|_V|v|_V<\infty.
\end{equation}

\end{proof}
\begin{remark}
Let us remark that the above method is also applied for $ d>3$, but then $ \alpha >2$. Let us summarize:
\begin{itemize}
\item $ d=1 $ and $ \alpha >1$,
\item $ d=2 $ and $ \frac43\leq  \alpha \leq 2$,
\item $ d=3$ and $ \frac53\leq  \alpha \leq 2$,
\item $ d=4 $ and  $ \alpha = 2$,
\end{itemize}
\end{remark}

Our main result in this section is the following Theorem,
\begin{theorem}\label{Main-theorem-approxi-2}
 Equation \eqref{FSB-approx-Xi} admits a unique solution $ u$ (in the sense of Definition \ref{Def-solution}).
Moreover,  for all $ p>1$, there exists $ c>0$  independent of $ \xi$, such that
\begin{equation}\label{est-solution-variational}
 \mathbb{E}\left(\sup_{[0, T]}| u(t)|_{L^2}^p +\int_0^T|u(t)|_V^2dt \right)\leq c<\infty.
\end{equation}
\end{theorem}

\begin{proof}
\noindent It is easy to see that the coefficients of equation \eqref{FSB-approx-Xi} satisfy the assumptions  $ (H_1)-(H_3)$, with $ p'=2$
and $ \lambda =1$.  We are enable to get  $ (H_4')$. Then we try to get other estimation. Indeed, using the first inequality in
\eqref{Eq-B-u-v-V-dual}  and Young inequality, we get
\begin{eqnarray}
 | A(v)|_{V^*} &\leq & c( | v|_{H^{\frac\alpha2, 2}}+ |\xi|_{H^{\frac{d+2-\alpha}{4}}} | v|_{H^{\frac{d+2-\alpha}{4}}})\nonumber\\
&\leq & c\left( 2|v|_{H^{\frac\alpha2, 2}}+ |v|_{L^2} |\xi_t|_{H^{\frac{d+2-\alpha}{4}}}^{\frac{2\alpha}{3\alpha-(d+2)}} \right).\nonumber\\
\end{eqnarray}
We replace the assumption $ (H_4')$ in Theorem \ref{Teorem-Krylov-Rozovsky} by  the assumption
\begin{itemize}
\item  $(H_4)$ (Boundedness of the growth of $ A(t, .)$),
\begin{equation}
 | A(v)|^2_{V^*}\leq  c(N_t^{2}+ K|v|_V^{2})(1+ |v|_H^{2}),
\end{equation}
\end{itemize}
where here $ N_t := |\xi_t|_{H^{\frac{d+2-\alpha}{4}}}^{\frac{2\alpha}{3\alpha-(d+2)}} \in L^2([0, T]\times\Omega; dt\times P)$.
Using the same proof as in  \cite{Lui-Roekner-nonlocal-monotonicity-10}, we infer the existence of a unique strong solution of equation
\eqref{FSB-approx-Xi} in the sense of definition \ref{Def-solution-variational} satisfying the estimation
\begin{equation}\label{est-solution-variational-c-xi}
 \mathbb{E}\left(\sup_{[0, T]}| v(t)|_{L^2}^p +\int_0^T|v(t)|_V^2dt \right)\leq c_\xi<\infty.
\end{equation}

One also can apply the result in  \cite{Zhang-random-coe-09}, to check all the details.

The last step is to prove that the LHS in \eqref{est-solution-variational-c-xi} is uniformly bounded (independently of $ \xi$). In fact, by
application of Ito formula,
see, e.g. \cite[Theorem 4.2.5]{ Roeckner-Pevot-06},  we get,

\begin{equation}
|u(t)|_{H} ^p= |u_0|_H^p+ 2\int_0^t\left(|u(s)|_H^{p-2}({}_{V^*}\langle -A_\alpha u(s)+ \xi_s\cdots \nabla u(s), u(s)\rangle_V)+ 2\int_0^t
|u(s)|_H^{p-2}||G(u(s))||_Q^2\right)ds + 2 \int_0^t{}_{V^*}\langle  u(s), G(u(s)) dW(s)\rangle_V.
\end{equation}
 Thanks to  property  \eqref{B-v-v-v} (the divergence free  property of $ \xi$), we infer
\begin{equation}
\mathbb{E}\left[\sup_{[0, T]}|u(t)|_{H} ^p +  \mathbb{E}\int_0^T|u(t)|_{H} ^{p-2}| u(s)|^2_Vds \right] \leq C(\mathbb{E}|u_0|_H^2+ 1).
\end{equation}

\end{proof}
  Thanks to our hypothesis $ $
as it is announced above.
The proof is a mixture of the proofs in \cite{Krylov-Rozovski-monotonocity-2007, Lui-Roekner-nonlocal-monotonicity-10, Roeckner-Pevot-06}, hence
some standard facts will be only omitted.

\noindent We use the Galerkin approximation. Let $ P_n$ be the projection of $ V^*$ onto the span $ \{e_1, e_2, \cdots, e_n\}$ and let
$ P_n$ be the projection of $\mathbb{E} $ onto the span $ \{h_1, h_2, \cdots, h_n\}$.
We set $ u^n:= \sum_{j=1}^n{}_{V^*}\langle u^n, e_j\rangle_Ve_j $. Then the following equation admits a unique solution,

\begin{equation}\label{Eq-Def-Kry-Rozo}
 u^n(t)= P_nu_0+ \int_0^t \left(-P_nA_\alpha u^n(s) + P_n(B(\xi(s), u^n(s))) \right) ds + \int_0^t \tilde{P}_nG(u^n(s))\,dW(s).
\end{equation}}}

\appendix
\section{Equivalence between FSNS and SFNS equations.}\label{Appendix-Equivalence}
Recall that we have proved in Section \ref{sec-formulation} that Equation \eqref{Main-stoch-eq} with $ A_\alpha:= (A^S)^{\frac{\alpha}{2}}$ is well defined. This equation can be seen as the fractional version of the stochastic Navier-Stokes equation (FSNSE). 
A stochastic version of the fractional Navier-Stokes equation (SFNSE)
 can also be constructed by taking 
$ A_\alpha:= \Pi(-\Delta)^{\frac{\alpha}{2}}$ on $ \mathbb{L}^q(O)$. 
For simplicity, let us keep in mind for a short time that the two equations, FSNSE and SFNSE, are different. Later on, we shall prove that they are equivalent. By a fractional Navier-Stokes equation (FNSE), we mean Equation 
 \eqref{Eq-classical-SNSE-O-p}, with $ \Delta$ replaced by $ -(-\Delta)^\frac\alpha2$. The SFNSE is a 
stochastic perturbation of FNSE. Thanks to theorems \ref{Prop-1-Laplace-Stokes}-\ref{Lem-semigroup} 
and to the calculus above, the SFNSE is also well defined.\del{The main question now is whether or not the two equations FSNSE and SFNSE are equivalent. Let us  before moving to this last question,  
clarify some practical and theoretical matters for the two versions. 
Indeed,} As the derivation of equations describing physical phenomena is mainely based on the deterministic case it is intuitively seen that the stochastic version of the fractional Navier-Stokes equation is more suitable for physical modeling, rather than the fractional version of the stochastic Navier-Stokes equation, see e.g. \cite{Caffarelli-2009, SugKak, Sug, Sug-Frac-cal-89}. As we shall prove the equivalence of the two versions, we conclude that the FSNSE, seems intuitively more theoretical, is  of practical interest as well.

\noindent  In the case $ O= \mathbb{T}^d$, the operators  $ \Delta $ and $ div$ are commuting. 
Therefore, the Stokes operator $A^S$
is minus the Laplacian $ \Delta$ on $ \mathbb{L}^q( \mathbb{T}^d)$, see e.g.
\cite{Gigaweak-strong83},
\cite[p. 48]{Foias-book-2001}, \cite[p. 9]{Temam-NS-Functional-95} and \cite[p 105]{Temam-Inf-dim-88} for the torus,
 \cite{Kato-Ponce-86} for $ O=\mathbb{R}^d$  and see also a direct  proof in Section \ref{sec-Torus}. Combining this statement with the results in Theorem \ref{theorem-domains-A-D-A-S}, we conclude that  
\begin{eqnarray}\label{eq-A-S-alpha-Delta-alpha}
D((A^S)^\frac\alpha2) &=& D(\Pi_q(-\Delta)^\frac\alpha2\Pi_q) = H_d^{\alpha, q}(\mathbb{T}^d)\cap \mathbb{L}^q(\mathbb{T}^d),\nonumber\\
(A^S)^\frac\alpha2 u &=& \Pi(-\Delta)^\frac\alpha2 u= (-\Delta)^\frac\alpha2 u, \;\; \forall u \in D((A^S)^\frac\alpha2).
\end{eqnarray}
This proves that the FSNSE and the SFNSE  defined on the torus are ''equivalent''. 
In the case $ O\subset \mathbb{R}^d$  being bounded, the Stokes operator $ A^S$ is not equal to $ -\Delta$.  In fact,
 as  we can not in general expect that if $ u \in D(A^S)$ we also have
$ \Delta u\cdot \vec{n} =0 $ on $\partial O$ it is not obvious whether or not $ \Delta u \in \mathbb{L}^q(O)$. Our claim here is that $ (A^S)^\frac\alpha2 =  \Pi(A^D)^\frac\alpha2 \Pi$.
In deed, thanks to \eqref{def-A-D} and \eqref{def-A-S}, it is easy to deduce that 
 $ A^S = \Pi A^D\Pi$, see also \cite{Fujiwara-Morimoto-L-r-Helmholtz-decomposition-77}.\del{In fact, this latter is a simple 
consequence of \eqref{def-A-D} and \eqref{def-A-S}.
using \eqref{def-A-D} and \eqref{def-A-S}, we get
\begin{eqnarray}
 D(\Pi A^D\Pi)&=& D(A^D)\cap \mathbb{L}^q(O) = D(A^S)\nonumber\\
\Pi A^D u &=& -\Pi \Delta u = A^Su, \;\;\;  \forall u \in  D(\Pi A^D)\cap \mathbb{L}^q(O).
\end{eqnarray}
In the second step, we prove that $ (A^S)^\frac\alpha2 =  \Pi(A^D)^\frac\alpha2 \Pi$.}
Using Theorem \ref{theorem-domains-A-D-A-S}, we infer that
\begin{equation}
 D(\Pi (A^D)^\frac\alpha2\Pi)=  D((A^D)^\frac\alpha2)\cap \mathbb{L}^q(O) = D((A^S)^\frac\alpha2).
\end{equation}
\del{Recall the following definition of the negative power of $ A$, with  
$ A$ could be either $ A^S$ or $ A^D$,  see e.g. 
\cite{Giga-Doamian-fract-Stokes-Laplace, Pazy-83}, 
\begin{equation}\label{negative-fractional-A}
 A^{-\frac\alpha2} u = \frac1{2\pi i}\int_\Gamma z^{-\frac\alpha2}(A-zI_{L_d^q})^{-1} u dz, \; \forall u \in \mathbb{L}^q(O).
\end{equation}
where  $ \Gamma $ is the path running the resolvent set from $ \infty e^{-i\theta}$ to  $ \infty e^{i\theta}$, $ <0\theta <\pi$,  avoiding the 
negative real axis and the origin and such that the branch $  z^{-\frac\alpha2}$ is taken to be positive for real for real positive values of $z$ and $ I_X$ is the identity on the space $ X$.
The integral in the RHS of \eqref{negative-fractional-A-D} converges in the uniform operator topology.}
Moreover, using the definition of the negative power of $ A^S$ and $ A^D$ via the resolvent, see e.g. 
\cite{Giga-Doamian-fract-Stokes-Laplace, Pazy-83} and the definition of the Helmholtz projection, we infer that 
\begin{equation}\label{negative-fractional-A-D}
 (A^D)^{-\frac\alpha2} \Pi^{-1}u = \frac1{2\pi i}
 \int_\Gamma z^{-\frac\alpha2}(A^D-zI_{L_d^q})^{-1}\Pi^{-1}u dz,\;\;\; \forall u \in \mathbb{L}^q(O),
\end{equation}
where  $ \Gamma $ is the path running the resolvent set from $ \infty e^{-i\theta}$ to  
$ \infty e^{i\theta}$, $ 0<\theta <\pi$,  avoiding the 
negative real axis and the origin and such that the branch $  z^{-\frac\alpha2}$ is taken to be positive
for real  positive values of $z$ and $ I_X$ is the identity on the space $ X$.
The integral in the RHS of \eqref{negative-fractional-A-D} 
converges in the uniform operator topology. Furthermore, we have for all $ u \in \mathbb{L}^q(O)$,
\begin{eqnarray}\label{negative-fractional-A-D-last}
 (A^D)^{-\frac\alpha2} \Pi^{-1}u &=& \frac1{2\pi i}\int_\Gamma z^{-\frac\alpha2}(\Pi A^D-z\Pi)^{-1}u dz
 =\frac1{2\pi i}\int_\Gamma z^{-\frac\alpha2}(A^S-zI_{\mathbb{L}^q})^{-1}u dz
= :(A^S)^{-\frac\alpha2}.\nonumber\\
\end{eqnarray}
As the operators $(A^S)^{-1}$ and $(A^D)^{-1}$ are one-to-one this achieved the proof of the equivalence between the FSNSE and the 
SFNSE.

\section{Local  mild solution of the multidimensional FSNSE.}\label{appendix-local-solution}

\noindent\del{ In this Section, we prove Theorem \ref{Main-theorem-mild-solution-d}.} Let us first recall the following result,
\begin{lem}\label{lem-Lipschitz-pi-n}
 \noindent  Let $ n$ be fixed and let $ X$ be a Banach space.
We define the application $ \pi_n: X \rightarrow  B(0, n)\subset X$, by
\begin{equation}\label{Eq-Pi-n-X}
\pi_n(x)=
\Bigg\{
\begin{array}{lr}
x, \;\;  |x|_{X}\leq n,\\
nx|x|^{-1}_{X}, \;\;  |x|_{X}> n.
\end{array}
\end{equation}
\noindent Then $  \pi_n$ satisfies the following estimates,
\begin{equation}\label{inea-Pi-n}
 |\pi_n x- \pi_n y|_{X}\leq 2 |x-y|_{X}\;\;\; \text{and} \;\;\;
 |\pi_n x|_{X}\leq  \min\{n, |x|_{X}\}.
\end{equation}
\end{lem}
\noindent  For $ n\in \mathbb{N}_0$, we introduce the following sequence of equations
\begin{equation}\label{Eq-approx-n}
\Bigg\{
\begin{array}{lr}
 du_n= \left(-
A_{\alpha}u_n(t) + B(\pi_nu_n(t))\right)dt+ G(\pi_n(u_n(t)))dW(t), \; 0< t\leq T,\\
u_n(0)= u_0.
\end{array}
\end{equation}
\noindent For the fixed parameters $ 2< p <\infty$, $ 2< q\leq q_0$ and $ T_1\leq T$ (for simplicity reasons, we omit the subscription of all  these parameters),  we define the spaces $ \mathcal{E}_T$ and $ E_T$ as in \eqref{E-cal-T} and \eqref{E-T} respectively with $ X:= D(A_q^{\frac\delta2})$. First let us mention that, using  Formula \eqref{construction-of-fract-bounded} and the fact that $ div \pi_m u=0$, it is easy to see that  $ \pi_n$ is well defined\del{Recall that as we have seen in Section \ref{sec-Domain}, $ \pi_n$ is well defined} on $D(A_q^{\frac\delta2})$. Furthermore, we define the map  $ \Phi_{n}$ on $ \mathcal{E}_T$ by\del{ We denote by $ \mathcal{M}_{q, \delta, p}$
the space of adapted  stochastic
processes with values in $ D(A_q^{\frac {\delta}2})$,  (later, we use for simplicity, the
notation $\mathcal{M}_{\delta} $),  endowed  by the norm
\begin{equation}\label{norm-M}
|\theta |_{\mathcal{M}_{\delta}}:= \left(\mathbb{E}\sup_{[0,T_1]}|\theta(t)|^p_{D(A_q^{\frac{\delta}2})}\right)^\frac 1p<\infty.
\end{equation}}

\begin{equation}\label{Eq-def-Phi}
\Phi_{n}(u)(t) := e^{-A_\alpha t}u_0 + \int_0^te^{-A_\alpha (t-s)}B(\pi_n(u(s)))ds + \int_0^te^{-A_\alpha (t-s)}
G(\pi_n(u(s)))W(ds).
\end{equation}
\del{The aim now  is to prove that the map $ \Phi_{n} $ is a contraction on $ \mathcal{E}_{T}$ and than we apply the
fixed point Theorem. The claim is}

\noindent {\bf Claim}
\noindent Under the conditions of Theorem \ref{Main-theorem-mild-solution-d} the map $\Phi_{n}$  is well defined and it is a
contraction on $ \mathcal{E}_{T}$.

\noindent {\bf Proof of the claim}
Let us denote by $ I_0,\; I_1(u), \; I_2(u)$ respectively the terms in the RHS of \eqref{Eq-def-Phi} and prove that $I_j(u)
\in \mathcal{E}_{T}$, for $ j \in\{0, 1, 2\}$.
In fact, $ I_0 \in \mathcal{E}_{T} $ thanks to \eqref{Eq-initial-cond} and to the boundedness of the operators of the semigroup. From Lemma \ref{lem-est-z-t} and the second
estimate in \eqref{inea-Pi-n}, we conclude that
$ I_2(u)  \in \mathcal{E}_{T}$.  In particular, there exists $ c_n >0$, such that
\begin{eqnarray}\label{est-2-term-B-0-2}
|I_2 (u)|_{\mathcal{E}_{T_1}} & \leq & c_nT_1 (1+|u|_{\mathcal{E}_{T_1}}).
\end{eqnarray}
Moreover, thanks to Proposition \ref{Prop-Main-I}, with $ \beta =\eta=\delta$ and
the second estimate in  \eqref{inea-Pi-n}, we infer the existence of $ c, \mu>0$ such that
\begin{eqnarray}\label{est-2-term-B-0-1}
|I_1(u)|_{\mathcal{E}_{T_1}} &\leq& CT_1^{\mu}\left(\mathbb{E}\sup_{[0, T_1]}|\pi_nu(s)|^{2p}_{D(A_q^\frac{\delta}2)} \right)^\frac1p
\leq  ncT_1^{\mu}|u|_{\mathcal{E}_{T_1}}<\infty. \nonumber\\
\end{eqnarray}
The $ D(A^\frac\delta2)$-time continuity of $ \Phi_{n}(u)(\cdot) $ (i.e. $\Phi_{n}(u) \in C([0, T]; D(A^\frac\delta2))$)  follows from the continuity of the terms $ I_j, \; j\in \{0,1,2\}$ in \eqref{Eq-def-Phi}. In fact, the continuity of these latter \del{$ I_0 \& I_2$} are consequences of the regularization effect of the semi group respectively of Proposition \ref{Prop-Main-I}, with $ \beta =\eta=\delta$, \eqref{inea-Pi-n} and \cite[Lemma 2]{Daprato-Kwapien-Zab-Regu-Sol-88} respectively of Lemma \ref{lem-est-z-t}. \del{To show the continuity of the term $ I_1$, we follow the proof of \cite[Lemma 2]{Daprato-Kwapien-Zab-Regu-Sol-88}, use}Thus $ \Phi_{n} $ is well defined. To prove that $ \Phi_{n} $ is a contraction, we first remark that
\begin{equation}\label{formula-B-v1-B-v2}
 B(u_1, u_1) - B(u_2, u_2) = B(u_1, u_1-u_2)+ B(u_1-u_2, u_2).
\end{equation}
Thanks to Proposition \ref{Prop-Main-I}, with $ \beta =\eta=\delta$ and the first estimate in \eqref{inea-Pi-n}, we get\del{ ( see also a similar calculus in the proof of Theorem
\ref{prop-global-mild-solution} bellow)}
\begin{equation}\label{eq-lipschitz-B}
\Big(\mathbb{E} \sup_{[0, T_1]}|\int_0^te^{-A_\alpha (t-s)}\big(B(\pi_n(u_1(s)))-B(\pi_n(u_2(s)))\big)
ds|^p_{D(A_q^{\frac\delta2})}\Big)^\frac1p\leq cnT_1^{1-\frac1\alpha(1+\frac dq)}
|u_1-u_2|_{\mathcal{E}_{T_1}}.
\end{equation}
Using Assumption $(\mathcal{C})$, \cite[Proposition 4.2]{Neerven-Evolution-Eq-08}, the calculus as in Lemma \ref{lem-est-z-t}
 and the fact that $ \pi_n$ is globally Lipschitz, we get
\begin{equation}\label{eq-lipschitz-G}
\Big(\mathbb{E} \sup_{[0, T_1]}|\int_0^te^{-A_\alpha (t-s)}\big(G(\pi_n(u_1(s))) - G(\pi_n(u_2(s)))\big)W(ds)|^p_{D(A_q^{\frac\delta2})}\Big)^\frac1p
\leq c_n T_1|u_1-u_2|_{\mathcal{E}_{T_1}}.
\end{equation}
Consequently, thanks to \eqref{eq-lipschitz-B} and \eqref{eq-lipschitz-G} we infer that
\begin{equation}
 | \Phi_{n}(u_1) -\Phi_{n}(u_2)|_{\mathcal{E}_{T_1}}\leq C_nT_1^\mu | u_1 -u_2|_{\mathcal{E}_{T_1}},\;\; \mu>0.
\end{equation}
The time interval $ [0, T_1]$ is chosen such that $ C_nT_1^\mu <1$. Therefore by application of the fixed point theorem, we infer the existence
 of $ (u_n(t), t\in [0, T_1]) \in \mathcal{E}_{T}$  solution of \eqref{Eq-approx-n}, for all $ t \leq T_1$. Using the semigroup property of
$ (e^{-tA_\alpha}, t \in [0, T])$, we extend $  (u_n(t), t\in [0, T_1]) $ to be defined on $ [0, T]$.
\noindent We define the following stopping time, see Appendix \ref{append-stop-time} for the proof,
\begin{equation}\label{Eq-def-tau-n-delta}
 \tau_n := \inf\{t\in (0, T), \; s.t. |u_n(t) |_{D(A^{\frac\delta2})}> n\}\wedge T,
\end{equation}
with the understanding that $ \inf(\emptyset)=+\infty$.\del{and prove that $ \tau_n $ is a stopping time. Indeed, it is easy to check,  using Lemma \ref{lem-est-z-t}, Estimate \eqref{Eq-Pi-n-X} and a similar calculus as in \eqref{est-1-semi-group-second-term-1}, that the real  
$ \mathcal{F}_t-$adapted stochastic processes $ (\langle u_n(s), e_k\rangle_{\mathbb{L}^2})_{k\in\Sigma}$ are continuous for all $ k\in \Sigma$, see also Remark \ref{Rem-1}.
Therefore, \del{ the positive $ \mathcal{F}_t-$adapted stochastic real processes
\begin{equation}
X_j(\omega, s):= |\sum_{k\in \Sigma_j}|k|^{\delta} \langle u_n(\omega, s), e_k\rangle_{\mathbb{L}^2}e_k|_{\mathbb{L}^q},
\end{equation} }the positive $ \mathcal{F}_t-$adapted stochastic real processes
\begin{equation}
X_j(\omega, s):= |\sum_{k\in \Sigma_j}|k|^{\delta} \langle u_n(\omega, s), e_k\rangle_{\mathbb{L}^2}e_k|_{\mathbb{L}^q},
\end{equation} 
are also continuous for all $ j\in \mathbb{N}_1$, with $ (\Sigma_j)_j\subset \Sigma$  being an increasing sequence of finite subsets converging to $ \Sigma$. Therefore, see e.g. \cite[Propositions 4.5 or 4.6 ]{Revuz-Yor}, (Recall that we have assumed that the filtration $ (\mathcal{F}_t)_{t\in[0, T]}$ is right continuous), \del{therefore an optional time is also a stopping time.}  
\del{\begin{eqnarray}
\{\tau_n> t\} &=& \cap_{s\leq t}\{\omega, |u_n(s, \omega)|_{D(A_q^\frac\delta2)}\leq n\}\nonumber\\
&=& \cap_{s\leq t}\{\omega, |\sum_{k\in \Sigma}|k|^{\delta} \langle u_n(\omega, s), e_k\rangle_{\mathbb{L}^2}e_k|_{\mathbb{L}^q}\leq n\}
\nonumber\\
&=& \cap_{s\leq t}\cup_{m\in\mathbb{N}}\cap_{j\geq m}\{\omega, |\sum_{k\in \Sigma_j}|k|^{\delta} \langle u_n(\omega, s), e_k\rangle_{\mathbb{L}^2}e_k|_{\mathbb{L}^q}\leq n\}.
\end{eqnarray}}
\begin{eqnarray}
\{\tau_n\leq t\} &=& \cup_{s\leq t}\{\omega, |u_n(s, \omega)|_{D(A_q^\frac\delta2)}\geq n\}\nonumber\\
&=& \cup_{s\leq t}\{\omega, |\sum_{k\in \Sigma}|k|^{\delta} \langle u_n(\omega, s), e_k\rangle_{\mathbb{L}^2}e_k|_{\mathbb{L}^q}\geq n\}
\nonumber\\
&=& \cup_{s\leq t}\cap_{m\in\mathbb{N}}\cup_{j\geq m}\{\omega, |\sum_{k\in \Sigma_j}|k|^{\delta} \langle u_n(\omega, s), e_k\rangle_{\mathbb{L}^2}e_k|_{\mathbb{L}^q}\geq n\}
\nonumber\\
&=& \cup_{s\leq t}\cap_{m\in\mathbb{N}}\cup_{j\geq m}\{\omega, X_j(\omega, s):= \sum_{k\in \Sigma_j}|k|^{\delta} \langle u_n(\omega, s), e_k\rangle_{\mathbb{L}^2}e_k \in (B_{\mathbb{L}^q}(0, n))^c\}
\nonumber\\
&=& \cup_{s\in \mathbb{Q} \&\leq t}\cap_{m\in\mathbb{N}}\cup_{j\geq m}\{\omega, X_j(\omega, s) \in (B_{\mathbb{L}^q}(0, n))^c\} \in \mathcal{F}_t.
\end{eqnarray}
\del{Thanks to the continuity of the positive $ \mathcal{F}_t-$adapted process $X_j(\omega, s):=  \sum_{k\in \Sigma_j}|k|^{\delta} \langle u_n(\omega, s), e_k\rangle_{\mathbb{L}^2}e_k|_{\mathbb{L}^q}$ 
As the filtration $\mathcal{F}_t$ is right continuous, it is sufficient to prove that for all $ t\in [0, T]$, $ \{\tau_n\geq t\}\in \mathcal{F}_t$.}} The sequence $ (\tau_n)_n $ increases, hence there exists a stopping time $ \tau_\infty$ such that
\begin{equation}\label{Eq-def-tau-delta}
 \tau_\infty = \lim_{n\nearrow +\infty}\tau_n\leq T.
\end{equation}
Thanks to the uniqueness of the fixed point in $ \mathcal{E}_{T}$, we can construct $ (u(t), t< \tau_\infty)$,
by $ u(t)= u_n(t), \; \forall t < \tau_n$.
Then   $ (u(t), \; t < \tau_\infty)$  is a mild solution of Equation \eqref{Main-stoch-eq} up to $ \tau_\infty$. Moreover, as  $(u(t\wedge \tau_n), t\in [o, T])$ is an increasing sequence, then  $ (u(t), \; t < \tau_\infty)$ is maximal local mild solution.
\del{The estimate \eqref{eq-cond-norm-mild-solution} follows from the construction of the solution as an element of $ \mathcal{E}_{T}$ with $ X= D(A_q^\frac\delta2)$. The regularity of the trajectories in \eqref{eq-cond-cont-mild-solution} is deduced,
by a standard way, thanks to Lemma \ref{lem-est-z-t}, Proposition \ref{Prop-Main-I}, Assumption \eqref{Eq-initial-cond} and the strong continuity
of the semigroup $ (e^{-tA^{\frac{\alpha}{2}}})_{t\geq 0}$ with $ X_1:= H^{-\delta'', q}(O)$ and $ \delta'' \geq \alpha+1+\frac{d}{q}-\delta$.}

\del{\section{Passage Velocity-Vorticity forms of 2D-FSNSEs.}\label{sec-Passage Velocity-Vorticity}}
\del{We define the operator "curl" as follow
\begin{eqnarray}
 curl: v\in H_2^{s, q}(O) \rightarrow  H_1^{s-1, q}(O) \ni curlv:= \partial_1v_2-  \partial_2v_1,
\end{eqnarray}
with $ O $ being either $ \mathbb{T}^2$ or a bounded domain from $ \mathbb{R}^2$, $ 1<q<\infty$ and $ s\in \mathbb{R}$. 
The following result characterizes an intrinsic property between $ curlv$ and the Sobolev regularity of $ v$, 
\begin{lem}\label{lem-basic-curl-gradient}
 Let  $ \mathbb{R}^2$, $ 1<q<\infty$ and $ s\in \mathbb{R}$. Then there exists a constant $ c>0$ such that for all $ v \in H_2^{s+1, q}(O)$
\begin{equation}
 |\nabla v|_{H_2^{s, q}}\leq c |curl v|_{H_1^{s, q}}.
\end{equation}
\end{lem}
\begin{proof}
Let \begin{equation}
     \zeta_{jk}:= \partial_jv_k- \partial_kv_j.
    \end{equation}
By a straightforward calculus, we get
\begin{equation}
     \partial_jv_k=  \partial_j \partial_i\Delta^{-1}\zeta_{ik}.
    \end{equation}
For $ O=\mathbb{T}^2$, it is easy to see, using Marcinkiewicz's theorem, that the  pseudodifferential operators
operator $  \partial_j \partial_i\Delta^{-1}$, with symbol
$ k_ik_j|k|^{-2}$ are bounded  on 
$ H^{s, q}(\mathbb{T}^2)_1$, $ 1<q<\infty $ and $ s\in \mathbb{R}$.  Therefore,
\begin{eqnarray}
 | \nabla v|_{H_2^{s, q}}&\leq& c|\partial_jv_{k}|_{H_1^{s, q}} \leq c|\partial_j \partial_i\Delta^{-1}\zeta_{ik}|_{H_1^{s, q}}\nonumber\\
&\leq& c|\zeta_{ik}|_{H_1^{s, q}} \leq c |curl v|_{H_1^{s, q}}.
\end{eqnarray}
 See also \cite[Lemma 3.1]{Kato-Ponce-86} for the proof for $ O=\mathbb{R}^2$.
\end{proof}}
\del{ Later on, we unify the representation of $ u $ and $ \theta $ as follow
\begin{eqnarray}
\mathcal{R}^{1}:H_1^{s, q}(O) &\rightarrow & \mathbb{H}^{s, q}(O)\nonumber\\
\theta &\mapsto& \mathcal{R}^{1}\theta:= \nabla^\perp\cdot \Delta^{-1}\theta = \int_O \nabla^\perp_xg_{\mathbb{T}^2}(\cdot, y)\theta(t, y)dy.
\end{eqnarray}
Let us now separate the discussion of the cases $ O$ bounded and  $ O=\mathbb{T}^2$.
{\bf The case $ O=\mathbb{T}^2$.} 
In this case, the 
\begin{equation}
\left\{
 \begin{array}{lr}
  \Delta \psi = \theta
 \end{array}
\right.
\end{equation}
\begin{equation}\label{eq-def-u-theta}
u = \mathcal{R}^{1} \theta.
\end{equation}
{\bf The case $ O\subset \mathbb{R}^2$ bounded.} 
\begin{equation}
\left\{
 \begin{array}{lr}
  \Delta \psi = \theta
 \end{array}
\right.
\end{equation}
The following result characterizes an intrinsic property between $ curlv$ and the Sobolev regularity of $ v$, 
\begin{lem}\label{lem-basic-curl-gradient}
 Let  $ \mathbb{R}^2$, $ 1<q<\infty$ and $ s\in \mathbb{R}$. Then there exists a constant $ c>0$ such that for all $ v \in H_2^{s+1, q}(O)$
\begin{equation}
 |\nabla v|_{H_2^{s, q}}\leq c |curl v|_{H_1^{s, q}}.
\end{equation}
\end{lem}
\begin{proof}
Let \begin{equation}
     \zeta_{jk}:= \partial_jv_k- \partial_kv_j.
    \end{equation}
By a straightforward calculus, we get
\begin{equation}
     \partial_jv_k=  \partial_j \partial_i\Delta^{-1}\zeta_{ik}.
    \end{equation}
For $ O=\mathbb{T}^2$, it is easy to see, using Marcinkiewicz's theorem, that the  pseudodifferential operators
operator $  \partial_j \partial_i\Delta^{-1}$, with symbol
$ k_ik_j|k|^{-2}$ are bounded  on 
$ H^{s, q}(\mathbb{T}^2)_1$, $ 1<q<\infty $ and $ s\in \mathbb{R}$.  Therefore,
\begin{eqnarray}
 | \nabla v|_{H_2^{s, q}}&\leq& c|\partial_jv_{k}|_{H_1^{s, q}} \leq c|\partial_j \partial_i\Delta^{-1}\zeta_{ik}|_{H_1^{s, q}}\nonumber\\
&\leq& c|\zeta_{ik}|_{H_1^{s, q}} \leq c |curl v|_{H_1^{s, q}}.
\end{eqnarray}
 See also 
 \cite[Lemma 3.1]{Kato-Ponce-86} for the proof for $ O=\mathbb{R}^2$.
\end{proof}

In this section we consider only the case $ d=2$. \del{The passage from velocity to vorticity forms and vise versa is done in 4 steps. 

{\bf Step 1. General framework.} } We defined the operator "curl" as follow
\del{\begin{eqnarray}
 curl: v\in \mathbb{H}^{s, q}(O) \rightarrow  H_1^{s-1, q}(O) \ni \theta = curlv:= \partial_1v_2-  \partial_2v_1,
\end{eqnarray}}
\begin{eqnarray}
 curl: v\in H_2^{s, q}(O) \rightarrow  H_1^{s-1, q}(O) \ni \theta = curlv:= \partial_1v_2-  \partial_2v_1,
\end{eqnarray}
with $ O $ being either $ \mathbb{T}^2$ or a bounded domain from $ \mathbb{R}^2$, $ 1<q<\infty$ and $ s\in \mathbb{R}$.
To recuperate the  velocity $ u$ of the fluid once the vorticity $ \theta$ is known, we introduce the stream function $ \psi$,
which is the solution of the 
Poisson equation endowed with relevant boundary conditions in the case $ O\subset \mathbb{R}^2$ being bounded. In particular, 
for  $ O =\mathbb{T}^2$, the problem is formulated as follow,
\begin{equation}\label{recuper-eq-u-theta-Torus}
\left\{
 \begin{array}{lr}
 \Delta \psi = \theta,\\
u= \nabla^\perp \psi, \;\;\; \text{and} \;\; \nabla^\perp:=(-\frac{\partial }{\partial x_2}, \frac{\partial
}{\partial x_1}).
 \end{array}
\right.
\end{equation}
\del{\begin{equation}\label{recuper-eq-u-theta}
\left\{
 \begin{array}{lr}
 \Delta \psi = \theta,\\
\psi/\partial O =0,\\
u= \nabla^\perp \psi, \;\;\; \text{and} \;\; \nabla^\perp:=(-\frac{\partial }{\partial x_2}, \frac{\partial
}{\partial x_1}).
 \end{array}
\right.
\end{equation}
\begin{equation}\label{psi-theta-u}
 \Delta \psi = \theta, \;\;\; \text{and} \;\; u= \nabla^\perp \psi, 
\end{equation}}
Thanks to the vanishing average condition, Poisson equation is well posed. The velocity $ u $ is then obtained by a direct calculus using the 
second equation in \eqref{recuper-eq-u-theta-Torus}.
For  $ O\subset \mathbb{R}^2$ bounded, we consider \eqref{recuper-eq-u-theta-Torus} and we conclude from Dirichlet boundary conditions of the 
velocity $ u$ that 
$ \psi$ should satisfy vanishing Neumann boundary conditions. Therefore, $ \psi = const.$ on $ \partial O$. We suppose that this constant is null, 
for further discussion, see e.g. \cite{Marchioro-Puvirenti-Vortex-84}. Let us denote by $ A_1$ either the Laplacian on $ O \subset \mathbb{R}^2 $ 
with Dirichlet boundary conditions or  a component of Stokes operator in the case $O= \mathbb{T}^2$,  then we formulate the recuperation  problem 
for both cases as 
\begin{equation}\label{recuper-eq-u-theta}
\left\{
\begin{array}{lr}
 A_1 \psi = \theta,\\
u= \nabla^\perp \psi.
\end{array}
\right.
\end{equation}
The problem, as mentioned before, is well posed, see also Section \ref{sec-formulation} and we have
\begin{equation}\label{repres-u-theta}
 u(t, x) = \nabla^\perp A_1^{-1}\theta(t, x)= \int_O \nabla^\perp_xg_O(x, y)\theta(t, y)dy,
\end{equation}
where $ g_O$ is the Green function corresponding to the Poisson equation with Dirichlet boundary conditions. 
In the case $  O= \mathbb{T}^2 $, the Green function  $g_{\mathbb{T}^2}$ is explicitly given by
\begin{equation}
 g_{\mathbb{T}^2}(x, y):= -\frac1{(2\pi)^2}\sum_{k\in\mathbb{Z}^2_0}\frac{1}{|k|^2}e^{k\cdot(x-y)},\;\; x, y \in \mathbb{T}^2.
\end{equation}
Moreover,  the operator 
\del{\begin{eqnarray}
\mathcal{R}^{1}:H_1^{s, q}(\mathbb{T}^2) &\rightarrow & H_2^{s, q}(\mathbb{T}^2)\nonumber\\
\theta &\mapsto& \mathcal{R}^{1}\theta:= \nabla^\perp\cdot \Delta^{-1}\theta
\end{eqnarray}}
\begin{eqnarray}
\mathcal{R}^{1}:H_1^{s, q}(O) &\rightarrow & H_2^{s+1, q}(O)\nonumber\\
\theta &\mapsto& \mathcal{R}^{1}\theta:= u= \nabla^\perp\cdot A_1^{-1}\theta = \int_O \nabla^\perp_{\cdot}g_{O}(\cdot, y)\theta(y)dy.
\end{eqnarray}
is well defined and bounded for all $ 1<q<\infty$ and $ s\in \mathbb{R}$. In fact,
\begin{eqnarray}
\del{ |u|_{\mathbb{H}^{s+1, q}}&\leq& }|u|_{H_2^{s+1, q}} &\leq& c | \nabla^\perp A_1^{-1}\theta|_{H_2^{s+1, q}}
\leq c|\partial_j A_1^{-1}\theta|_{H_2^{s+1, q}} \leq c|A_1^{-1}\theta|_{H_2^{s+2, q}}\leq c|\theta|_{H_2^{s, q}}.
\end{eqnarray}
The proof of this statement 
for a more large class of operators on  $\mathbb{T}^d, d\in \mathbb{N}_0$ which includes 
 $ \mathcal{R}^{1}$,  can be found in \cite{Debbi-scalar-active} and for the case $ O=\mathbb{R}^2$, one can  see \cite[Lemma 3.1]{Kato-Ponce-86}.
 For the convenience of the reader, let us mention here that  $ \mathcal{R}^{1} $ is a 
pseudoodifferential operator of Calderon-Zygmund Reisz type. Otherwise, we can rewrite  
$\mathcal{R}^{1} = -\mathcal{R}^\perp (-\Delta)^{-\frac12}$, where $ \mathcal{R}$ is Riesz transform and $ v^\perp := (-v_2, v_1)$. 
 One can also use the representation of $ u$ via Green function (recall that $ \Delta_xg_O(x, y)= \delta_x(y)$). \\}
\section{The Biot-Savart's law and the corresponding fractional stochastic vorticity equation.}\label{sec-Passage Velocity-Vorticity}
In this appendix we consider only the case $ d=2$, for the multidimensional case see e.g. \cite[Chap.3]{Chemin-Book-98}, \cite[Chap.2]{Majda-Bertozzi-02}, \cite{Marchioro-Puvirenti-Vortex-84} and \cite{Mikulevicius-H1-NS-solution-2004}. The Biot-Savart law determines the  velocity $ u$ from the vorticity $ \theta$. This law is given as a pseudo-differential operator of order $-1$ in the cases $ O=\mathbb{R}^d$ and $ O=\mathbb{T}^d$. The case $ O\subsetneq\mathbb{R}^d$, as mentioned before is much involved. here we give a survey and some results about this law in the cases  $ O\subsetneq\mathbb{R}^d$ and $ O=\mathbb{T}^d$ than we move to the derivation of the stochastic vorticity equation for the case $ O=\mathbb{T}^d$. A generalization of the Biot-Savart's law to a nonlocal pseudo-differential operators of fractional order $ \gamma \leq 0$ has been investigated in \cite{Debbi-scalar-active}.\del{The passage from velocity to vorticity forms and vise versa is done in 4 steps. 

{\bf Step 1. General framework.} } We define the operator "curl" as follow, see Preliminary Notations,
\del{\begin{eqnarray}
 curl: v\in \mathbb{H}^{s, q}(O) \rightarrow  H_1^{s-1, q}(O) \ni \theta = curlv:= \partial_1v_2-  \partial_2v_1,\;\; 1<q<\infty\;\; s\in \mathbb{R}.
\end{eqnarray}}
\begin{eqnarray}
 curl: v\in H_2^{\beta, q}(O) \rightarrow  H_1^{\beta-1, q}(O) \ni \theta = curlv:= \partial_1v_2-  \partial_2v_1,\; \beta\in \mathbb{R}, \; 1<q<\infty.\nonumber\\
\end{eqnarray}
\del{with $ O $ being either $ \mathbb{T}^2$ or a bounded domain from $ \mathbb{R}^2$, $ 1<q<\infty$ and $ s\in \mathbb{R}$.
To recuperate the  velocity $ u$ of the fluid once the vorticity $ \theta$ is known, }We introduce the stream function $ \psi$,
 as the solution of the 
Poisson equation endowed with a relevant boundary condition in the case $ O\subset \mathbb{R}^2$ being bounded. In deed,
we conclude from Dirichlet boundary condition of the 
velocity $ u$ and the third equation in \eqref{recuper-eq-u-theta-Torus} bellow, that 
$ \psi$ should satisfy vanishing Neumann boundary condition. Therefore, $ \psi/ \partial O= const.$\del{ on $ \partial O$.} We suppose that this constant is null, 
for further discussion, see e.g. \cite{ Marchioro-Puvirenti-Vortex-84}. The problem is then formulated as follow,
\begin{equation}\label{recuper-eq-u-theta-Torus}
\left\{
 \begin{array}{lr}
 \Delta \psi = \theta,\\
\psi/\partial O =0,\\
u= \nabla^\perp \psi, \;\;\; \text{and} \;\; \nabla^\perp:=(-\frac{\partial }{\partial x_2}, \frac{\partial
}{\partial x_1}).
 \end{array}
\right.
\end{equation}
The formulation  \eqref{recuper-eq-u-theta-Torus} is still valid for  $O= \mathbb{T}^2$ without the boundary condition.
\del{\begin{equation}\label{recuper-eq-u-theta}
\left\{
 \begin{array}{lr}
 \Delta \psi = \theta,\\
\psi/\partial O =0,\\
u= \nabla^\perp \psi, \;\;\; \text{and} \;\; \nabla^\perp:=(-\frac{\partial }{\partial x_2}, \frac{\partial
}{\partial x_1}).
 \end{array}
\right.
\end{equation}
\begin{equation}\label{psi-theta-u}
 \Delta \psi = \theta, \;\;\; \text{and} \;\; u= \nabla^\perp \psi, 
\end{equation}}Let us denote by $ A_1$ either the Laplacian on $O= \mathbb{T}^2$ or the Laplacian  
with Dirichlet boundary condition on $ O \subset \mathbb{R}^2 $,  then we formulate the recuperation  problem 
for both cases as 
\begin{equation}\label{recuper-eq-u-theta}
\left\{
\begin{array}{lr}
 A_1 \psi = \theta,\\
u= \nabla^\perp \psi.
\end{array}
\right.
\end{equation}
Problem \eqref{recuper-eq-u-theta} is well posed, see also Section \ref{sec-formulation}. Recall that for  $ O=\mathbb{T}^2$, the wellposdness is guaranteed   
thanks to the vanishing average condition for the torus.\\ The velocity $ u $ is obtained by a direct calculus,\del{ using the 
second equation in \eqref{recuper-eq-u-theta-Torus}. We get}
\begin{equation}\label{repres-u-theta}
 u(t, x) = \nabla^\perp A_1^{-1}\theta(t, x)= \int_O \nabla^\perp_xg_O(x, y)\theta(t, y)dy,
\end{equation}
where $ g_O$ is the Green function corresponding to the Poisson equation with Dirichlet boundary conditions for $ O$ bounded. 
In the case $  O= \mathbb{T}^2 $, the Green function  $g_{\mathbb{T}^2}$ is explicitly given by
\begin{equation}
 g_{\mathbb{T}^2}(x, y):= -\frac1{(2\pi)^2}\sum_{k\in\mathbb{Z}^2_0}\frac{1}{|k|^2}e^{k\cdot(x-y)},\;\; x, y \in \mathbb{T}^2.
\end{equation}
Moreover,
\begin{lem}\label{lem-R1-bounded}
The operator 
\del{\begin{eqnarray}
\mathcal{R}^{1}:H^{s, q}(\mathbb{T}^2) &\rightarrow & H_2^{s, q}(\mathbb{T}^2)\nonumber\\
\theta &\mapsto& \mathcal{R}^{1}\theta:= \nabla^\perp\cdot \Delta^{-1}\theta
\end{eqnarray}}
\begin{eqnarray}
\mathcal{R}^{1}:H_1^{\beta, q}(O) &\rightarrow & H_2^{\beta+1, q}(O)\nonumber\\
\theta &\mapsto& \mathcal{R}^{1}\theta:= u= \nabla^\perp\cdot A_1^{-1}\theta = \int_O \nabla^\perp_{\cdot}g_{O}(\cdot, y)\theta(y)dy
\end{eqnarray}
is well defined and bounded for all $ 1<q<\infty$ and $ \beta\in \mathbb{R}$. 
\end{lem}
\begin{proof}
In fact,
\begin{eqnarray}
\del{ |u|_{\mathbb{H}^{s+1, q}}&\leq& }|u|_{H_2^{\beta+1, q}} &\leq& c | \nabla^\perp A_1^{-1}\theta|_{H_2^{\beta+1, q}}
\leq c|\partial_j A_1^{-1}\theta|_{H_1^{\beta+1, q}} \leq\del{ c|A_1^{-1}\theta|_{H^{\beta+2, q}}\leq} c|\theta|_{H^{\beta, q}}.
\end{eqnarray}
One can also use the representation of $ u$ via Green function (recall that $ \Delta_xg_O(x, y)= \delta_x(y)$). For 
$ O=\mathbb{T}^2$, it is also convenient to remark that $ \mathcal{R}^{1} $ is a 
pseudo-differential operator of Calderon-Zygmund Reisz type, see 
definitions, results and further discussions in \cite{Debbi-scalar-active}. In deed, we can rewrite  
$\mathcal{R}^{1} = -\mathcal{R}^\perp (-\Delta)^{-\frac12}$, where $ \mathcal{R}$ is Riesz transform and $ \mathcal{R}^\perp := (-\mathcal{R}_2, \mathcal{R}_1)$. 
The proof of the above statement 
for a larger class of operators on  $\mathbb{T}^d, d\in \mathbb{N}_0$ which includes 
 $ \mathcal{R}^{1}$  can be found in \cite{Debbi-scalar-active}.\del{ For the case $ O=\mathbb{R}^2$ see e.g. \cite[Lemma 3.1]{Kato-Ponce-86}.}
\end{proof}

The following result characterizes an intrinsic property between $ curlv$ and the Sobolev regularity of $ v$.\del{ Recall that $ \nabla v$ is a matrix and we use the following notation for the norm matrix
$ | \nabla v|_{H_{2\times d}^{s, q}}:= | \partial_j v_i|_{H^{s, q}}$.}
\del{It is an easy consequence of the above discussion.} 
\begin{lem}\label{lem-basic-curl-gradient}
 Let  $O \subset \mathbb{R}^2$ bounded or $O = \mathbb{T}^2$, $ 1<q<\infty$ and $ \beta\in \mathbb{R}$. 
Then there exists a constant $ c>0$ such that for all $ v \in H_2^{\beta+1, q}(O)$
\begin{equation}
 c|\nabla v|_{H^{\beta, q}}\leq |curl v|_{H^{\beta, q}}\leq |\nabla v|_{H^{\beta, q}}.
\end{equation}
\del{Moreover, for $O \subset \mathbb{R}^2$ bounded or $O = \mathbb{T}^2$, $ 1<q<\infty$ and $ \beta\in \mathbb{R}$, there exists a constant $ c>0$ such that for all $ v \in \mathbb{H}_d^{\beta+1, q}(O)$
\begin{equation}
 |\nabla v|_{H_{2\times d}^{\beta, q}}\leq c |curl v|_{H^{\beta, q}}.
\end{equation}}
\end{lem}
\begin{proof}
Using the definition of the curl operator, the sobolev spaces and Lemma 
\ref{lem-R1-bounded}, we infer that there exists $ c>0$ such that 
\begin{eqnarray}
|\nabla v|_{H^{\beta, q}}&\leq & |\partial_j v_i|_{H^{\beta, q}}
\leq c|v|_{H_{2}^{\beta+1, q}}\leq c|curl v|_{H^{\beta, q}}.
\end{eqnarray}
Moreover, we have, 
\begin{eqnarray}
|curl v|_{H^{s, q}} &\leq & |\partial_j v_i|_{H^{\beta, q}}
\leq |\nabla v|_{H^{\beta, q}}.
\end{eqnarray}
\end{proof}
\begin{remark}
If we assume that $ v $ is of divergence free, i.e. $ v \in \mathbb{H}^{\beta, q}(\mathbb{T}^d)$, $d\in \mathbb{N}_1$, $O = \mathbb{T}^2$, $ 1<q<\infty$ and $ \beta \in \mathbb{R}_+$, then it is easy to adapt the proof of  \cite[Lemma 3.1]{Kato-Ponce-86}. 
\end{remark}

 \del{Now, we turn to the stochastic term. Let $(u(t), t \in [0, T])$ be a solution of FSNSE satisfying 
 \eqref{cond-solu-torus-H1}.  We denote by
\begin{equation}\label{eq-def-sigma-k}
\sigma^k(u):= G(u)Q^{\frac12}e_k= q_{k}^\frac12G(u)e_k,  for\;\;  k\in\Sigma.
\end{equation}
First, we claim that for  $P-a.s.$ the following stochastic integral $ \int_0^t  \sum_{k\in\Sigma}curl \sigma^k(u(s))d\beta_k(s)$ 
is well defined and
\begin{eqnarray}
 curl \int_0^tG(u(s))dW(s)= \int_0^t  \sum_{k\in\Sigma} curl \sigma^k(u(s))d\beta_k(s), \forall  t \in [o, T].
\end{eqnarray}

\noindent In fact, using the stochastic isometry,\del{ Lemma \ref{lem-basic-curl-gradient},} Assumption $ (\mathcal{C})$
and \eqref{cond-solu-torus-H1} we infer that for $ \beta$ either equals $ 1$, or $0$.
\del{\begin{eqnarray}\label{eq-def-sima-integ}
 \mathbb{E}\!\!\!\!&|&\!\!\!\!\int_0^t \sum_{k\in\Sigma} curl\sigma^k(u(s))d\beta_k(s)|_{L^2}^2\leq
\mathbb{E}\int_0^t\sum_{k\in\Sigma} |curl \sigma^k(u(s))|_{L^2}^2ds \nonumber\\
&\leq&
c\mathbb{E}\int_0^t\sum_{k\in\Sigma}|\partial_j \sigma_i^k(u(s))|_{L^2}^2ds
\del{\leq
c\mathbb{E}\int_0^t |G(u(s))Q^\frac12|_{HS(H^{1,2}) }^2ds \nonumber\\}
\leq c\mathbb{E}\int_0^t\sum_{k\in\Sigma}  
|\sigma^k(u(s))|_{H^{1,2}}^2ds\nonumber\\
&\leq&
c\mathbb{E}\int_0^t |G(u(s))|_{L_Q(H^{1,2}) }^2ds 
\leq
c\int_0^t(1+ \mathbb{E}|u(s)|_{\mathbb{H}^{1, 2}}^2)ds <\infty.
\end{eqnarray}}
\begin{eqnarray}\label{eq-def-sima-integ}
 \mathbb{E}\!\!\!\!&|&\!\!\!\!\int_0^t \sum_{k\in\Sigma} curl\sigma^k(u(s))d\beta_k(s)|_{H^{\beta-1, 2}}^2\leq
\mathbb{E}\int_0^t\sum_{k\in\Sigma} |curl \sigma^k(u(s))|_{H^{\beta-1, 2}}^2ds \nonumber\\
&\leq&
c\mathbb{E}\int_0^t\sum_{k\in\Sigma}|\partial_j \sigma_i^k(u(s))|_{H^{\beta-1, 2}}^2ds
\del{\leq
c\mathbb{E}\int_0^t |G(u(s))Q^\frac12|_{HS(H^{1,2}) }^2ds \nonumber\\}
\leq c\mathbb{E}\int_0^t\sum_{k\in\Sigma}  
|\sigma^k(u(s))|_{H^{\beta,2}}^2ds\nonumber\\
&\leq&
c\mathbb{E}\int_0^t |G(u(s))|_{L_Q(H^{\beta,2}) }^2ds 
\leq
c\int_0^t(1+ \mathbb{E}|u(s)|_{\mathbb{H}^{\beta, 2}}^2)ds <\infty.
\end{eqnarray}
\noindent Moreover, thanks to  \eqref{eq-W-n} and \eqref{eq-def-sigma-k}, we infer on one hand that 
\begin{eqnarray}
curl \sum_{k\in\Sigma_n} \int_0^t  \sigma^k(u(s))d\beta_k(s) \rightarrow
curl  \int_0^t \sum_{k\in\Sigma_n} \sigma^k(u(s))d\beta_k(s),\,\,  in \, L^2(\Omega, \mathcal{D}'(O)),  \nonumber\\
\end{eqnarray}
where  $ (\Sigma_n)_n$ is a sequence of subsets converging to $\Sigma$ and $ \mathcal{D}'(O)$ is the dual of $ \mathcal{D}(O)$. 
On the other hand, using the linearity of the operator $ curl$, the stochastic isometry identity, 
\eqref{eq-W-n},  Assumption $ (\mathcal{C})$, \eqref{cond-solu-torus-H1} and \eqref{eq-def-sima-integ}, we end up with
\begin{eqnarray}
curl \sum_{k\in\Sigma_n} \int_0^t  \sigma^k(u(s))d\beta_k(s) &=&\sum_{k\in\Sigma_n} \int_0^t curl \sigma^k(u(s))d\beta_k(s)\nonumber\\
&{}& \rightarrow
 \int_0^t \sum_{k\in\Sigma}curl  \sigma^k(u(s))d\beta_k(s),\,\,  in \, L^2(\Omega, \mathcal{D}'(O)). 
\end{eqnarray}
The uniqueness of the limit confirm the result. We use the following notation for the stochastic integral
\begin{eqnarray}\label{inte-g-tilde}
\int_0^t \tilde{G}(\theta(s))dW(s) := \sum_{k\in\Sigma} \int_0^t curl \sigma_k(u(s))d\beta_k(s)
& =& \sum_{k\in\Sigma} \int_0^t curl \sigma_k(\mathcal{R}^{1}(\theta(s))d\beta_k(s).\nonumber\\
\end{eqnarray}
The same calculus above is still valid for Orstein-Uhlenbeck stochastic process.\\
\noindent Using the definition of Helmholtz projection, in particular, the fact that $ \mathcal{Y}_q \subset Ker(Curl)$, we prove 
\begin{equation}
 curl B(u) = u\cdot \nabla \theta.
\end{equation}}
Now, we derive the stochastic vorticity equation. Let $(u, \tau)$ be a maximal weak solution of FSNSE satisfying \eqref{cond-solu-torus-H1}, 
 up to the stopping time $ \tau $.  
First, we claim that for  $P-a.s.$ the following stochastic integral 
$ \int_0^{t\wedge \tau}  \sum_{k\in\Sigma}curl \sigma^k(u(s))d\beta_k(s)$, with  $\sigma^k(u(s))$ given by
\eqref{eq-def-sigma-k}, is well defined and
\begin{eqnarray}
 \int_0^{t\wedge \tau} \sum_{k\in\Sigma} curl \sigma^k(u(s))d
 \beta_k(s)= curl \int_0^{t\wedge \tau}G(u(s))dW(s),\; \forall  t \in [o, T].
\end{eqnarray}

\noindent In fact, using the stochastic isometry,\del{ Lemma \ref{lem-basic-curl-gradient}} Assumption $ (\mathcal{C})$ (\eqref{Eq-Cond-Linear-Q-G}, with $ 2\leq q<\infty$ and $ \delta \in\{0, 1\}$)
and \eqref{cond-solu-torus-H1}, we infer that for $ \beta$  equals either $ 1$ or $0$,
\begin{eqnarray}\label{eq-def-sima-integ}
 \mathbb{E}\!\!\!\!&|&\!\!\!\!\int_0^{t\wedge \tau} \sum_{k\in\Sigma} curl\sigma^k(u(s))d\beta_k(s)|_{H^{\beta-1, q}}^2\leq
\del{\mathbb{E}\int_0^t\sum_{k\in\Sigma} |curl \sigma^k(u(s))|_{H^{\beta-1, q}}^2ds \nonumber\\
&\leq&}
c\mathbb{E}\int_0^{t\wedge \tau}\sum_{k\in\Sigma}|\partial_j \sigma_i^k(u(s))|_{H^{\beta-1, q}}^2ds\nonumber\\
\del{\leq
c\mathbb{E}\int_0^t |G(u(s))Q^\frac12|_{HS(H^{1,2}) }^2ds \nonumber\\}
&\leq& c\mathbb{E}\int_0^{t\wedge \tau}\sum_{k\in\Sigma}  
|\sigma^k(u(s))|_{H^{\beta, q}}^2ds
\leq
c\mathbb{E}\int_0^{t\wedge \tau} |G(u(s))|_{R_Q(H^{\beta, q}) }^2ds \nonumber\\
&\leq&
c\int_0^{t\wedge \tau}(1+ \mathbb{E}|u(s)|_{\mathbb{H}^{\beta, q}}^2)ds <\infty.
\end{eqnarray}
\noindent Moreover, thanks to  \eqref{eq-W-n} and \eqref{eq-def-sigma-k}, we infer on one hand that 
\begin{eqnarray}
curl \sum_{k\in\Sigma_n} \int_0^{t\wedge \tau}  \sigma^k(u(s))d\beta_k(s) \rightarrow
curl  \int_0^{t\wedge \tau} \sum_{k\in\Sigma} \sigma^k(u(s))d\beta_k(s),\,\,  in \, L^2(\Omega, \mathcal{D}'(O)),  \nonumber\\
\end{eqnarray}
where  $ (\Sigma_n)_n$ is a sequence of subsets converging to $\Sigma$ and $ \mathcal{D}'(O)$ is the dual of $ \mathcal{D}(O)$. 
On the other hand, using the linearity of the operator $ curl$, the stochastic isometry identity, 
\eqref{eq-W-n},  Assumption $ (\mathcal{C})$, \eqref{cond-solu-torus-H1} and \eqref{eq-def-sima-integ}, we end up with
\begin{eqnarray}
curl \sum_{k\in\Sigma_n} \int_0^{t\wedge \tau}  \sigma^k(u(s))d\beta_k(s) &=&\sum_{k\in\Sigma_n} \int_0^{t\wedge \tau} 
curl \sigma^k(u(s))d\beta_k(s)\nonumber\\
&{}& \rightarrow
 \int_0^{t\wedge \tau} \sum_{k\in\Sigma}curl  \sigma^k(u(s))d\beta_k(s),\,\,  in \, L^2(\Omega, \mathcal{D}'(O)). 
\end{eqnarray}
The uniqueness of the limit confirm the result. We use the following notation\del{ for the stochastic integral
\begin{eqnarray}\label{inte-g-tilde}
\int_0^{t\wedge \tau} \tilde{G}(\theta(s))dW(s) := \sum_{k\in\Sigma} \int_0^{t\wedge \tau} curl \sigma^k(u(s))d\beta_k(s)
& =& \sum_{k\in\Sigma} \int_0^{t\wedge \tau} curl \sigma^k(\mathcal{R}^{1}(\theta(s))d\beta_k(s).\nonumber\\
\end{eqnarray}}
\begin{eqnarray}\label{inte-g-tilde}
\tilde{G}(\theta):= curl G(\mathcal{R}^{1}(\theta)).
\end{eqnarray}
\noindent Using the definition of Helmholtz projection, in particular, the fact that $ \mathcal{Y}_q \subset Ker(Curl)$, an elementary calculus
yields to 
\begin{equation}
 curl B(u) = u\cdot \nabla \theta.
\end{equation}

 Now, we  assume that $ O= \mathbb{T}^2$,  using Fourier transform, it is easy to prove that 
\begin{equation}
 curl A_\alpha u = (-\Delta)^{\frac\alpha2}curl u, \; \; \forall u\in D(A_\alpha).
\end{equation}
In fact the relation above is also true for all $ u\in \mathbb{H}^{\beta+\alpha}(\mathbb{T}^2),  \beta\in \mathbb{R}$.
Applying the operator curl on the integral representation of Equation \eqref{Main-stoch-eq} stopped at the  stopping time $ \tau$
 and using the calculus above, we infer that if $ (u, \tau)$ is a local weak solution of \eqref{Main-stoch-eq}, then
 $ \theta:= curl u$ is a weak (strong in probability) solution of 
\begin{equation}\label{Eq-vorticity-Torus-2-diff}
\left\{
\begin{array}{lr}
 d\theta(t)= \left(- A_{\alpha}\theta(t) + u(t)\cdot \nabla \theta(t)\right)dt+ \tilde{G}(\theta(t))dW(t), \; 0< t\leq \tau.\\
\theta(t) = curl u_0. 
\end{array}
\right.
\end{equation}
By the same way, we can  prove that if $ (u, \tau)$ is a local mild solution of \eqref{Main-stoch-eq}, then the same calculus above is still valid\del{ for Orstein-Uhlenbeck stochastic process and \del{$ u$ is solution of \eqref{Eq-Mild-Solution}}} and $ \theta:= curl u$ is a mild solution 
to equation \eqref{Eq-vorticity-Torus-2-diff}. In the proof of this case, we use the commutativity property between the operators $ \partial_j$ and the 
semigroup $ (e^{-tA_\alpha})_{t\geq0}$. A general formula of Equation \eqref{Eq-vorticity-Torus-2-diff} has been studied in \cite{Debbi-scalar-active}.
\del{For the case $ O=\mathbb{R}^2$, we use the integral formula of the fractional Laplacian. } 

\del{For the case of bounded domain $ O\subset \mathbb{R}^2$, we prove, roughly speaking, a weak commutativity  between the the operator $ A_\alpha$
and the curl. It is easy to see that 
\begin{equation}
 \langle curl A^S u, \phi\langle =  \langle A^S curl u, \phi\langle, \;\;  \forall \phi \in C_0^\infty,
\end{equation}
with the notation  $ A^S$ stands for both 1D and 2D  Stokes operator.}

\del{the local solution $ (u, \tau)$ of \eqref{Eq-weak-Solution}. Using
the commutativity of $ A^\frac\alpha2$ and $ \partial_j, j=1,2,$, which is inherited from the commutativity of
the Laplacian $ A^S= -\Delta$ and $ \partial_j, j=1,2,$}
\del{satisfies in $ V^*$
\begin{equation}\label{Eq-vorticity-Torus-2}
\theta(t)= curl u_0 + \int_0^t\left(-A_{\alpha}\theta(s) + u(s)\cdot \nabla \theta(s)\right)ds+ \int_0^t \tilde{G}(\theta(s))dW(s), \; 0< t\leq T.\\
\end{equation}
Or equivalently, $ \theta $ is a weak (strong in probability) solution\del{ (in distribution sense)} of
\begin{equation}\label{Eq-vorticity-Torus-2-diff}
\left\{
\begin{array}{lr}
 d\theta(t)= \left(- A_{\alpha}\theta(t) + u(t)\cdot \nabla \theta(t)\right)dt+ \tilde{G}(\theta(t))dW(t), \; 0< t\leq T.\\
\theta(t) = curl u_0. 
\end{array}
\right.
\end{equation}}

\del{\section{Stopping times.}\label{append-stop-time}
\del{In this paper, we have defined in many places random times and claimed that they are stopping time. Some of them are easily seen that 
Here, we give the proof. Recall that  Remark \ref{Rem-1} confirms that the mild and weak solutions are weakly continuous.

\begin{lem}
Let $(Y(t), t\in [0, T])$ be a weakly continuous process.  We define the following random time
\begin{equation}\label{Eq-def-tau-n-delta-Y}
 \tau_n := \inf\{t\in (0, T), \; s.t. |Y(t)|_{D(A_q^{\frac\beta2})}\geq n\}\wedge T,
\end{equation}
with the understanding that $ \inf(\emptyset)=+\infty$. Then that $ \tau_n $ is a stopping time. 
\end{lem}

\begin{proof}
\del{Indeed, it is easy to check,  using Lemma \ref{lem-est-z-t}, Estimation \eqref{Eq-Pi-n-X} and a similar calculus as in \eqref{est-1-semi-group-second-term-1}, that the real  
$ \mathcal{F}_t-$adapted stochastic processes $ (\langle u_n(s), e_k\rangle_{\mathbb{L}^2})_{k\in\Sigma}$ are continuous for all $ k\in \Sigma$, see also Remark \ref{Rem-1}.}

Thanks to the weak continuity of $ Y$, the real $ \mathcal{F}_t-$adapted stochastic processes $ (\langle u_n(s), e_k\rangle_{\mathbb{L}^2})_{k\in\Sigma}$ are continuous for all $ k\in \Sigma$. Therefore, \del{the positive $ \mathcal{F}_t-$adapted stochastic real processes
\begin{equation}
X_j(\omega, s):= |\sum_{k\in \Sigma_j}|k|^{\delta} \langle u_n(\omega, s), e_k\rangle_{\mathbb{L}^2}e_k|_{\mathbb{L}^q},
\end{equation} }the positive $ \mathcal{F}_t-$adapted stochastic real processes
\begin{equation}
X_j(\omega, s):= |\sum_{k\in \Sigma_j}|k|^{\delta} \langle u_n(\omega, s), e_k\rangle_{\mathbb{L}^2}e_k|_{\mathbb{L}^q},
\end{equation} 
are also continuous for all $ j\in \mathbb{N}_1$, with $ (\Sigma_j)_j\subset \Sigma$  being an increasing sequence of finite subsets converging to $ \Sigma$. Therefore, see e.g. \cite[Propositions 4.5 or 4.6 ]{Revuz-Yor}, (Recall that we have assumed that the filtration $ (\mathcal{F}_t)_{t\in[0, T]}$ is right continuous), \del{therefore an optional time is also a stopping time.}  
\del{\begin{eqnarray}
\{\tau_n> t\} &=& \cap_{s\leq t}\{\omega, |u_n(s, \omega)|_{D(A_q^\frac\delta2)}\leq n\}\nonumber\\
&=& \cap_{s\leq t}\{\omega, |\sum_{k\in \Sigma}|k|^{\delta} \langle u_n(\omega, s), e_k\rangle_{\mathbb{L}^2}e_k|_{\mathbb{L}^q}\leq n\}
\nonumber\\
&=& \cap_{s\leq t}\cup_{m\in\mathbb{N}}\cap_{j\geq m}\{\omega, |\sum_{k\in \Sigma_j}|k|^{\delta} \langle u_n(\omega, s), e_k\rangle_{\mathbb{L}^2}e_k|_{\mathbb{L}^q}\leq n\}.
\end{eqnarray}}
\begin{eqnarray}
\{\tau_n\leq t\} &=& \cup_{s\leq t}\{\omega, |u_n(s, \omega)|_{D(A_q^\frac\delta2)}\geq n\}\nonumber\\
&=& \cup_{s\leq t}\{\omega, |\sum_{k\in \Sigma}|k|^{\delta} \langle u_n(\omega, s), e_k\rangle_{\mathbb{L}^2}e_k|_{\mathbb{L}^q}\geq n\}
\nonumber\\
&=& \cup_{s\leq t}\cap_{m\in\mathbb{N}}\cup_{j\geq m}\{\omega, |\sum_{k\in \Sigma_j}|k|^{\delta} \langle u_n(\omega, s), e_k\rangle_{\mathbb{L}^2}e_k|_{\mathbb{L}^q}\geq n\}
\nonumber\\
&=& \cup_{s\leq t}\cap_{m\in\mathbb{N}}\cup_{j\geq m}\{\omega, X_j(\omega, s):= \sum_{k\in \Sigma_j}|k|^{\delta} \langle u_n(\omega, s), e_k\rangle_{\mathbb{L}^2}e_k \in (B_{\mathbb{L}^q}(0, n))^c\}
\nonumber\\
&=& \cup_{s\in \mathbb{Q} \&\leq t}\cap_{m\in\mathbb{N}}\cup_{j\geq m}\{\omega, X_j(\omega, s) \in (B_{\mathbb{L}^q}(0, n))^c\} \in \mathcal{F}_t.
\end{eqnarray}
\del{Thanks to the continuity of the positive $ \mathcal{F}_t-$adapted process $X_j(\omega, s):=  \sum_{k\in \Sigma_j}|k|^{\delta} \langle u_n(\omega, s), e_k\rangle_{\mathbb{L}^2}e_k|_{\mathbb{L}^q}$ 
As the filtration $\mathcal{F}_t$ is right continuous, it is sufficient to prove that for all $ t\in [0, T]$, $ \{\tau_n\geq t\}\in \mathcal{F}_t$.}

\end{proof}}
In this work, we have defined in several places hitting times and we have claimed that they are stopping times.  The proofs of some of them are easily deduced thanks to the continuity of the norm-process and the right continuity of the filtration. This is the case for example for the stopping times\del{ $\tau_N^i$} defined in the proof of the uniqueness of the solution in Section \ref{sec-Torus}.\del{ the stopping times
$ \tau_N^i: \inf\{t\leq T; |u^i(t)|_{\mathbb{L}^{2}}\geq N\}\wedge T, i=1, 2$} Other proofs are consequences of the right continuity of the filtration and the $X-$weak continuity of the process.  Recall that Remark \ref{Rem-1} confirms that the local mild and weak solutions are $X-$weakly continuous. This is the case for the stopping time defined by \eqref{Eq-def-tau-n-delta} in Section \ref{appendix-local-solution}. Recall that this  stopping time is used implicitly in the proof of the local mild solution in Section \ref{sec-1-approx-local-solution}. Other proofs are more sophisticated as is the case for  $ \xi_N$ defined by \eqref{eq-stop-time-weak-solu-L2} in  Section \ref{sec-Domain}. In deed, in this case, the weak continuity concerns the $\mathbb{L}^2-$norm and the hitting time is defined via the $ \mathbb{H}^{\frac{d+2-\alpha}{4}, 2}-$norm. Here, we give the proof for this case.\del{ For simplicity, we consider the case $ q=2$. This is a direct proof of}   

\begin{lem}
Let $(Y(t), t\in [0, T])$ be an $\mathbb{L}^2-$weakly continuous  $ \mathcal{F}_t-$adapted $\mathbb{L}^2-$valued process.  We define the following random time
\begin{equation}\label{Eq-def-tau-n-delta-Y}
 \tau_n := \inf\{t\in (0, T), \; s.t.\;\;  |Y(t)|_{D(A_2^{\frac\delta2})}> n\}\wedge T,
\end{equation}
with the understanding that $ \inf(\emptyset)=+\infty$. Then $ \tau_n $ is a stopping time. 
\end{lem}

\begin{proof}
\del{Indeed, it is easy to check,  using Lemma \ref{lem-est-z-t}, Estimation \eqref{Eq-Pi-n-X} and a similar calculus as in \eqref{est-1-semi-group-second-term-1}, that the real  
$ \mathcal{F}_t-$adapted stochastic processes $ (\langle u_n(s), e_k\rangle_{\mathbb{L}^2})_{k\in\Sigma}$ are continuous for all $ k\in \Sigma$, see also Remark \ref{Rem-1}.}

Thanks to the weak continuity of $ Y$,\del{ the real $ \mathcal{F}_t-$adapted stochastic processes $ (\langle u_n(s), e_k\rangle_{\mathbb{L}^2})_{k\in\Sigma}$ are continuous for all $ k\in \Sigma$. Therefore, \del{the positive $ \mathcal{F}_t-$adapted stochastic real processes
\begin{equation}
X_j(\omega, s):= \sum_{k\in \Sigma_j}\lambda_k^{\beta} \langle u_n(\omega, s), e_k\rangle_{\mathbb{L}^2}^2,
\end{equation} }} the positive $ \mathcal{F}_t-$adapted stochastic processes
\begin{equation}
X_j(\omega, s):= \sum_{k\in \Sigma_j}\lambda_k^{\delta} \langle Y(\omega, s), e_k\rangle_{\mathbb{L}^2}^2
\end{equation} 
are continuous for all $ j\in \mathbb{N}_1$, with $ (\Sigma_j)_j\subset \Sigma$  being an increasing sequence of finite subsets converging to $ \Sigma$. Therefore,  
\begin{eqnarray}
\{\tau_n> t\} &=& \cap_{s\leq t}\{\omega, |Y(s, \omega)|_{D(A_2^\frac\delta2)}\leq n\}
\del{= \cap_{s\leq t}\{\omega, \sum_{k\in \Sigma}|k|^{2\beta} \langle Y(\omega, s), e_k\rangle_{\mathbb{L}^2}^2< n^2\}
\nonumber\\
=\cap_{j}\cap_{s\leq t}\{\omega, \sum_{k\in \Sigma_j}|k|^{2\beta} \langle Y(\omega, s), e_k\rangle_{\mathbb{L}^2}^2<\leq n^2\}.\nonumber\\}
=\cap_{j}\cap_{s\leq t}\{\omega, X_j(\omega, s)< n^2\}
= \cap_{j}\{ T_n^j>t\},\nonumber\\
\del{&=& \cap_{s\leq t}\cap_{j\geq m}\{\omega, |\sum_{k\in \Sigma_j}|k|^{2\beta} \langle u_n(\omega, s), e_k\rangle_{\mathbb{L}^2}^2\leq n\}.}
\end{eqnarray}
\del{\begin{eqnarray}
\{\tau_n\leq t\} &=& \cup_{s\leq t}\{\omega, |Y(s, \omega)|_{D(A^\frac\beta2)}\geq n\}\nonumber\\
&=& \cup_{s\leq t}\{\omega, |\sum_{k\in \Sigma}|k|^{\delta} \langle u_n(\omega, s), e_k\rangle_{\mathbb{L}^2}e_k|_{\mathbb{L}^q}\geq n\}
\nonumber\\
&=& \cup_{s\leq t}\cap_{m\in\mathbb{N}}\cup_{j\geq m}\{\omega, |\sum_{k\in \Sigma_j}|k|^{\delta} \langle u_n(\omega, s), e_k\rangle_{\mathbb{L}^2}e_k|_{\mathbb{L}^q}\geq n\}
\nonumber\\
&=& \cup_{s\leq t}\cap_{m\in\mathbb{N}}\cup_{j\geq m}\{\omega, X_j(\omega, s):= \sum_{k\in \Sigma_j}|k|^{\delta} \langle u_n(\omega, s), e_k\rangle_{\mathbb{L}^2}e_k \in (B_{\mathbb{L}^q}(0, n))^c\}
\nonumber\\
&=& \cup_{s\in \mathbb{Q} \&\leq t}\cap_{m\in\mathbb{N}}\cup_{j\geq m}\{\omega, X_j(\omega, s) \in (B_{\mathbb{L}^q}(0, n))^c\} \in \mathcal{F}_t.
\end{eqnarray}}
where\del{$ T_n^j:=\inf\{t, X_j \notin B_{\mathbb{R}_+}(0, n^2)\}$  (open ball)} $ T_n^j:=\inf\{t, X_j >n^2\}$. As $ X_j$ is an $ \mathcal{F}_t$-adapted and continuous and thanks to \cite[Proposition 4.6]{Revuz-Yor}, we conclude that $ T_n^j$ is an optional time with respect to the filteration $ \sigma(X_j(s), s\leq t)\subset \mathcal{F}_t$. Using the right continuity property of the filtration $\mathcal{F}_t$, see the assumption in Section \ref{sec-formulation}, we conclude that $ T_n^j$ is a stopping time with respect to this latter. Therefore, $ \{\tau_n\leq t\} = \{\tau_n> t\}^c \in \mathcal{F}_t$. 

\del{We can also use other tools like, \cite[Problems 2.6 \& 2.7]{Karatzas-Book},  \cite[Proposition 4.14]{Metivier-book-Mart-82} and \cite[Proposition 4.5]{Revuz-Yor}. In particualar, we can also use \cite[Theorem 1.6]{Metivier-book-Mart-82} to  prove that $X_j $, as an adapted continuous process, is progressively measurable therefore, thanks to \cite[Proposition 4.15]{Metivier-book-Mart-82} and to the right continuity the filtration $\mathcal{F}_t$, we confirm that for all $ n, j$ the random time $ T_n^j$ is a stopping time with respect to $\mathcal{F}_t$. }

\del{Thanks to the continuity of the trajectories of  $X_j$\del{$X_j(\omega, s):=  \sum_{k\in \Sigma_j}|k|^{\delta} \langle u_n(\omega, s), e_k\rangle_{\mathbb{L}^2}e_k|_{\mathbb{L}^q}$} 
As the filtration $\mathcal{F}_t$ is right continuous, it is sufficient to prove that for all $ t\in [0, T]$, $ \{\tau_n\geq t\}\in \mathcal{F}_t$.}
\end{proof} }

\del{\section{Complementary of the proof of the martingale solution.}\label{Appendix-Martingale-solu}
To prove the existence of a martingale solution, we use \del{consider the Gelfand triplet \eqref{Gelfand-triple-Domain} and 
 use lemmas \ref{lem-unif-bound-theta-n-H-1-domain} and \ref{lem-bounded-W-gamma-p} 
and }the following compact embedding, see \cite[Theorem 2.1]{Flandoli-Gatarek-95},
\begin{equation}
 W^{\gamma, 2}(0, T; \mathbb{H}^{-\delta', 2}(O))\cap \mathbb{L}^2(0, T; \mathbb{H}^{\frac\alpha2, 2}(O)) 
 \hookrightarrow L^2(0, T; \mathbb{L}^2(O)).
\end{equation}
\del{$ L^2(0, T; \mathbb{H}^{\frac\alpha2, 2}(O)) \hookrightarrow L^2(0, T; \mathbb{L}^2(O))$,} Therefore, we deduce that the sequence of laws $ (\mathcal{L}(u_n))_n$  is tight on $ L^2(0, T; \mathbb{L}^2(O))$. 
Thanks to Prokhorov's theorem there exists a  subsequence, still denoted $ (u_n)_n$, for which  the sequence of laws $ (\mathcal{L}(u_n))_n$ converges  weakly on $ L^2(0, T; \mathbb{L}^2(O))$  to a probability measure $ \mu$. By Skorokhod's embedding theorem, we can construct a probability basis
$ (\Omega^*, F^*, \mathbb{F}^*,  P^*)$  and a sequence of $ L^2(0, T; \mathbb{L}^2(O))\cap C([0, T]; \mathbb{H}^{-\delta', 2}(O))-$random variables
$ (u^*_n)_n$ and $ u^*$ such that  $\mathcal{L}(u^*_n) = \mathcal{L}(u_n), \forall n \in \mathbb{N}_0$,  $\mathcal{L}(u^*) = \mu$ and
$ u^*_n \rightarrow u^* a.s.$ in $ L^2(0, T; \mathbb{L}^2(O))\cap C([0, T]; \mathbb{H}^{-\delta', 2}(O))$. Moreover,  $ u^*_n(\cdot, \omega) \in C([0, T]; H_n)$. Thanks to  Lemma \ref{lem-unif-bound-theta-n-H-1-domain} and to the equality in law, we infer that
 for all $ n\in \mathbb{N}$,
\begin{eqnarray}\label{eq-bound-u-*-n-u-*}
\mathbb{E}\sup_{[0, T]}| u_n^*(s)|^p_{\mathbb{L}^2}+ \mathbb{E}\int_0^T| u_n^*(s)|^2_{ \mathbb{H}^{\frac\alpha2, 2}}ds \leq c<\infty.
\end{eqnarray}
Consequently, the sequence  $ u^*_n$ converges weakly in $  L^2(\Omega\times [0, T]; \mathbb{H}^{\frac\alpha2, 2}(O))$ to a limit $ u^{**}$. It is easy to see that $u^{*} = u^{**},  P\times dt a.e.$ and 
\begin{equation}
u^{*}(\cdot, \omega)\in L^2(0, T; \mathbb{H}^{\frac\alpha2, 2}(O))\cap L^\infty(0, T; \mathbb{L}^2(O)).
\end{equation} 
We introduce the filtration 
\begin{equation}
(\mathit{G}_n^*)_t:= \sigma\{u^{*}_n(s), s\leq t\}
\end{equation}
and construct with respect to the filtration $ (\mathit{G}_n^*)_t$ the time continuous square integrable martingale
$ (M_n(t), t\in [0, T])$ with trajectories in
$ C([0, T]; \mathbb{L}^2(O))$ by
\begin{equation}
M_n(t):= u_n^*(t) - P_nu_0+\int_0^t A_\alpha u_n^*(s) ds -\int_0^t P_nB(u_n^*(s))ds
\end{equation}
with the quadratic variation 
\begin{equation}
\langle\langle M_n\rangle\rangle_t= \int_0^tP_nG(u^*_n(s))QG(u^*_n(s))^*ds,
\end{equation}
where $ G(u^*_n(s))^*$ is the adjoint of $G(u^*_n(s))^*$. The proof yields as a consequence of the equality in law. The main task now is to prove that for $ a.s.$, $ M_n(t)$ converges weakly in $ \mathbb{H}^{-\delta', 2}(O)$ to a martingale $ M(t)$, for all $ t\in [0, T]$, where $ M(t)$ given by 
\begin{equation}\label{eq-M(t)}
M(t):= u^*(t) - u_0+\int_0^t A_\alpha u^*(s) ds -\int_0^t B(u^*(s))ds.
\end{equation}
In fact, for all $ v\in V_2:= \mathbb{H}^{\delta', 2}$ with $ \delta'>1+\frac d2$, we have $ P-a.s.$,  $ \langle P_nu_0, v\rangle $  converges to $ \langle u_0, v\rangle $, thanks to the fact that $ P_n \rightarrow I$ in $ \mathbb{L}^2(O)$, $ \langle u_n^*(t), v\rangle $  converges to $ \langle u^*(t), v\rangle $ as a consequence of the 
the a.s. convergence in $ L^2(0, T; \mathbb{L}^2(O))$ (in fact, we speak about the convergence of a subsequence but as usual we keep the same notation) and  the weak convergence and the continuity in $\mathbb{H}^{-\delta', 2}(O))$, the term $\int_0^t A_\alpha u_n^*(s) ds$
converges thanks to the weak convergence $ L^2(0, t; \mathbb{L}^2(O))$, for all $ t\in [0, T]$ and the elementry inequality $ \langle A_\alpha u_n^*(t), v\rangle = \langle u_n^*(t), A_\alpha v\rangle $ with $ v \in \mathbb{H}^{\delta', 2}(O)$ and  $\delta'>1+\frac d2>\alpha $. The convergence of $\int_0^t \langle B(u_n^*(s)), v\rangle ds$ is completely described in \cite[Appendix 2]{Flandoli-Gatarek-95}, in particular the condition $ \delta'>1+\frac d2$ implies that $ \partial_j v \in C^0(O)$, which we need to do the calculus. To prove that $ M(t)$ is  a quadratic martingale, we see that for all $ \phi \in C_b(L^2(0, s; \mathbb{L}^2(O)))$ and $ v\in \mathcal{D}(O)$
\begin{equation}
\mathbb{E}(\langle M(t)- M(s), v\rangle \phi(u^*|_{[0, s]}))= \lim_{n\rightarrow +\infty} \mathbb{E}(\langle M_n(t)- M_n(s), v\rangle \phi(u^*|_{[0, s]}))=0
\end{equation}
and 
\begin{eqnarray}
&{}&\mathbb{E}(\langle M(t), v\rangle \langle M(t), y\rangle - \langle M(s), v\rangle \langle M(s), y\rangle -\int_s^t\langle G^*(u^*(r)P_nv, G^*(u^*(r)P_ny \rangle dr)\phi(u^*|_{[0, r]}))\nonumber\\
&=& \lim_{n\rightarrow +\infty} \mathbb{E}(\langle M_n(t), v\rangle \langle M_n(t), y\rangle - \langle M_n(s), v\rangle \langle M_n(s), y\rangle -\int_s^t\langle G^*(u_n^*(r)P_nv, G^*(u_n^*(r)P_ny \rangle dr)\phi(u_n^*|_{[0, r]}))\nonumber\\
&=& 0
\end{eqnarray}
The main ingredeints are formula \eqref{eq-M(t)}, $ \partial_j\phi \in C^0 $ and therefore we can estimate $\int_0^t \langle B(u_n^*(s)), v\rangle ds$ by $\int_0^t|B(u_n^*(s))|_{L^1} |v|_{C^1} ds$.

\del{We follow the same steps as in \cite{Flandoli-Gatarek-95} (hence we omit here the details), we end up with the statement that $ u^*$ is a solution in $ V_2:= \mathbb{H}^{1+\frac d2, 2}(O)$ of Equation \eqref{Eq-weak-Solution} with $ W$ being replaced by $ W^*$. 

It is easy to see, thanks to Equation \eqref{FSBE-Galerkin-approxi} and to the equality in Law of $ u_n $ and $ u_n^*$,
that $ (M_n(t), t\in [0, T])$ is a square integrable martingale with respect to the filtration 
$(\mathcal{G}_n)_t:=\sigma\{u_n^*(s), s\leq t\}$. The remain part of the proof follows the same steps as in \cite{Flandoli-Gatarek-95},
hence we omit here.
\del{$ \mathbb{G}^* := (\mathcal{G}^*_t)_t$,  with $ \mathcal{G}^*_t $ is the $ \sigma-$algebra generated by
$ \cup_n (\mathcal{G}_n)_t:=\sigma\{u_n^*(s), s\leq t\}$.}}
}


\section{Some Sobolev inequalities.}\label{Appendix-Sobolev}
\del{\subsection{Sobolev pointwise multiplication}\label{Sobolev pointwise multiplication}
Let us recall the following classical result, see e.g. \cite[Theorem 1.4.6.1, p 190-191]{R&S-96}
\begin{theorem}\label{theor-1-SR}
 Let   $ s_1, s_2 \in \mathbb{R}$, satisfy  $ s_1\leq  s_2$ and  $ s_1+ s_2>0$. Then
\begin{itemize}
 \item (i) if $ s_2> \frac dq$, $ H^{s_1, 2}\cdot H^{s_2, 2} \hookrightarrow H^{s_1, 2}$.
\item (ii) if $ s_2< \frac dq$, $ H^{s_1, 2}\cdot H^{s_2, 2} \hookrightarrow H^{s_1+s_2-\frac dq, 2}$.
\end{itemize}
\end{theorem}}

\subsection{Sobolev pointwise multiplication on bounded sets}\label{Sobolev pointwise multiplication-Bounded-Domain}
Assume that $ O \subset \mathbb{R}^d$ is a bounded  $ C^\infty$ domain,(recall, domain means an open subset, see e.g.
\cite[5.2 p43]{Triebel-Structure-Function-vol97}). The notation
$A^s_{pq}(\mathbb{R}^d), s\in \mathbb{R}, 0<q\leq \infty, 0<p<\infty,$ stands either for Triebel-Lizorkin spaces $F^s_{pq}(\mathbb{R}^d) $ or for
Besov spaces $B^s_{pq}(\mathbb{R}^d)$, see
the definition in \cite[p.8]{R&S-96}. We know that, see e.g. \cite[Proposition Tr.6, 2.3.5, p 14]{R&S-96}, 
\begin{eqnarray}\label{main-relation-spaces}
F^s_{p2}(\mathbb{R}^d) &=& H^{s, p}(\mathbb{R}^d), \; 1<p<\infty, \; s\in \mathbb{R}, \nonumber\\
F^s_{pp}(\mathbb{R}^d) &=& B^s_{pp}(\mathbb{R}^d) =  W^{s, p}(\mathbb{R}^d), \; 1\leq p<\infty, \; 0<s\neq\; integer,
\end{eqnarray}
where $ H^{s, p}(\mathbb{R}^d)$ is the Bessel potential spaces or called also Sobolev spaces of fractional order and
$ W^{s, p}(\mathbb{R}^d), \; 1\leq p<\infty, \; 0<s\neq\; integer$  is Slobodeckij spaces. We define Triebel-Lizorkin and Besov spaces $ A^s_{pq}(O)$
on bounded sets by, see e.g. \cite[Definition 5.3. p 44]{Triebel-Structure-Function-vol97}
\begin{equation}
 A^s_{pq}(O)= \{ f \in D'(O); \;\; \text{there is a } g \in A^s_{pq}(\mathbb{R}^d),\; \text{with}\; g/O =f \; \text{in distribution sense}\},
\end{equation}
endowed with the norm
\begin{equation}
 |f|_{A^s_{pq}(O)}= \inf_{g\in A^s_{pq}(\mathbb{R}^d),\; g/O=f}|g |_{A^s_{pq}(\mathbb{R}^d)}.
\end{equation}
The relations in \eqref{main-relation-spaces} still also valid for
bounded sets, see e.g. \cite[5.8 p 52]{Triebel-Structure-Function-vol97}. Our main theorem is the following
\del{Let us recall the following classical result, see e.g. \cite[Theorem 1.4.6.1, p 190-191]{R&S-96}}
\begin{theorem}\label{Theo-pointwiseMulti-Bounded-Domain}
 Let\del{   $ s_1, s_2, s_3 \in \mathbb{R}$ and} $p, s, q, p_i, s_i, q_i, i=1,2$, such that the following pointwise multiplication is satisfied for
$ A^{s_i}_{p_iq_i}(\mathbb{R}^d)$
\begin{equation}\label{Eq-pointwise-R-d}
 |f_1f_2|_{A^{s}_{pq}}\leq c |f_1|_{A^{s_1}_{p_1q_1}} |f_2|_{A^{s_2}_{p_2q_2}}.
\end{equation}
Then Inequality \eqref{Eq-pointwise-R-d} is also valid for $ O \subset \mathbb{R}^d$ being a bounded open  $ C^\infty$ set.
\end{theorem}
\begin{proof}
Let $f_i\in A^{s_i}_{p_iq_i}(O)$, \del{and $g_i\in A^{s_i}_{p_iq_i}(\mathbb{R}^d)$, such that  $ g_i/O =f_i$}
then
\begin{equation}
 |f_1f_2|_{A^s_{pq}(O)}= \inf_{g\in A^s_{pq}(\mathbb{R}^d), g/O=(f_1f_2)}|g |_{A^s_{pq}(\mathbb{R}^d)}\leq
\inf_{g_i\in A^{s_i}_{p_iq_i}(\mathbb{R}^d), g_i/O=f_i} |g_1g_2|_{A^s_{pq}(\mathbb{R}^d)}.
\end{equation}
 Applying Estimate  \eqref{Eq-pointwise-R-d}, we infer that
\begin{equation}
 |f_1f_2|_{A^s_{pq}(O)}\leq c
\inf_{g_i\in A^{s_i}_{p_iq_i}(\mathbb{R}^d), g_i/O=f_i} (|g_1|_{A^{s_1}_{p_1q_1}(\mathbb{R}^d)}|g_2|_{A^{s_2}_{p_2q_2}(\mathbb{R}^d)})
\leq c|f_1|_{A^{s_1}_{p_1q_1}(O)}|f_2|_{A^{s_2}_{p_2q_2}(O)}
\end{equation}

\end{proof}

\subsection{Sobolev embedding }\label{lem-appendix-sobolev-embedding}
\begin{theorem} Let  $ O$ be  either the whole space $ \mathbb{R}^d$, or the torus $ \mathbb{T}^d$, or an arbitrary domain 
 $O\subset \mathbb{R}^d$.
 If $ t\leq s$ and $ 1<p\leq q\leq \frac{dp}{d-(s-t)p}<\infty$, then
\begin{equation}
 H^{s, p}(O) \hookrightarrow  H^{t, q}(O).
\end{equation}
\end{theorem}

\begin{proof}
For the proof see \cite[Theorem 7.63. p221 + 7.66 p222]{Adams-Hedberg-94}. For   $O= \mathbb{R}^d $ and $ q=\frac{dp}{d-sp}$, see \cite[Theorem 1, p 119 or Theorem 2 p 124]{Stein} and
\cite[Proposition 6.4. p 24]{Taylor-PDE-III}. For  $O= \mathbb{T}^d$, see e.g. \cite[pp 23-24]{Taylor-PDE-III}. 
\end{proof}
As a consequence, we have
\begin{equation}
 H^{\frac\alpha2, 2}(O) \hookrightarrow H^{\frac\alpha2-\frac d2+\frac dq, q}( O), \;\;\; \forall q \geq 2.
\end{equation}
See also the above result for  the Sobolev solenoidal spaces in \cite[Theorem 3.10]{Amann-solvability-NSE-2000}.

\del{\subsection{Sobolev embedding }\label{lem-appendix-sobolev-embedding}
\begin{theorem} We take  $ O$ to be  the whole space $ \mathbb{R}^d$, or the torus $ \mathbb{T}^d$,
or an arbitrary regular (see \cite{Adams-Hedberg-94} for the definition of the regularity) domain (open set)
 $O\subset \mathbb{R}^d$.
 If $ t\leq s$ and $ 1<p\leq q\leq \frac{dp}{d-(s-t)p}<\infty$, then
\begin{equation}
 H^{s, p}( O) \hookrightarrow  H^{t, q}(O).
\end{equation}
\end{theorem}

\begin{proof}
For the proof, in the case  $O= \mathbb{R}^d $,  see \cite[Theorem 7.63]{Adams-Hedberg-94}
and in the case $ O$ open regular domain, see \cite[Theorem 7.63. p221 + 7.66 p222]{Adams-Hedberg-94}.
Moreover, for   $O= \mathbb{R}^d $ and $ q=\frac{dp}{d-sp}$, see \cite[Theorem 1, p 119 or Theorem 2 p 124]{Stein} and
\cite[Proposition 6.4. p 24]{Taylor-PDE-III}. For  $O= O $, see e.g. \cite[pp 23-24]{Taylor-PDE-III}.
\end{proof}
As a consequence, we have
\begin{equation}
  H^{\frac\alpha2, 2}(O) \hookrightarrow H^{\frac\alpha2-\frac d2+\frac dq, q}( O), \;\;\; \forall q \geq 2.
\end{equation}}

\del{\eqref{B-u-v-h-alpha-2-d} with $ \eta =0$ (or \eqref{Eq-B-H-alpha-2-est})}}}


\subsection*{Acknowledgement} This work is supported by Leverhulme trust.


\end{document}